\documentclass[article, 12pt]{amsart}

\usepackage{amssymb,latexsym}
\usepackage{amsmath}
\usepackage[left=1in,top=1in,right=1in,bottom=1in]{geometry} 
\usepackage{fancyhdr,lastpage} 
\usepackage{color}
\usepackage{leftidx}
\usepackage{mathrsfs}
\usepackage{verbatim}
\usepackage{amsthm}
\usepackage{wasysym}
\usepackage{upgreek}
\usepackage{accents}
\usepackage{fancyhdr}
\usepackage{times}

\begin{document}

\sloppy
\renewcommand{\theequation}{\arabic{section}.\arabic{equation}}
\thinmuskip = 0.5\thinmuskip
\medmuskip = 0.5\medmuskip
\thickmuskip = 0.5\thickmuskip
\arraycolsep = 0.3\arraycolsep

\newtheorem{theorem}{Theorem}[section]
\newtheorem{lemma}[theorem]{Lemma}
\newtheorem{proposition}[theorem]{Proposition}
\newtheorem{corollary}[theorem]{Corollary}
\newtheorem{definition}{Definition}[section]
\newtheorem{remark}{Remark}[section]
\renewcommand{\thetheorem}{\arabic{section}.\arabic{theorem}}
\renewcommand{\thefootnote}{\alph{footnote}}
\newcommand{\prf}{\noindent{\bf Proof.}\ }
\def\prfe{\hspace*{\fill} $\square$


\smallskip \noindent}

\def\be{\begin{equation}}
\def\ee{\end{equation}}
\def\bea{\begin{eqnarray}}
\def\eea{\end{eqnarray}}
\def\beas{\begin{eqnarray*}}
\def\eeas{\end{eqnarray*}}

\newcommand{\R}{\mathbb R}
\newcommand{\T}{\mathbb T}
\newcommand{\K}{\mathbb S}
\newcommand{\leftexp}[2]{{\vphantom{#2}}^{#1}{#2}}
\newcommand{\eqdef}{\overset{\mbox{\tiny{def}}}{=}}

\newcommand{\gzerozeronorm}[1]{S_{(g_{00},\partial g_{00});#1}}
\newcommand{\gzerozerounorm}[1]{S_{\u(g_{00},\partial g_{00});#1}}
\newcommand{\gzerozeroellipticnorm}[1]{S_{\underpartial \partial g_{00};#1}^e}

\newcommand{\gzerostarnorm}[1]{S_{(g_{0*},\partial g_{0*});#1}}
\newcommand{\gzerostarunorm}[1]{S_{\u(g_{0*}, \partial g_{0*});#1}}
\newcommand{\gzerostarellipticnorm}[1]{S_{\underpartial \partial g_{0*};#1}^e}

\newcommand{\hstarstarnorm}[1]{S_{\partial h_{**};#1}}
\newcommand{\hstarstarunorm}[1]{S_{\u h_{**};#1}}
\newcommand{\partialhstarstarunorm}[1]{S_{\u \partial h_{**};#1}}
\newcommand{\combinedpartialupartialhstarstarnorm}[1]{S_{\u (h,\partial h_{**});#1}}
\newcommand{\hstarstarellipticnorm}[1]{S_{\underpartial \partial h_{**};#1}^e}

\newcommand{\totalellipticnorm}[1]{S_{\underpartial \partial g;#1}^e}

\newcommand{\gnorm}[1]{S_{(g,\partial g);#1}}
\newcommand{\gunorm}[1]{S_{\u(g,\partial g);#1}}

\newcommand{\totalnorm}{\mathbf{S}_{(Total)}}
\newcommand{\totalbelowtopnorm}[1]{S_{(g,\partial g,u^*,\varrho);N-1}}
\newcommand{\totalbelowtopunorm}[1]{S_{\u (g,\partial g, u^*, \varrho);N-1}}

\newcommand{\velocitynorm}[1]{S_{u^*;#1}}
\newcommand{\topvelocitynorm}[1]{S_{\underpartial u^*;#1}}
\newcommand{\velocityunorm}[1]{S_{\u u^*;#1}}
\newcommand{\densnorm}[1]{S_{\varrho;#1}}
\newcommand{\densunorm}[1]{S_{\u \varrho;#1}}
\newcommand{\fluidnorm}[1]{S_{(u,\varrho);#1}}
\newcommand{\fluidunorm}[1]{S_{\u (u^*,\varrho);#1}}

\newcommand{\newufluidenergy}[1]{\mathscr{E}_{u^*;#1}}
\newcommand{\gzerozeroenergy}[1]{E_{(g_{00},\partial g_{00});#1}}
\newcommand{\gzerostarenergy}[1]{E_{(g_{0*},\partial g_{0*});#1}}
\newcommand{\partialhstarstarenergy}[1]{E_{\partial h_{**};#1}}
\newcommand{\hstarstarenergy}[1]{E_{(Alternate);\underpartial h_{**};#1}}

\newcommand{\gzerozeropartialuenergy}[1]{E_{(\u g_{00}, \partial \u g_{00});#1}}
\newcommand{\gzerostarpartialuenergy}[1]{E_{(\u g_{0*},\partial \u g_{0*});#1}}
\newcommand{\partialhstarstarpartialuenergy}[1]{E_{\partial \u h_{**};#1}}
\newcommand{\hstarstarpartialuenergy}[1]{E_{\u h_{**};#1}}

\newcommand{\fluidenergy}[1]{E_{\text{fluid};#1}}
\newcommand{\totalenergy}{\mathbf{E}_{(Total)}}

\newcommand{\velocityenergy}[1]{E_{u^*;#1}}
\newcommand{\topordervelocityenergy}[1]{E_{\underpartial u^*;#1}}
\newcommand{\densenergy}[1]{E_{\dens;#1}}

\newcommand{\genergy}[1]{E_{(g,\partial g);#1}}
\newcommand{\guenergy}[1]{{E_{(\u g,\partial \u g);#1}}}

\newcommand{\totalbelowtopenergy}[1]{E_{(g,\partial g,u^*,\varrho);N-1}}
\newcommand{\totalbelowtopuenergy}[1]{E_{(\u g,\partial \u g, \u u^*, \u \varrho);N-1}}

\newcommand{\speed}{c_s}
\newcommand{\underpartial}{\underline{\partial}}

\newcommand{\todo}[1]{\vspace{5 mm}\par \noindent

\marginpar{\textsc{ \hspace{.2 in}   \textcolor{red}{ To Fix}}} \framebox{\begin{minipage}[c]{0.95
\textwidth} \tt #1 \end{minipage}}\vspace{5 mm}\par}

\def\dens{\varrho}
\def\g{\partial}
\def\u{\partial_{{\bf u}}}
\def\H{\mathcal{H}_{N-1}}
\def\A{\mathcal{A}}
\def\E{\mathcal{ E}}
\def\D{\mathcal{D}}
\def\F{\mathcal{F}}
\def\M{\mathcal{M}}
\def\N{\mathcal{N}}
\def\div{\mbox{div}}
\def\curl{\mbox{curl}}
\def\a{\partial_{\vec{\alpha}}}
\def\b{\partial_{\vec{\beta}}}
\def\open#1{\setbox0=\hbox{$#1$}
\baselineskip = 0pt
\vbox{\hbox{\hspace*{0.4 \wd0}\tiny $\circ$}\hbox{$#1$}}
\baselineskip = 11pt\!}

\numberwithin{equation}{section} 

\title{The Global Future Stability of the FLRW Solutions to the Dust-Einstein System with a Positive Cosmological Constant}
\author{Mahir Had\v zi\'c $^{*}$
\and 
Jared Speck$^{**}$}

\thanks{$^*$King's College London, Department of Mathematics, Strand, London, WC2R2LS, UK. 
\texttt{mahir.hadzic@kcl.ac.uk}}

\thanks{$^{**}$Massachusetts Institute of Technology, Department of Mathematics, 77 Massachusetts Ave, Room E18-328, Cambridge, MA 02139-4307, USA. \texttt{jspeck@math.mit.edu}}

\thanks{$^{*}$ MH gratefully acknowledges support from NSF grant \# DMS-1211517.}

\thanks{$^{**}$ JS gratefully acknowledges support from NSF grant \# DMS-1162211 
and from a Solomon Buchsbaum grant administered by the Massachusetts Institute of Technology.
}

\begin{abstract}
  We study small perturbations of the well-known family of
  Friedman-Lema\^{\i}tre-Robertson-Walker (FLRW) solutions to the dust-Einstein system with a positive cosmological constant
  in the case that the spacelike Cauchy hypersurfaces are diffeomorphic to $\mathbb{T}^3.$
  These solutions model a quiet pressureless fluid 
  in a dynamic spacetime undergoing accelerated expansion.
  We show that the FLRW solutions are nonlinearly globally future-stable under small perturbations of
	their initial data.
	Our analysis takes place relative to a harmonic-type coordinate system, in which 
	the cosmological constant results in the presence of dissipative terms in the evolution equations.
	Our result extends the results of ~\cite{iRjS2012,jS2012,cLjVK}, 
	where analogous results were proved for the Euler-Einstein system under the equations of state $p = c_s^2 \rho$, 
	$0<c_s^2 \leq 1/3$. 
	The dust-Einstein system is the Euler-Einstein system with $c_s=0.$ The main difficulty that we
	overcome is that the energy density of the dust loses one degree of differentiability
	compared to the cases $0 < c_s^2 \leq 1/3.$ Because the dust-Einstein equations are coupled, 
	this loss of differentiability introduces new obstacles for
	deriving estimates for the top-order derivatives of all solution variables.
	To resolve this difficulty, we commute the equations with a well-chosen 
	differential operator and derive a collection of elliptic estimates
	that complement the energy estimates of ~\cite{iRjS2012,jS2012}. 
	An important feature of our analysis is that we are able to close our estimates even though
	the top-order derivatives of all solution variables can grow much more rapidly 
	than in the cases $0<c_s^2 \leq 1/3.$
	Our results apply in particular to small compact perturbations of the vanishing dust state.

\bigskip


\bigskip


\end{abstract}

\setcounter{tocdepth}{1}

\maketitle

\tableofcontents

\section{Introduction}
A commonly used cosmological model for the evolution of a matter-containing universe is that of a fluid coupled to the 
Einstein field equations of general relativity (see Ch. 5 of~\cite{rW1984} or Ch. 10 of~\cite{sHgE1973}). 
Such systems of equations go under the general 
name of Euler-Einstein systems, and in this article, we study a particular case: the so-called
dust-Einstein system. Dust is a fluid in which the pressure is zero.
The dust model plays an important role in cosmology, where it is often
used to model the matter content of the universe starting in the so-called ``matter-dominated'' era.
The dynamic quantities in the dust-Einstein system 
are the spacetime manifold $\mathcal{M},$ the Lorentzian spacetime metric $g_{\mu \nu}$, 
the dust mass-energy density $\rho$, and the dust four-velocity $u^{\mu}.$
Relative to an arbitrary coordinate system, the dust-Einstein system can be expressed as\footnote{ Throughout the article, we use Greek letters to denote ``spacetime'' indices varying from 0 to 3, and Latin letters to denote ``spatial" indices varying between 1 and 3.}
\begin{subequations}
\label{E:dege}
\begin{align}
\text{Ric}_{\mu\nu}-\frac{1}{2}Rg_{\mu\nu}+\Lambda g_{\mu\nu}&=T_{\mu\nu}, && (\mu,\,\nu=\,0,1,2,3), \label{E:metricge}\\
D_{\alpha}T^{\alpha\mu}&=0, && \quad (\mu=0,1,2,3),\label{E:velocityge}\\
g_{\alpha \beta} u^{\alpha} u^{\beta}&=-1. &&\,\label{E:massshell1}
\end{align}
\end{subequations}
Above, $\text{Ric}_{\mu\nu}$ denotes the Ricci tensor of $g_{\mu \nu},$ 
$R$ is the scalar curvature of $g_{\mu \nu},$ 
$\Lambda>0$ is a fixed positive constant known as the {\em cosmological constant}, 
the energy momentum tensor $T_{\mu\nu}$ of the dust is given by
\begin{align} \label{E:energymomentumdust}
T_{\mu\nu} = \rho\, u_{\mu}u_{\nu}, \quad \ \ (\mu,\,\nu=\,0,1,2,3)\,,
\end{align}
and $D_{\mu}$ denotes the covariant derivative corresponding to $g_{\mu \nu}.$
Equations~\eqref{E:metricge} are the Einstein-field equations, while  
Eqs.~\eqref{E:velocityge}-\eqref{E:massshell1}
model the evolution of the pressureless dust. 
As we mentioned above, the system~(\ref{E:dege}) is a special case of a family of PDEs
known as Euler-Einstein systems; the family is parameterized by 
the choice of an equation of state, which is an equation that relates the fluid variables.
Euler-Einstein systems comprise Eqs. \eqref{E:metricge}-\eqref{E:massshell1},
but the dust energy-momentum tensor \eqref{E:energymomentumdust} is replaced with
the energy-momentum tensor of a perfect fluid:
\[
T_{\mu\nu}=(\rho+p)u_{\mu}u_{\nu}+pg_{\mu\nu}.
\]
Above, $p$ denotes the fluid pressure. The Euler-Einstein equations are not
closed because there are too many fluid variables. In order to close the equations, one can,
for example, prescribe a barotropic equation of state $p = f(\rho).$
The dust-Einstein system is a therefore a special case of the Euler-Einstein system in which 
the equation of state is $p = 0.$

The dust-Einstein system \eqref{E:dege}-\eqref{E:energymomentumdust} admits an important family of explicit
solutions known as the Friedman-Lema\^{\i}tre-Robertson-Walker solutions (from now on FLRW solutions). The FLRW solutions
serve as a model of a spatially homogeneous, isotropic,
dust-containing universe exhibiting accelerated expansion.
For a fixed constant $\Lambda>0$, the FLRW solutions are the quadruple $((-\infty,\infty)\times\T^3,\tilde{g}, \tilde{u},\tilde{\rho})$, where 
$\T^3=[-\pi,\pi]^3$ (with the endpoints identified) is the three-dimensional torus and 
\be\label{E:babyFLRW}
\tilde{g}=-dt^2+a^2(t)\sum_{i=1}^3(dx^i)^2, \ \ \tilde{u}=(1,0,0,0), \ \ \tilde{\rho}=a^{-3}(t)\bar{\dens},
\ \ a(t)\sim C e^{H t}.
\ee
We remark that corresponding FLRW solutions exist for many other spatial topologies besides $\mathbb{T}^3;$
for simplicity, we restrict our attention to the case of $\mathbb{T}^3.$
The scale factor $a(t)$ above captures the spatial expansion rate of the spacetime metric $\tilde{g}.$
Its asymptotic behavior is $a(t) \sim C e^{Ht}$ (see Lemma \ref{L:backgroundaoftestimate}).
The FLRW fluid is quiet, that is, $\tilde{u}=(1,0,0,0)$, $\bar{\dens} \geq 0$ is a non-negative constant, and
\begin{align}
H=\sqrt{\frac{\Lambda}{3}}
\end{align}
is known as the Hubble constant.
As we will explain in Sect.~\ref{S:backgroundsolution}, the FLRW solutions 
$((-\infty,\infty)\times\T^3,\tilde{g}, \tilde{u},\tilde{\rho})$
indeed solve \eqref{E:dege}, and they exhibit the following important accelerated expansion property:
\[
\frac{d^2}{dt^2}a(t)>0.
\]
\textbf{The main goal of this article is to show that the FLRW solutions \eqref{E:babyFLRW} are globally future-stable
solutions to the dust-Einstein system.} We now present an informal version of our main result, 
which is rigorously stated and proved in Theorem~\ref{T:maintheorem}.
\begin{theorem}[\textbf{Informal statement of the main result}]
Let $\Lambda>0$ be a fixed cosmological constant. 
The FLRW solutions $([0,\infty)\times\T^3,\tilde{g},\tilde{u},\tilde{\rho})$ of the dust-Einstein system 
given by~\eqref{E:babyFLRW} are globally future-stable. More precisely,
small perturbations of the data (given on $\mathbb{T}^3$) of the
FLRW solutions launch maximal globally hyperbolic developments that are
future geodesically complete and whose future halves are diffeomorphic to $[0,\infty)\times\T^3.$
\end{theorem}

\begin{remark} [\textbf{Other FLRW solutions}]
Under the most commonly studied equations of state in cosmology
\begin{align*}
	p = c_s^2 \rho, \quad \ \  0\leq c_s\leq 1,
\end{align*}
there exist associated FLRW solutions to the Euler-Einstein system.
The constant $c_s$ is the speed of sound and the case $c_s=0$ corresponds precisely to the dust-Einstein case.
It has been shown in~\cite{iRjS2012, jS2012,cLjVK} 
that the FLRW family is globally
future-stable when $0<c_s^2\le1/3$. The cases $c_s^2=0$ and $c_s^2=1/3$ are of particular importance 
in cosmology literature and are called the {\em dust} and {\em pure radiation} case respectively.
In the present article, we address the ``end-point'' case $c_s^2=0.$
\noindent
We also remark that in~\cite{aR2004b}, Rendall uncovered heuristic evidence suggesting instability when $\speed ^2> 1/3.$
The heuristics were based on formal series expansions.
\end{remark}

\begin{remark} [\textbf{Compactly supported data}]
	Unlike the future stability results derived in ~\cite{iRjS2012,jS2012,cLjVK} 
	in the cases $0<c_s^2 \leq 1/3,$ the results derived in the 
	present article apply to compactly supported fluid energy-density data.
	This is the one redeeming feature of the degenerate nature of the dust model.
\end{remark}

\begin{remark} [\textbf{Reduced differentiability for $\rho$}]
The main difficulty that we overcome is that the density $\rho$
in the dust model exhibits a ``loss of a derivative"-phenomenon in contrast to the
cases $0 < c_s^2 \leq 1/3.$ This inherent degeneracy introduces difficulties even for the local well-posedness theory;
we will discuss these difficulties below. In~\cite{yCBhF2006}, Choquet-Bruhat and Friedrich proved a local well-posedness result for the dust-Einstein system 
by showing that in wave coordinate gauge, the equations form
a Leray-hyperbolic system in the unknowns $g,\,\rho,\,u$. 
The important point is that the Leray method leads to the availability of a priori Sobolev estimates for the unknowns.
For an overview of the Leray approach see~\cite{yCB2009}. 

In contrast to the methods of ~\cite{yCBhF2006}, our analysis is 
based on a method that couples elliptic estimates to the $L^2-$type energy estimates.
Our approach handles the degeneracy of the dust matter in a different way and 
in particular yields local well-posedness by a different method. Of course, the important point is that
our approach also yields estimates that allow us to prove our main future stability theorem.

\end{remark}

\begin{remark} [\textbf{No proof of convergence to FLRW}]
	We do not claim that the perturbed solution converges to an exact FLRW solution
	as $t \to \infty.$ However, one can show that relative to the wave coordinate system $(t,x^1,x^2,x^3) \in [0,\infty)\times\T^3$ 
 	introduced in Sect. \ref{SS:dustEinstein}, suitably time-rescaled components of the perturbed metric $g_{\mu \nu},$ 
	its inverse $g^{\mu \nu},$ $u^{\mu},$ $\rho,$ 
	and various coordinate derivatives of these quantities each converge to functions of the 
	spatial coordinates $(x^1,x^2,x^3)$ as $t \rightarrow \infty.$ 
	The limiting functions are close to time-rescaled components of the FLRW solution, which are constant in $t$ and $(x^1,x^2,x^3).$
	We do not prove these convergence results in this article, but they can be proved using the arguments given in
	the last section of \cite{iRjS2012}. The main idea of the proof is to revisit the equations after one has proved global existence
	and to use the already derived estimates to treat them as ``ODEs'' with small errors; such ODE-type treatment
	leads to sharper estimates. Of course, the equations are not actually ODEs and hence to carry out this approach, one must have already 
	derived suitable bounds for all terms involving spatial derivatives. We derive such suitable bounds in
	Theorem \ref{T:maintheorem}.
\end{remark}

There is a large amount of literature on the Euler-Einstein and dust-Einstein systems, and 
for additional mathematical overview, we point the reader to~\cite{aR2005b, rW1984}. For the physical context,
readers may consult the cosmological references \cite{sC2001, jPbR2003, vS2004}.
The dust matter model is also referred to as the {\em pure matter}, {\em incoherent matter}, or simply {\em matter}~\cite{yCB2009,sC2001}. 
It is a simple fluid matter model consisting of massive particles whose relative velocities are considered to be negligible.
In particular, in the dust model, on the cosmic scale, we idealize each galaxy as a ``grain of dust"~\cite{rW1984} whose individual velocities are so small that the ``pressure" thus created is negligible. 
In ~(\ref{E:dege}), we have coupled such a pressureless fluid to the Einstein-field equations with a {\em positive} cosmological constant $\Lambda.$ The cosmological constant was first added to the Einstein equations by Einstein himself~\cite{aE1917} in his effort to find a static cosmological \footnote{``Cosmological" refers to a physical theory describing the evolution of a universe as a whole.}
solution. His pursuit proved to be physically wrong, for 
Hubble's famous measurements \cite{eH1929} of the redshift of distant galaxies 
revealed that the universe is in fact expanding. Moreover, measurements of type Ia Supernovae redshift 
at the end of 1990s established that the universe is expanding in an {\em accelerated} fashion.
These experimental findings underlie the theoretical importance of the FLRW solutions: when $\Lambda > 0,$ 
they are among the simplest spatially homogeneous and isotropic solutions undergoing accelerated expansion.
The positive cosmological constant is also important in high-energy physics, where it is often 
identified with the ``vacuum energy." This formal identification follows 
if one views the vacuum as a perfect fluid under the equation of state $p = - \rho$~\cite{sC2001}. 
In particular, the ensuing cosmological models, which go under the name of 
$\Lambda$CDM-models, account for the ``dark energy'' contribution to the expansion of the universe.
%
\subsection{The initial value problem}
It was one of the basic insights of Choquet-Bruhat~\cite{CB1952} that the question of the existence of solutions to Einstein-matter
models can be formulated as an {\em initial value problem}. The conceptual difficulty in addressing this
issue arises from the diffeomorphism invariance of the Einstein equations and the lack of a canonical coordinate system.
In particular, the Einstein field equations are not hyperbolic in the standard sense.
Choquet-Bruhat overcame these difficulties by working in a so-called {\em wave coordinate system}
(these are also referred to as \emph{harmonic coordinates}), 
that is, by demanding that the constraint $\Gamma^{\mu} \eqdef g^{\alpha\beta} \Gamma_{\alpha \ \beta}^{\ \mu}=0$ is valid for
$\mu=0,1,2,3,$ where $\Gamma_{\alpha \ \beta}^{\ \mu}$ is a Christoffel symbol of $g_{\mu \nu}.$ 
The imposition of such a condition allows one to replace the Einstein-field equations 
with a ``modified'' system of equations. The rough idea is to judiciously set $\Gamma^{\mu} \equiv 0$ in certain terms. 
The modified Einstein field equations
form a manifestly hyperbolic system of quasilinear wave equations in the spacetime metric components $g_{\mu \nu}.$ 
For a variety of matter models, one can then prove local well-posedness by applying standard energy methods.
The main insight of~\cite{CB1952} was that the wave coordinate condition
is preserved by the flow of the modified equations if it is assumed to hold initially.
Hence, solutions to the modified system are also solutions to the Einstein equations.
In particular, as mentioned above, 
these ideas have been used to
prove a local well-posedness theorem for the coupled dust-Einstein system
\cite{yCB2009,yCBhF2006}.

In the present work, we employ a modification of the classic wave coordinate condition. Our modified version is
$\Gamma^{\mu}=\tilde{\Gamma}^{\mu},$ where
$\tilde{\Gamma}^{\mu}=3\omega\delta^{\mu}_0$, $(\mu=0,1,2,3),$ are
the contracted Christoffel symbols of the FLRW background metric
\eqref{E:babyFLRW} and $\omega = \omega(t) = a^{-1}(t) \frac{d}{dt} a(t) \sim H$
(see Lemma \ref{L:backgroundaoftestimate}). 
Our choice of gauge is closely related 
to the one used by Ringstr\"om~\cite{hR2008}, and it belongs to a general class of 
modified wave coordinate gauges introduced
by Friedrich and Rendall~\cite{hFaR2000}. 
As we will see, the gauge $\Gamma^{\mu}=\tilde{\Gamma}^{\mu}$ 
leads to the presence of dissipative terms in the modified equations and hence is
well-suited for studying the global structure of near-FLRW solutions
(for details see Sect.~\ref{S:modified}).

Let us now recall some basic facts concerning the initial data for the dust-Einstein system.
An initial data set consists of a $3-$dimensional Riemannian manifold $\mathring{\Sigma}$
(throughout this article $\mathring{\Sigma} = \mathbb{T}^3$)
and the following tensorfields on $\mathring{\Sigma}:$ a Riemannian metric $\mathring{\underline{g}}_{jk},$ 
$(j,k = 1,2,3),$ 
a symmetric covariant two-tensor $\mathring{\underline{K}}_{jk},$ a function $\mathring{\rho},$ and a vectorfield $\underline{\mathring{u}}^j.$ 
A solution consists of a $4-$dimensional manifold $\mathcal{M},$ a Lorentzian metric $g_{\mu \nu},$ a function $\rho,$ a future-directed unit-normalized 
vectorfield $u^{\mu},$ $(\mu, \nu = 0,1,2,3),$ on $\mathcal{M}$ satisfying \eqref{E:metricge}-\eqref{E:massshell1}, and an 
embedding $\mathring{\Sigma} \hookrightarrow \mathcal{M}$ such that $\mathring{\underline{g}}_{jk}$ is the first fundamental form
\footnote{Recall that $\mathring{\underline{g}}$ is defined at the point $x$ by $\mathring{\underline{g}}(X,Y) = g(X,Y)$ for all $X,Y \in T_x \mathring{\Sigma}.$} 
of $\mathring{\Sigma},$ $\mathring{\underline{K}}_{jk}$ is the second fundamental form
\footnote{Recall that $\mathring{\underline{K}}$ is 
defined at the point $x$ by $\mathring{\underline{K}}(X,Y) \eqdef g(D_{X} \hat{N},Y)$ for all $X,Y \in T_x \mathring{\Sigma},$ where $\hat{N}$ is the 
future-directed unit normal to $\mathring{\Sigma}$ at $x.$} 
of $\mathring{\Sigma},$ the restriction of $\rho$ to $\mathring{\Sigma}$ is $\mathring{\rho},$ 
and $(\underline{\mathring{u}}^1, \underline{\mathring{u}}^2, \underline{\mathring{u}}^3)$ is the $g-$orthogonal projection of $(u^0,u^1,u^2,u^3)$ 
onto $\mathring{\Sigma}$. 

It is well known that the initial value problem for the Einstein equations is overdetermined
and that additional constraint equations on
the initial data must be imposed; see Sect.~\ref{SS:dustEinstein}.
In this article, we do not address the problem
of proving the existence of solutions the constraint equations for the dust-Einstein system.

In 1969, Choquet-Bruhat and Geroch~\cite{cBgR1969} proved a fundamental result for Einstein-matter systems. They
showed that every initial data set satisfying the constraint equations
launches a unique {\em maximal globally hyperbolic development}, which roughly corresponds
to the largest possible spacetime solution uniquely determined by the data. However,
this abstract existence result does not provide any quantitative information about the nature of the
maximal solution. An important consequence
of the main theorem of this paper is our characterization 
of the maximal globally hyperbolic developments of near-FLRW data
as {\em future causally geodesically complete}. In particular, their future halves are 
free of singularities in both the metric and the dust.

\subsection{Previous results and related work}
Roughly speaking, there are only two solution regimes for the $1+3$ dimensional 
Einstein equations in which global results have been proved without any symmetry assumptions.
The first regime corresponds to $\Lambda = 0$ and 
concerns the stability of the Minkowski spacetime solution.
The groundbreaking, fully covariant proof 
of the stability of Minkowski spacetime as a solution to the Einstein-vacuum equations 
was first provided by Christodoulou and Klainerman \cite{dCsK1993} relative to a
foliation of spacetime by outgoing null cones and maximal constant-time hypersurfaces
$\Sigma_t.$
A shorter, less precise proof, which
also applies to the Einstein-scalar field system, was provided
by Lindblad-Rodnianski \cite{hLiR2010} relative to a wave coordinate gauge.
The following extensions of these results have also been obtained: 
\textbf{i)} Bieri's improvement of Christodoulou-Klainerman's
regularity and decay assumptions on the data \cite{lB2007}, 
\textbf{ii)}
Zipser's \cite{nZ2000} and 
Loizelet's \cite{jL2008, jL2009}
allowing for coupling to the Maxwell equations,
\textbf{iii)} the second author's allowing for coupling 
to a large family of nonlinear electromagnetic equations \cite{jS2010b}, 
and \textbf{iv)} extensions to higher spatial dimension \cite{yCBpCjL2006}.
All of these results except for \cite{yCBpCjL2006} are based on the availability of
a complicated hierarchy of \emph{dispersive-type} estimates for nearly-Minkowskian
solutions to the Einstein equations. In this regime, the Einstein equations can be
roughly viewed as dispersive wave equations on $\mathbb{R}^{1+3}$ with some very delicate
quadratic nonlinearities that are on the border of what allow for global existence.

The second regime in which global results have been obtained requires $\Lambda > 0$
and concerns the future stability of FLRW-like solutions, such as the ones considered in the present
article. As we will see, in contrast to the first class, these results can be obtained
by choosing a gauge such that $\Lambda > 0$ induces
\emph{energy dissipation} in the solutions.
In particular, the mathematical analysis used to analyze this regime differs markedly from
the dispersive analysis used to prove the stability of Minkowski spacetime.
In particular, it is possible to prove future stability results even when the constant-time
slices are diffeomorphic to a compact manifold, as is the case in the present article.
The first global stability results for the Einstein equations with $\Lambda>0$ were established by Friedrich for the
Einstein-vacuum equations~\cite{hF1986a} and then for the Einstein-Maxwell and Einstein-Yang-Mills equations~\cite{hF1991}. 
These results rely on the so-called {\em conformal method}, which 
was developed by Friedrich and which
reduces the question of global future stability 
to a question of local-in-time solvability for an equivalent problem obtained via a change of state-space variables.
Friedrich's stability results for vacuum spacetimes were later extended by Anderson 
\cite{mA2005} to apply to $1 + n$ dimensional spacetimes, where $n$ is odd. 
Although the conformal method is well suited for analyzing matter models with trace-free
energy momentum tensors, it unfortunately seems to be less well suited for studying general matter models.

A more robust method for addressing the question of stability in the presence of a positive cosmological
constant was introduced by Ringstr\"om~\cite{hR2008} in his study of  
Einstein-nonlinear scalar field systems. 
In particular, he introduced a version of the wave coordinate condition that is essentially equivalent
to the one used in this paper. The coordinate system allowed him to 
expose the strong dissipative effect associated with the positive cosmological constant, 
and this dissipation played a crucial role in his work. The scalar field in \cite{hR2008}
was assumed to verify the evolution equation $g^{\alpha \beta}D_{\alpha} D_{\beta} \Phi=V'(\Phi),$ 
and the corresponding energy-momentum tensor is
$T_{\mu\nu}=\g_{\mu}\Phi\g_{\nu}\Phi-[\frac{1}{2}g^{\alpha\beta}\g_{\alpha}\Phi\g_{\beta}\Phi+V(\Phi)]g_{\mu\nu}$.
The ``potential'' $V$ was assumed to verify $V(0)= V'(0)=0$ and $V''(0)>0$. 
Although Ringstr\"om assumed that $\Lambda=0$, the presence of the potential $V$, when $\Phi$ is small, effectively
emulates the presence of a positive cosmological constant. Such models are closely related to the $\Lambda$CDM models mentioned above.
Ringstr\"om's main result was a proof of the global future stability of a large family of FLRW-like solutions to 
Einstein-nonlinear scalar field systems. 
We remark that prior to Ringstr\"om's work, 
the long time asymptotics of spatially homogeneous solutions 
to Einstein-nonlinear scalar field systems
had been studied by Rendall~\cite{aR2004c}.
Ringstr\"om also proved global stability results without symmetry assumptions
for the Einstein-Vlasov system with a positive cosmological constant 
in his recent monograph ~\cite{hR2013}. 
Prior to the monograph, stability results with various symmetry assumptions 
had been proved in~\cite{aRsbT2003,sbTnN}.

Motivated by~\cite{hR2008}, Rodnianski and the second author initiated a systematic 
study of the future stability properties of the FLRW solutions to the Euler-Einstein system
with a positive cosmological constant under the equation of state $p=c_s^2\rho$. 
The main new difficulty that the authors had to
overcome is that unlike the the scalar field models studied
by Ringstr\"om \cite{hR2008},
the fluid models admit solutions that develop shock singularities in finite time.
We will elaborate on this issue just below.
Under the assumptions that the fluid is irrotational, 
that $0<c_s^2<1/3,$ and that the initial (uniform) FLRW fluid pressure is strictly positive,
they proved~\cite{iRjS2012} the global future stability of the FLRW family.
The assumption of irrotationality was later removed by the second author~\cite{jS2012}.
The main new contribution of those articles was showing that 
in the parameter regime $0<c_s^2<1/3,$ the rapid spacetime expansion 
provides enough dissipation to suppress the formation of shocks in the near-FLRW regime.

We now describe some related future stability results.
The global future stability of the FLRW family in the pure radiation case $c_s^2=1/3$ 
with a positive cosmological constant
was recently established by Valiente Kroon and L\"ubbe~\cite{cLjVK} 
using Friedrich's conformal method. As we mentioned above, the conformal method is often applicable when the energy-momentum tensor
of the matter is trace-free. The important point is that the energy-momentum tensor of the fluid is trace-free if and only if 
$p = (1/3) \rho.$ 
A global future stability result for the dust-Einstein system with a positive cosmological constant under the assumption of planar symmetry
was proved in \cite{sT2008}. The case of planar symmetry is much simpler than the $1+3$ dimensional
case treated in the present article. In particular, the plane-symmetric case can be treated without
the elliptic estimates that play a fundamental role in the present article (see below). 
Similarly, a global future stability result for the
the plane symmetric Einstein-stiff fluid system 
($c_s = 1$) with a positive cosmological constant
was proved in \cite{pLsT2011}.

We also mention the global stability result \cite{lAvM2004}, which does not fit into either of  
the previously mentioned classes. In \cite{lAvM2004}, Andersson and Moncrief 
proved a global stability result in the expanding direction 
for a compactified version of the vacuum FLRW solution with $\Lambda = 0.$
The spatial slices of the spacetimes are hyperboloidal (that is, they have constant negative sectional curvature),
and the perturbed solutions were shown to be future geodesically complete and to decay towards the background
solution.

An important conclusion of the above sequence of works is that nearly-quiet fluids with 
$0 < c_s^2 \leq 1/3$ are {\em stabilized} in the presence of a positive cosmological constant.
More precisely, if one were to fix an FLRW background metric 
$\tilde{g}$ and to analyze the relativistic Euler equations
on such a fixed background with $0<c_s^2<1/3$, then the methods of~\cite{jS2012} could easily be adapted to show that the 
FLRW-type fluid solutions are globally future-stable. This
stands in stark contrast to the analogous situation on the Minkowski spacetime background, 
in which Christodoulou~\cite{dC2007b,dC2007} showed that under a general physical equation of state\footnote{Christodoulou identified one exceptional equation of state for the relativistic Euler equations that in the irrotational case leads to the minimal surface equation in Minkowski spacetime. The corresponding quasilinear wave equation verifies the well-known \emph{null condition} and hence admits 
global small-data solutions.}, 
shocks singularities will develop in solutions launched by smooth data belonging to 
an open set that contains data arbitrarily close to the uniform quiet fluid states\footnote{The results of \cite{dC2007} do not apply to the dust.}. 
In a related article that addressed fluids 
verifying $0 \leq c_s^2 \leq 1/3$ 
on a {\em fixed} spacetime background 
$\big([0,\infty) \times \mathbb{T}^3, - dt^2 + a^2(t) \sum_{i=1}^3 (dx^i)^2 \big),$
the fluid's future behavior has recently been characterized~\cite{jS2012b} 
by the second author in a near-FLRW-type fluid solution regime.
In particular, the second author identified a sharp time-integrability criterion for the reciprocal expansion factor $a^{-1}(t)$ that characterizes the fluid's future behavior. If the condition is violated, then in the case 
$c_s^2=1/3,$ it was shown that arbitrarily small perturbations of the FLRW-type fluid data can lead to the formation
of a shock singularity in finite time. The main idea of the proof was to use the conformal invariance of the relativistic Euler equations (which holds if and only if $c_s^2=1/3$)
to reduce the result to the well-known shock-formation result of Christodoulou~\cite{dC2007}. On the other hand, if the 
condition on $a^{-1}(t)$ is fulfilled, then when $c_s^2 = 1/3,$ 
the FLRW-type fluid solution was shown to be globally future-stable. In the remaining cases
$0 \leq c_s^2 < 1/3,$ it was shown that if the integrability
condition on $a^{-1}(t)$ is fulfilled, and a few additional mild technical assumptions on $a(t)$ also hold,
then the fluid is globally future-stable.

A Newtonian version of the future stability theorem of~\cite{jS2012} was proved in 1994 by Brauer, Rendall, and Reula~\cite{uBaRoR1994}.
They studied Newtonian cosmological models equipped with a positive cosmological constant and containing a perfect fluid verifying the equation of state $p = C \rho^{\gamma},$ where $\rho \geq 0$ is the Newtonian mass density, and $C > 0, \gamma > 1$ are constants. They proved that the uniform quiet fluid states of constant positive density are globally future-stable. 

\subsection{Methodology and a model problem}
One of the key insights of the series of works~\cite{hR2008, iRjS2012, jS2012} is that the presence of the positive cosmological constant 
induces a strong dissipative effect that stabilizes solutions. 
Such an effect becomes visible upon reformulating the system in a special version 
of wave coordinates; this reformulation is a starting point 
for our analysis as well.

In this article, we will use the following modification of the classic wave coordinate condition:
\begin{align} \label{E:gaugeintro}
\Gamma^{\mu} & =\tilde{\Gamma}^{\mu} &&(\mu=0,1,2,3),
\end{align}
where $\tilde{\Gamma}^{\mu}=\tilde{g}^{\alpha\beta}\tilde{\Gamma}^{\ \mu}_{\alpha \ \beta}$
are the contracted Christoffel symbols of the background FLRW metric and in particular are known functions of $t.$
Under the gauge condition~\eqref{E:gaugeintro}, the components $g_{00},$ $g_{0i}$ of the metric tensor, as well 
as the rescaled metric components $h_{ij}=a^{-2}(t)g_{ij}$, $(1\leq i,j\leq3),$ satisfy a 
{\em nonlinear damped wave-type equation} [recall that the scale factor $a(t)$ is introduced in~\eqref{E:babyFLRW}].
More precisely, the scalar unknowns $\varphi \in \lbrace g_{00},\,g_{0i},\,h_{ij} \rbrace_{1\le i,j\le 3},$ 
satisfy wave equations of the following form
(see Prop. \ref{P:Decomposition} for the precise form):
\be\label{E:schematic0}
\hat{\square}_g \varphi= A \,\g_t\varphi+ B \, \varphi+\mathcal{N}(\varphi,\g\varphi,u,\rho),
\ee
where $A > 0$ and $B \geq 0$ are constants (depending on the particular metric component), 
$\mathcal{N}(\varphi,\g\varphi,u,\rho)$ is a nonlinearity, 
$\hat{\square}_g \eqdef g^{\alpha\beta}\g_{\alpha}\g_{\beta}$ is the \emph{reduced} wave-operator of $g$,
and $\g$ denotes the spacetime coordinate gradient. 
We stress that the constants $A$ and $B,$ which generate the dissipative effects
that lead to our main future stability theorem, are present because 
\textbf{i)} the cosmological constant is positive
and 
\textbf{ii)} we are using the gauge \eqref{E:gaugeintro}, which allows us to replace the Einstein field
equations with equivalent equations of the form \eqref{E:schematic0}.
We also remark that the nonlinear term 
$\mathcal{N}$ depends on all of the metric components and thus Eq. \eqref{E:schematic0}
is only schematically correct. In order to illustrate the dissipative effect mentioned above, for simplicity, 
we replace $\hat{\square}_g$ in~\eqref{E:schematic0} with the reduced wave operator $\hat{\square}_{g_{(Model)}}$
of the pre-specified Lorentzian metric
\[
	g_{(Model)}=-dt^2+e^{2t}\sum_{i=1}^3(dx^i)^2,
\]
(that is, $\hat{\square}_{g_{(Model)}} \eqdef -\g_{tt}+e^{-2t}\triangle,$
where $\triangle$ denotes the Laplacian corresponding to the standard Euclidean metric on $\mathbb{T}^3$). 
The main point is that $g_{(Model)}$ is a good model for the perturbed FLRW solution metrics that we will encounter in
our actual problem of study.
For simplicity, we set $B=0$ in \eqref{E:schematic0}. 
We then obtain the following model semilinear problem that we will, for the sake of illustration, study in the remainder of this section:
\be\label{E:schematic}
	\hat{\square}_{g_{(Model)}} \varphi = A \partial_t \varphi + \mathcal{N}(\varphi,\g\varphi,u,\rho).
\ee
If we multiply~\eqref{E:schematic} by $-\g_t\varphi$ and integrate by parts over $\T^3$, we obtain
\[
\frac{d}{dt}E^2(t)+\int_{\T^3} \left\lbrace 
	A(\g_t\varphi)^2
		+ |e^{-t} \underpartial \varphi|^2
	\right\rbrace	
	\,dx
	= -\int_{\T^3}
			\mathcal{N}(\varphi,\g\varphi,u,\rho)
		\g_t\varphi \,dx,
\]
where here and throughout, $\underpartial$ denotes the spatial coordinate gradient, 
$dx$ denotes the volume form of the standard Euclidean metric on $\mathbb{T}^3,$ 
and the energy $E$ is defined by
\begin{align} \label{E:MODELENERGY}
E^2(t)\eqdef\frac{1}{2}\int_{\T^3}
	\left\lbrace
		(\g_t\varphi)^2+|e^{-t} \underpartial \varphi|^2
	\right\rbrace \,dx.
\end{align}
Applying the Cauchy-Schwarz inequality, we deduce that there exist universal constants $C_1 > 0$ and $C_2 > 0$ such that
\[
\frac{d}{dt}E^2(t)+C_1 E^2(t) \le C_2 \, \|\mathcal{N}\|_{L^2}(t) \, E(t).
\]
Such a coercive energy structure is ``stable" under differentiation, which means that for any integer 
$N\geq0,$ we can form a high-order energy functional
\[
E_{N}^2(t)\eqdef\frac{1}{2}\sum_{|\vec{\alpha}|\leq N}\int_{\T^3}\{(\g_t\a\varphi)^2 + |e^{-t} \underpartial \a\varphi|^2\}\,dx
\] 
that satisfies an analogous inequality
\[
\frac{d}{dt}E_{N}^2(t)+C_1E_N^2(t) \le C_2 \sum_{|\vec{\alpha}|\leq N} \|\a \mathcal{N}\|_{L^2}(t) E_N(t).
\]
Of course, in the full dust-Einstein problem, we also have to derive energy inequalities for the fluid variables; we will
return to this issue shortly.

The technical core of our approach is to effectively estimate the ``error'' terms $\|\a\mathcal{N}\|_{L^2}(t)$ in terms 
of the energies. That is, suitable estimates for the error terms 
$\|\a \mathcal{N}\|_{L^2}(t)$ in terms of the energies will
lead to the availability of an integral inequality for 
the energies that is amenable to Gronwall's inequality;
such analysis provides us with suitable a priori estimate for the energies.
As we will see, in the full dust-Einstein problem, we will 
define a family of energies for the metric and fluid variable components, and we will
derive a somewhat intricate \emph{hierarchy} of integral inequalities for the various solution variable energies. 
Moreover, in order to derive suitable a priori estimates for the energies,
\emph{we will have to understand the precise manner in which these inequalities are coupled.} 
The complete structure is revealed in Props. \ref{P:integralinequalities} and \ref{P:integralinequalitiesmetric}.
We stress that \emph{the derivation of suitable a priori estimates for the various energies constructed out of the metric
and fluid variables is the main step in our proof of future-global existence}. 
For once we have shown that the energies cannot blow-up in finite time, 
only a few other simple estimates are needed to completely rule out the possibility of singularity formation.
We give a precise statement of this ``continuation principle" in Sect. \ref{SS:LWPANDCONTINUATION}; 
see Prop. \ref{P:continuationprinciple}.
We remark that in the proof of our main future stability theorem, we will work with rescaled energies
that are approximately constant in time.

We stress that the most important technical estimates in this article, which are derived in Sect.~\ref{S:sobolev}, 
involve Sobolev estimates of the error terms analogous to $\mathcal{N}$ above; 
these error term estimates are the ones that allow us to derive
Gronwall-amenable estimates for our collection of solution energies. 
Our derivation of these error term estimates is based on 
Sobolev-Moser type estimates (which are laid out in Appendix~\ref{B:SobolevMoser})
that take into account the various growth/decay rates of the metric/fluid components. To
facilitate our analysis of the growth/decay rates of many of the error terms, we introduce a Counting Principle,
which we explain in detail below.

There are two major caveats concerning this approach to estimating the metric.
The {\bf first} caveat is connected to a dangerous term  
in the wave equation~\eqref{E:metric0j} for $g_{0j};$ the right-hand side contains a linearly large 
error term of the form $-2Hg^{ab}\Gamma_{ajb}$ that does not decay sufficiently fast to immediately allow 
closure of our energy estimate for $g_{0j};$ see the ``dangerous'' integrand on the right-hand side of \eqref{E:mathfrakENg0*integral}.
This is a potentially damaging problem, but the saving grace is the observation that $-2Hg^{ab}\Gamma_{ajb}$ 
can be bounded by the energies of the {\em purely spatial} metric components $g_{ij}$. Fortunately, an effectively independent 
estimate for the $g_{ij}$ energy can be derived from its wave equation~\eqref{E:metricjk}, which does not contain 
any dangerous linear terms on the right-hand side. This is one manifestation of the energy hierarchy mentioned above.
Such a {\em partial decoupling} phenomenon was already observed in~\cite{hR2008, iRjS2012, jS2012}.  
In the present article, we will, out of necessity, uncover and exploit a much more sophisticated form of 
partial decoupling of the energy estimates. We will return to this issue later in the introduction.

The {\bf second} caveat is potentially even more dangerous, and it is intrinsically tied to the degenerate nature of the dust model. 
In short, to overcome the degenerate derivative structure of the dust model, 
we expend a great deal of effort to avoid losing derivatives.
To elaborate, we will pursue the high-order energy method as laid out above and then point out the difficulties that arise.
To begin, we note that the above discussion already provides a natural $L^2-$based energy
framework for controlling the high-order derivatives of the metric components
$g_{\mu \nu}$ (which are caricatured by the variable $\varphi$).
However, there is one possible obstruction to implementing this strategy: since the metric equations are coupled to the fluid variables [expressed schematically through the presence of the term $\mathcal{N}(\varphi,\g\varphi,u,\rho)$ 
in Eq. \eqref{E:schematic}],
we must understand the energy structure corresponding to the fluid unknowns $\rho,u^{\mu}$, $(\mu=0,1,2,3)$. 
As we will see, this structure degenerates in the case of the dust. 

For the sake of contrast, let us briefly discuss the relativistic Euler equations under the equation of state $p=c_s^2\rho,$ 
$c_s^2 > 0.$ At the moment, we are only concerned with the issue of avoiding the loss of derivatives.
When $c_s^2 > 0,$ there are no obstacles.
Specifically, one can associate \emph{energy currents} $\dot{J}^{\mu}$ to the fluid solutions. 
The energy currents are vectorfields that can be used via the divergence theorem to control the evolution of the $L^2$ norms of the fluid variables and their derivatives. The energy current method is essentially a geometrically inspired version of integration by parts.
This approach to analyzing the Euler equations differs from the well-known \emph{symmetric hyperbolic} framework 
(see, for example, \cite{rCdH1962, cD2010, kF1954}) in that it allows the analysis to take place directly in the \emph{Eulerian} variables $(p,u)$ rather than an artificially introduced collection of state-space variables.
This framework, which was first applied to the relativistic Euler equations by Christodoulou in \cite{dC2007},
has been previously applied by the second author in various fluid contexts~\cite{jS2008b,jS2008a,jSrS2011}. 
A general framework for the energy current method was developed by Christodoulou in ~\cite{dC2000}, 
who showed that it can be applied to a class of systems known as \emph{regularly hyperbolic} PDEs.
In the particular case of the relativistic Euler equations, the \emph{full coerciveness} 
of the energies generated by the currents\footnote{Actually, the relativistic Euler equations fall just outside of the scope of the regularly hyperbolic PDEs studied in \cite{dC2000}.
Nonetheless, in \cite{dC2007}, Christodoulou showed that the energy current framework can be extended to apply to the relativistic Euler equations in Eulerian coordinates.}
requires the \emph{strict positivity} of $c_s^2.$ 
Hence, this method does not directly apply to the dust model, for in this case, the method only yields \emph{positive semi-definite} energies. Specifically, the energy currents used in \cite{jS2012} do not provide control of the energy density $\rho$ when $c_s = 0;$ 
this is the main {\em new} difficulty that we encounter in comparison to~\cite{jS2012}. 

\subsubsection{Fighting to avoid derivative loss}
We now illustrate this new difficulty in greater detail.
To begin, we first note that the dust equations for $u^j$ and $\rho$ can be viewed as transport-type equations of the following schematic form [see \eqref{E:revol}-\eqref{E:velevol} for the precise form in terms of the rescaled energy density $\dens = e^{3 \Omega} \rho$]
\footnote{We sum over the repeated upstairs-downstairs spatial indices.}:
\begin{subequations}
\begin{align} \label{E:velocityintro}
\g_tu^j + v^{a}\g_au^j &=-2\omega u^j+\mathcal{S}^j(\varphi, \g \varphi,u), & & (j=1,2,3),
\\
\g_t\rho + v^a\g_a\rho& = - 3 \omega \rho + \mathcal{S}(\varphi, \g\varphi, u, \partial u, \rho), & & 
\label{E:densityintro}
\end{align}
\end{subequations}
where $\omega = a^{-1}(t) \frac{d}{dt} a(t),$
the $\mathcal{S}^j$ and $\mathcal{S}$ are nonlinear source terms,  
the $v^a$ are functions of the metric and $u,$
and $\g$ denotes the spacetime coordinate gradient. 
We note that $\omega(t) \sim H = \sqrt{\Lambda/3}$ 
(see Lemma \ref{L:backgroundaoftestimate})
is a known function of time
and that the dissipative terms 
$-2\omega u^j$ 
and
$- 3 \omega \rho$
on the right-hand sides
of 
\eqref{E:velocityintro}
and \eqref{E:densityintro}
\emph{are precisely what stabilizes the dust.} 
In particular, the term $-2\omega u^j$ leads to the 
rapid exponential decay of $u^j$ 
towards the quite state $u^j \equiv 0.$
We stress that the dissipative terms,
which are analogs of
the dissipative terms present in the model metric component wave Eq. \eqref{E:schematic0},
are available only because $\Lambda > 0.$
For simplicity, at the moment, we will not concern ourselves
with issues connected to time decay or the issue of putting the correct $t-$weights in the norms. 
Instead, we will only focus on the number and \emph{kinds} of derivatives involved in the estimates.
Note that the $\mathcal{S}^j$ do not depend on $\rho$ or $\g u$
and that $\mathcal{S}$ depends on $\g u;$
these observations are crucially important for understanding the energy structure that we now discuss. 
Assume now that our high-order energy framework is set up in such a way that 
$\| \partial \varphi \|_{H^{N}}$
is the top-order Sobolev norm of the metric caricature quantity $\varphi$
that is controlled by the energies corresponding to the model Eq. \eqref{E:schematic}.
Then we can bound the right-hand side of \eqref{E:velocityintro} only in $H^N$-norm
because it depends on $\partial \varphi.$
It is therefore straightforward to use Eq. \eqref{E:velocityintro} to derive energy estimates 
that yield control of $\sum_{j=1}^3\|u^j\|_{H^N},$ and this is the highest Sobolev norm of $u^j$ that we can 
expect to control.
Next, we note that we can bound the right-hand side of~\eqref{E:densityintro} only in $H^{N-1}$-norm
because of its dependence on $\g u.$
Using this bound and equation ~\eqref{E:densityintro}, a straightforward $L^2$-type energy argument involving integration by parts 
over $\mathbb{T}^3$
allows us to control $\rho,$ \emph{but only in $H^{N-1}$-norm.} 
This ``loss of one derivative" is the degenerate energy structure mentioned above.
The essential difficulty is now apparent: \emph{in order to bound the norm $\|\g \varphi\|_{H^N}$
for solutions to~\eqref{E:schematic} by implementing the $L^2$ based energy method described above,
we must control its right-hand side in the norm
$\|\mathcal{N}(\varphi,\g \varphi,\rho, u)\|_{H^N}$}. 
This in turn requires us to estimate $\|\rho\|_{H^N}$ in terms of the available energy quantities,
\emph{which is one more derivative of $\rho$ than we expect to have control of}.

We now sketch a proof of how to resolve these difficulties 
and in particular how to obtain control over both
$\| \partial \varphi \|_{H^{N}}$
and $\sum_{j=1}^3\|u^j\|_{H^N}.$
Our resolution requires several ingredients.
First, we note that Eqs.\eqref{E:velocityintro}-\eqref{E:densityintro}
for be rewritten as follows [see \eqref{E:dustalternative1}-\eqref{E:dustalternative2}]:
\begin{align} \label{E:velintro}
	\u u^j 
	&	= \tilde{S}^j(u,\varphi,\partial \varphi), \\
	\u \rho
	&=\tilde{S}(\rho,\partial u,\varphi,\partial\varphi), 
	\label{E:dustintro}
\end{align}
where $\u$ is the first-order differential operator defined by
\be\label{E:deltau}
\u \eqdef u^{\mu}\g_{\mu}.
\ee
As we mentioned above, it is straightforward to use Eq. \eqref{E:dustintro} to derive energy estimates 
that yield control of $\|\rho\|_{H^{N-1}}.$ 
We in fact need to derive such estimates for $\rho$ in order to prove our main future stability theorem.
However, as we will see, 
\emph{we also need to treat \eqref{E:dustintro} as a ``constraint" for $\rho$ in order to close some of the estimates.} 
Similarly, we can treat \eqref{E:velintro} as a ``constraint'' for $u.$
That is, we can use \eqref{E:velintro}-\eqref{E:dustintro}
to also obtain control of the norms
\[
\|\u \rho \|_{H^{N-1}},
\, \|\u u \|_{H^{N-1}}
\]
in terms of 
$\sum_{j=1}^3\|u^j\|_{H^N},$
$\| \rho \|_{H^{N-1}},$
and 
$\| \partial \varphi \|_{H^N}.$
Actually, using \eqref{E:dustintro}, we could obtain control over
$ \|\u u \|_{H^N},$ but as we will see, we do not need this top-order norm to close our estimates.
Hence, it is natural to try to close the estimates using the norms
\be
\label{E:controlintro}
\| \varphi \|_{H^{N-1}}
+
\| \g \varphi \|_{H^N}
+ \sum_{j=1}^3\|u^j\|_{H^N}
+ \sum_{j=1}^3\|\u u^j\|_{H^{N-1}}
+\| \rho \|_{H^{N-1}}
+ \| \u \rho \|_{H^{N-1}}.
\ee 
However, we have still not resolved the aforementioned problem of how estimate 
$\|\g \varphi\|_{H^N}.$ 
To resolve this problem, we organize and build our high-order energy estimates around two central insights.

\textbf{i}) First, we commute~\eqref{E:schematic} with the operator $\u$.
As a consequence, we obtain an equation that is essentially of the same type as~\eqref{E:schematic},
that is, of the schematic form
\be\label{E:waveintro}
\hat{\square}_{g_{(Model)}} \u \varphi = A \,\g_t \u \varphi \ +  \text{ ``error term,"}
\ee
where the ``error term'' above, by the Leibniz rule and some basic commutator estimates, has the structure
\begin{align} \label{E:ANNOYINGMODELPARTIALUDEPENDENTERRORTERM}
\text{ ``error term" }  & =
\mathcal{N}_2(\varphi, \g \varphi, \partial \u \varphi, \rho, \u \rho, u, \u u, \underpartial u)
\end{align}
for some nonlinear inhomogeneous terms $\mathcal{N}_2(\cdots).$
We remind the reader that here and throughout, $\underpartial$ denotes the \emph{spatial} coordinate gradient.
By applying to equation~\eqref{E:waveintro}  
the same energy methods strategy that we first applied to~\eqref{E:schematic},
we can derive energy estimates for
$\| \g_t \u \varphi\|_{H^{N-1}}$
and 
$\| \underpartial \u \varphi\|_{H^{N-1}}$
(we again stress that we are currently ignoring the issue of tracking the correct $t-$weights). 
Hence, assuming that we have control of the norms listed in \eqref{E:controlintro},
we also gain control of the norms
$\| \g_t \u \varphi \|_{H^{N-1}}$
and
$\| \underpartial \u \varphi \|_{H^{N-1}}.$
This argument requires in particular that we have control
over $\sum_{j=1}^3\| u^j\|_{H^N}$
because the right-hand side of \eqref{E:ANNOYINGMODELPARTIALUDEPENDENTERRORTERM}
depends on $\underpartial u.$
We already described how to derive such a bound for $\sum_{j=1}^3\|u^j\|_{H^N}:$ 
we  use equation ~\eqref{E:velocityintro} 
and a straightforward $L^2$-type energy argument involving integration by parts 
over $\mathbb{T}^3,$ 
but we again stress that \emph{this approach is viable only if we are able to control the right-hand side of 
\eqref{E:velocityintro} in the norm} $H^N.$
In view of the fact that the right-hand side of
\eqref{E:velocityintro} depends on $\partial \varphi,$
we see that in order to control $\sum_{j=1}^3\|u^j\|_{H^N},$
we in particular have to control $\| \partial^{(2)} \varphi\|_{H^{N-1}}$
(here and throughout, $ \partial^{(2)} \varphi$ denotes the array of second-order
spacetime coordinate derivatives).
Thus, the main difficulty that remains,
from the point of view of avoiding derivative loss,
is how to control  $\| \partial^{(2)} \varphi\|_{H^{N-1}}.$

\textbf{ii)} In order to estimate the top-order spacetime derivatives 
$\| \partial^{(2)} \varphi\|_{H^{N-1}}$ and thus avoid top-order
derivative loss, we first use the decomposition $\g_t = \frac{1}{u^0}\u - \frac{u^a}{u^0} \g_a$
(and for sake of illustration we completely ignore the lower-order terms $\frac{1}{u^0}$ and $u^a$)
to deduce the schematic relation 
$|\partial^{(2)} \varphi| \lesssim |\g_t \u \varphi| + |\underpartial \u \varphi| + |\underpartial^{(2)} \varphi| + l.o.t.,$
where the terms $l.o.t$ involve at most one derivative of $u$ and $\varphi$ and hence are not top-order.
As we explained in \textbf{i)}, the quantities $\| \g_t \u \varphi \|_{H^{N-1}}$ and $\| \underpartial \u \varphi \|_{H^{N-1}}$
are controlled by the energies corresponding to Eq. \eqref{E:waveintro}. 
Hence, it only remains for us to explain how we control
the top-order pure spatial derivatives $\| \underpartial^{(2)} \varphi\|_{H^{N-1}}.$
To this end, we will prove a \emph{key elliptic estimate} (see Lemma~\ref{L:elliptic}), 
which takes the following
form in the model semilinear problem:
\begin{align} \label{E:MODELELLIPTICRECOVERY}
\|\underpartial^{(2)} \varphi\|_{H^{N-1}}
\lesssim \, 
\| \hat{\square}_{g_{(Model)}} \varphi\|_{H^{N-1}}
+\| \g_t \u \varphi\|_{H^{N-1}}
+ \| \underpartial \u \varphi\|_{H^{N-1}}
+ \mbox{lower order terms},
\end{align}
where we again stress that we have, for the time being, 
completely ignored the issue of including the correct $t-$weights
in the estimates.
The first term on the right-hand side of \eqref{E:MODELELLIPTICRECOVERY}
is controllable because we can use
Eq. \eqref{E:schematic} to substitute 
$\hat{\square}_{g_{(Model)}} \varphi$ 
with terms that are lower-order
in the sense that they depend on non-top-order
derivatives and hence can be bounded in terms of energies that we have already shown how to estimate.
The second and third
terms on the right-hand side of \eqref{E:MODELELLIPTICRECOVERY}
are controlled by the energies corresponding to Eq. \eqref{E:waveintro}.
The remaining terms are also lower-order in terms of the number of derivatives.
We have thus given a schematic outline of how to close the estimates without losing derivatives.
This is roughly the outline that we will follow in our analysis of the full dust-Einstein
problem, although many additional complications will arise because we prove a global result.

\subsubsection{Remarks on $t-$behavior}
We now make some important remarks concerning the time behavior of some important terms
that exhibit a new feature of the present problem compared to \cite{jS2012}.
First, a careful analysis of the right-hand side of \eqref{E:waveintro}
reveals the presence of the dangerous commutator term 
$- \left(\u g_{(Model)}^{ab} \right) \g_a \g_b \varphi$ which is {\em linear} to the leading order because 
$- \u g_{(Model)}^{ab} = 2 g_{(Model)}^{ab} 
+ \mbox{harmless error terms}$ 
(here we are assuming that $u^j$ and $u^0 - 1$ are rapidly decaying error terms, 
which are estimates that we will in fact derive during our analysis of the dust-Einstein problem).
This is precisely the second caveat referred to 
earlier in the introduction, as the presence of a linear term with an unfavorable growth rate
can in principle lead to a breakdown of our method. 
The main point is that if we were to directly bound
the $L^2$ norm of $- \left(\u g_{(Model)}^{ab}\right) \g_a \g_b \varphi$
when carrying out the energy estimates, then we would be unable to prove our 
future-global existence result because the expected time behavior of this product is borderline.
As is explained in detail in Lemma~\ref{L:COMMUTEDBASICSTRUCTURE}
and Remark~\ref{R:DANGEROUSTOPORDERCOMMUTATORTERM}, 
to circumvent this difficulty,
we first use the definition of the operator $\hat{\Box}_{g_{(Model)}}$
to algebraically replace
$2 g_{(Model)}^{ab} \g_a \g_b\varphi$ with 
$2(\hat{\Box}_{g_{(Model)}}\varphi - g_{(Model)}^{00} \g_{tt}\varphi)$. 
Next, we replace $\hat{\Box}_{g_{(Model)}}\varphi$ with the right-hand side of
\eqref{E:schematic}, while the 
decomposition $\g_t = \frac{1}{u^0}\u - \frac{u^a}{u^0} \g_a$
and the relation $g_{(Model)}^{00} = -1$ 
together allow us to replace the term
$-2 g_{(Model)}^{00}\g_{tt} \varphi$ 
with $+2 \g_t \u\varphi$ modulo small harmless error terms.
In total, we can derive the following equivalent version of~\eqref{E:waveintro}
(see Lemma \ref{L:COMMUTEDBASICSTRUCTURE} for the precise version):
\begin{align*}
\hat{\square}_{g_{(Model)}} \u \varphi = (A+2) \,\g_t \u \varphi 
\ + \underbrace{C \g_t\varphi}_{\mbox{\textnormal{dangerous}}}  \ + \text{harmless error terms}.
\end{align*}
We stress that the term $(A + 2) \,\g_t \u \varphi$ 
provides a {\em coercive} quadratic contribution to the energies
corresponding to $\u \varphi,$
and it is thus good for our purposes because it can only enhance the decay of $\u \varphi.$ 
The term 
$C \g_t \varphi$ is linear, but of {\em lower} order of differentiation. 
Hence $\|\g_t \varphi\|_{H^{N-1}}$ can effectively be independently controlled by
first deriving energy estimates for Eq. \eqref{E:schematic}. 
Thus, for the second time, we have
encountered an effective {\em partial decoupling} of the energy estimates.
This structure is essential for proving our future-global result.

Another important issue involving $t-$behavior is connected
to the elliptic estimate \eqref{E:MODELELLIPTICRECOVERY}.
The important point is that when we derive the correct elliptic estimate, that is,
the one with the actual $t-$weights that occur in the problem,
we will not obtain 
$\|\underpartial^{(2)} \varphi\|_{H^{N-1}}$ on the left-hand side of \eqref{E:MODELELLIPTICRECOVERY},
but rather only 
$e^{-2t} \|\underpartial^{(2)} \varphi\|_{H^{N-1}};$
see Lemma \ref{L:elliptic} for the full details.
Note that this is a \emph{worse} $t-$weight by a factor 
of $e^{-t}$ than the factor $e^{-t}$ found in the model
energy \eqref{E:MODELENERGY}.
That is, in reality, we are able to prove \emph{only that}
$\|\underpartial^{(2)} \varphi\|_{H^{N-1}}$
is bounded by $\lesssim$ $e^{2t}$ times 
the right-hand side of \eqref{E:MODELELLIPTICRECOVERY}
[we also need to put the correct $t-$weights into the right-hand side of \eqref{E:MODELELLIPTICRECOVERY}].
Roughly, the reason that we
only obtain $e^{-2t} \|\underpartial^{(2)} \varphi\|_{H^{N-1}}$ on the left-hand side 
is that the inverse spatial metric components 
$g_{(Model)}^{jk}$ are decaying like $e^{-2t},$
and these components are the main coefficients that govern
the behavior of the top-order spatial derivatives.
These top-order spatial derivatives estimates 
with unfavorable $t-$weights
are another {\em new} feature relative to the works
\cite{iRjS2012,jS2012}, where elliptic estimates did not play a role.
Our solution norms, which are rescaled by time factors 
so as to remain  
approximately constant (see Sect.~\ref{S:norms}), 
reflect this worsened top-order behavior.
In total, we are able to control all of the norms necessary to close the bootstrap scheme
at the expense of additional exponentially-in-time growing factors at the top-order of differentiation
in our norm hierarchy. 
Roughly speaking, if the total derivative count for  $u^j$ in our problem is $N$, then our proof reveals that 
typical size-$\epsilon$ norms of a given metric tensor component $\varphi$ are of the form
\[
	e^{rt} (\|\g_t\varphi\|_{H^{N-1}}+\|\u\g_t\varphi\|_{H^{N-1}})+e^{(r-1)t}\|\underpartial \g_t \varphi\|_{H^{N-1}}
\]
and
\[
	e^{s t}(\|\underline{\partial}\varphi\|_{H^{N-1}}+\|\u\underline{\partial}\varphi\|_{H^{N-1}})+e^{(s-1)t}\| \underpartial^{(2)} \varphi\|_{H^N},
\]
where $r,s \in \mathbb{R}$ are rates that capture the decay/growth properties of 
the various derivatives of $\varphi$ (we recall that $\underpartial$ denotes the spatial coordinate gradient). 
A similar hierarchy holds for $u^j.$

A remarkable feature of the dissipation generated by the positive cosmological constant is \emph{its ability to 
overcome the possible time growth caused by the unfavorable $t-$weights in the elliptic
estimates.} Roughly speaking, the reason that the unfavorable $t-$weights in the elliptic estimates
are harmless is the following: 
as our above discussion has suggested, the
only time that we need to use the elliptic estimates for the metric components is when we need
to estimate the top-order spatial derivatives of $u^j.$ That is, we need to 
control $\sum_{j=1}^3\|u^j\|_{H^N}$ via energy estimates 
and the right-hand side of \eqref{E:velocityintro} depends on $\partial \varphi.$
The main point is that in the near-FLRW solution regime, the coupling of 
the metric quantity $\partial \varphi$ to $u$
[as expressed through the nonlinear term $\mathcal{S}^j(\varphi, \g \varphi,u)$ in \eqref{E:velocityintro}]
is ``sufficiently weak." 
More precisely, in our main problem of interest, the analogs of the model nonlinearity 
$\mathcal{S}^j$ are the error terms $\triangle^j$ given by
\eqref{AE:trianglejdef}, and the main top-order spatial derivative Sobolev
estimate of $\triangle^j$ is \eqref{E:PARTIALIFLUIDTRIANGLEJHNMINUSONE}
(the norms ``$\totalellipticnorm{N-1}$'' on the right-hand side of \eqref{E:PARTIALIFLUIDTRIANGLEJHNMINUSONE}
are norms that are controlled with the help of elliptic estimates).
Our proof of \eqref{E:PARTIALIFLUIDTRIANGLEJHNMINUSONE}
reveals that the terms in $\triangle^j$ 
involving the derivatives of the metric
(and that therefore need to be bounded with elliptic estimates at the top order)
are quadratically small and have factors $u^a$ that 
[thanks to the dissipation provided by the cosmological constant and in particular 
the term $-2\omega u^j$ on the right-hand side of \eqref{E:velocityintro}]
are \emph{exponentially decaying so rapidly that they more than counter the bad top-order metric component growth} 
that can arise because of the elliptic estimates with unfavorable $t-$weights.
That is, roughly speaking, the terms $\triangle^j$ are of the form
$(\mbox{growing metric term}) \times (\mbox{even more rapidly decaying terms}).$
This is what we mean by ``sufficiently weak" coupling.

\subsection{On ``dangerous linear terms" and decoupling phenomena}
One of the recurring themes in the preset article is the presence of the so-called ``dangerous linear terms"
with borderline $t-$behavior, 
which were explained above in the case of the simplified model problem. 
In the full dust-Einstein problem under consideration, the remarkable fact is that
in each case where such a linear term arises, we find that the linear term can be controlled by energies that
we can \emph{first effectively bound by independent estimates.} 
This crucial partial decoupling phenomenon, 
which is present due to the particular tensorial structure of the dust-Einstein equations, 
is the main reason that our global estimates close.
We stress that the full structure is somewhat intricate and it is revealed in 
Props. \ref{P:integralinequalities}
and \ref{P:integralinequalitiesmetric}. 
Moreover, to prove the main theorem we must use the Gronwall estimates in the {\em correct} order and this is best reflected in the proof of the a-priori estimates
in Prop.~\ref{P:auxiliary}.

\subsection{Counting Principle for the error terms}
The most difficult aspects of our proof of future-global existence 
are our derivation of energy identities for 
the metric and fluid variables and 
our derivation 
of suitable Sobolev estimates
for the error terms present in the integrals on the right-hand sides of the 
energy identities. We carry out this analysis in
Sects. ~\ref{S:differentialinequalities} and ~\ref{S:sobolev} respectively.
In order to facilitate an otherwise lengthy and technical process of estimating the error terms, 
we introduce a Counting Principle (see Lemma~\ref{L:countingprinciple} and Cor. \ref{C:COUNTINGPRINCIPLE}). 
The Counting Principle has been designed to reduce the verification of many of the $H^{N-1}$ estimates
of Sect.~\ref{S:sobolev} to \emph{counting net downstairs spatial indices in a product}. 
Our Counting Principle naturally incorporates the ``loss of one derivative"-phenomenon explained above
as well as the aforementioned potential rapid $t-$growth in the top-order spatial derivatives compared to \cite{jS2012}
(which is connected to the unfavorable $t-$weights in the elliptic estimates).
The ``counting" refers to the difference of the total number of downstairs and upstairs spatial indices in a given
error term product, and this integer will roughly correspond to the total growth/decay rate of the product.
For instance, if we are trying to bound the products
\[
\|(\g_tg^{a0})\u g_{ab}\|_{H^{N-1}},\quad \|(\g_tg^{a0})\g_i\u g_{ab}\|_{H^{N-1}},
\]
we note that the differences of the total number of downstairs and upstairs spatial indices are $1$ and $2$ respectively. 
A single application of the operator
$\g_t$ or the operator $\u$ is
neutral from the point of view of index counting.
The Counting Principle will then imply that the first norm above
can grow at most like $\epsilon e^{\Omega(t)},$ while the second term
potentially grows at a rate $\epsilon e^{2\Omega(t)}.$
Here $e^{\Omega(t)}=a(t) \sim e^{Ht}$ and $\epsilon$ is a positive number corresponding to  
the small size of our solution norms. In fact, it turns out that we can organize the terms into two sets 
(called $\mathcal{G}_{N-1}$ and $\mathcal{H}_{N-1}$) so that the number 
$M\in\mathbb{N}_0$ of factors from the set $\mathcal{G}_{N-1}$ in a given product to be bounded 
will collectively contribute an \emph{additional} decay factor of $e^{-Mq\Omega(t)}$, where $q>0$ is a small number. 
The extra decay, which is present in many products, is an essential ingredient 
in our derivation of suitable a priori energy estimates 
(see the important factors $e^{-q H \tau}$ present 
in some terms on the right-hand sides of the energy 
inequalities of Props. \ref{P:integralinequalities} and \ref{P:integralinequalitiesmetric}). 
The exact definitions of $\mathcal{G}_{N-1}$ and $\mathcal{H}_{N-1}$ are provided in Def.~\ref{D:GandH}. 
We remark that the Counting Principle does not provide {\em sharp} decay/growth rates for the lower-order derivatives of the unknowns,
nor is it very helpful for estimating commutator-type error terms or terms involving the derivatives 
$\g_{tt},$ $\g_{ttt},$ or $\g_t \u;$ 
we handle such terms separately, ``by hand." 
We remark that more precise rates of decay/growth for the lower-order derivatives of
various quantities can be obtained as in the Asymptotics section of~\cite{jS2012}, 
but we will not concern ourselves with the sharp rates in this paper.

\subsection{Plan of the paper}
In Sect.~\ref{S:Notation}, we explain various notational conventions used throughout the paper.
In Sect.~\ref{S:dustEinstein}, we introduce the dust-Einstein system, the FLRW solutions, and the modified 
equations (which reflect our particular wave coordinate gauge). 
We also discuss the local well-posedness of the initial value problem for the modified equations 
and a closely related continuation principle.
In Sect.~\ref{S:norms}, we define the $t-$weighted norms and the energies that are crucial to our analysis.
In Sect.~\ref{S:differentialinequalities}, we provide the basic differential inequalities for our energies and 
formulate the bootstrap assumptions. Sect.~\ref{S:sobolev} is devoted to a series of
error term estimates that are essential for closing our high-order energy estimates. In Sect.~\ref{S:equivalence}, 
we prove the equivalence of the norms and the energies. We 
derive a priori energy estimates for near-FLRW solutions and
prove our main future stability theorem in Sect.~\ref{S:global}. 

\section{Notation} \label{S:Notation}
Our notation and conventions in this article are the same as those used in \cite{iRjS2012,jS2012}; we repeat them for convenience.

\subsection{Index conventions}
Greek ``spacetime'' indices $\alpha, \beta, \cdots$ take on the values $0,1,2,3,$ while Latin ``spatial'' indices $a,b,\cdots$ 
take on the values $1,2,3.$ Repeated indices are summed over (from $0$ to $3$ if they are Greek, and from $1$ to $3$ if they are Latin). Indices are lowered and raised with the spacetime metric $g_{\mu \nu}$ and its inverse $g^{\mu \nu}.$ Exceptions to this rule include the constraint Eqs. \eqref{E:Gauss}-\eqref{E:Codazzi}, in which we use the $3-$metric $\mathring{\underline{g}}_{jk}$ and its inverse $\mathring{\underline{g}}^{jk}$ to lower and raise indices.

\subsection{Coordinate systems and differential operators} \label{SS:COORDINATES}
We often work in a fixed standard local coordinate system $(x^1,x^2,x^3)$ on $\mathbb{T}^3.$ The vectorfields $\partial_j \eqdef \frac{\partial}{\partial x^{j}}$ are globally well-defined even though the coordinates themselves are not. This coordinate system extends to a local coordinate system $(x^0,x^1,x^2,x^3)$ on manifolds-with-boundary of the form $[0,T) \times \mathbb{T}^3,$ and we often write $t$ instead of $x^0.$ Relative to this coordinate system, the FLRW metric $\widetilde{g}$ is of the form \eqref{E:backgroundmetricform}. The symbol $\partial_{\mu}$ denotes the coordinate partial derivative $\frac{\partial}{\partial x^{\mu}},$ and we often write $\partial_t$ instead of $\partial_0.$ All of our estimates are derived in the fixed frame $\big\lbrace \partial_{\mu} \big\rbrace_{\mu = 0,1,2,3}.$ 
We also use the following shorthand notation for two and three repeated applications of the time differentiation operator $\g_t:$
\[
\g_{tt}\eqdef\g_t\g_t,\quad \g_{ttt}\eqdef\g_t\g_t\g_t.
\]
If $\vec{\alpha} = (n_1,n_2,n_3)$ is a triple of non-negative integers, then we define the spatial multi-index coordinate differential operator $\partial_{\vec{\alpha}}$ by $\partial_{\vec{\alpha}} \eqdef \partial_1^{n_1} \partial_2^{n_2} \partial_3^{n_3}.$ The symbol $|\vec{\alpha}| \eqdef n_1 + n_2 + n_3$ denotes the order of $\vec{\alpha}.$ 
If $\vec{\beta}$ is another triple of non-negative integers $\vec{\beta}=(m_1,m_2,m_3)$, then we define the
multi-index version of the binomial coefficient
\[
	{\vec{\alpha} \choose \vec{\beta}} \eqdef \prod_{i=1}^3{n_i \choose m_i}.
\]
We write 
\begin{align}
	D_{\mu} T_{\mu_1 \cdots \mu_s}^{\nu_1 \cdots \nu_r} = 
		\partial_{\mu} T_{\mu_1 \cdots \mu_s}^{\nu_1 \cdots \nu_r} + 
		\sum_{a=1}^r \Gamma_{\mu \ \alpha}^{\ \nu_{a}} T_{\mu_1 \cdots \mu_s}^{\nu_1 \cdots \nu_{a-1} \alpha \nu_{a+1} \nu_r} - 
		\sum_{a=1}^s \Gamma_{\mu \ \mu_{a}}^{\ \alpha} T_{\mu_1 \cdots \mu_{a-1} \alpha \mu_{a+1} \mu_s}^{\nu_1 \cdots \nu_r} \nonumber
\end{align} 
to denote the components of the covariant derivative of a tensorfield $T_{\mu_1 \cdots \mu_s}^{\nu_1 \cdots \nu_r}$ defined on
$\mathcal{M}.$ The Christoffel symbol $\Gamma_{\mu \ \nu}^{\ \alpha}$ is defined in \eqref{E:EMBIChristoffeldef}.

The notation $\partial^{(N)} T_{\mu_1 \cdots \mu_s}^{\nu_1 \cdots \nu_r}$ denotes the array containing all $N^{th}$ order \emph{spacetime} coordinate derivatives (including time derivatives) of the component $T_{\mu_1 \cdots \mu_s}^{\nu_1 \cdots \nu_r}.$ $\underpartial^{(N)} T_{\mu_1 \cdots \mu_s}^{\nu_1 \cdots \nu_r}$ denotes the array containing all $N^{th}$ order \emph{spatial coordinate} derivatives of the component $T_{\mu_1 \cdots \mu_s}^{\nu_1 \cdots \nu_r}.$ When $N=1,$ we omit the superscript.

\subsection{Identification of spacetime tensors and spatial tensors} 
\label{SS:IDENTIFICATIONS} 
It will often be convenient to view $\mathbb{T}^3$ as an embedded submanifold of the spacetime under an embedding
$\iota_t: \mathbb{T}^3 \hookrightarrow \lbrace t \rbrace \times \mathbb{T}^3 \subset \mathcal{M},$ 
$\iota_t(x^1,x^2,x^3) \eqdef (t,x^1,x^2,x^3).$ We often suppress the embedding by identifying $\mathbb{T}^3$ with its image $\iota_t(\mathbb{T}^3),$ which is a constant-time hypersurface in $\mathcal{M}.$ If $T_{k_1 \cdots k_s}^{j_1 \cdots j_r}$ is a $\mathbb{T}^3-$inherent ``spatial'' tensorfield, then there is a unique ``spacetime'' tensorfield  $T_{\mu_1 \cdots \mu_s}^{'\nu_1 \cdots \nu_r}$ defined along $\iota_t (\mathbb{T}^3) \simeq \mathbb{T}^3$ such that $\iota_t^* T' = T$ and such that $T'$ is tangent to $\iota_t(\mathbb{T}^3).$ Here $\iota_t^*$ denotes the pullback by $\iota_t,$ and we recall that $T_{\mu_1 \cdots \mu_s}^{'\nu_1 \cdots \nu_r}$ is tangent to $\iota_t (\mathbb{T}^3)$ if any contraction of any upstairs (downstairs) index with the unit normal covector $\hat{N}_{\mu}$ (unit normal vector $\hat{N}^{\mu}$) results in $0;$ for downstairs indices, this notion depends on the spacetime metric $g_{\mu \nu}.$ 
To pullback the upstairs indices of $T'$, we first lower them, then pull them back, and then we raise them again. We sometimes identify $T$ with $T'$ and use the same symbol to denote both, e.g. $T_{k_
1 \cdots k_s}^{j_1 \cdots j_r} \simeq T_{\mu_1 \cdots \mu_s}^{\nu_1 \cdots \nu_r}.$ We often use this identification along the initial data Cauchy hypersurface $\mathring{\Sigma} \simeq \mathbb{T}^3.$ For example, we alternate between viewing $\mathring{\underline{g}}$ as a $\mathring{\Sigma}-$inherent Riemannian metric $\mathring{\underline{g}}_{jk},$ and as a spacetime tensorfield $\mathring{\underline{g}}_{\mu \nu}$ defined along the embedded hypersurface $\mathring{\Sigma} \subset \mathcal{M}$ 
[that is, viewing $\mathring{\underline{g}}_{\mu \nu}$ as the first fundamental form of $\mathring{\Sigma}$ relative to $(\mathcal{M},g)$]. All of these standard identifications should be clear in context.

\subsection{Norms} \label{SS:NORMS}
All of the Sobolev norms we use are defined relative to the local coordinate system $(x^1,x^2,x^3)$ on $\mathbb{T}^3$ introduced in Sect. \ref{SS:COORDINATES}. Our norms are not coordinate invariant quantities, since we work with the norms of the scalar-valued \emph{components} of tensorfields relative to this coordinate system. If $f$ is a function defined on the hypersurface $\lbrace x \in \mathcal{M} \ | \ t = const \rbrace \simeq \mathbb{T}^3,$ then (relative to this coordinate system), we define the standard Sobolev norm $\big\| f \big\|_{H^N}$ as follows:

\begin{align} 
	\big\| f \big\|_{H^N} \eqdef
		\bigg( \sum_{|\vec{\alpha}| \leq N} 
		\int_{\mathbb{T}^3} \big|\partial_{\vec{\alpha}} f(t,x^1,x^2,x^3) \big|^2 
		dx \bigg)^{1/2}. \nonumber
\end{align}
Above, the notation $``\int_{\mathbb{T}^3} f \, dx"$ denotes the integral of $f$ over $\mathbb{T}^3$ with respect to the measure corresponding to the volume form of the \emph{standard Euclidean metric} on $\mathbb{T}^3.$
We denote the $N^{th}$ order homogeneous Sobolev norm of $f$ by
\begin{align}
	\big\| \underpartial^{(N)} f \big\|_{L^2} 
		\eqdef
		\bigg( \sum_{|\vec{\alpha}| = N} 
		\int_{\mathbb{T}^3} \big|\partial_{\vec{\alpha}} f(t,x^1,x^2,x^3) \big|^2 
		dx \bigg)^{1/2}. \nonumber
\end{align}
We often omit the superscript when $N=1.$ We also use the common notation:
    \begin{align}
    	\| F \|_{L^{\infty}} & \eqdef \mbox{ess} \sup_{x \in \mathbb{T}^3} |F(x)|, \nonumber \\
    	\| F \|_{C_b^N} & \eqdef \sum_{|\vec{\alpha}| \leq N} \big\| \partial_{\vec{\alpha}} F \|_{L^{\infty}}. \nonumber
    \end{align} 		
		If $I \subset \mathbb{R}$ is an interval and $X$ is a normed function space, then $C^N(I,X)$ denotes the set of 
		$N$-times continuously differentiable maps from $I$ into $X.$
	
We use $C$ to denote a generic positive constant that may change from line to line. It may depend on $N$ and $\Lambda$, but never on $(g_{\mu \nu}, \rho, u^j),$ $(\mu, \nu = 0,1,2,3),$ $(j=1,2,3).$

Given two quantities $A$ and $B,$ we write
\begin{align*}
	A \lesssim B
\end{align*}	
to mean that there exists a constant $C > 0$ such that $A \leq C B.$
We write 
\begin{align*}
	A \approx B
\end{align*}
whenever $A \lesssim B$ and $B \lesssim A.$

\setcounter{equation}{0}
\section{The dust-Einstein system}\label{S:dustEinstein}
\subsection{The initial value problem for the dust-Einstein system}\label{SS:dustEinstein}
In this section we formulate the initial value problem for
and provide the definition of a solution to the dust-Einstein system.
\footnote{
For more detailed information on the physical derivation
of the dust-Einstein system and a general introduction to the perfect fluids in the context of
general relativity, we refer the reader to Sect. 3.1 of \cite{jS2012}, the
standard references~\cite{sHgE1973,rW1984}, as well as the
Christodoulou's survey article~\cite{dC2008}.}
The Einstein-field equation~\eqref{E:metricge} can be rewritten in the following manner:
\begin{align} \label{E:metricgealternative}
	\text{Ric}_{\mu\nu}-\Lambda g_{\mu\nu}+\frac{1}{2}T g_{\mu\nu}-T_{\mu\nu} & = 0, & & (\mu,\nu=0,1,2,3),
\end{align}
where $T_{\mu\nu}=\rho u_{\mu}u_{\nu}$. To see that~\eqref{E:metricgealternative} holds, we take
the trace of each side of \eqref{E:metricge} and easily obtain
the relationship $R=4\Lambda-T$, where $T$ denotes the trace of the tensor $T_{\mu\nu}$. 
Inserting this relation into \eqref{E:metricge}, we deduce~\eqref{E:metricgealternative}.
It follows that the dust-Einstein system comprises equation
\eqref{E:metricgealternative} 
together with the equations of motion for a perfect fluid:
\begin{align} \label{E:dustalternative}
	D_{\alpha}T^{\alpha\mu} &=0, && (\mu=0,1,2,3),
\end{align}
and the normalization condition $g_{\alpha\beta}u^{\alpha}u^{\beta}=-1$.
By first contracting \eqref{E:dustalternative} against $u^{\mu}$
and then projecting \eqref{E:dustalternative} onto 
the $g-$orthogonal complement of $u$ via the projection
\[
\Pi_{\mu\nu} \eqdef u_{\mu}u_{\nu}+g_{\mu\nu},
\]
we obtain the following equivalent formulation of \eqref{E:dustalternative}:
\begin{subequations}
\begin{align}
u^{\alpha} D_{\alpha}u^j & = 0, &&  (j=1,2,3), \label{E:dustalternative1}\\
u^{\alpha}D_{\alpha}\rho+\rho D_{\alpha}u^{\alpha} & = 0, && \label{E:dustalternative2}
\end{align}
\end{subequations}
where 
\be\label{E:massshellalternative}
g_{\alpha\beta}u^{\alpha}u^{\beta}=-1.
\ee
Regarding the gravitational quantities, $\mbox{Ric}_{\mu \nu}$ is the \emph{Ricci curvature tensor}, $R$ is the \emph{scalar curvature}, and $\Lambda$ is the \emph{cosmological constant}. $\mbox{Ric}_{\mu \nu}$ and $R$ can be expressed in terms of the \emph{Riemann curvature tensor}\footnote{Under our sign convention, $D_{\mu} D_{\nu} X_{\alpha} - D_{\nu} D_{\mu} X_{\alpha} = \mbox{Riem}_{\mu \nu \alpha}^{\ \ \ \ \beta} X_{\beta}.$} $\mbox{Riem}_{\mu \alpha \nu}^{\ \ \ \ \beta}$, which is expressible in terms of the \emph{Christoffel symbols} $\Gamma_{\mu \ \nu}^{\ \alpha}$ of the \emph{spacetime metric} $g_{\mu \nu}.$ 
Relative to an arbitrary local coordinate system, these quantities can be expressed as follows:
\begin{subequations}
\begin{align}
\mbox{Riem}_{\mu \alpha \nu}^{\ \ \ \ \beta} & \eqdef 
	\partial_{\alpha} \Gamma_{\mu \ \nu}^{\ \beta}
	- \partial_{\mu} \Gamma_{\alpha \ \nu}^{\ \beta}
  + \Gamma_{\alpha \ \lambda}^{\ \beta} \Gamma_{\mu \ \nu}^{\ \lambda}
	- \Gamma_{\mu \ \lambda}^{\ \beta} \Gamma_{\alpha \ \nu}^{\ \lambda},	
	\nonumber \\
\mbox{Ric}_{\mu \nu} & \eqdef \mbox{Riem}_{\mu \alpha \nu}^{\ \ \ \ \alpha} 
	= \partial_{\alpha} \Gamma_{\mu \ \nu}^{\ \alpha}
	- \partial_{\mu} \Gamma_{\alpha \ \nu}^{\ \alpha}
  + \Gamma_{\alpha \ \lambda}^{\ \alpha} \Gamma_{\mu \ \nu}^{\ \lambda}
	- \Gamma_{\mu \ \lambda}^{\ \alpha} \Gamma_{\alpha \ \nu}^{\ \lambda},
	\nonumber \\
R & \eqdef g^{\alpha \beta} \mbox{Ric}_{\alpha \beta}, \nonumber\\
\Gamma_{\mu \ \nu}^{\ \alpha} & \eqdef \frac{1}{2} g^{\alpha \lambda}(\partial_{\mu} g_{\lambda \nu} 
	+ \partial_{\nu} g_{\mu \lambda} - \partial_{\lambda} g_{\mu \nu}). \label{E:EMBIChristoffeldef}
\end{align}
\end{subequations}

An initial data set for the dust-Einstein Eqs. \eqref{E:metricgealternative} + \eqref{E:dustalternative1}-\eqref{E:dustalternative2}
consists of the quintuple 
$(\mathring{\Sigma}, \mathring{\underline{g}}_{jk}, \mathring{\underline{K}}_{jk}, \mathring{\rho}, \underline{\mathring{u}}^j),$ $(j,k = 1,2,3)$.
Here $\mathring{\Sigma}$ is the initial data hypersurface, $\mathring{\underline{g}}_{jk}$ is the restriction 
of the spacetime metric $g$ to $\mathring{\Sigma}$, $\mathring{\underline{K}}_{jk}$ is the
prescribed second fundamental form of $\mathring{\Sigma}$ relative to $g$,
$\mathring{\rho}:\mathring{\Sigma}\to\R_+$ is the initial dust-density, and 
$\underline{\mathring{u}}^j$ is the $g-$orthogonal projection of the initial 
four-velocity onto $\mathring{\Sigma}.$ Solving the initial value problem means constructing the solution launched by the data.
The solution consists of a $4-$dimensional manifold $\mathcal{M},$ a Lorentzian metric $g_{\mu \nu},$ a function $\rho,$ a future-directed unit-normalized vectorfield $u^{\mu},$ $(\mu, \nu = 0,1,2,3),$ on $\mathcal{M}$ satisfying 
\eqref{E:metricgealternative} + \eqref{E:dustalternative1}-\eqref{E:dustalternative2}, 
and an embedding $\mathring{\Sigma} \hookrightarrow \mathcal{M}$ such that $\mathring{\underline{g}}_{jk}$ is the first fundamental form\footnote{Recall that $\mathring{\underline{g}}$ is defined at the point $x$ by $\mathring{\underline{g}}(X,Y) = g(X,Y)$ for all $X,Y \in T_x \mathring{\Sigma}.$} of $\mathring{\Sigma},$ $\mathring{\underline{K}}_{jk}$ is the second fundamental form\footnote{Recall that $\mathring{\underline{K}}$ is defined at the point $x$ by $\mathring{\underline{K}}(X,Y) \eqdef g(D_{X} \hat{N},Y)$ for all $X,Y \in T_x \mathring{\Sigma},$ where $\hat{N}$ is the future-directed unit normal to $\mathring{\Sigma}$ at $x.$} of $\mathring{\Sigma},$ the restriction of $\rho$ to $\mathring{\Sigma}$ is $\mathring{\rho},$ and $(\underline{\mathring{u}}^1, \underline{\mathring{u}}^2, \underline{\mathring{u}}^3)$ is the $g-$orthogonal projection of $(u^0,u^1,u^2,u^3)$ onto $\mathring{\Sigma};$ see Sect. \ref{SS:IDENTIFICATIONS} for a summary of the conventions we use for identifying tensors inherent to $\mathring{\Sigma}$ with spacetime tensors.

In order for $(\mathring{\Sigma}, \mathring{\underline{g}}_{jk}, \mathring{\underline{K}}_{jk}, \mathring{\rho}, \underline{\mathring{u}}^j)$ to be viable data for the Einstein equations,
the following constraint equations have to be satisfied (see for example, \cite[Ch. 10]{rW1984}):
\begin{subequations}
\begin{align}
	\mathring{\underline{R}} - \mathring{\underline{K}}_{ab} \mathring{\underline{K}}^{ab} + (\mathring{\underline{g}}^{ab} \mathring{\underline{K}}_{ab})^2 - 2 \Lambda & = 2 T(\hat{N},\hat{N})|_{\mathring{\Sigma}}, \label{E:Gauss} \\
	\mathring{\underline{D}}^a \mathring{\underline{K}}_{aj} - \mathring{\underline{g}}^{ab}  \mathring{\underline{D}}_j \mathring{\underline{K}}_{ab}  & = 
		T(\hat{N},\frac{\partial}{\partial x^j})|_{\mathring{\Sigma}}. \label{E:Codazzi}
\end{align}
\end{subequations}
The above equations are known as the {\em Gauss} and the {\em Codazzi} equation respectively. Above, $\mathring{\underline{D}}$
denotes the Levi-Civita connection of $\mathring{\underline{g}},$ while $\hat{N}$ denotes the future-directed unit normal to $\mathring{\Sigma}$.

\subsection{FLRW Background Solutions} \label{S:backgroundsolution}
We now briefly recall some standard facts concerning the FLRW background solutions to the Euler-Einstein equations; 
readers can consult \cite[Ch. 5]{rW1984} or \cite{iRjS2012} for more details. The FLRW solutions are derived under the ansatz that 
the background metric $\widetilde{g}$ has the following warped product form relative to a standard local coordinate system on 
$(-\infty,\infty) \times \mathbb{T}^3:$
\begin{align} \label{E:backgroundmetricform}
	\widetilde{g} = -dt^2 + a^2(t) \sum_{i=1}^3 (dx^i)^2,
\end{align}
where $a(t) \geq 0$ is the scale factor. For simplicity, we will assume for the remainder of the article that
$	a(0) = 1.$
It is then easy to verify [using definition \eqref{E:EMBIChristoffeldef}] that the only non-vanishing Christoffel symbols of $\widetilde{g}$ are
\begin{align*} 
	\widetilde{\Gamma}_{j \ k}^{\ 0} = \widetilde{\Gamma}_{k \ j}^{\ 0}  = a(t) \frac{d}{dt} a(t) \delta_{jk}, 
		&& \widetilde{\Gamma}_{j \ 0}^{\ k} = \widetilde{\Gamma}_{0 \ j}^{\ k}  = \omega(t) \delta_j^k, && (j,k=1,2,3), 
\end{align*}
where 
\begin{align}
	\omega(t) \eqdef \frac{1}{a(t)} \frac{d}{dt} a(t). \label{E:omegadef}
\end{align}
We then make the following assumptions about the background fluid solution variables: $\widetilde{\rho} = \widetilde{\rho}(t)$ and $\widetilde{u}^{\mu} \equiv (1,0,0,0).$ Inserting these assumptions 
into the dust-Einstein Eqs. \eqref{E:metricgealternative} +~\eqref{E:dustalternative1}-~\eqref{E:dustalternative2} and performing routine computations, one derives the \emph{Friedmann equations}:
\begin{subequations}
\begin{align}
	\widetilde{\rho}(t) [a(t)]^{3} & \equiv \bar{\dens},
		\label{E:backgroundrhoafact} \\
	\frac{d}{dt} a(t) = a(t) \sqrt{\frac{\Lambda}{3} + \frac{\widetilde{\rho}(t)}{3}} & = a(t) \sqrt{\frac{\Lambda}{3} 
		+ \frac{\bar{\dens}}{3[a(t)]^{3}}}, \label{E:backgroundaequation}
\end{align} 
\end{subequations}
where the constant $\bar{\dens} \geq 0$ denotes the initial energy density.  

\begin{remark}[\textbf{Compactly supported data}]
	We stress that in the present article, because of the special structure
	of the dust model, our methods apply in particular to case $\bar{\dens} = 0.$
	This stands in contrast to the cases $0 < c_s^2 \leq 1/3$ studied in \cite{iRjS2012,jS2012,cLjVK},
	where it was assumed that $\bar{\dens} > 0.$
\end{remark}

Eq.~\eqref{E:backgroundrhoafact} is a consequence of~\eqref{E:dustalternative2}, while~\eqref{E:backgroundaequation}
follows from plugging the above ansatz into the $00$-component of the Einstein-field equations. Furthermore, all the $0j$-components are equal to zero. To satisfy the $jk$-components of the field equations
we discover (see, for example, ~\cite[Sect. 4.1]{iRjS2012}) that the following second-order ODE must hold:
\[
2a\frac{d^2}{dt^2}a + \left(\frac{d}{dt}a\right)^2 -\Lambda a^2 = 0.
\]
However, it can be checked that the above relationship automatically follows from~\eqref{E:backgroundrhoafact} and~\eqref{E:backgroundaequation}.
Thus, the dust-Einstein system reduces to the Eqs. \eqref{E:backgroundrhoafact}-\eqref{E:backgroundaequation}. 
We now note that the ODE \eqref{E:backgroundaequation} suggests the asymptotic behavior $a(t) \sim e^{Ht}, H \eqdef \sqrt{\Lambda/3}.$ 
This behavior of $a(t),$ which is analyzed in greater detail in Lemma \ref{L:backgroundaoftestimate}, is the source of the rapid spacetime expansion.

For aesthetic reasons, we also introduce the quantity
\begin{align}
	\Omega(t) & \eqdef \ln a(t), \label{E:BigOmega}
\end{align}
which implies that
\begin{align}
	a(t)  = e^{\Omega(t)},  \ \
	\omega(t)  = \frac{d}{dt} \Omega(t). \nonumber
\end{align}
For future use, we also note the following identities, which follow from the above discussion and from straightforward computations; we leave it to the reader to verify the details:
\begin{subequations}
\begin{align}
	\frac{d}{dt} \omega(t) & = - \frac{1}{2} \widetilde{\rho}(t), \label{E:omegadotidentity} \\
	3 \frac{d}{dt} \omega(t) + 3 \omega^2(t) - \Lambda & = -\frac{1 }{2} \widetilde{\rho}(t) \label{E:omegadotidentity2}.
\end{align}
\end{subequations}

\subsubsection{Friedmann's equation}
The following lemma summarizes the asymptotic behavior of solutions to the ODE \eqref{E:backgroundaequation}. 
It will be used in derivation of error estimates in
Sect.~\ref{S:sobolev}.
\begin{lemma} [\textbf{Analysis of Friedmann's equation}] \cite[Sect. 4.2]{iRjS2012} \label{L:backgroundaoftestimate}
	Let $\Lambda > 0,$ $\bar{\dens} \geq 0$ be constants, and let $a(t)$ be the solution to the following ODE: 
	\begin{align*}
		\frac{d}{dt}a(t) & = a(t) \sqrt{\frac{\Lambda}{3} + \frac{\bar{\dens}}{3[a(t)]^3}}, && a(0) = 1.
	\end{align*}
	Then with $H \eqdef \sqrt{\Lambda/3},$ the solution $a(t)$ is given by
	\begin{align*}
		a(t) & = \bigg\lbrace\sinh \Big(\frac{3 H t}{2}\Big) 
			\sqrt{\frac{\bar{\dens}}{3H^2} + 1}
		 	+	\cosh \Big(\frac{3 H t}{2} \Big) \bigg\rbrace^{2/3}, 
	\end{align*}
	and for all integers $N \geq 0,$ there exists a constant $C_N > 0$ such that for all $t \geq 0,$ with 
	$A \eqdef \bigg\lbrace \frac{1}{2} \Big(\sqrt{\frac{\bar{\dens}}{3H^2} + 1} + 1 \Big) \bigg\rbrace^{2/3},$
	we have that
	\begin{subequations}
	\begin{align*}
		(1/2)^{2/3} e^{Ht} \leq a(t) & \leq A e^{Ht}, \\
		\Big| e^{-Ht} \frac{d^N}{dt^N}a(t) - A H^N \Big| & \leq C_Ne^{-3 Ht}.
	\end{align*}
	\end{subequations}
	Furthermore, for all integers $N \geq 0,$ there exists a constant $\widetilde{C}_N > 0$ such that
	for all $t \geq 0,$ with 
	\begin{align}  \label{E:LITTLEOMEGADEFINED}
		\omega(t) \eqdef \frac{1}{a(t)} \frac{d}{dt} a(t) = \frac{d}{dt} \Omega(t), 
	\end{align}	
	 we have that
	\begin{align*}
		H \leq \omega(t)  \leq \sqrt{H^2 + \frac{\bar{\dens}}{3}}, \ \ \ \
		\Big| \frac{d^N}{dt^N}\big(\omega(t) - H\big) \Big|  \leq \widetilde{C}_N e^{-3 Ht}.
	\end{align*}
\end{lemma}

%
\subsection{Modified Dust-Einstein equations}\label{S:modified}
In this section we reformulate the dust-Einstein system \eqref{E:metricgealternative} + 
\eqref{E:dustalternative1}-\eqref{E:massshellalternative} relative to a {\em wave coordinate}
system. More precisely, we impose the gauge condition
\be\label{E:gauge}
\Gamma^{\mu}=\tilde{\Gamma}^{\mu},
\ee
where $\Gamma^{\mu} \eqdef g^{\alpha\beta}\Gamma_{\alpha \ \beta}^{\ \mu}$ is a contracted Christoffel symbol
of the spacetime metric $g$ and
$\tilde{\Gamma}^{\mu} \eqdef \tilde{g}^{\alpha\beta}\tilde{\Gamma}^{\mu}_{\alpha\beta}$
is a contracted Christoffel symbol of the FLRW background-metric $\tilde{g}=-dt^2+e^{2\Omega(t)}\sum_{i=1}^3(dx^i)^2$.
The wave coordinate condition \eqref{E:gauge} has several remarkable features.
First, it allows one to replace the dust-Einstein equations with a manifestly hyperbolic ``modified'' system of quasilinear wave equations coupled to the dust equations. The following fact is essential: when the gauge condition holds, the modified system is equivalent to the original system.
Second, the condition \eqref{E:gauge} leads to the presence of 
\emph{dissipative} terms in the modified equations (see Sect. \ref{SS:MODIFIEDEQNS}). This dissipation is the primary
reason that we are able to prove our main future stability theorem.
Finally, the gauge condition is preserved by the flow of the modified equations 
if it is satisfied initially. The latter observation allows us to infer that the solutions of the modified problem are indeed the solutions to the original dust-Einstein system. 

Without presenting the complete details of this calculation 
(for which we refer the reader to Sect. 5 of~\cite{jS2012}) 
we now sketch the derivation of the modified system. 
To this end, we first introduce the
{\em modified Ricci tensor} $\widehat{\text{Ric}}_{\mu\nu}:$ 
\begin{align} \label{E:modifiedRicci}
	\widehat{\mbox{Ric}}_{\mu \nu} 
	& \eqdef \mbox{Ric}_{\mu \nu} 
	+ \frac{1}{2} \left\lbrace D_{\mu} (3 \omega(t) g_{\nu 0} - \Gamma_{\nu}) + D_{\nu} (3 \omega(t) g_{\mu 0} - \Gamma_{\mu}) \right\rbrace  			\\
	& = -\frac{1}{2} \hat{\square}_g g_{\mu \nu} 
		+ \frac{3}{2} \omega(t) \big(D_{\mu} g_{\nu 0} + D_{\nu} g_{\mu 0} \big) 
		+ g^{\alpha \beta} g^{\gamma \delta} (\Gamma_{\alpha \gamma \mu} \Gamma_{\beta \delta \nu}
		+ \Gamma_{\alpha \gamma \mu} \Gamma_{\beta \nu \delta}
		+ \Gamma_{\alpha \gamma \nu} \Gamma_{\beta \mu \delta}), \notag
\end{align}
where $3 \omega(t) g_{\nu 0} - \Gamma_{\nu}$ and $g_{\nu 0}$ are treated as one-forms 
for the purposes of covariant differentiation and
\begin{align*}
\hat{\square}_g \eqdef g^{\alpha \beta} \partial_{\alpha} \partial_{\beta}
\end{align*}
is the {\em reduced wave operator} corresponding to $g$.
Roughly speaking, the modified Ricci tensor is defined by making well-chosen replacements of
$\Gamma^{\mu}$ with $\tilde{\Gamma}^{\mu}$ in the coordinate frame expansion of
$\mbox{Ric}_{\mu \nu}.$

We now simply state our definition of the modified dust-Einstein equations,
which, roughly speaking, is defined by making well-chosen replacements of
$\Gamma^{\mu}$ with $\tilde{\Gamma}^{\mu}$ 
in the dust-Einstein Eqs.
\eqref{E:metricgealternative} + \eqref{E:dustalternative1}-\eqref{E:massshellalternative}
(see Sect. 5 of~\cite{jS2012} for a detailed derivation):
\begin{subequations}
\label{E:modified}
\begin{align}
	\widehat{\mbox{Ric}}_{00} + 2 \omega \Gamma^0 - 6 \omega^2 - \Lambda g_{00} -\rho \Big(u_0^2
	+\frac{1}{2}g_{00}\Big)& = 0, \label{E:Rhat00} \\
	\widehat{\mbox{Ric}}_{0j} - 2 \omega (\Gamma_j - 3 \omega g_{0j}) - \Lambda g_{0j} - \rho \Big(u_0u_j+\frac{1}{2}g_{0j} \Big) & = 0, \\
	\widehat{\mbox{Ric}}_{jk} - \Lambda g_{jk} - \rho \Big(u_ju_k+\frac{1}{2}g_{jk}\Big) & = 0, \label{E:Rhatjk} \\
	u^{\alpha} \partial_{\alpha} \rho +\rho\,\partial_{\alpha}u^{\alpha}+\rho\, \Gamma_{\alpha \ \beta}^{\ \alpha}u^{\beta} & = 0, \label{E:firstEulersummary} \\
	u^{\alpha} \partial_{\alpha} u^{j} + \Gamma_{\alpha \ \beta}^{\ j}u^{\alpha}u^{\beta} & = 0, \label{E:secondEulersummary} 
\end{align}
\end{subequations}
where $g_{\alpha \beta} u^{\alpha} u^{\beta} = -1$ 
and $\omega$ is defined in~\eqref{E:omegadef}. We stress that 
under the gauge condition \eqref{E:gauge}, the modified equations are equivalent to the original dust-Einstein equations.

\begin{remark} [\textbf{Hyperbolic equations}] \label{E:MODIFIEDSYSTEMROUGHCLASSIFICATION}
The relation \eqref{E:modifiedRicci} shows that the modified system is a system of quasilinear 
wave equations for the metric components coupled to the dust equations.
\end{remark}

\subsection{Initial data for the modified system}
As we explained in Sect.~\ref{SS:dustEinstein}, an initial data set 
for the dust-Einstein Eqs. \eqref{E:metricgealternative} + \eqref{E:dustalternative1}-\eqref{E:massshellalternative} consists of the quintuple 
$(\mathring{\Sigma}, \mathring{\underline{g}}_{jk}, \mathring{\underline{K}}_{jk}, \mathring{\rho}, \underline{\mathring{u}}^j),$ $(j,k = 1,2,3)$.
We assume that the data have
the properties described at the end of Sect. \ref{SS:dustEinstein} and that
in particular, they verify the constraint Eqs. \eqref{E:Gauss}-\eqref{E:Codazzi}.
Our present goal is to use these data to construct suitable data for the modified 
Eqs. \eqref{E:Rhat00}-\eqref{E:secondEulersummary}.
In particular, we demand that the modified data verify the gauge condition $\Gamma^{\mu}=\tilde{\Gamma}^{\mu}$ 
$(\mu = 0,1,2,3)$ at $t=0.$ In order to satisfy all conditions, we must specify the full spacetime metric along the initial 
time slice $\mathring{\Sigma}$ as well as the transversal derivatives $\g_tg_{\mu\nu}|_{t=0}.$
More precisely, standard calculations imply (see, for example \cite[Sect. 5.3]{jS2012})
that in order to meet all of our requirements including the gauge condition, it suffices to set
\begin{subequations}
\label{E:initialdata}
\begin{align}
	g_{00}|_{t=0} & = -1, \ \ g_{0j}|_{t=0} = 0, \ \ g_{jk}|_{t=0} = \mathring{\underline{g}}_{jk}, \label{E:initialdata1}\\
	\rho|_{t=0} & = \mathring{\rho}, \ \ u^j |_{t=0} = \underline{\mathring{u}}^j, \ \ u^0|_{t=0} = \sqrt{1 + 
		\mathring{\underline{g}}_{ab}\underline{\mathring{u}}^a 
		\underline{\mathring{u}}^b}, \notag\\
	(\partial_t g_{jk})|_{t=0} & = 2 \mathring{\underline{K}}_{jk}.\notag
\end{align}
\end{subequations}
\begin{subequations}
\label{E:initialconstraint}
\begin{align}
	(\partial_{t} g_{00})|_{t=0} & =  2 \big(3 \omega(0)  - \mathring{\underline{g}}^{ab} 
		\mathring{\underline{K}}_{ab}\big),  \\
	(\partial_t g_{0j})|_{t=0} & = 
		\mathring{\underline{g}}^{ab}( \partial_a \mathring{\underline{g}}_{bj} - \frac{1}{2} \partial_j 
		\mathring{\underline{g}}_{ab}).
\end{align}
\end{subequations}
Above, $\mathring{\underline{g}}^{jk}$ denotes the inverse of $\mathring{\underline{g}}_{jk}.$ 
The fields $(\mathring{\Sigma}, g_{\mu \nu}|_{t=0}, \partial_t g_{\mu \nu}|_{t=0}, \rho|_{t=0}, u^j|_{t=0}),$ 
$(\mu,\nu = 0,1,2,3),$ $(j=1,2,3)$ as defined above form a complete set of data for the modified Eqs.
\eqref{E:Rhat00}-\eqref{E:secondEulersummary}, which have the basic structure described in Remark \ref{E:MODIFIEDSYSTEMROUGHCLASSIFICATION}.
As mentioned in the introduction, in order to ensure that the solutions to the modified 
system are also solutions to the original dust-Einstein equations, we have to show that the gauge condition \eqref{E:gauge}
is verified for all times and not merely initially. The proof is by now 
considered to be a standard result
that follows from the methods of \cite{CB1952}. 
The main idea is to show that for solutions to the modified equations, the quantities
$\Gamma_{\mu} - \tilde{\Gamma}_{\mu} \eqdef g_{\mu \alpha}(\Gamma^{\alpha} - \tilde{\Gamma}^{\alpha})$ 
verify a system of homogeneous wave equations with trivial initial data. From the standard uniqueness theorem
for wave equations, one concludes that $\Gamma_{\mu} - \tilde{\Gamma}_{\mu}$ completely vanishes.

\subsection{Decomposition of the modified equations}  \label{SS:MODIFIEDEQNS}
As we mentioned in the introduction, the basic mechanism that will allow us to prove the stability of the FLRW solutions is 
the dissipative effect induced by the positive cosmological constant.
To fully expose this effect, we will now decompose the modified dust-Einstein system \eqref{E:modified}.
We will work with the following rescaled version of the dust density $\rho:$
\be\label{E:r}
\dens \eqdef e^{3\Omega}\rho,
\ee
where we recall that $\Omega$ is defined in~\eqref{E:BigOmega}. We expect the quantity $\dens-\bar{\dens}$ to be small since we are trying to prove the stability of the background solution.
Recall that $\bar{\dens}$ is a constant associated with the background solution, introduced in~\eqref{E:backgroundrhoafact}.
Furthermore, we also rescale the spatial components $g_{jk}$ of the spacetime metric in the following fashion:
\be\label{E:h}
h_{jk}\eqdef e^{-2\Omega}g_{jk}.
\ee
Note that we expect $h_{jk}-\delta_{jk}$ to be small for the same reason as above.
The corresponding rescaled FLRW background solution variables verify $\dens \equiv \bar{\dens}$ and $h_{jk} \equiv \delta_{jk}$. 
In the remainder of the article, various error terms are denoted as indexed versions of the symbol $\triangle$.
\begin{proposition} [\textbf{Decomposition of the modified equations}] \label{P:Decomposition} 
Let $(g,u,\rho)$ be a solution to the modified dust-Einstein system~\eqref{E:Rhat00}-\eqref{E:secondEulersummary} and let the 
rescaled density $\dens$ and the rescaled metric components $h_{jk}$ be defined by~\eqref{E:r} and~\eqref{E:h} respectively.
Then the unknowns $(g_{00},g_{0j}, h_{jk}, u^j, \dens),$ $(j,k=1,2,3),$ 
satisfy the following system of equations:
\begin{subequations} 
\label{E:modifiedsystem}
\begin{alignat}{2}
\hat{\square}_g g_{00}&=&&5H\g_t g_{00}+6H^2(g_{00}+1)+\triangle_{00}\,,\label{E:metric00}\\
\hat{\square}_g g_{0j}&=&&3H\g_tg_{0j}+2H^2g_{0j}-2Hg^{ab}\Gamma_{ajb}+\triangle_{0j}\,,\label{E:metric0j}\\
\hat{\square}_g h_{jk}&=&&3H\g_th_{jk}+\triangle_{jk}\,,\label{E:metricjk} \\
u^{\alpha}\g_{\alpha}(\dens-\bar{\dens})
&=&&\triangle\,,\label{E:revol}\\
u^{\alpha}\g_{\alpha}u^j&=&&-2 u^0 \omega u^j - u^0 \triangle^{\ j}_{0 \ 0} +\triangle^j\,,\label{E:velevol}
\end{alignat}
\end{subequations}
where the harmless error terms $\triangle_{00}$, $\triangle_{0j}$, and $\triangle_{jk}$
are given by~(\ref{AE:triangle00}),~(\ref{AE:triangle0j}), and~(\ref{AE:trianglejk}) respectively,
the harmless error terms $\triangle$ and $\triangle^j$
are given by~(\ref{AE:triangledef}) and~(\ref{AE:trianglejdef}) respectively,
and the important linearly small Christoffel symbol error term $\triangle^{\ j}_{0 \ 0}$ is given by \eqref{AE:TRIANGLE0JUPPER0DEF}.
\end{proposition}
\begin{proof}
A detailed formulation and proof of the above proposition are provided in Prop.~\ref{AP:Decomposition}
of Appendix~\ref{S:APPENDIXA}.
\end{proof}
\subsection{Local well-posedness and the continuation principle}
\label{SS:LWPANDCONTINUATION}
We now state a local well-posedness theorem that is specialized to the data of interest. 
In particular, the small perturbations of the standard FLRW data, which we will consider in our global stability analysis, 
satisfy the assumptions of the following theorem.
\begin{theorem}[\textbf{Local well-posedness for the modified system}]\label{T:localwellposedness}
Let $N \ge 4$. Assume that the initial data $(\mathring{\Sigma}, \mathring{\underline{g}}_{jk}, \mathring{\underline{K}}_{jk}, \mathring{\rho}, \underline{\mathring{u}}^j)$ $(j,k = 1,2,3)$ verify the Einstein constraints \eqref{E:Gauss}-\eqref{E:Codazzi}.
Consider the corresponding initial data set for the modified equations 
\eqref{E:Rhat00}-\eqref{E:secondEulersummary}
as defined in Sect.~\ref{SS:dustEinstein}.
In particular, assume the relations ~\eqref{E:initialdata} and~\eqref{E:initialconstraint}, which imply that
the wave coordinate condition \eqref{E:gauge} holds at $t=0.$
Assume that the data have the following regularity properties:
\begin{subequations}
	\begin{align*}
		\underpartial \mathring{g}_{jk} & \in H^{N}; \ 	\mathring{K}_{jk} - \omega(0)\mathring{g}_{jk}  \in H^{N}; \\
	 \mathring{\rho} - \bar{\dens} & \in H^{N-1};
		 \ \mathring{u}^j \in H^N,
	\end{align*}
	\end{subequations}
	and assume further that there is a constant $C > 1$ such that
	\begin{align} \label{E:localexistencemathringgjklowerequivalenttostandardmetricagain}
		C^{-1} \delta_{ab} X^a X^b \leq \mathring{g}_{ab}X^a X^b \leq C \delta_{ab} X^a X^b,  \ \
		\forall(X^1,X^2,X^3) \in \mathbb{R}^3.
	\end{align}
Then these data launch a unique classical solution $(g_{\mu\nu}, u^{\mu}, \rho)$ 
$(\mu,\nu = 0,1,2,3)$
to the modified system on a slab of the form $(T_-,T_+)\times \T^3$, $T_-<0<T_+$, such that $(g_{jk})_{j,k=1,2,3}$
is uniformly positive definite and $g_{00}<0$ on $(T_-,T_+)\times\T^3$. Moreover, 
the solution has the following regularity properties:
%
%
\begin{subequations}
	\begin{align*}
		g_{00} + 1, \ g_{0j} \in C^0((T_-, T_+),H^{N+1}); \, \underpartial g_{jk} & \in C^0((T_-, T_+),H^{N}), \\
		\partial_t g_{00}, \ \partial_t g_{0j}, \, \partial_t g_{jk} - 2\omega(t) g_{jk} & \in C^0((T_-, T_+),H^{N}),  \\ 
	   \partial_{tt}g_{00}, \ \partial_{tt}g_{0j}, \ \partial_t (\partial_t g_{jk} - 2\omega(t) g_{jk}) & \in C^0((T_-, T_+),H^{N-1}), \\
	   \partial_t \u g_{00}, \ \partial_t \u g_{0j}, \ \u (\partial_t g_{jk} - 2\omega(t) g_{jk}) & \in C^0((T_-, T_+),H^{N-1}), \\
		e^{3 \Omega} \rho - \bar{\dens}, \ \u (e^{3 \Omega} \rho) & \in C^0((T_-, T_+),H^{N-1}), \\
		 u^0 - 1, \ u^j & \in C^0((T_-, T_+),H^{N}), \\
		 	\u u^j & \in C^0((T_-, T_+),H^{N-1}). 
	\end{align*}
	\end{subequations}	
Furthermore,  $g_{\mu\nu}$ is a Lorentzian metric on $(T_-,T_+)\times\T^3$, and the sets $\{t\}\times\T^3$ are Cauchy hypersurfaces
in the Lorentzian manifold $(\mathcal{M}\eqdef (T_-,T_+)\times\T^3,\,g)$ for any $t\in(T_-,T_+).$
In addition, the solution depends continuously on the initial data relative to the above norms. Furthermore, 
the wave coordinate condition~\eqref{E:initialconstraint} is verified on $(T_-,T_+)\times\T^3.$
In particular, the solution $(g_{\mu\nu},u^{\mu},\rho)$, ($\mu,\nu=0,1,2,3$) 
\underline{also solves the unmodified system}~\eqref{E:metricgealternative} + \eqref{E:dustalternative1} - \eqref{E:massshellalternative}.
\end{theorem}
\begin{remark} [\textbf{Initial metric is Lorentzian}]
The assumptions~\eqref{E:initialdata1} and~\eqref{E:localexistencemathringgjklowerequivalenttostandardmetricagain} 
guarantee that the initial spacetime metric $g_{\mu\nu}|_{t=0}$ is Lorentzian.
\end{remark}
\begin{remark} [\textbf{The regularity of the solution}]
The regularity properties of the solution stated in the conclusions of the theorem
are consistent with the heuristics argument for avoiding derivative loss that was explained in the introduction.
We will derive these estimates in detail in Sects.~\ref{S:norms}-\ref{S:equivalence}.
\end{remark}

\begin{remark} [\textbf{On data that do not verify the constraints}]
It is possible to formulate a local well-posedness theorem for the modified equations with general initial data sets that
\textbf{do not} necessarily satisfy the Einstein constraint equations. The solutions thus obtained do not necessarily
verify the unmodified dust-Einstein equations. For brevity,
we choose not to state such a general result. However, we point out that, 
as shown by Ringstr\"om ~\cite{hR2008}, such a theorem would be useful for studying 
future stability in settings where the initial $3-$manifold is topologically more complicated than 
$\mathbb{T}^3.$ Roughly, for certain topologies, one can locally reduce the problem
to the study of the problem on $\mathbb{T}^3$ and then patch together the different regions.
The point is that the Einstein constraint equations might not be verified in part the overlapping regions
(these unphysical regions are of course ignored because they have no connection to 
the maximal globally hyperbolic development of the physical data.)
\end{remark}
\begin{proof}
The proof of the theorem relies on a high-order energy scheme heuristically described in the introduction and
carried out in detail in the next section. For this reason, we shall leave out the proof to avoid repetition.
The preservation of wave coordinates is a now standard result that was first proved in the fundamental work~\cite{CB1952}.
For the proof of the analogous statement in the context of Euler-Einstein system, we refer to ~\cite[Theorem 51.]{jS2012}.
\end{proof}
We now state a continuation principle for the modified system \eqref{E:modifiedsystem}. This is 
a rather standard result that exhaustively classifies the scenarios that could lead to
singularity formation in a solution to the modified dust-Einstein equations.
See, for example, \cite[Ch. VI]{lH1997}, \cite[Ch. 1]{cS2008}, \cite{jS2008b} for the main ideas behind a proof.
The main point is that in order to prove future-global existence, we only have to rule out the 
three breakdown scenarios stated in the continuation principle.
\begin{proposition}[\textbf{Continuation Principle}]\label{P:continuationprinciple}
Let $T_{\text{max}}>0$ be the supremum over all times $T_+$ such that the solution 
$(g_{\mu\nu},u^{\mu},\rho)$ to \eqref{E:Rhat00}-\eqref{E:secondEulersummary} 
exists on the the time interval
$[0,T_+)$ and has the properties stated in Theorem \ref{T:localwellposedness}. If $T_{\text{max}}<\infty$, then
as $t\to T_{\text{max}}^-,$ one of the following three breakdown scenarios must occur:
\begin{enumerate}
\item
There exists a sequence $(t_n,x_n) \in [0,T_{max}) \times \mathbb{T}^3$ such that 
$\lim_{n \to \infty} g_{00}(t_n,x_n) = 0.$
\item
There exists a sequence $(t_n,x_n) \in [0,T_{max}) \times \mathbb{T}^3$ such that the smallest eigenvalue of
$g_{jk}(t_n,x_n)$ converges to $0$ as $n \to \infty.$
\item
$	\lim_{t \to T_{max}^-} \sup_{0 \leq \tau \leq t} 
			\Bigg\lbrace
			\sum_{\mu,\nu =0}^3 
					\| g_{\mu \nu}\|_{C_b^2}(\tau) 
					+ \| \partial_t g_{\mu \nu}\|_{C_b^1}(\tau)  
			+ \| \rho \|_{C_b^1}(\tau)  
					+ \sum_{j=1}^3 \| u^j\|_{C_b^1}(\tau)  
			\Bigg\rbrace
			 = 
				\infty.$
\end{enumerate}
\end{proposition}

\section{Norms and energies}\label{S:norms}
In this section, we introduce suitable Sobolev norms and energies for the metric components and
fluid components; we therefore fix our functional-analytic framework for addressing future-global existence for near-FLRW dust-Einstein solutions. Our use of the word ``energy" is reserved for the $L^2-$based quantities that 
arise dynamically from the dust-Einstein system as the natural coercive quadratic forms associated with 
equations themselves. In Sect.~\ref{S:equivalence}, we will show that the energies in fact control the 
Sobolev norms introduced in Sect.~\ref{SS:norms}. The proof of the equivalence is
somewhat non-trivial and is based in part on elliptic estimates. The main point is that the norms control 
some additional quantities, such as the top-order spatial derivatives of the metric components, 
that are not manifestly controlled by the energies.
%
\subsection{Norms} \label{SS:norms}
We now define the norms that we use to study near-FLRW solutions.

\begin{remark}[\textbf{On the large number of norms}]
Below we define a large number of norms. This is important because we need to
fully understand the structure of the coupling of the solution variables in the
evolution equations.
This understanding is essential in order for us to avoid losing derivatives 
(that is, from the point of view of local well-posedness)
and also for us to close our global existence argument.
This will become apparent during the proofs of 
Props. \ref{P:energynormcomparison}
and \ref{P:integralinequalitiesmetric}.
\end{remark}

\begin{remark}	[\textbf{The norms remain small}]
	The proof of our main future stability theorem will show that all of the norms 
	introduce below are bounded by $\lesssim \epsilon$ for $t \in [0,\infty)$ 
	if their initial size is $\leq \epsilon$ (and $\epsilon$ is sufficiently small).
\end{remark}	

\begin{remark} [\textbf{Norms for the elliptic quantities}]
In the following definitions, 
quantities indexed with a superscript ``$e$'' will  
later be shown to be bounded in terms of the energies with the aid of 
elliptic estimates.
\end{remark}

\begin{definition} [\textbf{Norms for the metric components}]
\label{D:METRICCOMPONENTNORMS}
Let $N$ be a positive integer, and let $q$ be the small positive constant defined in~\eqref{E:qdef}.
We define
\begin{subequations}
\begin{align}
\gzerozeronorm{N-1}
& \eqdef e^{q\Omega} \|\g_t g_{00}\|_{H^{N-1}}
	+ e^{q\Omega} \|g_{00} + 1\|_{H^{N-1}}
	+ \sum_{i=1}^3 e^{(q-1) \Omega} \| \partial_i g_{00}\|_{H^{N-1}},
	\label{E:G00NORM} \\
\gzerozerounorm{N-1}	
& \eqdef 
	e^{q\Omega} \| \u \g_t g_{00}\|_{H^{N-1}}
	+ e^{q\Omega} \| \u g_{00}\|_{H^{N-1}}
	+  \sum_{i=1}^3 e^{(q-1)\Omega} \| \u \partial_i g_{00}\|_{H^{N-1}},
		\label{E:G00UNORM} \\
\gzerozeroellipticnorm{N-1}
&	\eqdef \sum_{i=1}^3 e^{(q-1)\Omega} \| \g_i \g_t g_{00}\|_{H^{N-1}}
		+ \sum_{i,j=1}^3 e^{(q-2)\Omega}\| \g_i \g_j g_{00}\|_{H^{N-1}}.
		\label{E:G00ELLIPTICNORM}
\end{align}
\end{subequations}
We define
\begin{subequations}
\begin{align}
\gzerostarnorm{N-1}
& \eqdef \sum_{j=1}^3 e^{(q-1)\Omega} \|\g_t g_{0j}\|_{H^{N-1}}
	+ \sum_{j=1}^3 e^{(q-1)\Omega} \|g_{0j} \|_{H^{N-1}}
	+ \sum_{i,j=1}^3 e^{(q-2) \Omega} \| \partial_i g_{0j}\|_{H^{N-1}},
		\label{E:G0STARNORM}\\
\gzerostarunorm{N-1}
& \eqdef \sum_{j=1}^3 e^{(q-1)\Omega} \| \u \g_t g_{0j}\|_{H^{N-1}}
+ \sum_{j=1}^3 e^{(q-1)\Omega} \| \u g_{0j}\|_{H^{N-1}}
+ \sum_{i,j=1}^3 e^{(q-2)\Omega} \| \u \partial_i g_{0j}\|_{H^{N-1}},
\label{E:G0STARUNORM} \\
\gzerostarellipticnorm{N-1}
& \eqdef \sum_{i,j=1}^3 e^{(q-2)\Omega}\| \partial_i \g_t g_{0j}\|_{H^{N-1}}
		+ \sum_{i,j,k=1}^3 e^{(q-3)\Omega}\| \partial_i \partial_j g_{0k}\|_{H^{N-1}}.
		\label{E:G0STARELLIPTICNORM}
\end{align}
\end{subequations}
We define
\begin{subequations}
\begin{align}
\hstarstarnorm{N-1}
& \eqdef 
	\sum_{j,k=1}^3 e^{q\Omega} \|\g_t h_{jk}\|_{H^{N-1}}
		+ \sum_{i,j,k=1}^3 \|\partial_i h_{jk}\|_{H^{N-2}}
		+ \sum_{i,j,k=1}^3 e^{(q-1)\Omega} \|\partial_i h_{jk}\|_{H^{N-1}},
				\label{E:HSTARSTARNORM} \\
\combinedpartialupartialhstarstarnorm{N-1}
& \eqdef
	\sum_{j,k=1}^3 e^{q\Omega} \|\u \g_t h_{jk}\|_{H^{N-1}}
	+ \sum_{j,k=1}^3 e^{q\Omega} \| \u h_{jk} \|_{H^{N-1}}
		\label{E:DERIVATIVEHSTARSTARUNORM} \\
& \ \ + \sum_{i,j,k=1}^3 e^{q\Omega} \| \u \partial_i h_{jk} \|_{H^{N-2}}
	+ \sum_{i,j,k=1}^3 e^{(q-1)\Omega} \| \u \partial_i h_{jk}\|_{H^{N-1}},
	 \notag  \\	
\hstarstarellipticnorm{N-1}
& \eqdef \sum_{i,j,k=1}^3 e^{(q-1)\Omega} \|\g_i \g_t h_{jk}\|_{H^{N-1}}
		+ \sum_{i,j,k,l=1}^3 e^{(q-2)\Omega}\| \g_i \g_j h_{kl}\|_{H^{N-1}}.
		\label{E:HSTARSTARELLIPTICNORM}
\end{align}
\end{subequations}
The operator $\u$ from above is defined in~\eqref{E:deltau}.
\end{definition}

\begin{remark} [\textbf{Norms do not control $h_{jk}$}]
	None of the norms control the components $h_{jk}$ themselves. Rather, they only control the  derivatives of $h_{jk}.$
	We will make separate bootstrap assumptions to control the components $h_{jk}$
	[see Eq. \eqref{E:gjkBAvsstandardmetric}]. We remark that we will improve these bootstrap assumptions by 
	integrating $\g_t h_{jk}$ in time.
\end{remark}

\begin{remark} [\textbf{The top-order $t-$weights}]
Note, for example, that the exponential weights are ``worse" by a factor of $e^{\Omega}$
in the elliptic norm $\gzerozeroellipticnorm{N-1}$ compared to the lower order norm $\gzerozeronorm{N-1}.$
That is, the top-order spatial derivatives of $g_{00}$ are allowed to be larger
by a factor of $e^{\Omega}$ compared to its just-below-top-order spatial derivatives.
This worsened behavior reflects the degenerate hyperbolic nature of our modified dust-Einstein equations 
as explained in the introduction.
Similar remarks apply to the other norms defined above.
\end{remark}
\begin{definition}[\textbf{Norms for the fluid variable components}]
\label{D:FLUIDVARIABLENORMS}
Let $N$ be a positive integer, and let $q$ be the small positive constant defined in~\eqref{E:qdef}.
We define
\begin{subequations}
\begin{align} 
	\velocitynorm{N-1}
	& \eqdef 
		\sum_{j=1}^3 e^{(1+q)\Omega} \|u^j\|_{H^{N-1}},
		\\
	\topvelocitynorm{N-1}
	& \eqdef 
		\sum_{i,j=1}^3 e^{q\Omega} \| \partial_i u^j\|_{H^{N-1}},
			 \label{E:VELOCITYNORMTOPORDER}
			 \\
	\velocityunorm{N-1}
	& \eqdef 
		\sum_{j=1}^3 e^{(1+q)\Omega} \| \partial_{\mathbf{u}} u^j\|_{H^{N-1}},
			\label{E:VELOCITYNORMLOWERORDER} 
			\\
	\densnorm{N-1}
	& \eqdef \|\dens - \bar{\dens}\|_{H^{N-1}},
		\\
	\densunorm{N-1}
	& \eqdef e^{q \Omega} \|\u \dens \|_{H^{N-1}}.
\end{align}
\end{subequations}
\end{definition}
\begin{definition}[\textbf{Aggregate metric norms}]
\label{D:TOTALMETRICNORM}
We define
\begin{subequations}
\begin{align}
	\gnorm{N-1}
	& \eqdef
		\gzerozeronorm{N-1}
		+ \gzerostarnorm{N-1}
		+ \hstarstarnorm{N-1},
			\label{E:TOTALBELOWTOPSPATIALDERIVATIVEMETRICNORM} \\
	\gunorm{N-1}
	& \eqdef
		\gzerozerounorm{N-1}
		+	\gzerostarunorm{N-1}
		+ \combinedpartialupartialhstarstarnorm{N-1},
			\label{E:TOTALUDERIVATIVEMETRICNORM} \\
	\totalellipticnorm{N-1}
	& \eqdef
		\gzerozeroellipticnorm{N-1}
		+
		\gzerostarellipticnorm{N-1}
		+
		\hstarstarellipticnorm{N-1}.
\end{align}
\end{subequations}
\end{definition}

\begin{definition} [\textbf{Aggregate metric $+$ below-top-order fluid norms}]
\label{D:TOTALBELOWTOPSOLUTIONNORM}
We define
\begin{subequations}
\begin{align} \label{E:totalbelowtopnorm}
	\totalbelowtopnorm{N-1}
	& \eqdef \gnorm{N-1} 
		+ \velocitynorm{N-1}
		+ \densnorm{N-1},
			\\
	\totalbelowtopunorm{N-1}
	& \eqdef 
		\gunorm{N-1}
		+ \velocityunorm{N-1}
		+ \densunorm{N-1}.
\end{align}
\end{subequations}
\end{definition}

\begin{definition} [\textbf{Up-to-top-order metric $+$ fluid total solution norm}]
\label{D:TOTALSOLUTIONNORM}
We define
\begin{align} \label{E:totalnorm}
	\totalnorm
& \eqdef \totalbelowtopnorm{N-1}
		+ \totalbelowtopunorm{N-1}
		+ \totalellipticnorm{N-1}
		+ \topvelocitynorm{N-1}.
\end{align}
	
\end{definition}

\begin{remark}
Our global existence proof will show that if the above norms
are initially of size $\epsilon,$ then they never grow larger than $C \epsilon$
(for $\epsilon$ sufficiently small).
These norm bounds imply, for example,
that if $K$ spatial derivatives are applied to the quantities $g_{00}$, $g_{0j}$, and $h_{jk}$, 
where $0 \leq K \leq 2,$
then their $H^{N-1}$-norms grow at most like $e^{K\Omega}$, $e^{(K-1)\Omega}$, and $e^{K\Omega}$ 
(times a possible ``additional decay factor'' of the form $e^{- q \Omega}$) respectively. This structure hints at a systematic way of organizing our estimates in later sections, and we will formalize it in our {\bf Counting Principle} 
(see Lemma~\ref{L:countingprinciple} and Cor. \ref{C:COUNTINGPRINCIPLE}).
\end{remark}

\subsection{Energies}
\label{SS:metricenergies}
In this section, we define the energies that we use to study near-FLRW solutions.
The energies are coercive and will be used to control the norms from the previous sections.
This is a somewhat non-trivial fact that will be proved in Prop. \ref{P:energynormcomparison}.
The main point is that the time derivatives of the energies can be directly 
estimated with the help of integration by parts because their structure is carefully chosen 
to complement the structure of the evolution equations;
the same statement is not generally true for the norms introduced in the previous section.

\subsubsection{The building block energy for $g_{\mu \nu}$}

We now define the building blocks of the energies that we will use to estimate solutions to Eqs. 
\eqref{E:metric00}-\eqref{E:metricjk}.

\begin{lemma} [\textbf{Properties of the building blocks of energies for the metric}] \cite[Lemma 15]{hR2008} \label{L:buildingblockmetricenergy} 
Let $v$ be a solution to the scalar equation
\begin{align} \label{E:vscalar}
	\hat{\square}_g v = \upalpha H \partial_t v + \upbeta H^2 v + F,
\end{align}
where $\hat{\square}_g = g^{\lambda \kappa} \partial_{\lambda} \partial_{\kappa},$ 
$\upalpha > 0$ and $\upbeta \geq 0$ are constants, and define $\mathcal{E}_{(\upgamma,\updelta)}[v, \partial v] \geq 0$ by
\begin{align}  \label{E:mathcalEdef}
		\mathcal{E}_{(\upgamma,\updelta)}^2[v,\partial v] \eqdef \frac{1}{2} \int_{\mathbb{T}^3} 
		\Big\lbrace -g^{00} (\partial_t v)^2 + g^{ab}(\partial_a v)(\partial_b v) - 2 \upgamma H g^{00} v \partial_t v
		+ \updelta H^2 v^2 \Big\rbrace \, dx.
\end{align}
Then there exist constants 
$\upeta > 0,$ $C >0,$ $C_{(\upbeta)} \geq 0,$ $\updelta \geq 0,$ and $\upgamma \geq 0$ such that 
\begin{align*}
	|g^{00} + 1| \leq \upeta
\end{align*}
implies that
\begin{align} \label{E:mathcalEfirstlowerbound}
	\mathcal{E}_{(\upgamma,\updelta)}^2[v,\partial v] \geq C \int_{\mathbb{T}^3} \Big\lbrace (\partial_t v)^2 + g^{ab}(\partial_a 
		v)(\partial_b v) + C_{(\upbeta)} v^2 \Big\rbrace \, dx.
\end{align} 
The constants $\updelta$ and $\upgamma$ depend on $\upalpha$ and $\upbeta,$ while $\upeta,$ $C,$ and $C_{(\upbeta)}$ depend on $\upalpha,$ $\upbeta,$ $\upgamma$ and $\updelta.$ Furthermore, $C_{(\upbeta)} = 0$ if $\upbeta = 0$ and 
$C_{(\upbeta)} = 1$ if $\upbeta > 0.$ In addition, if $\upbeta = 0,$ then $\upgamma = \updelta = 0,$ while if $\upbeta > 0,$ then we can arrange for $\upgamma > 0$ and $\updelta > 0.$ Finally, we have that
\begin{align*}
	\frac{d}{dt} (\mathcal{E}_{(\upgamma,\updelta)}^2[v,\partial v]) & \leq - \upeta H \mathcal{E}_{(\upgamma,\updelta)}^2[v,\partial v] 
		+ \int_{\mathbb{T}^3} \Big\lbrace - (\partial_t v + \upgamma H v)F + \triangle_{\mathcal{E};(\upgamma, 
			\updelta)}[v,\partial v] \Big\rbrace \, dx,
\end{align*}
where
\begin{align} \label{E:trianglemathscrEdef}
	\triangle_{\mathcal{E};(\upgamma, \updelta)}[v,\partial v] & = - \upgamma H (\partial_a g^{ab}) v \partial_b v
		- 2 \upgamma H (\partial_a g^{0a}) v \partial_t v - 2 \upgamma H g^{0a}(\partial_a v)(\partial_t v) \\
	& \ \ - (\partial_a g^{0a})(\partial_t v)^2 - (\partial_a g^{ab})(\partial_b v)(\partial_t v)
		- \frac{1}{2}(\partial_t g^{00})(\partial_t v)^2 \notag \\
	& \ \ + \bigg(\frac{1}{2} \partial_t g^{ab} + H g^{ab} \bigg) (\partial_a v) (\partial_b v)
		- \upgamma H (\partial_t g^{00}) v \partial_t v  \notag \\ 
	& \ \ - \upgamma H (g^{00} + 1)(\partial_t v)^2. \notag
\end{align}

\end{lemma}

\begin{proof}
	
	The proof is a standard integration by parts argument that begins with the multiplication of both sides of Eq.
	\eqref{E:vscalar} by $- (\partial_t v + \upgamma H v);$ see \cite[Lemma 15]{hR2008} for the details. For later use, we
	quote the following identity from the proof:
	\begin{align} 
	\label{E:mathcalEtimederivativeformula}
		\frac{d}{dt} & (\mathcal{E}_{(\upgamma,\updelta)}^2[v,\partial v]) \\ 
		& = \int_{\mathbb{T}^3} \Big\lbrace -(\upalpha - \upgamma) H (\partial_t v)^2 
		+ (\updelta - \upbeta - \upgamma \upalpha) H^2 v \partial_t v - \upbeta \upgamma H^3 v^2 \notag \\ 
	& \ \ \ \ \ - (1 + \upgamma) H g^{ab}(\partial_a v)(\partial_b v) - (\partial_t v + \upgamma H v)F 
		+ \triangle_{\mathcal{E};(\upgamma, \updelta)}[v,\partial v] \Big\rbrace \, dx. \notag
	\end{align}
\end{proof}

\subsubsection{Energies for the metric components}
We now use the building block energies $\E$ to define energies for the metric components.

\begin{definition}[\textbf{Energies for the metric components}] \label{D:energiesforg}
Let $N$ be a positive integer, and let $q$ be the small positive constant defined in~\eqref{E:qdef}.
Let $\mathcal{E}_{\cdots}[\cdots]$ be the building block energy \eqref{E:mathcalEdef}.
We define the positive definite energies
$\gzerozeroenergy{N-1},$
$\gzerozeropartialuenergy{N-1},$
$\gzerostarenergy{N-1},$ 
$\gzerostarpartialuenergy{N-1},$
$\hstarstarenergy{N-2},$ 
$\partialhstarstarenergy{N-1},$ 
$\hstarstarpartialuenergy{N-1},$
and
$\partialhstarstarpartialuenergy{N-1}$
as follows:
\begin{subequations}
\begin{align}
\gzerozeroenergy{N-1}^2 
& \eqdef
	\sum_{|\vec{\alpha}|\leq N-1} 
		e^{2q\Omega}\E^2_{\upgamma_{00},\updelta_{00}} [\a(g_{00}+1),\g\a g_{00}], 
		\label{E:en00} \\
\gzerozeropartialuenergy{N-1}^2 
& \eqdef 
	\sum_{|\vec{\alpha}| \leq N-1} 
		e^{2q\Omega} 
		\E^2_{\widetilde{\upgamma}_{00},\widetilde{\updelta}_{00}}[\u \a g_{00},\g \u \a g_{00}],
	\label{E:enpartialu00}
\end{align}
%
%
\begin{align}
\gzerostarenergy{N-1}^2 
& \eqdef\sum_{|\vec{\alpha}|\leq N-1} \sum_{j=1}^3 e^{2(q-1)\Omega} 
	\E^2_{\upgamma_{0*},\updelta_{0*}} [\a g_{0j},\g \a g_{0j}],
	\label{E:en0*} \\
\gzerostarpartialuenergy{N-1}^2 
& \eqdef \sum_{|\vec{\alpha}| \leq N-1}\sum_{j=1}^3 
				 e^{2(q-1)\Omega}
				 \E^2_{\widetilde{\upgamma}_{0*},\widetilde{\updelta}_{0*}}[\u \a g_{0j} ,\g \u \a g_{0j}],
\label{E:enpartialu0*}
\end{align}
\begin{align}
\hstarstarenergy{N-2}^2 
& \eqdef 
		\frac{1}{2}
		\sum_{1 \leq |\vec{\alpha}|\leq N-1} \sum_{j,k=1}^3
		\int_{\T^3} 
			H^2(\a h_{jk})^2  
		\,dx,
	\label{E:en**} \\
\partialhstarstarenergy{N-1}^2 
& \eqdef \sum_{|\vec{\alpha}|\leq N-1} \sum_{j,k=1}^3
		e^{2q\Omega}
		\E^2_{\upgamma_{**},\updelta_{**}}[0,\g \a h_{jk}],
	\label{E:partialen**} \\
\partialhstarstarpartialuenergy{N-1}^2
& \eqdef \sum_{|\vec{\alpha}| \leq N-1} \sum_{j,k=1}^3 
					e^{2q\Omega} \E^2_{\widetilde{\upgamma}_{**},\widetilde{\updelta}_{**}}[0,\g \u \a h_{jk}].
				\label{E:partialenpartialu**}   
\end{align}
\end{subequations}
%
%
Above, 
\begin{align*}
h_{jk} & \eqdef  e^{-2\Omega}g_{jk} && (j,k=1,2,3).
\end{align*}

The number pairs $(\upgamma_{00},\updelta_{00})$, $(\upgamma_{0,*},\updelta_{0,*})$, and $(\upgamma_{**},\updelta_{**})$ 
are the constants generated by applying Lemma \ref{L:buildingblockmetricenergy} 
to the right-hand side of~\eqref{E:metric00}-\eqref{E:metricjk} respectively, while
the number pairs $(\widetilde{\upgamma}_{00},\widetilde{\updelta}_{00})$, $(\widetilde{\upgamma}_{0,*},\widetilde{\updelta}_{0,*})$, and $(\widetilde{\upgamma}_{**},\widetilde{\updelta}_{**})$ 
are the constants generated by applying Lemma \ref{L:buildingblockmetricenergy} 
the right-hand side of~\eqref{E:metric00}-\eqref{E:metricjk} with $\upalpha$ replaced by $\upalpha + 2$
[see equation  \eqref{E:vscalar2} for a justification of the constants $\widetilde{\upgamma}, \widetilde{\updelta}$].
In particular, we have that $(\upgamma_{**},\updelta_{**}) = (\widetilde{\upgamma}_{**},\widetilde{\updelta}_{**}) = (0,0).$
\end{definition}

\begin{remark}[\textbf{Justification of the necessity of} $\hstarstarenergy{N-2}$]
	We need the energy $\hstarstarenergy{N-2}$ defined in \eqref{E:en**}  
	in order to control the second sum in the norm \eqref{E:HSTARSTARNORM};
	this second
	sum is not controlled by the energy $\partialhstarstarenergy{N-1}$ 
	defined in \eqref{E:partialen**}
	because the constant $C_{(\upbeta)}$ from 
	\eqref{E:mathcalEfirstlowerbound} is $0$ for the energies corresponding to the wave equations verified by $h_{ij}.$
\end{remark}

\begin{remark}[\textbf{Order of} $\u$ \textbf{differentiation matters}]
	We stress that our norms control $\a \u$ derivatives of various quantities
	while our energies control the $\u \a$ derivatives of various quantities.
	These operators agree up to commutation terms that we are able to control.
	The reason that we do not define an energy for the quantity 
	$\a \u$ is that in order to derive $L^2$ energy estimates for the derivative 
	$\partial \a \u$ of the metric components, we would have to commute the wave
	equations \eqref{E:metric00}-\eqref{E:metricjk}
	with $\a \u.$ The important point is that the 
	top-order operator
	$\a \u$ (that is, $|\vec{\alpha}| = N-1$)
	when commuted with the reduced wave operator $\hat{\square}_g,$ would generate a commutator error term
	that involves too many derivatives of the components $u^{\mu}$ 
	(that is, $N+1$ spatial derivatives of these components)
	to close even a local well-posedness argument.
\end{remark}

\begin{definition}[\textbf{Aggregate metric energies}]
\label{D:TOTALMETRICENERGIES}
We define
\begin{subequations}
\begin{align}
	\genergy{N-1}
	& \eqdef
		\gzerozeroenergy{N-1}
		+ \gzerostarenergy{N-1}
		+ \partialhstarstarenergy{N-1}
		+  \hstarstarenergy{N-2},
			\label{E:METRICTOTALBELOWTOPORDERSPATIALDERIVATIVEENERGY} \\
	\guenergy{N-1}
	& \eqdef
		\gzerozeropartialuenergy{N-1}
		+ \gzerostarpartialuenergy{N-1}
		+ \partialhstarstarpartialuenergy{N-1}.
		\label{E:TOTALMETRICUDERIVATIVEENERGY}
\end{align}
\end{subequations}
\end{definition}

\subsubsection{Fluid energies}

We now define energies for the fluid variables.

\begin{definition}[\textbf{Energies for the fluid variable components}] \label{D:FLUIDENERGIES}
Let $N$ be a positive integer, and let $q$ be the small positive constant defined in~\eqref{E:qdef}.
We define the positive definite energies 
$\velocityenergy{N-1},$
$\topordervelocityenergy{N-1},$ 
and $\densenergy{N-1}$ as follows:
\begin{subequations}
\begin{align}
	\velocityenergy{N-1}^2
	& \eqdef \sum_{j=1}^3 e^{2 (1 + q)\Omega}  \left\| u^j  \right\|_{H^{N-1}}^2,
		  \label{E:VELOCITYENERGYDEF} \\
	\topordervelocityenergy{N-1}^2
	& \eqdef \sum_{i,j=1}^3 e^{2 q \Omega} \left\| \partial_i u^j  \right\|_{H^{N-1}}^2,  
			\label{E:TOPORDERVELOCITYENERGYDEF} \\
	\densenergy{N-1}^2
	& \eqdef \|\dens-\bar{\dens}\|_{H^{N-1}}^2.
		\label{E:DENSITYENERGYDEF}
\end{align}
\end{subequations}

\end{definition}


\subsubsection{Aggregate total energies}
We begin by defining a solution energy that 
\emph{does not control the top-order spatial derivatives of $\partial g$ or $u^j.$}

\begin{definition} [\textbf{Below-top-order aggregate metric $+$ fluid energies}]
\label{D:TOTALBELOWTOPSOLUTIONENERGY}
We define
\begin{align} \label{E:totalbelowtopenergy}
	\totalbelowtopenergy{N-1}
	& \eqdef 
		\genergy{N-1} 
		+ \velocityenergy{N-1}
		+ \densenergy{N-1}.
\end{align}
\end{definition}

Finally, we define the total energy of the solution.
It controls all quantities that we able to dynamically estimate
via integration by parts, but it does not directly control
the top-order spatial derivatives of $\partial g$
(we will use elliptic estimates to control these latter quantities).
The energy is $0$ for the background FLRW solution.
\begin{definition}[\textbf{Total solution energy}]
\label{D:TOTALENERGY}
We define
\begin{align} \label{E:TOTALENERGYDEF}
	\totalenergy
	& \eqdef
		\totalbelowtopenergy{N-1}
		+ \guenergy{N-1}
 		+ \topordervelocityenergy{N-1}. 
	\end{align}
\end{definition}

\section{Preliminary differential inequalities for the energies} \label{S:differentialinequalities}

\subsection{Preliminary differential inequalities for the metric energies}
We will soon apply Lemma~\ref{L:buildingblockmetricenergy} to derive 
our basic differential inequalities for the metric energies
$\gzerozeroenergy{N-1},$
$\gzerozeropartialuenergy{N-1},$
$\cdots$
defined in
Def. \ref{D:energiesforg}. 
However, we first need to understand
the structure of the inhomogeneous terms in the commuted equations.
To this end, we commute~\eqref{E:vscalar} with the operators $\a$ 
and $\u \a,$ and obtain the following lemma.

\begin{lemma}[\textbf{The basic structure of the commuted metric equations}] \label{L:COMMUTEDBASICSTRUCTURE}
Assume that the scalar-valued function $v$ verifies equation~\eqref{E:vscalar}. Then
the differentiated quantities $\a v$ and $\u \a v$ verify the following equations:
\begin{subequations}
\begin{align} \label{E:vscalar1}
\hat{\square}_g \a v 
& = \upalpha H \g_t \a v 
	+ \upbeta H^2\a v
	+ \a F 
	+ \{\hat{\square}_g,\a\} v,
	\\
\label{E:vscalar2}
\hat{\square}_g (\u \a v) 
& = (\upalpha + 2) H (\g_t \u \a v)
	+ \upbeta H^2 \u \a v
	 \underbrace
			{ +2 \upalpha H^2 \g_t \a v 
				+ 2 \upbeta H^3 \a v
			}_{\mbox{\textnormal{dangerous}}}
		\\
& \ \ 
			+ \a \u F 
			+ \{\u, \a \}F 
			+ 2 H \a F
			+  \upalpha H \{\partial_{\mathbf{u}}, \partial_t \} \a v
			+ \partial_{\mathbf{u}} (\{\hat{\square}_g,\a\} v) 
			+ \triangle_{\text{Ell}}[\partial^{(2)} \a v],
	\notag
\end{align}
\end{subequations}
where $\{ A,B \} \eqdef AB - BA$ denotes the commutator of the operators $A$ and $B.$
Furthermore, the error term
$\triangle_{\text{Ell}}[\partial^{(2)} \a v]$
from the right-hand side of \eqref{E:vscalar2}
can be decomposed as
\begin{align} \label{E:NEWQUADRATICERROR}
\triangle_{\text{Ell}}[\partial^{(2)} \a v]	
& \eqdef  - (\u g^{00})\g_{tt}\a v 
	- 2 (\u g^{0a}) \g_t\g_a\a v
	+ 2 (g^{\mu \nu} \g_{\mu} u^{\delta}) \g_{\nu} \g_{\delta} \a v
	+ (g^{\mu \nu} \g_{\mu} \g_{\nu} u^{\delta}) \partial_{\delta} \a v
	\\
& \ \ - \left(
					\partial_{\mathbf{u}} g^{ab} + 2 \omega g^{ab}
				\right) 
				\g_a \g_b \a v
			+ 2 (\omega - H) g^{ab} \g_a \g_b \a v
	\notag \\
& \ \ + 2H 
		\left(
			- 2 g^{0a} \g_a \g_t \a v
			+ \{\hat{\square}_g,\a\} v
		\right)
	\notag \\
& \ \ - 2H (g^{00} + 1) \g_{tt} \a v
	+ 2H \left(\frac{1}{u^0} - 1 \right) \g_t \partial_{\mathbf{u}} \a v
	- 2H\frac{(\g_t u^{\delta})}{u^0}\g_{\delta}\a v 
	- 2H \frac{u^a}{u^0} \g_t \g_a \a v.
	\notag
\end{align}

\end{lemma}
\begin{proof}
The lemma follows from a series of somewhat tedious calculations.
First, we note that identity \eqref{E:vscalar1} follows easily from commuting the equation~\eqref{E:vscalar} with the operator $\a$. 

In order to prove~\eqref{E:vscalar2}, we commute~\eqref{E:vscalar1} with the operator $\u$ and immediately obtain
\begin{align}
\hat{\Box}_g(\u\a v) 
 = & \underbrace{\u(\hat{\Box}_g \a v)}_{ = I} + 2g^{\mu\nu}(\g_{\mu}u^{\delta}) \g_{\nu}\g_{\delta}\a v \nonumber \\
& + (\hat{\Box}_g u^{\delta}) \g_{\delta} \a v  \underbrace{- \left(\u g^{ab}\right)\g_a\g_b \a v}_{ = II} 
- \left(\u g^{00} \right) \g_{tt}\a v - 2 (\u g^{0a}) \g_t\g_a\a v. \label{E:vscalarproof1}
\end{align}
We now substitute $\hat{\Box}_g \a v$ in the term $I$ in~\eqref{E:vscalarproof1} with
the right-hand side of~\eqref{E:vscalar1}. This will account for the presence of the terms
$\upalpha H \g_t \u \a v$, $\upalpha H \{\g_{\bf u},\g_t\} \a v$, $\upbeta H^2 \u \a v$, 
$\u\a F = \a \u F + \{\u, \a \}F$,
and $\g_u(\{\hat{\Box}_g, \a\}v)$
on the right-hand side of~\eqref{E:vscalar2}. The right-hand side of~\eqref{E:vscalarproof1} except for the terms $I$ and $II$ accounts for the first line of~\eqref{E:NEWQUADRATICERROR}.
We now analyze the term $II$ in~\eqref{E:vscalarproof1}. We rewrite it as
\begin{align}
- (\u g^{ab}) \g_a\g_b \a v
=-(\u g^{ab}+2\omega)\g_a\g_b \a v +(2\omega-2H)g^{ab}\g_a\g_b \a v \underbrace{+2Hg^{ab}\g_a\g_b \a v}_{=III}. \label{E:vscalarproof2}
\end{align}
The first two terms on the right-hand side of~\eqref{E:vscalarproof2} are present in~\eqref{E:NEWQUADRATICERROR}.
To handle the term $III$ in~\eqref{E:vscalarproof2}, we re-express $g^{ab}\g_a\g_b\a v$ as
$
\hat{\Box}_g\a v - g^{00}\g_{tt} \a v - 2g^{0a} \g_t \g_a \a v.
$
Using~\eqref{E:vscalar1}, we deduce that
\begin{align}
2H g^{ab}\g_a\g_b \a v = & 2\upalpha H^2 \g_t\a v + 2\upbeta H^3 \a v + 2H \a F + 2H \{\hat{\Box}_g, \a\} v 
- 4 H g^{0a} \g_t \g_a \a v \nonumber\\
&  \underbrace{- 2 H g^{00} \g_{tt} \a v}_{=IV}. \label{E:vscalarproof3}
\end{align}
The first two terms on the right-hand side of~\eqref{E:vscalarproof3} are accounted for in~\eqref{E:vscalar2} as the ``dangerous linear terms", as well as $2H \a F$. 
Term $ 2H \{\hat{\Box}_g, \a\} v 
-4 H g^{0a} \g_t \g_a \a v$ constitutes an error term present in~\eqref{E:NEWQUADRATICERROR}.
Finally, we re-express the term $IV$ in~\eqref{E:vscalarproof3} with
the help of the relation 
$\partial_{tt} 
= \frac{1}{u^0} \g_t \g_{\mathbf{u}} 
- \frac{1}{u^0} (\g_t u^{\delta})\g_{\delta}
- \frac{1}{u^0} u^a \g_a \g_t 
$
to obtain
\begin{align}
-2H g^{00}\g_{tt} \a v & =  -2H (g^{00}+1)\g_{tt}\a v + 2H\g_{tt} \a v \nonumber \\
& = -2H (g^{00}+1)\g_{tt}\a v +2H \left(\frac{1}{u^0}-1 \right)\g_t \u \a v- 2H\frac{(\g_t u^{\delta})}{u^0}\g_{\delta}\a v - 2H\frac{u^a}{u^0}\g_a\g_t \a v \nonumber \\
& \ \ + 2H \g_t \u \a v. \label{E:vscalarproof4}
\end{align}
The first four terms on the right-most side of~\eqref{E:vscalarproof4} are present in the error term~\eqref{E:NEWQUADRATICERROR}, whereas
the term $+ 2H \g_t \u \a v$ appears as the ``additional two copies of $H \g_t \u \a v$" 
present in the first term on the right-hand side of~\eqref{E:vscalar2}.
This finishes the proof of the lemma.
\end{proof}
\begin{remark} [\textbf{The structure of the $\u$-commuted wave equation}] \label{R:DANGEROUSTOPORDERCOMMUTATORTERM}
The structure of the terms in the $\u$-commuted wave equation is important. 
Specifically, equation~\eqref{E:vscalar2} captures a new technical feature,
which was not present in the earlier work \cite{jS2012}. 
The main issue of concern is that the term
$(\u g^{ab}) \g_a\g_b\a v$ is not quadratic, 
but rather a {\bf linear} contribution to the right-hand side and thus has to be handled with great care.
To expand upon this point, we note that 
$- (\u g^{ab}) \g_a\g_b\a v$ was carefully replaced by $2H g^{ab}\g_a\g_b\a v$ modulo harmless error terms, cf.~\eqref{E:vscalarproof2}. 
In the next step [cf.~\eqref{E:vscalarproof3}] we used the wave equation to replace $ 2H g^{ab}\g_a \g_b \a v$ with $- 2H g^{00}\g_{tt} \a v - 2 H g^{0a} \g_t \a v$ 
plus inhomogeneous terms, some of which are ``dangerous linear error terms.'' Our proof
of Theorem \ref{T:maintheorem} will show that these dangerous linear terms 
can be controlled by a hierarchy of decoupled energies that can effectively be independently bounded.
The structure of this hierarchy is a consequence of
the special structure of the dust-Einstein equations in our wave coordinate gauge.
Finally,~\eqref{E:vscalarproof4} allows us to write $- 2H g^{00}\g_{tt} \a v$  as $2H \g_t \u \a v$ plus harmless quadratic error terms. We stress that the term $2H \g_t \u \a v$ is 
a {\bf positive} multiple of $\g_t \u \a v$ and thus \underline{can only enhance} 
[see \eqref{E:mathcalEtimederivativeformula}]
the decay of the building block energy \eqref{E:mathcalEdef} associated with~\eqref{E:vscalar2}.
\end{remark}

In the next lemma, we state preliminary differential inequalities for the metric energies.
These inequalities will be used in Sect.~\ref{S:global} to deduce future-global stability.
\begin{lemma}[\textbf{A first differential inequality for the metric energies}]	
	\label{L:metricfirstdifferentialenergyinequality}
	Assume that $g_{00},$ $g_{0j},$ $h_{jk} = e^{-2 \Omega} g_{jk},$ $(j,k=1,2,3),$ are solutions to the modified Eqs.
	\eqref{E:metric00}-\eqref{E:metricjk}, and let 
	$\gzerozeroenergy{N-1},$
	$\gzerozeropartialuenergy{N-1},$
	$\cdots$
	be the metric energies from Def.~\ref{D:energiesforg}. 
	Then under the assumptions of Lemma \ref{L:buildingblockmetricenergy},
	the following differential inequalities are satisfied, where 
	$\triangle_{\mathcal{E};(\gamma, \delta)}[\cdot, \partial(\cdot )]$ is defined in \eqref{E:trianglemathscrEdef}, the
	constants $(\upgamma_{00}, \updelta_{00}),$ $(\upgamma_{0*}, \updelta_{0*}),$ $(\upgamma_{**}, \updelta_{**}),$
	$(\widetilde{\upgamma}_{00}, \widetilde{\updelta}_{00}),$ $(\widetilde{\upgamma}_{0*}, \widetilde{\updelta}_{0*}),$ and 
	$(\widetilde{\upgamma}_{**}, \widetilde{\updelta}_{**})$
	[note that $(\upgamma_{**}, \updelta_{**}) = (\widetilde{\upgamma}_{**}, \widetilde{\updelta}_{**}) = (0,0)$]
	are the constants from Def.~\ref{D:energiesforg}, and 
	$\upeta_{00},$ $\upeta_{0*},$ $\upeta_{**},$ 
	$\widetilde{\upeta}_{00},$ $\widetilde{\upeta}_{0*},$ $\widetilde{\upeta}_{**}$
	are the positive constants ``$\upeta$'' produced by applying Lemma \ref{L:buildingblockmetricenergy} to each of Eqs. 
	\eqref{E:metric00}-\eqref{E:metricjk} and their $\partial_{\mathbf{u}}-$commuted version 
	[see~\eqref{E:vscalar2}] respectively:
\begin{subequations}
\begin{align}
\frac{d}{dt}\gzerozeroenergy{N-1}^2
& \leq (2q - \upeta_{00})H\gzerozeroenergy{N-1}^2
	+2q(\omega-H)\gzerozeroenergy{N-1}^2 
	\label{E:00est} \\
& \ \ - \sum_{|\vec{\alpha}|\leq N-1}\int_{\T^3}e^{2q\Omega}
		\left(
			\g_t \a g_{00} + \upgamma_{00} H \a (g_{00}+1)
		\right)
		\left(
			\a\triangle_{00}
			+ \{\hat{\square}_g, \a\} (g_{00}+1)
		\right)
		\,dx 
		\nonumber \\
& \ \
	+ \sum_{|\vec{\alpha}|\leq N-1} 
		\int_{\T^3} 
			e^{2q\Omega} 
			\triangle_{\E;(\upgamma_{00},\updelta_{00})}[\a(g_{00}+1),\g(\a g_{00})]
		\, dx, \notag
\end{align}
\begin{align}
\frac{d}{dt} \gzerozeropartialuenergy{N-1}^2
& \leq (2q - \widetilde{\upeta}_{00})H E^2_{\partial_{\mathbf{u}} g_{00}}
	+ 2q (\omega-H)E^2_{\partial_{\mathbf{u}} g_{00}} 
	\label{E:partialu00est} \\
& \ \
	- \sum_{|\vec{\alpha}| \leq N-1}\int_{\T^3}
	e^{2q\Omega}
	\left(
		\g_t\u \a g_{00}
		+ \widetilde{\upgamma}_{00} H \u \a g_{00}
	\right)
	\left(
		\underbrace{
			10 H^2 \g_t \a g_{00} 
			+ 12 H^3 \a (g_{00} + 1)
		}_{\mbox{\textnormal{dangerous}}}
	\right)
	\, dx
		\nonumber \\
& \ \
	- \sum_{|\vec{\alpha}| \leq N-1}\int_{\T^3}
	e^{2q\Omega}
	\left(
		\g_t\u \a g_{00}
		+ \widetilde{\upgamma}_{00} H \u \a g_{00}
	\right)
	\left(
		\a \u \triangle_{00} 
		+ 2 H \a \triangle_{00}
	\right)
	\, dx
		\nonumber \\
& \ \
	- \sum_{|\vec{\alpha}| \leq N-1}\int_{\T^3}
	e^{2q\Omega}
	\left(
		\g_t\u \a g_{00}
		+ \widetilde{\upgamma}_{00} H \u \a g_{00}
	\right)
	\left(
		5 H \{\partial_{\mathbf{u}}, \partial_t \} \a g_{00} 
		+ \{ \u, \a \} \triangle_{00}
	\right)
	\, dx
		\nonumber \\
& \ \
	- \sum_{|\vec{\alpha}| \leq N-1}\int_{\T^3}
	e^{2q\Omega}
	\left(
		\g_t\u \a g_{00}
		+ \widetilde{\upgamma}_{00} H \u \a g_{00}
	\right)
	\left(
	 \partial_{\mathbf{u}} (\{\hat{\square}_g,\a\} g_{00})
		+ \triangle_{\text{Ell}}[\partial^{(2)} \a g_{00}]
	\right)
	\, dx
		\nonumber \\
& \ \
	+ \sum_{|\vec{\alpha}|\leq N-1} 
		\int_{\T^3} 
		e^{2q\Omega} 
		\triangle_{\E;(\widetilde{\upgamma}_{00}, \widetilde{\updelta}_{00})}[\u\a g_{00},\g(\u\a g_{00})]
		\, dx, \notag
\end{align}
%
\begin{align}
\frac{d}{dt}\gzerostarenergy{N-1}^2
& \leq \left(
				2(q-1)-\upeta_{0*}
		 \right)
		 H\gzerostarenergy{N-1}^2 
		+ 2 (q-1) (\omega-H) \gzerostarenergy{N-1}^2 
		\label{E:0*est} \\
& \ \ - \sum_{|\vec{\alpha}|\leq N-1}
			\sum_{j=1}^3
			\int_{\T^3}e^{2(q-1)\Omega}
		\left(
			\g_t \a g_{0j} + \upgamma_{0*} H \a g_{0j}
		\right)
		\left(
			\underbrace{- 2 H \a (g^{ab} \Gamma_{ajb})}_{\mbox{\textnormal{dangerous}}}
			+ \a \triangle_{0j}
			+ \{\hat{\square}_g, \a\} g_{0j}
		\right)
		\,dx 
		\nonumber \\
& \ \
	+ \sum_{|\vec{\alpha}|\leq N-1} 
		\sum_{j=1}^3
		\int_{\T^3} 
			e^{2(q-1)\Omega}
			\triangle_{\E;(\upgamma_{0*},\updelta_{0*})}[\a g_{0j}, \g(\a g_{0j})]
			\, dx, \notag
\end{align}
\begin{align}
\frac{d}{dt}\gzerostarpartialuenergy{N-1}^2
& \leq \left(
				2(q-1)-\upeta_{0*}
		 \right)
		 H E^2_{ \u  g_{0*}} 
		+ 2 (q-1) (\omega-H) E^2_{ \u  g_{0*}} 
		\label{E:partialu0*est} \\
& \ \
	- \sum_{|\vec{\alpha}| \leq N-1}
		\sum_{j=1}^3
		\int_{\T^3}
	e^{2 (q-1) \Omega}
	\left(
		\g_t \u \a g_{0j}
		+ \widetilde{\upgamma}_{0*} H \u \a g_{0j}
	\right)
		\nonumber \\
&  \ \ \ \ \ \ \ \ \ \ \ \ \ \ \ \ \ \ \ \ \ \ \ 
	\times
	\left(
		 \underbrace{
		 	- 2 H  \a \u (g^{ab} \Gamma_{ajb})
		 	- 4 H^2  \a (g^{ab} \Gamma_{ajb})
		 	+ 6 H^2 \g_t \a g_{0j}
			+ 4 H^3 \a g_{0j}
			}_{\mbox{\textnormal{dangerous}}}
	\right)
	\, dx
		\nonumber \\
& \ \
	- \sum_{|\vec{\alpha}| \leq N-1}
		\sum_{j=1}^3
		\int_{\T^3}
	e^{2 (q-1) \Omega}
	\left(
		\g_t \u \a g_{0j}
		+ \widetilde{\upgamma}_{0*} H \u \a g_{0j}
	\right)
	\left(
		\a \u \triangle_{0j}
		 + 2 H \a \triangle_{0j}
	\right)
	\, dx
		\nonumber \\
& \ \
	- \sum_{|\vec{\alpha}| \leq N-1}
		\sum_{j=1}^3
		\int_{\T^3}
	e^{2 (q-1) \Omega}
	\left(
		\g_t \u \a g_{0j}
		+ \widetilde{\upgamma}_{0*} H \u \a g_{0j}
	\right)
		\nonumber \\
& \ \ \ \ \ \ \ \ \ \ \ \ \ \ \ \ \ \ \ \ \ \ \ 
	\times
	\left(
		3 H \{ \partial_{\mathbf{u}}, \partial_t \} \a g_{0j}  
		- 2 H  \{ \u, \a \} (g^{ab} \Gamma_{ajb})
	\right)
		\, dx
	\nonumber \\
& \ \
	- \sum_{|\vec{\alpha}| \leq N-1}
		\sum_{j=1}^3
		\int_{\T^3}
	e^{2 (q-1) \Omega}
	\left(
		\g_t \u \a g_{0j}
		+ \widetilde{\upgamma}_{0*} H \u \a g_{0j}
	\right)
		\nonumber \\
&  \ \ \ \ \ \ \ \ \ \ \ \ \ \ \ \ \ \ \ \ \ \ \
		\times
		\left(
		\{ \u, \a \} \triangle_{0j}
		+ \partial_{\mathbf{u}} ( \{\hat{\square}_g, \a\} g_{0j})
		+ \triangle_{\text{Ell}}[\partial^{(2)} \a g_{0j}]
	\right)
		\, dx
	\nonumber \\
& \ \
	+ \sum_{|\vec{\alpha}|\leq N-1} 
		\sum_{j=1}^3
		\int_{\T^3} 
			e^{2(q-1)\Omega}
			\triangle_{\E;(\upgamma_{0*},\updelta_{0*})}[\u \a g_{0j},\g(\u\a g_{0j})]
			\, dx, \notag
\end{align}
\begin{align}
	\frac{d}{dt} \hstarstarenergy{N-2}^2
	& \leq 
		\sum_{1 \leq|\vec{\alpha}|\leq N - 1} \sum_{j,k=1}^3 
			\int_{\T^3} 
			H^2(\a\g_th_{jk})(\a h_{jk}) 
		\, dx,
	\label{**est}
\end{align}
\begin{align} \label{E:partial**est}
\frac{d}{dt} \partialhstarstarenergy{N-1}^2
	& \leq (2q-\upeta_{**}) H \partialhstarstarenergy{N-1}^2
		+ 2q(\omega-H) \partialhstarstarenergy{N-1}^2
		\\
& \ \ - \sum_{|\vec{\alpha}|\leq N-1}\sum_{j,k=1}^3\int_{\T^3}
			e^{2q\Omega} 
			(\g_t \a h_{jk})
			\left(
				\a \triangle_{jk}
				+ \{
						\hat{\square}_g, 
						\a
					\} 
					h_{jk}
			\right )\, dx 
	\nonumber \\
& \ \ + \sum_{|\vec{\alpha}|\leq N-1}\sum_{j,k=1}^3
				\int_{\T^3} 
				e^{2q\Omega} 
				\triangle_{\E;(\upgamma_{**},\updelta_{**})}[0,\g(\a h_{jk})] \, dx,
	\notag
\end{align}
%
%
%
\begin{align}
\frac{d}{dt} \partialhstarstarpartialuenergy{N-1}^2
	& \leq (2q-\upeta_{**}) H \partialhstarstarpartialuenergy{N-1}^2
			+ 2q(\omega-H) \partialhstarstarpartialuenergy{N-1}^2
		\label{E:partialpartialu**est} \\
& \ \
	- \sum_{|\vec{\alpha}| \leq N-1} \sum_{j,k=1}^3
	\int_{\T^3}
	e^{2q\Omega}
	\left(
		\g_t \u \a h_{jk}
	\right)
	\left(
		 \underbrace{
		 		6 H^2 \g_t \a h_{jk}
			}_{\mbox{\textnormal{dangerous}}}
		+ \a \u \triangle_{jk}
		+ 2 H \a \triangle_{jk}
	\right)
	\, dx
	\nonumber \\
& \ \
	- \sum_{|\vec{\alpha}| \leq N-1} \sum_{j,k=1}^3
	\int_{\T^3}
	e^{2q\Omega}
	\left(
		\g_t\u \a h_{jk}
	\right)
	\left(
		3H \{\u, \partial_t \} \a h_{jk}
		+ \{\u, \a \} \triangle_{jk}
	\right)
	\, dx
		\nonumber \\
& \ \
	- \sum_{|\vec{\alpha}| \leq N-1} \sum_{j,k=1}^3
	\int_{\T^3}
	e^{2q\Omega}
	\left(
		\g_t\u \a h_{jk}
	\right)
	\left(
		\partial_{\mathbf{u}} (\{\hat{\square}_g,\a\} h_{jk})
		+ \triangle_{\text{Ell}}[\partial^{(2)} \a h_{jk}]
	\right)
	\, dx
		\nonumber \\
& \ \	+ 
		\sum_{|\vec{\alpha}|\leq N-1}\sum_{j,k=1}^3
		\int_{\T^3} ~\
			e^{2q\Omega} \triangle_{\E;(\upgamma_{**},\updelta_{**})} [0,\g(\u\a h_{jk})]
			\, dx.
	\notag
\end{align} 
\end{subequations}
\end{lemma}
\begin{proof}
The estimate \eqref{**est} is very elementary and follows from differentiating in time under the integral 
inherent in the definitions of $\hstarstarenergy{N-2}^2.$
The remaining relations are a direct consequence of definitions
~\eqref{E:en00}-\eqref{E:partialenpartialu**} 
and Lemma~\ref{L:buildingblockmetricenergy} applied to the commuted metric evolution equations,
which are of the form \eqref{E:vscalar1} and~\eqref{E:vscalar2}.
\end{proof}

\subsection{Preliminary differential inequalities for the fluid energies}
We would now like to provide the fluid analog of Lemma \ref{L:metricfirstdifferentialenergyinequality}, that is,
preliminary differential inequalities for the fluid energies. 
We first provide the following lemma, which captures the structure of the commuted fluid equations.

\begin{lemma}[\textbf{The basic structure of the commuted fluid equations}] \label{L:FLUIDCOMMUTEDBASICSTRUCTURE}
Assume that the rescaled energy density $\dens=e^{3 \Omega} \rho$ and the fluid velocity components $u^1,u^2,u^3$ verify the fluid Eqs. 
\eqref{E:revol}-\eqref{E:velevol}. Then the $\a-$differentiated quantities $\a \dens, \a u^1, \a u^2, \a u^3$
verify the following identities:
\begin{subequations}
\begin{align}
	u^{\alpha}\g_{\alpha} \a \dens
	&= u^0 \a \left(\frac{1}{u^0} \triangle \right) 
		+ u^0 \{ \frac{u^a}{u^0} \partial_a, \a \} \dens
		\,,\label{E:revolcommuted} \\
	u^{\alpha}\g_{\alpha} \a u^j 
	& = - 2 u^0 \omega \a u^j 
		-  u^0 \a \triangle_{0 \ 0}^{\ j}
		+ u^0 \a \left(\frac{1}{u^0} \triangle^j \right)
		+ u^0 \{ \frac{u^a}{u^0} \partial_a, \a \} u^j \,. \label{E:velevolcommuted}
\end{align}	
\end{subequations}
Above, $\{A,B\}=AB-BA$ denotes a commutator. 
\end{lemma}

\begin{proof}
	We divide Eqs. \eqref{E:revol}-\eqref{E:velevol} by $u^0,$ 
	apply the operator $\a$ to each side,
	multiply by $u^0,$ and then organize the terms as in \eqref{E:revolcommuted}-\eqref{E:velevolcommuted}.
\end{proof}

With the help of Lemma \ref{L:FLUIDCOMMUTEDBASICSTRUCTURE}, we now derive the desired
preliminary differential inequalities for the fluid energies.
\begin{lemma}[\textbf{Energy identity for the fluid energies}] 
\label{L:fluidfirstdifferentialenergyinequality}
Assume that the rescaled energy density $\dens$ and the fluid velocity components $u^1,u^2,u^3$ verify the fluid Eqs. 
\eqref{E:revol}-\eqref{E:velevol}. Then fluid energies 
$\velocityenergy{N-1},$ 
$\topordervelocityenergy{N-1},$
and $\densenergy{N-1}$ 
from Def.~\ref{D:FLUIDENERGIES} verify the following differential identities:
\begin{subequations}
\begin{align} \label{E:VELOCITYDIFFINEQ}
	\frac{d}{dt} \velocityenergy{N-1}^2
	& = -2(1-q) \omega  
			\sum_{j=1}^3 e^{2 (1 + q)\Omega} \left\| u^j  \right\|_{H^{N-1}}^2
			\\
	& \ \ 
			- 2 
			\sum_{|\vec{\alpha}|\leq N-1} 
			\sum_{j=1}^3	
				\int 	e^{2 (1 + q)\Omega}
					(\a u^j)
					\underbrace{\a \triangle_{0 \ 0}^{ \ j}}_{\mbox{\textnormal{dangerous}}}
				\, dx
				\notag \\
	& \ \ 
			+ 2 
			\sum_{|\vec{\alpha}|\leq N-1} 
			\sum_{j=1}^3
				\int	 
					e^{2 (1 + q)\Omega} (\a u^j)
					\a \left(\frac{1}{u^0} \triangle^j\right)
				\, dx
				\notag \\
	& \ \ 
			+ 2 
			\sum_{j=1}^3
			\sum_{|\vec{\alpha}|\leq N-1} 
				\int
					e^{2 (1 + q)\Omega}
					(\a u^j)
					\{ \frac{u^a}{u^0} \partial_a, \a \} u^j
			\, dx
			\notag \\
	& \ \ + 
		\sum_{|\vec{\alpha}|\leq N-1} 
		\sum_{j=1}^3
				\int	
					e^{2(1 + q)\Omega}
					\left[\partial_a \left(\frac{u^a}{u^0} \right)\right]
					(\a u^j)^2
				\, dx,
		\notag   \\
 \frac{d}{dt} \topordervelocityenergy{N-1}^2
	& = - 2(2 - q) \omega 
				\sum_{i,j=1}^3 e^{2 q \Omega} \left\| \g_i u^j  \right\|_{H^{N-1}}^2
			\label{E:TOPORDERVELOCITYDIFFINEQ} \\
	& \ \ 
			- 2 
			\sum_{|\vec{\alpha}|\leq N-1} 
			\sum_{j=1}^3	
				\int 
					e^{2 q \Omega} 
					(\a \partial_i u^j) 
					\underbrace{\a \partial_i \triangle_{0 \ 0}^{ \ j}}_{\mbox{\textnormal{dangerous}}}
				\, dx
								\notag \\
	& \ \ 
			+ 2 
			\sum_{|\vec{\alpha}|\leq N-1} 
			\sum_{i,j=1}^3
				e^{2 q \Omega}	
					(\a \partial_i u^j) 
					\a \partial_i \left(\frac{1}{u^0} \triangle^j \right) 
			\, dx
				\notag \\
	& \ \ 
			+ 2 
			\sum_{|\vec{\alpha}|\leq N-1} 
			\sum_{i,j=1}^3
				\int
				e^{2 q \Omega}
					(\a \partial_i u^j) 
					\{ \frac{u^a}{u^0} \partial_a, \a \partial_i \} u^j
				\, dx
				\notag \\
	& \ \ + 
		\sum_{|\vec{\alpha}|\leq N-1} 
		\sum_{i,j=1}^3
				\int
					e^{2 q\Omega}
					\left[\partial_a \left(\frac{u^a}{u^0} \right)\right]
					(\a \partial_i u^j)^2 
				\, dx,
				\notag
\end{align}
\begin{align}  \label{E:DENSITYDIFFINEQ}
	\frac{d}{dt} \densenergy{N-1}^2
	& = \sum_{|\vec{\alpha}|\le N-1}\int_{\T^3}
		\left[\g_a\left(\frac{u^a}{u^0}\right)\right]
	|\a (\dens-\bar{\dens})|^2 \\
	& \ \ + 2 \sum_{|\vec{\alpha}| \leq N - 1}
				\int_{\mathbb{T}^3}
					\left[\a(\dens - \bar{\dens})\right]
          \a\left(\frac{1}{u^0} \triangle +\{\frac{u^a}{u^0}\g_a,\,\a\}\dens\right)
					\, dx. \notag
\end{align}
\end{subequations}
\end{lemma}

\begin{proof}
To prove \eqref{E:VELOCITYDIFFINEQ}, we take the time derivative of the right-hand side of \eqref{E:VELOCITYENERGYDEF}
and bring $\g_t$ under the integral over $\mathbb{T}^3.$ We then use 
\eqref{E:velevol} and \eqref{E:velevolcommuted} to replace all $\partial_t \a u^j$ terms with spatial
derivatives of $\a u^j$ plus the inhomogeneous terms on the right-hand sides. 
Finally, we integrate by parts over $\mathbb{T}^3$ to remove the spatial derivatives
off of the terms $\a u^j.$ The proof of \eqref{E:TOPORDERVELOCITYDIFFINEQ} is based
on definition \eqref{E:TOPORDERVELOCITYENERGYDEF} and is similar.

The proof of \eqref{E:DENSITYDIFFINEQ}, which is based on the energy definition \eqref{E:DENSITYENERGYDEF}
and the evolution Eqs. \eqref{E:revol} and \eqref{E:revolcommuted}, is similar but much simpler.
\end{proof}

\subsection{Linear-Algebraic Estimates of $g_{\mu \nu}$ and $g^{\mu \nu}$} 

In this section, we provide basic linear-algebraic estimates for the metric and its inverse.
In particular, they will ensure that $g_{\mu \nu}$ is a Lorentzian metric under assumptions 
compatible with our future stability theorem. 
We remark that the conclusions of the local well-posedness
Theorem \ref{T:localwellposedness} above were based in part on Lemma \ref{L:ginverseformluas}.

\begin{lemma} [\textbf{The Lorentzian nature of $g^{\mu \nu}$}] \cite[Lemmas 1 and 2]{hR2008}\label{L:ginverseformluas}
	Let $g_{\mu \nu}$ be a symmetric $4 \times 4$ matrix of real numbers. 
	Let $(g_{\flat})_{jk}$ be the $3 \times 3$ matrix defined by $(g_{\flat})_{jk} = g_{jk},$ and let $(g_{\flat}^{-1})^{jk}$
	be the $3 \times 3$ inverse of $(g_{\flat})_{jk}.$ Assume that $g_{00} < 0$ and that $(g_{\flat})_{jk}$ is positive definite. 
	Then $g_{\mu \nu}$ is a Lorentzian metric with inverse $g^{\mu \nu},$ $g^{00} < 0,$ and the $3 \times 3$ matrix 
	$(g^{\#})^{jk}$ defined by $(g^{\#})^{jk} \eqdef g^{jk}$ is positive definite. Furthermore, the following relations hold:
	\begin{subequations}
	\begin{align*}
		g^{00} & = \frac{1}{g_{00} - d^2}, \\
		\frac{g_{00}}{g_{00} - d^2} (g_{\flat}^{-1})^{ab} X_{a}X_{b} 
			& \leq (g^{\#})^{ab} X_{a}X_{b} \leq (g_{\flat}^{-1})^{ab} X_{a}X_{b}, && \forall (X_1,X_2,X_3) \in \mathbb{R}^3, \\
		g^{0j} & = \frac{1}{d^2 - g_{00}} (g_{\flat}^{-1})^{aj} g_{0a}, && (j=1,2,3),
	\end{align*}
	\end{subequations}
	where 
$		d^2 = (g_{\flat}^{-1})^{ab} g_{0a} g_{0b}.
$	%
\end{lemma}

The estimates in the next lemma are based on the following rough assumptions, 
which we will improve during our global existence argument.
\begin{center}
	\large{Rough Bootstrap Assumptions for $g_{\mu \nu}:$}
\end{center}
We assume that there are constants $\upeta > 0$ and $c_1 \geq 1$ such that
\begin{subequations}
\begin{align}
	|g_{00} + 1| & \leq \upeta, && \label{E:metricBAeta} \\
	c_1^{-1} \delta_{ab} X^a X^b 
		& \leq e^{-2 \Omega} g_{ab}X^{a}X^{b} 
		\leq c_1 \delta_{ab} X^a X^b, && \forall (X^1,X^2,X^3) \in \mathbb{R}^3, \label{E:gjkBAvsstandardmetric}   \\
	\sum_{a=1}^3 |g_{0a}|^2 & \leq \upeta c_1^{-1} e^{2(1 - q) \Omega}. && \label{E:g0jBALinfinity}
\end{align}
\end{subequations}
For our global existence argument, we will assume that $\upeta \leq q,$ where $q$ is defined in Sect. \ref{SS:bootstrapassumptions}.
\begin{lemma} [\textbf{First estimates of $g^{\mu \nu}$}] \cite[Lemma 7]{hR2008} \label{L:ginverseestimates}
	Let $g_{\mu \nu}$ be a symmetric $4 \times 4$ matrix of real numbers satisfying 
	\eqref{E:metricBAeta}-\eqref{E:g0jBALinfinity}, where $\Omega \geq 0$ and $0 \leq q < 1.$ 
	Then $g_{\mu \nu}$ is a Lorentzian metric, and there exists a constant $\upeta_0 > 0$ 
	such that $0 \leq \upeta \leq \upeta_0$ implies that the following estimates hold for 
	its inverse $g^{\mu \nu}:$ 
	\begin{subequations}
	\begin{align}
		|g^{00} + 1| & \leq 4 \upeta, && \label{E:g00upperplusoneroughestimate} \\
		\sqrt{\sum_{a=1}^3 |g^{0a}|^2} & \leq 2 c_1 e^{-2 \Omega} \sqrt{\sum_{a=1}^3 |g_{0a}|^2}, \\
		|g^{0a} g_{0a}| & \leq 2 c_1 e^{-2 \Omega} \sum_{a=1}^3 |g_{0a}|^2, && \\
		\frac{2}{3c_1} \delta^{ab}X_{a} X_{b}
		& \leq e^{2 \Omega} g^{ab}X_{a}X_{b} 
			\leq \frac{3c_1}{2}\delta^{ab}X_{a}X_{b}, && \forall (X_1,X_2,X_3) \in \mathbb{R}^3. 
				\label{E:gjkuppercomparetostandard}  
	\end{align}
	\end{subequations}
\end{lemma}
\subsection{Bootstrap assumptions}\label{SS:bootstrapassumptions}
We shall henceforth assume that $N\ge 4$, which will suffice for all conclusions in the remainder of the paper.
We start by introducing our main Sobolev norm bootstrap assumption:
\begin{align} \label{E:bootstrap}
	\totalnorm \leq \epsilon,
\end{align}
where $\totalnorm$ is the total solution norm defined in \eqref{E:totalnorm} and $\epsilon$ is a sufficiently small positive number, that will be adjusted throughout the course of our analysis in the next sections. 
We now define the constant $q$ appearing in the norms and energies from Sect.~\ref{S:norms}.
\begin{definition}[\textbf{Definition of} $q$]
	Let $\upeta_{00},$ $\upeta_{0*},$ $\upeta_{**}$ be the positive constants appearing in the conclusions of Lemma 
	\ref{L:metricfirstdifferentialenergyinequality} and $\upeta_0$ the constant from Lemma~\ref{L:ginverseestimates}. We set
	 \begin{align}
	 	q & \eqdef \frac{1}{8} \mbox{min} \big\lbrace 1, \upeta_0, \upeta_{00}, \upeta_{0*}, \upeta_{**} \big\rbrace.
	 		\label{E:qdef}
	 		 \end{align}
\end{definition}
The constant $q$ has been chosen to be small enough to close our bootstrap argument for global existence. 
In particular, inequality \eqref{E:g00upperplusoneroughestimate} with $\upeta \leq q$ implies that the metric energy building blocks $\mathcal{E}_{(\upgamma,\updelta)}[\cdot,\partial (\cdot)]$ are coercive in the sense of inequality \eqref{E:mathcalEfirstlowerbound}.
\begin{remark} [\textbf{Some redundancy in the bootstrap assumptions}] \label{R:RoughBootstrapAutomatic}
If $\epsilon$ is sufficiently small, then the inequalities \eqref{E:metricBAeta} and \eqref{E:g0jBALinfinity} (for $\upeta \leq q$) are implied by 
the definition of $\totalnorm,$ the bootstrap assumption $\totalnorm \leq \epsilon,$ and
Sobolev embedding.
\end{remark}
\section{Error term estimates}\label{S:sobolev}
In this section, we use the bootstrap assumptions to
derive $L^{\infty}$ and Sobolev estimates for all of the error terms that arise
in the expressions for the time derivatives of the metric and fluid energies
(that is, for the terms appearing on the right-hand sides of the inequalities in 
Lemmas \ref{L:metricfirstdifferentialenergyinequality} and \ref{L:FLUIDCOMMUTEDBASICSTRUCTURE}).
The main results are stated and proved as Prop. \ref{P:Sobolev}
and Prop. \ref{P:Sobolevtwo}, but
we first prove two preliminary propositions
in which we analyze the metric and the four-velocity components.
In these first two propositions, we build up a collection of quantities
that can be effectively estimated by counting spatial indices; 
the precise statements are contained in the Counting Principle of
Lemma \ref{L:countingprinciple} and
Cor. \ref{C:COUNTINGPRINCIPLE}.
We then use the Counting Principle to help shorten the proof of
many of the error term estimates of 
Prop. \ref{P:Sobolev}
and Prop. \ref{P:Sobolevtwo}.

\begin{remark}[\textbf{Philosophy regarding the norms}]
	In deriving the estimates in this section, our basic philosophy is that whenever
	possible, we ``lazily'' bound quantities in terms of the 
	total solution norm $\totalnorm$ defined in \eqref{E:totalnorm}, 
	which controls all of the relevant derivatives of all of the solution variables.
	However, in many cases, we will have to carefully avoid
	placing the full solution norm $\totalnorm$ on the right-hand side of the estimates and instead
	only place the relevant sub-norms.
	There are two main reasons why we sometimes have to be careful in this fashion:
	\textbf{i)} in order to avoid losing derivatives, we have to make sure that certain norms
	are not involved in some of the estimates, and
	$\textbf{ii)}$ to prove our future stability theorem, we have to uncover a hierarchy of
	effectively decoupled estimates.
\end{remark}

\subsection{Estimates for $g_{\mu\nu},$ $g^{\mu\nu}$ $u^{\mu},$ and their $\u$-derivatives}
In the next proposition, we derive Sobolev estimates for the metric, its inverse, and 
their non-$\u$ derivatives in terms of the solution norms.
%
\begin{proposition}[\textbf{Metric estimates}] \label{P:metric}
Let $N \geq 4$ be an integer and assume that the bootstrap assumption \eqref{E:gjkBAvsstandardmetric} holds 
on the spacetime slab $[0,T) \times \mathbb{T}^3$ for some constant $c_1 \geq 1.$ 
Then there exists a constant $\epsilon' > 0$ such that if $\totalnorm(t) \leq \epsilon'$ for $t \in [0,T),$ 
then the following estimates also hold on $[0,T)$
(and the implicit constants in the estimates can depend on $N$ and $c_1$):
\begin{subequations}
\begin{alignat}{2}	
\|g_{jk}\|_{L^{\infty}} 
& \lesssim && \,e^{2\Omega},
 	\label{E:GJKLOWERLINFTY} \\
\|\g_t g_{jk} - 2 \omega g_{jk} \|_{H^{N-1}} 
& \lesssim &&\, e^{(2-q)\Omega}\gnorm{N-1},
 \label{E:PARTIALTGJKMINUS2OMEGAGJKHNMINUSONE} \\
\|\g_t \g_i g_{jk} \|_{H^{N-1}} 
& \lesssim &&\, e^{(3-q)\Omega}
	\left\lbrace
		\gnorm{N-1}
		+ \hstarstarellipticnorm{N-1}
	\right\rbrace,
 	\label{E:PARTIALTPARTIALIGJKHNMINUSONE} \\
\| g^{00} + 1 \|_{H^{N-1}} 
+ \| \partial_t g^{00} \|_{H^{N-1}} 
+ e^{-\Omega}\|\g_i g^{00}\|_{H^{N-1}} 
& \lesssim &&\, e^{-q\Omega} \gnorm{N-1},
 		\label{E:G00UPPERPLUSONEHNMINUSONE}
 		\\
\| g^{0j} \|_{H^{N-1}} 
+ \| \partial_t g^{0j} \|_{H^{N-1}} 
+ e^{-\Omega} \|\g_i g^{0j}\|_{H^{N-1}} 
& \lesssim  && \, e^{-(1 + q)\Omega}\gnorm{N-1},	
	\label{E:G0JUPPERHNMINUSONE} \\ 
 \|g^{jk}\|_{L^{\infty}} 
+ \|\partial_t g^{jk}\|_{L^{\infty}} 
& \lesssim && \, e^{-2\Omega}, 
	\label{E:GJKUPPERLINFTY} \\
\| \underpartial g^{jk} \|_{H^{N-2}} 
& \lesssim &&  e^{-2\Omega} \gnorm{N-1},
	\label{E:DERIVATIVESOFGJKUPPERHNMINUS2} \\
 \| \partial_i g^{jk} \|_{H^{N-1}} 
& \lesssim && \, e^{-(1 + q)\Omega} \gnorm{N-1},
	\label{E:partialu2}  \\
\|\g_t g^{jk} + 2 \omega g^{jk} \|_{H^{N-1}}
& \lesssim && \, e^{-(2+q)\Omega}\gnorm{N-1}, 
	\label{E:SUPERIMPORTANTCOMMUTATORFACTOR} \\
\|g^{aj}\g_tg_{ak}-2\omega\delta^j_k\|_{H^{N-1}} 
& \lesssim && e^{-q\Omega}\gnorm{N-1}.
 \label{E:partialu5} 
\end{alignat}
\end{subequations}
The norm $\gnorm{N-1}$ is defined in \eqref{E:TOTALBELOWTOPSPATIALDERIVATIVEMETRICNORM}.
\end{proposition}

\begin{proof}
The proof of the proposition, which is based in part on Lemma \ref{L:ginverseformluas}, 
is exactly the same as the proof of the analogous proposition
in \cite[Sect. 9.1]{iRjS2012} and relies on the ideas
of \cite[Lemmas 9,11,18]{hR2008}.
We shall illustrate the basic ideas by
proving \eqref{E:SUPERIMPORTANTCOMMUTATORFACTOR}
under the assumption that 
\eqref{E:G0JUPPERHNMINUSONE},
\eqref{E:GJKUPPERLINFTY},
and
\eqref{E:DERIVATIVESOFGJKUPPERHNMINUS2}
have already been proved;
we direct the reader to the aforementioned references for complete proofs.
To proceed, we first note the following matrix identity (for $j,k = 1,2,3$):
\begin{align}
\g_tg^{jk}+2\omega g^{jk}
& =-g^{j\alpha}g^{k\beta}\g_tg_{\alpha\beta}+2 \omega g^{j\alpha}g^{k\beta}g_{\alpha\beta} \label{E:SUPERIMPORTANT1}\\
& = 
-g^{0j}g^{0k} \g_tg_{00}
-g^{0j}g^{ka} \g_tg_{0a}
-g^{ja}g^{0k} \g_tg_{0a}
- g^{ja}g^{kb}\underbrace{(\g_tg_{ab}-2\omega g_{ab}) }_{=e^{2\Omega}\g_th_{ab}}
	\notag \\
& \ \ + 2 \omega g^{0j}g^{0k} g_{00}  
+ 2 \omega g^{0j}g^{ak} g_{0a} 
+ 2 \omega g^{aj}g^{0k} g_{0a}.
\notag
\end{align}
Hence, by Lemma \ref{L:backgroundaoftestimate}, 
Prop. \ref{P:F1FkLinfinityHN}, 
and Sobolev embedding, we have
\begin{align} \label{E:PARTIALTGJKUPPERFIRSTESTIMATE}
	\| \g_tg^{jk}+2\omega g^{jk} \|_{H^{N-1}}
	& \lesssim 
		\| g^{0j} \|_{H^{N-1}} \| g^{0k} \|_{H^{N-1}} \| \g_t g_{00} \|_{H^{N-1}}
			\\
	& \ \ 
		+ \sum_{j,k=1}^3
			\left\lbrace
				 \| g^{ka} \|_{L^{\infty}}
				 +
				 \| \underpartial g^{ka} \|_{H^{N-2}} 
			\right\rbrace
			\| g^{0j} \|_{H^{N-1}} 
			\| \g_t g_{0a} \|_{H^{N-1}}
				\notag \\
	& \ \ +
		 e^{2 \Omega}
		 \sum_{a,b,j,k=1}^3
			\left\lbrace
				 \| g^{ab} \|_{L^{\infty}}
				 +
				 \| \underpartial g^{ab} \|_{H^{N-2}} 
			\right\rbrace^2
			\| \g_t h_{jk} \|_{H^{N-1}}
			\notag \\
	& \ \ +
		\left\lbrace
			1 + \| g_{00} + 1 \|_{H^{N-1}}
		\right\rbrace
		\| g^{0j} \|_{H^{N-1}} \| g^{0k} \|_{H^{N-1}} 
		\notag \\
	& \ \ 
		+ \sum_{j,k=1}^3
			\left\lbrace
				 \| g^{ak} \|_{L^{\infty}}
				 +
				 \| \underpartial g^{ak} \|_{H^{N-2}} 
			\right\rbrace
			\| g^{0j} \|_{H^{N-1}} 
			\| g_{0a} \|_{H^{N-1}}.
			\notag
\end{align}
Inserting the estimates 
\eqref{E:G0JUPPERHNMINUSONE},
\eqref{E:GJKUPPERLINFTY},
and
\eqref{E:DERIVATIVESOFGJKUPPERHNMINUS2}
and the inequalities
$\| g_{00} + 1 \|_{H^{N-1}} \lesssim e^{-q \Omega} \gnorm{N-1},$
$\| \g_t g_{00} \|_{H^{N-1}} \lesssim e^{-q \Omega} \gnorm{N-1},$
$\| g_{0a} \|_{H^{N-1}} \lesssim e^{(1-q) \Omega} \gnorm{N-1},$
$\| \g_t g_{0a} \|_{H^{N-1}} \lesssim e^{(1-q) \Omega} \gnorm{N-1},$
and
$\| \g_t h_{jk} \|_{H^{N-1}} \lesssim e^{-q \Omega} \gnorm{N-1}$
into \eqref{E:PARTIALTGJKUPPERFIRSTESTIMATE},
we conclude \eqref{E:SUPERIMPORTANTCOMMUTATORFACTOR}.
\end{proof}

The following proposition, in which 
we derive estimates for 
$u,$ 
$\partial_{\mathbf{u}} u,$ 
$\partial_{\mathbf{u}} g,$
and $\partial_{\mathbf{u}}(g^{-1}),$
complements Prop. \ref{P:metric}.

\begin{proposition} [\textbf{Estimates for} $u$ \textbf{and the} $\u$ 
\textbf{derivative of} $u$ \textbf{and the metric}] \label{P:SobolevMetricNEW}
Let $N \geq 4$ be an integer and assume that the bootstrap assumption \eqref{E:gjkBAvsstandardmetric} holds 
on the spacetime slab $[0,T) \times \mathbb{T}^3$ for some constant $c_1 \geq 1.$ 
Then there exists a constant $\epsilon' > 0$ such that if $\totalnorm(t) \leq \epsilon'$ for $t \in [0,T),$ 
then the following estimates also hold on $[0,T)$
(and the implicit constants in the estimates can depend on $N$ and $c_1$):
\begin{subequations}
\begin{alignat}{2}
\| u^0 - 1 \|_{H^{N-1}}
&\lesssim &&\, e^{-q\Omega} \totalbelowtopnorm{N-1},
\label{E:u0MINUSONEHNMINUSONE} 
\\	
 \| \partial_{\mathbf{u}} u^0 \|_{H^{N-1}}
& \lesssim &&\, e^{-q\Omega} 
			\left\lbrace
				\totalbelowtopnorm{N-1}
				+ \totalbelowtopunorm{N-1}
			\right\rbrace, 
		  \label{E:PARTIALUU0HNMINUSONE}   \\
 \| \partial_i u^0 \|_{H^{N-1}}
& \lesssim &&\, e^{(1-q)\Omega} 
	\left\lbrace 
		\totalbelowtopnorm{N-1} 
		+ \topvelocitynorm{N-1}
	\right\rbrace,      \label{E:PARTIALIU0HNMINUSONE}   \\
\| u_0 + 1 \|_{H^{N-1}} 
&\lesssim&& \,e^{-q\Omega}\totalbelowtopnorm{N-1},      
	\label{E:uzerolowerplusone} \\
\| \partial_{\mathbf{u}} u_0 \|_{H^{N-1}}
&\lesssim&& \,e^{-q\Omega}
	\left\lbrace
				\totalbelowtopnorm{N-1}
				+ \totalbelowtopunorm{N-1}
	\right\rbrace,      \label{E:PARTIALUuzerolowerplusone}
	\\
 \| \partial_i u_0 \|_{H^{N-1}}
& \lesssim &&\, e^{(1-q)\Omega} 
	\left\lbrace 
		\totalbelowtopnorm{N-1} 
		+ \topvelocitynorm{N-1}
	\right\rbrace,      \label{E:PARTIALIU0LOWERHNMINUSONE}  \\
\| u_j \|_{H^{N-1}} 
	&\lesssim&& \, e^{(1-q)\Omega}\totalbelowtopnorm{N-1}, \label{E:ujlower}
	\\
\| \partial_{\mathbf{u}} u_j \|_{H^{N-1}}
&\lesssim&& \, e^{(1-q)\Omega}
	\left\lbrace
				\totalbelowtopnorm{N-1}
				+ \totalbelowtopunorm{N-1}
	\right\rbrace, \label{E:PARTIALUujlower}
		\\
\| \partial_i u_j \|_{H^{N-1}}
&\lesssim&& \, e^{(2-q)\Omega}
	\left\lbrace
				\totalbelowtopnorm{N-1}
				+ \topvelocitynorm{N-1}
	\right\rbrace. \label{E:PARTIALIujlower}
\end{alignat}
\end{subequations}
For $\u g_{jk},$
$\u \g_i g_{jk},$
 and $\u \partial_t g_{jk},$
we have the following estimates on $[0,T):$
\begin{subequations}
\begin{align}
	 \|\u g_{jk}\|_{L^{\infty}} 
	& \lesssim \, e^{2\Omega},
		\label{E:PARTIALUGNKLOWERLINFINITY} \\
	 \|\u g_{jk} - 2 \omega g_{jk} \|_{H^{N-1}}
	& \lesssim 
				e^{(2-q)\Omega}
				\totalbelowtopnorm{N-1},
		\label{E:PARTIALUGNKLOWERHNMINUSONE}
		\\
	\|\u \g_i g_{jk} \|_{H^{N-1}}
	& \lesssim e^{(3-q)\Omega} 
		\left\lbrace
				\totalbelowtopnorm{N-1}
				+ \totalbelowtopunorm{N-1}
		\right\rbrace,
		\label{E:partialupartialiglowerjk} 
				\\
	\|\u(\g_t g_{jk} - 2 \omega g_{jk})\|_{H^{N-1}}
	& \lesssim e^{(2-q)\Omega} 
		\left\lbrace
				\totalbelowtopnorm{N-1}
				+ \totalbelowtopunorm{N-1}
		\right\rbrace.
\label{E:partialupartialtglowerjk} 
\end{align}
\end{subequations}
For the $\u$-derivatives of the inverse metric terms, 
we have the following estimates on $[0,T):$
\begin{subequations}
\begin{alignat}{2}
\|\u g^{00}\|_{H^{N-1}} + e^{\Omega} \|\u g^{0j}\|_{H^{N-1}}  & \lesssim &&  
	\, e^{-q\Omega} 
		\left\lbrace
				\totalbelowtopnorm{N-1}
				+ \totalbelowtopunorm{N-1}
		\right\rbrace, 
		\label{E:partialuNEW}\\
\|\u g^{jk}\|_{L^{\infty}} & \lesssim &&\, e^{-2\Omega}, 
		\label{E:partialugupperjk} \\ 
\|\underpartial \u g^{jk}\|_{H^{N-2}} & \lesssim &&\, e^{-2\Omega} 
	\left\lbrace
				\totalbelowtopnorm{N-1}
				+ \totalbelowtopunorm{N-1}
	\right\rbrace, 
		\label{E:partialugupperjk2} \\
\| \partial_{\mathbf{u}} g^{jk} + 2 \omega g^{jk} \|_{H^{N-1}} 
+ \| \partial_{\mathbf{u}} (\g_t g^{jk} + 2 \omega g^{jk}) \|_{H^{N-1}} 
& \lesssim && \, e^{-(2+q)\Omega} 
	\left\lbrace
				\totalbelowtopnorm{N-1}
				+ \totalbelowtopunorm{N-1}
	\right\rbrace, 
	\label{E:SUPERIMPORTANT} \\
\| \u (g^{aj} \g_tg_{ak}-2\omega\delta^j_k) \|_{H^{N-1}}
&\lesssim && e^{-q\Omega} 
			\left\lbrace
				\totalbelowtopnorm{N-1}
				+ \totalbelowtopunorm{N-1}
			\right\rbrace. 
		\label{E:partialu5NEW}
\end{alignat}
\end{subequations}
The norms appearing on the right-hand sides of the above inequalities are defined in Sect.~\ref{S:norms}.

\end{proposition}
\noindent

\prf
Throughout this proof, we will often make use of 
Lemma \ref{L:backgroundaoftestimate}
and the Sobolev embedding result $H^2(\mathbb{T}^3) \hookrightarrow L^{\infty}(\mathbb{T}^3)$
without mentioning it every time. \emph{We also stress that the order in which 
we prove the estimates is important in some cases.} In particular, 
order in which we prove 
the estimates is generally different
than the order in which we have stated them in the proposition.

\noindent
{\em Proof of~\eqref{E:u0MINUSONEHNMINUSONE}-\eqref{E:PARTIALIujlower} and  \eqref{E:PARTIALUGNKLOWERLINFINITY}-\eqref{E:partialupartialtglowerjk}.}
To prove \eqref{E:u0MINUSONEHNMINUSONE}, we first use $g_{\alpha \beta} u^{\alpha} u^{\beta} = - 1$
to isolate $u^0:$
\begin{align} \label{E:U0UPPERISOLATEDagain}
	u^0 = - \frac{g_{0a}u^a}{g_{00}} + \sqrt{1 + \Big(\frac{g_{0a}u^a}{g_{00}}\Big)^2 - \frac{g_{ab}u^a u^b}{g_{00}} 
		- \Big(\frac{g_{00} + 1}{g_{00}}\Big)}.
\end{align}
Applying Cor. \ref{C:SobolevTaylor} to the right-hand side of \eqref{E:U0UPPERISOLATEDagain}, we deduce that
\begin{align} \label{E:u0HNMINUSONEfirstinequality}
	\| u^0 - 1 \|_{H^{N-1}} 
	& \lesssim
		\bigg\lbrace \Big\| \frac{g_{0a}u^a}{g_{00}} \Big\|_{H^{N-1}}
		+ \Big\| \Big(\frac{g_{0a}u^a}{g_{00}}\Big)^2 - \frac{g_{ab}u^a u^b}{g_{00}} - \frac{g_{00} + 1}{g_{00}} \Big\|_{H^{N-1}}
		\bigg\rbrace.
\end{align}
We will show that 
$\left \| \frac{g_{0a}u^a}{g_{00}} \right\|_{H^{N-1}},$
$\left \| \frac{g_{ab}u^a u^b}{g_{00}} \right\|_{H^{N-1}},$
$\left \| \frac{g_{00} + 1}{g_{00}} \right\|_{H^{N-1}} \lesssim e^{-q \Omega} \totalbelowtopnorm{N-1}.$
The desired estimate \eqref{E:u0MINUSONEHNMINUSONE} then follows from these estimates and \eqref{E:u0HNMINUSONEfirstinequality}.

To proceed, we first
use Cor. \ref{C:SobolevTaylor} 
and the definition of
$\totalbelowtopnorm{N-1}$
to deduce that
$\left\| \frac{1}{g_{00}} + 1 \right\|_{H^{N-1}} \lesssim \| g_{00} + 1 \|_{H^{N-1}} \lesssim e^{-q \Omega} \totalbelowtopnorm{N-1}.$
Also using Prop. \ref{P:F1FkLinfinityHN}, 
and the inequalities 
$\| g_{0a} \|_{H^{N-1}} \lesssim e^{(1-q) \Omega} \totalbelowtopnorm{N-1}$
and
$\|u^a \|_{H^{N-1}} \lesssim e^{-(1+q) \Omega} \totalbelowtopnorm{N-1},$
we obtain
\begin{align}
\left \| \frac{g_{0a}u^a}{g_{00}} \right\|_{H^{N-1}} 
& \lesssim 
\left\lbrace 
	1 + \left\| \frac{1}{g_{00}} + 1 \right\|_{H^{N-1}} 
\right\rbrace
\|g_{0a}\|_{H^{N-1}}
\|u^a\|_{H^{N-1}}  \label{E:uestimate1} 
\lesssim e^{-q \Omega} \totalbelowtopnorm{N-1} 
\end{align}
as desired.

Next, from the bootstrap assumption \eqref{E:gjkBAvsstandardmetric},
the inequality $\|\underpartial h_{ab}\|_{H^{N-2}} \le \totalbelowtopnorm{N-1},$ 
and the relation $g_{ab} = e^{2 \Omega} h_{ab},$
we infer that
$\|g_{ab} \|_{L^{\infty}} + \|\underpartial g_{ab}\|_{H^{N-2}} \lesssim e^{2\Omega}$.
Hence, arguing as we did above, we conclude that
\begin{align}   \label{E:uestimate2}
\left\|\frac{g_{ab}u^a u^b}{g_{00}} \right\|_{H^{N-1}} 
& \lesssim 
	\left\lbrace 
	1 + \left\| \frac{1}{g_{00}} + 1 \right\|_{H^{N-1}} 
\right\rbrace
	\left\lbrace
		\|g_{ab} \|_{L^{\infty}} + \|\underpartial g_{ab}\|_{H^{N-2}} 
	\right\rbrace
	\|u^a\|_{H^{N-1}}
	\|u^b\|_{H^{N-1}}
		\\
& \lesssim e^{2\Omega} e^{-(2+2q)\Omega} \totalbelowtopnorm{N-1} \lesssim e^{- q \Omega} \totalbelowtopnorm{N-1}
		\notag
\end{align}
as desired.

Similarly, we deduce that
\begin{align} \label{E:uestimate3}
\left \| \frac{g_{00} + 1}{g_{00}} \right\|_{H^{N-1}}
& \lesssim 
	\left\lbrace 
	1 + \left\| \frac{1}{g_{00}} + 1 \right\|_{H^{N-1}} 
\right\rbrace
	\| g_{00} + 1 \|_{H^{N-1}}
	\lesssim e^{- q \Omega} \totalbelowtopnorm{N-1}
\end{align}
as desired, and the proof of \eqref{E:u0MINUSONEHNMINUSONE} is complete.

To prove \eqref{E:uzerolowerplusone}, we first use the identity 
$u_0 + 1 
= 1 - u^0 
+ (g_{00} + 1) u^0
+ g_{0a} u^a$ 
and Prop. \ref{P:F1FkLinfinityHN}
to deduce that
\begin{align} \label{E:U0LOWERFIRSTESTIMATE}
	\| u_0 + 1 \|_{H^{N-1}} 
	& \lesssim 
		\| u^0  - 1 \|_{H^{N-1}}
		+ \left\lbrace 1 + \| u^0 - 1 \|_{H^{N-1}} \right\rbrace \| g_{00} + 1 \|_{H^{N-1}}
		+  \| g_{0a} \|_{H^{N-1}} \| u^a \|_{H^{N-1}}.
\end{align}
Inserting the estimate \eqref{E:u0MINUSONEHNMINUSONE}
and the inequalities 
$\| g_{00} + 1 \|_{H^{N-1}} \lesssim e^{-q \Omega} \totalbelowtopnorm{N-1},$
$\| g_{0a} \|_{H^{N-1}} \lesssim e^{(1-q) \Omega} \totalbelowtopnorm{N-1},$
and $\| u^a \|_{H^{N-1}} \lesssim e^{-(1+q) \Omega} \totalbelowtopnorm{N-1}$
into the right-hand side of \eqref{E:U0LOWERFIRSTESTIMATE}, we conclude \eqref{E:uzerolowerplusone}.
Similarly, \eqref{E:ujlower} follows from  
Prop. \ref{P:F1FkLinfinityHN},
\eqref{E:u0MINUSONEHNMINUSONE},
and the identity $u_j = g_{j\alpha} u^{\alpha}.$

To prove \eqref{E:PARTIALUGNKLOWERLINFINITY}, we first observe that
$\u g_{jk} = e^{2\Omega}\u h_{jk}+2 u^0 \omega g_{jk}.$ 
From \eqref{E:GJKLOWERLINFTY}, \eqref{E:u0MINUSONEHNMINUSONE}, 
and the inequality $\|\u h_{jk}\|_{H^{N-1}} \lesssim e^{-q \Omega} \totalbelowtopunorm{N-1},$
we deduce that 
$\|\u g_{jk}\|_{L^{\infty}} \lesssim e^{2\Omega}\totalbelowtopunorm{N-1} + e^{2\Omega}\lesssim e^{2\Omega}$ as desired.

To prove \eqref{E:PARTIALUGNKLOWERHNMINUSONE}, we first use the identities
$g_{jk} = e^{2 \Omega} h_{jk}$
and
$
\u g_{jk} - 2\omega g_{jk} = 2 \omega (u^0 - 1) g_{jk}  + e^{2\Omega}u^a\g_ah_{jk} + u^0 (\g_t g_{jk}-2\omega g_{jk})
$ 
and Prop. \ref{P:F1FkLinfinityHN} to deduce that
\begin{align} \label{E:PARTIALUGJKFIRSTESTIMATE}
	\| \u g_{jk} - 2\omega g_{jk} \|_{H^{N-1}}
	& \lesssim \| u^0 - 1 \|_{H^{N-1}} 
		\left\lbrace 
			\| g_{jk} \|_{L^{\infty}} 
			+ e^{2 \Omega} \| \underpartial h_{jk} \|_{H^{N-2}} 
		\right\rbrace
		+ e^{2\Omega}\| u^a \|_{H^{N-1}} \| \g_a h_{jk} \|_{H^{N-1}}
		 \\
	& \ \ 
		 + \left\lbrace 1 + \| u^0 - 1 \|_{H^{N-1}} \right\rbrace
		 	\| \g_t g_{jk} - 2\omega g_{jk} \|_{H^{N-1}}.
		 \notag
\end{align}
We now insert the bounds
\eqref{E:GJKLOWERLINFTY}, 
\eqref{E:PARTIALTGJKMINUS2OMEGAGJKHNMINUSONE},
\eqref{E:u0MINUSONEHNMINUSONE}, 
$\| u^a \|_{H^{N-1}} \lesssim e^{-(1 + q) \Omega} \totalbelowtopnorm{N-1},$
and $\| \g_a h_{jk} \|_{H^{N-1}} \lesssim e^{(1 - q) \Omega} \totalbelowtopnorm{N-1}$
into the right-hand side of \eqref{E:PARTIALUGJKFIRSTESTIMATE}
to conclude \eqref{E:PARTIALUGNKLOWERHNMINUSONE}.

To prove \eqref{E:PARTIALUU0HNMINUSONE}, we first apply $\u$ to the equation $g_{\alpha \beta} u^{\alpha} u^{\beta} = -1$
and solve for $\u u^0:$ 
\begin{align} \label{E:PARTIALUU0FIRSTRELATION}
	\u u^0 
	& = - 
			\frac{1}{u_0}
			\left\lbrace
				\frac{1}{2}(\u g_{00}) (u^0)^2
				+ (\u g_{0a}) u^0 u^a
				+ \frac{1}{2}(\u g_{ab}) u^a u^b
				+ u_a \u u^a
			\right\rbrace.
\end{align}
From Cor. \ref{C:SobolevTaylor} and \eqref{E:uzerolowerplusone},
we deduce that
$\left\| \frac{1}{u_0} + 1 \right\|_{H^{N-1}} \lesssim 1.$
Using this estimate,
\eqref{E:u0MINUSONEHNMINUSONE},
\eqref{E:PARTIALUU0FIRSTRELATION}, and Prop. \ref{P:F1FkLinfinityHN}, 
and decomposing $\u g_{ab} = (\u g_{ab} - 2 \omega g_{ab}) + 2 \omega g_{ab},$
we deduce that
\begin{align} \label{E:PARTIALUU0FIRSTESTIMATE}
	\| 	\u u^0 \|_{H^{N-1}} 
	& \lesssim 
		\| \u g_{00} \|_{H^{N-1}}
		+ \| \u g_{0a} \|_{H^{N-1}} \| u^a \|_{H^{N-1}}
		+ \| \u g_{ab} - 2 \omega g_{ab} \|_{H^{N-1}}  
			\| u^a \|_{H^{N-1}} \| u^b \|_{H^{N-1}}
			\\
	& \ \ + \left\lbrace \| g_{ab} \|_{L^{\infty}} + \| \underpartial g_{ab} \|_{H^{N-2}}  \right\rbrace
			\| u^a \|_{H^{N-1}} \| u^b \|_{H^{N-1}}
		+ \| u_a \|_{H^{N-1}} \| u^a \|_{H^{N-1}}.
		\notag
\end{align}
Inserting the estimates 
\eqref{E:GJKLOWERLINFTY},
\eqref{E:ujlower},
\eqref{E:PARTIALUGNKLOWERHNMINUSONE},
$\| \u g_{00} \|_{H^{N-1}} \lesssim e^{-q\Omega} \totalbelowtopunorm{N-1},$
$\| \u g_{0a} \|_{H^{N-1}} \lesssim e^{(1 - q)\Omega} \totalbelowtopunorm{N-1},$
$\| \underpartial g_{ab} \|_{H^{N-2}} 
= e^{2 \Omega} \| \underpartial h_{ab} \|_{H^{N-2}}
\lesssim e^{2 \Omega} \totalbelowtopnorm{N-1},$
and $\| u^a \|_{H^{N-1}} \lesssim e^{-(1 + q)\Omega} \totalbelowtopnorm{N-1}$
into the right-hand side of \eqref{E:PARTIALUU0FIRSTESTIMATE},
we conclude \eqref{E:PARTIALUU0HNMINUSONE}.

To prove \eqref{E:PARTIALIU0HNMINUSONE}, 
we first note that \eqref{E:PARTIALUU0FIRSTRELATION} holds with $\g_i$ in place of $\u.$
Hence, using essentially the same reasoning used to prove \eqref{E:PARTIALUU0HNMINUSONE},
we conclude \eqref{E:PARTIALIU0HNMINUSONE}.
The main difference is that
the estimate for $\g_i u^0$ involves the top-order spatial derivatives of the $u^j$
and not any $\u$ derivative, so we have the norm $\topvelocitynorm{N-1}$
on the right-hand side of \eqref{E:PARTIALIU0HNMINUSONE} instead of the norm 
$\totalbelowtopunorm{N-1}$ appearing in \eqref{E:PARTIALUU0HNMINUSONE}.
We also note that the application of the operator $\g_i$
results in the extra factor of $e^{\Omega}$ in the estimate for $\| \partial_i u^0 \|_{H^{N-1}}$
compared to $\| \u u^0 \|_{H^{N-1}}.$
This is of course consistent with our definition of the norms as laid out in Sect.~\ref{SS:norms}.
This extra factor is closely tied to the counting principle mentioned in the introduction 
(and rigorously introduced in Sect.~\ref{SS:mainestimates}).

To prove \eqref{E:PARTIALUuzerolowerplusone},
we first use the identity 
$\u u_0 
= (\u g_{00})u^0
+ (\u g_{0a}) u^a
+ g_{00} \u u^0
+ g_{0a} \u u^a$ 
and Prop. \ref{P:F1FkLinfinityHN}
to deduce that
\begin{align} \label{E:PARTIALUU0LOWERFIRSTESTIMATE}
	\| \u u_0  \|_{H^{N-1}} 
	& \lesssim 
		\left\lbrace 1 + \| u^0 - 1 \|_{H^{N-1}} \right\rbrace \| \u g_{00} \|_{H^{N-1}} 
		+ \| \u g_{0a} \|_{H^{N-1}} \| u^a \|_{H^{N-1}} 
			\\
	& \ \
		+ \left\lbrace 1 + \| g_{00} + 1 \|_{H^{N-1}} \right\rbrace \| \u u^0 \|_{H^{N-1}} 
		+ \| g_{0a} \|_{H^{N-1}} \| \u u^a \|_{H^{N-1}}.
		\notag
\end{align}
Inserting the estimates 
\eqref{E:u0MINUSONEHNMINUSONE},
\eqref{E:PARTIALUU0HNMINUSONE},
and the inequalities 
$\| g_{00} + 1 \|_{H^{N-1}} \lesssim e^{-q \Omega} \totalbelowtopnorm{N-1},$
$\| \u g_{00} \|_{H^{N-1}} \lesssim e^{-q \Omega} \totalbelowtopunorm{N-1},$
$\| g_{0a} \|_{H^{N-1}} \lesssim e^{(1-q) \Omega} \totalbelowtopnorm{N-1},$
$\| \u g_{0a} \|_{H^{N-1}} \lesssim e^{(1-q) \Omega} \totalbelowtopunorm{N-1},$
$\| u^a \|_{H^{N-1}} \lesssim e^{-(1+q) \Omega} \totalbelowtopnorm{N-1},$
and $\| \u u^a \|_{H^{N-1}} \lesssim e^{-(1+q) \Omega} \totalbelowtopunorm{N-1}$
into the right-hand side of \eqref{E:PARTIALUU0LOWERFIRSTESTIMATE}, 
we conclude \eqref{E:PARTIALUuzerolowerplusone}.

The proof of \eqref{E:PARTIALIU0LOWERHNMINUSONE} is similar to the proof of 
\eqref{E:PARTIALUuzerolowerplusone}, so we omit the full details. The main difference is that
the estimate for $\g_i u_0$ involves the top-order spatial derivatives of the $u^j$
and not any $\u$ derivative, so we have the norm $\topvelocitynorm{N-1}$
on the right-hand side of \eqref{E:PARTIALIU0LOWERHNMINUSONE} instead of the norm 
$\totalbelowtopunorm{N-1}$ appearing in \eqref{E:PARTIALUuzerolowerplusone}.

To prove \eqref{E:PARTIALUujlower},  
we first use the identity $u_j = g_{j \alpha} u^{\alpha},$
the decomposition $\u g_{ab} = (\u g_{ab} - 2 \omega g_{ab}) + 2 \omega g_{ab},$
and Prop. \ref{P:derivativesofF1FkL2} to deduce that
\begin{align} \label{E:PARTIALUUJLOWERFIRSTESTIMATE}
	\| \u u_j \|_{H^{N-1}}
	& \lesssim \left\lbrace 1 + \| u^0 - 1 \|_{H^{N-1}} \right\rbrace \| \u g_{0j} \|_{H^{N-1}} 
		+ \| \u g_{ja} - 2 \omega g_{ja} \|_{H^{N-1}}
			\| u^a \|_{H^{N-1}} \\
	& \ \ 
		+ \left\lbrace \| g_{ja} \|_{L^{\infty}} + \| \underpartial g_{ja} \|_{H^{N-2}}  \right\rbrace
			\| u^a \|_{H^{N-1}}
		+ \| g_{j0} \|_{H^{N-1}} \| \u u^0 \|_{H^{N-1}}
		\notag \\
	& \ \ 
			+ \left\lbrace \| g_{ja} \|_{L^{\infty}} + \| \underpartial g_{ja} \|_{H^{N-2}}  \right\rbrace
			\| \u u^a \|_{H^{N-1}}.
			\notag
\end{align}
Inserting the estimates
\eqref{E:GJKLOWERLINFTY},
\eqref{E:u0MINUSONEHNMINUSONE},
\eqref{E:PARTIALUU0HNMINUSONE},
\eqref{E:PARTIALUGNKLOWERHNMINUSONE},
and the inequalities
$\| \underpartial g_{ja} \|_{H^{N-2}} = e^{2 \Omega} \| \underpartial h_{ja} \|_{H^{N-2}} \lesssim e^{2 \Omega} \totalbelowtopnorm{N-1},$
$\| g_{0j} \|_{H^{N-1}} \lesssim e^{(1-q) \Omega} \totalbelowtopnorm{N-1},$
$\| \u g_{0j} \|_{H^{N-1}} \lesssim e^{(1-q) \Omega} \totalbelowtopunorm{N-1},$
$\| u^a \|_{H^{N-1}} \lesssim e^{-(1+q) \Omega} \totalbelowtopnorm{N-1},$
and $\| \u u^a \|_{H^{N-1}} \lesssim e^{-(1+q) \Omega} \totalbelowtopunorm{N-1}$
into the right-hand side of \eqref{E:PARTIALUUJLOWERFIRSTESTIMATE},
we conclude \eqref{E:PARTIALUujlower}.

To prove \eqref{E:PARTIALIujlower},
we first use the identity $\g_i u_j = (\g_i g_{j \alpha}) u^{\alpha} + g_{j \alpha} \g_i u^{\alpha}$
and Prop. \ref{P:derivativesofF1FkL2} to deduce that
\begin{align}  \label{E:PARTIALIUJLOWERFIRSTESTIMATE}
	\| \g_i u_j \|_{H^{N-1}}
	& \lesssim
		\left\lbrace 1 + \| u^0 - 1 \|_{H^{N-1}} \right\rbrace 
		\| \g_i g_{j0} \|_{H^{N-1}} 
		+ \| \g_i g_{ja} \|_{H^{N-1}} \| u^a \|_{H^{N-1}} 
			\\
	& \ \ + \| g_{j0} \|_{H^{N-1}} \| \g_i u^0 \|_{H^{N-1}}
		+ \left\lbrace \| g_{ja} \|_{L^{\infty}} + \| \underpartial g_{ja} \|_{H^{N-2}} \right\rbrace 
			\| \g_i u^a \|_{H^{N-1}}.
			\notag
\end{align}
Inserting the estimates
\eqref{E:GJKLOWERLINFTY},
\eqref{E:u0MINUSONEHNMINUSONE},
and
\eqref{E:PARTIALIU0HNMINUSONE}
and the inequalities
$\| \underpartial g_{ja} \|_{H^{N-2}} = e^{2 \Omega} \| \underpartial h_{ja} \|_{H^{N-2}} \lesssim e^{2 \Omega} \totalbelowtopnorm{N-1},$
$\| \g_i g_{ja} \|_{H^{N-1}} = e^{2 \Omega} \| \g_i h_{ja} \|_{H^{N-1}} \lesssim e^{(3 - q) \Omega} \totalbelowtopnorm{N-1},$
$\| g_{0j} \|_{H^{N-1}} \lesssim e^{(1-q) \Omega} \totalbelowtopnorm{N-1},$
$\| \g_i g_{0j} \|_{H^{N-1}} \lesssim e^{(2-q) \Omega} \totalbelowtopnorm{N-1},$
$\| u^a \|_{H^{N-1}} \lesssim e^{-(1+q) \Omega} \totalbelowtopnorm{N-1},$
and 
 $\| \g_i u^a \|_{H^{N-1}} \lesssim e^{- q \Omega} \topvelocitynorm{N-1}$
into the right-hand side of \eqref{E:PARTIALIUJLOWERFIRSTESTIMATE}, 
we conclude \eqref{E:PARTIALIujlower}.

To prove \eqref{E:partialupartialiglowerjk},
we first use the identity 
$\u \g_i g_{jk} = e^{2 \Omega} \u \g_i h_{jk} + 2 \omega u^0 \g_i h_{jk}$
and Prop. \ref{P:F1FkLinfinityHN} to deduce
\begin{align} \label{E:PARTIALUPARTIALIGJKFIRSTESTIMATE}
	\| \u \g_i g_{jk} \|_{H^{N-1}}
	& \lesssim e^{2 \Omega} \| \u \g_i h_{jk} \|_{H^{N-1}}
		+ e^{2 \Omega} \left\lbrace 1 + \| u^0 - 1 \|_{H^{N-1}} \right\rbrace \| \g_i h_{jk} \|_{H^{N-1}}.
\end{align}
Inserting the estimate \eqref{E:u0MINUSONEHNMINUSONE}
and the inequalities 
$\| \u \g_i h_{jk} \|_{H^{N-1}} \lesssim e^{(1-q)\Omega} \totalbelowtopunorm{N-1}$
and 
$\| \g_i h_{jk} \|_{H^{N-1}} \lesssim e^{(1-q)\Omega} \totalbelowtopnorm{N-1}$
into the right-hand side of 
\eqref{E:PARTIALUPARTIALIGJKFIRSTESTIMATE}, we conclude \eqref{E:partialupartialiglowerjk}.

Finally, to prove \eqref{E:partialupartialtglowerjk}, we first note that 
$ \u(\g_t g_{jk}-2\omega g_{jk}) = \u (e^{2\Omega}\g_t h_{jk}) = e^{2 \Omega} \u \g_t h_{jk} + 2 u^0 \omega e^{2\Omega} \g_t h_{jk}.$
Hence, by Prop. \ref{P:F1FkLinfinityHN} we have that 
\begin{align} \label{E:PARTIALUPARTIALTGJKFIRSTESTIMATE}
\| \u(\g_t g_{jk}-2\omega g_{jk}) \|_{H^{N-1}} 
\lesssim e^{2 \Omega} \| \u \g_t h_{jk} \|_{H^{N-1}}
+ e^{2 \Omega} \left\lbrace 1 + \| u^0 - 1 \|_{H^{N-1}} \right\rbrace \| \g_t h_{jk} \|_{H^{N-1}}.
\end{align}
Inserting the estimate \eqref{E:u0MINUSONEHNMINUSONE} and the inequalities
$\| \g_t h_{jk} \|_{H^{N-1}} \lesssim e^{-q \Omega} \totalbelowtopnorm{N-1}$
and
$\| \u \g_t h_{jk} \|_{H^{N-1}} \lesssim e^{-q \Omega} \totalbelowtopunorm{N-1}$
into the right-hand side of \eqref{E:PARTIALUPARTIALTGJKFIRSTESTIMATE},
we conclude \eqref{E:partialupartialtglowerjk}.

\vspace{0.1in}
\noindent
{\em Proof of \eqref{E:partialuNEW}-\eqref{E:partialu5NEW}:}
To prove the first estimate in \eqref{E:partialuNEW}, we first use the matrix identity 
$
\u g^{\mu \nu} = -g^{\mu \alpha}g^{\nu \beta} \u g_{\alpha \beta},
$
the decomposition $\u g_{ab} = (\u g_{ab} - 2 \omega g_{ab}) + 2 \omega g_{ab},$
and Prop. \ref{P:F1FkLinfinityHN} to deduce that
\begin{align} \label{E:PARTIALUG00UPPERFIRSTESTIMATE}
	\| \u g^{00} \|_{H^{N-1}}
	& \lesssim 
		(1 + \| g^{00} + 1 \|_{H^{N-1}})^2 \| \u g_{00} \|_{H^{N-1}}
		+ (1 + \| g^{00} + 1 \|_{H^{N-1}}) \| g^{0a} \|_{H^{N-1}} \| \u g_{0a} \|_{H^{N-1}}
		\\
	& \ \ 
		+ \| g^{0a} \|_{H^{N-1}} \| g^{0b} \|_{H^{N-1}} \| \u g_{ab} - 2 \omega g_{ab} \|_{H^{N-1}}
			\notag \\
	& \ \ + \| g^{0a} \|_{H^{N-1}} \| g^{0b} \|_{H^{N-1}} 
		\left\lbrace 
			\| g_{ab} \|_{L^{\infty}} + \| \underpartial g_{ab} \|_{H^{N-2}}  
		\right\rbrace.
		\notag
\end{align}
Inserting the estimates
\eqref{E:GJKLOWERLINFTY},
\eqref{E:G00UPPERPLUSONEHNMINUSONE},
\eqref{E:G0JUPPERHNMINUSONE},
\eqref{E:PARTIALUGNKLOWERHNMINUSONE},
and the inequalities
$\| \u g_{00} \|_{H^{N-1}} \lesssim e^{-q \Omega} \totalbelowtopunorm{N-1},$
$\| \u g_{0j} \|_{H^{N-1}} \lesssim e^{(1-q) \Omega} \totalbelowtopunorm{N-1},$
and $\| \underpartial g_{ab} \|_{H^{N-2}} = e^{2 \Omega} \| \underpartial h_{ab} \|_{H^{N-2}} \lesssim e^{2\Omega} 	\totalbelowtopnorm{N-1}$
into the right-hand side of \eqref{E:PARTIALUG00UPPERFIRSTESTIMATE},
we conclude the first estimate in \eqref{E:partialuNEW}. 
The proofs of the second estimate in \eqref{E:partialuNEW} 
and \eqref{E:partialugupperjk}-\eqref{E:partialugupperjk2} are similar, and we omit the details.

To prove the first estimate in \eqref{E:SUPERIMPORTANT}, 
we first decompose
$\partial_{\mathbf{u}} g^{jk} + 2 \omega g^{jk} 
= \g_t g^{jk} + 2 \omega g^{jk}
	+ (u^0 - 1) (\g_t g^{jk} + 2 \omega g^{jk})
	+ 2 \omega (u^0 - 1) g^{jk}
	+ u^a \g_a g^{jk}.$
Hence, by Prop. \ref{P:F1FkLinfinityHN}, we have
\begin{align} \label{E:FIRSTESTIMATEPARTIALUGJKUPPER}
	\| \partial_{\mathbf{u}} g^{jk} + 2 \omega g^{jk} \|_{H^{N-1}}
	& \lesssim 
		\| \g_t g^{jk} + 2 \omega g^{jk} \|_{H^{N-1}}
		+	\| u^0 - 1 \|_{H^{N-1}}
			\| \g_t g^{jk} + 2 \omega g^{jk} \|_{H^{N-1}}
			\\
	& \ \
		+ \left\lbrace 
				\|  g^{jk} \|_{L^{\infty}}
				+ \| \underpartial g^{jk} \|_{H^{N-2}}
			\right\rbrace 
			\| u^0 - 1 \|_{H^{N-1}}
		+ \| u^a \|_{H^{N-1}}
			\| \g_a g^{jk} \|_{H^{N-1}}.
			\notag
\end{align}
We now insert the estimates
\eqref{E:GJKUPPERLINFTY},
\eqref{E:DERIVATIVESOFGJKUPPERHNMINUS2},
\eqref{E:partialu2},
\eqref{E:SUPERIMPORTANTCOMMUTATORFACTOR},
and
\eqref{E:u0MINUSONEHNMINUSONE}
and the inequality
$\| u^a \|_{H^{N-1}} \lesssim e^{-(1+q) \Omega} \totalbelowtopnorm{N-1}$
into the right-hand side of \eqref{E:FIRSTESTIMATEPARTIALUGJKUPPER}
and thus conclude the first estimate in \eqref{E:SUPERIMPORTANT}.

To prove \eqref{E:partialu5NEW}, we first note the identity
\begin{align} \label{E:PARTIALTGSPATIALONEUPONEDOWNIMPORTANTRELATION}
	g^{aj} \g_tg_{ak}-2\omega\delta^j_k
	& = e^{2 \Omega} g^{ja} \partial_t h_{ak} 
	- 2 \omega g^{0j} g_{0k}.
\end{align}
We now apply $\u$ to \eqref{E:PARTIALTGSPATIALONEUPONEDOWNIMPORTANTRELATION} and use
the relations 
$\u \Omega = u^0 \omega$
and
$\u \omega = u^0 \g_t \omega,$
Lemma \ref{L:backgroundaoftestimate},
and Prop. \ref{P:derivativesofF1FkL2} to deduce that
\begin{align} \label{E:PARTIALUGUPPERPARTIALTGLOWERFIRSTESTIMATE}
	\| \u (g^{aj} \g_t g_{ak} - 2 \omega \delta^j_k) \|_{H^{N-1}}
	& \lesssim 
		e^{2 \Omega} 
		\left\lbrace 
			\| \u g^{ja} \|_{L^{\infty}} 
			+ \| \underpartial \u g^{ja} \|_{H^{N-2}} 
		\right \rbrace 
		\| \partial_t h_{ak} \|_{H^{N-1}}
			\\
	& \ \ 
		+ e^{2 \Omega} 
		\left\lbrace 
			\| g^{ja} \|_{L^{\infty}} 
			+ \| \underpartial g^{ja} \|_{H^{N-2}} 
		\right \rbrace 
		\| \u \partial_t h_{ak} \|_{H^{N-1}}
			\notag \\
	& \ \ 
		+ e^{2 \Omega}
		\left\lbrace
			1 + \| u^0 - 1 \|_{H^{N-1}}
		\right\rbrace
		\left\lbrace 
			\| g^{ja} \|_{L^{\infty}} 
			+ \| \underpartial g^{ja} \|_{H^{N-2}} 
		\right \rbrace 
		\| \partial_t h_{ak} \|_{H^{N-1}}
			\notag \\	
& \ \ + \| \u g^{0j} \|_{H^{N-1}} \|g_{0k} \|_{H^{N-1}}
		+ \| g^{0j} \|_{H^{N-1}} \|\u g_{0k} \|_{H^{N-1}}
			\notag \\
& \ \ + e^{- 3 \Omega}
			\left\lbrace
				1 + \| u^0 - 1 \|_{H^{N-1}}
			\right\rbrace
			\| g^{0j} \|_{H^{N-1}} \|g_{0k} \|_{H^{N-1}}.
				\notag
\end{align}
We now insert the estimates
\eqref{E:G0JUPPERHNMINUSONE},
\eqref{E:GJKUPPERLINFTY},
\eqref{E:DERIVATIVESOFGJKUPPERHNMINUS2},
\eqref{E:u0MINUSONEHNMINUSONE},
\eqref{E:partialugupperjk},
\eqref{E:partialugupperjk2},
and the inequalities
$\| \partial_t h_{ak} \|_{H^{N-1}} \lesssim e^{- q \Omega} \totalbelowtopnorm{N-1},$
$\| \u \partial_t h_{ak} \|_{H^{N-1}} \lesssim e^{- q \Omega} \totalbelowtopunorm{N-1},$
$\|g_{0k} \|_{H^{N-1}} \lesssim e^{(1-q) \Omega} \totalbelowtopunorm{N-1},$
and
$\|\u g_{0k} \|_{H^{N-1}} \lesssim e^{(1-q) \Omega} \totalbelowtopunorm{N-1}$
into the right-hand side of \eqref{E:PARTIALUGUPPERPARTIALTGLOWERFIRSTESTIMATE},
which leads to the desired estimate \eqref{E:partialu5NEW}.

Finally, we note that the proof of the second estimate in \eqref{E:SUPERIMPORTANT}
is similar to the proof of \eqref{E:partialu5NEW}, but is based on 
applying $\u$ to the identity
\begin{align}  \label{E:PARTIALTGSPATIALTWOUPIMPORTANTRELATION}
	\g_t g^{jk} + 2 \omega g^{jk}
	& = - g^{aj}(g^{bk} \g_t g_{ab} - 2 \omega \delta_a^k)
		- g^{0j} g^{ak} \g_t g_{0a}
		- g^{aj} g^{0k} \g_t g_{0a}
		- g^{0j} g^{0k} \g_t g_{00}
\end{align}
instead of the identity \eqref{E:PARTIALTGSPATIALONEUPONEDOWNIMPORTANTRELATION}.
Although we omit the tedious details, we point out that the only slightly subtle
estimates are for the factors
$\u(g^{bk} \g_t g_{ab} - 2 \omega \delta_a^k)$
and 
$g^{bk} \g_t g_{ab} - 2 \omega \delta_a^k,$
which have already been suitably bounded in
\eqref{E:partialu5NEW}
and
\eqref{E:partialu5}.

\subsection{Estimates for the error terms}\label{SS:mainestimates}

Props. \ref{P:metric} and \ref{P:SobolevMetricNEW} provided
Sobolev estimates for the fundamental solution variable components in terms of the 
norms of Sect. \ref{S:norms}. 
We now use these propositions
to prove Props. \ref{P:Sobolev} and \ref{P:Sobolevtwo}, which 
contain the most important analysis in the article: 
Sobolev estimates for various error terms.
These estimates play a central role in our proof of Props. 
\ref{P:integralinequalities}
and \ref{P:integralinequalitiesmetric},
in which we derive our basic integral inequalities for the energies;
we need to show that the error terms are small enough that they do not
cause the energies to blow-up in finite time.
Props. \ref{P:Sobolev} and \ref{P:Sobolevtwo} 
are an analog of Lemma 9.2.1 from \cite{jS2012}, but we  
had to alter many of the statements to reflect the 
more complicated energy hierarchy in the present article. 
In particular, our estimates for many of the top-order derivatives are worse by a factor
of $e^{\Omega}$ compared to the cases $0 < c_s^2 < 1/3$ studied in \cite{jS2012}.

\begin{proposition}[\textbf{Sobolev estimates for the error terms}]\label{P:Sobolev}
	Let $N \geq 4$ be an integer and let $(g_{\mu \nu}, u^{\mu}, \dens),$ $(\mu,\nu = 0,1,2,3),$ be a solution to the modified
	Eqs. \eqref{E:metric00}-\eqref{E:velevol} on the spacetime slab $[0,T) \times \mathbb{T}^3$. Assume that 
	the bootstrap assumption \eqref{E:gjkBAvsstandardmetric} holds on the same slab for some constant $c_1 \geq 
	1.$ Then there exists a constant $\epsilon' > 0$ such that if $\totalnorm(t) \leq \epsilon'$ for $t \in [0,T),$ 
	then the following estimates 
for the error terms $\triangle_{A,\mu\nu}$, $\triangle_{C,\mu\nu}$ 
defined in \eqref{AE:triangleA00def}-\eqref{AE:triangleAjkdef}
and \eqref{AE:triangleC00def}-\eqref{AE:triangleC0jdef}
also hold on $[0,T)$
(and the implicit constants in the estimates can depend on $N$ and $c_1$): 
\begin{subequations}
\begin{align}
\label{E:TRIANGLELAHNMINUSONE}
	\|\triangle_{A,00}\|_{H^{N-1}} 
	+ e^{-\Omega} \|\triangle_{A,0j}\|_{H^{N-1}}
	+ \|\triangle_{A,jk}\|_{H^{N-1}}
	& \lesssim e^{-2q\Omega} \gnorm{N-1},
			\\
	\|\triangle_{C,00}\|_{H^{N-1}}
	+ e^{-\Omega} \|\triangle_{C,0j}\|_{H^{N-1}}
	& \lesssim e^{-2q\Omega} \gnorm{N-1},
		\label{E:TRIANGLELCHNMINUSONE}
\end{align}
\end{subequations}
\begin{subequations}
\begin{align}
	\|\u \triangle_{A,00}\|_{H^{N-1}} 
	+ e^{-\Omega} \|\u \triangle_{A,0j}\|_{H^{N-1}}
	+ \|\u \triangle_{A,jk}\|_{H^{N-1}}
	& \lesssim e^{-2q\Omega} 
		\left\lbrace 
			\gnorm{N-1} 
			+ \gunorm{N-1} 
		\right\rbrace,
			\label{E:PARTIALUTRIANGLELAHNMINUSONE} \\
	\|\u \triangle_{C,00}\|_{H^{N-1}}
	+ e^{-\Omega} \|\u \triangle_{C,0j}\|_{H^{N-1}}
	& \lesssim e^{-2q\Omega} 
		\left\lbrace 
			\gnorm{N-1} 
			+ \gunorm{N-1} 
		\right\rbrace.
		\label{E:PARTIALUTRIANGLELCHNMINUSONE}
\end{align}
\end{subequations}
For the Christoffel symbol error terms defined in
\eqref{AE:triangleGamma000}-\eqref{AE:triangleGammaikj},
we have the following estimates on $[0,T):$
\begin{subequations}
\begin{align}
\|\triangle_{0 \ 0}^{\ 0} \|_{H^{N-1}} 
& \lesssim e^{-q\Omega} \gnorm{N-1},
	\label{E:triangle000} 
	\\
\|\triangle_{j \ 0}^{\ 0} \|_{H^{N-1}} 
&\lesssim e^{(1-q)\Omega}  \gnorm{N-1},
	\label{E:triangle0j0} \\
\|\triangle_{0 \ 0}^{\ j} \|_{H^{N-1}} 
&\lesssim  
	e^{-(1 + q)\Omega} \gnorm{N-1},
	\label{E:VERYIMPORTANTCHRISTOFFELSYMBOLERRORESTIMATE} 
		\\
\|\triangle_{0 \ k}^{\ j} \|_{H^{N-1}} 
&\lesssim e^{-q \Omega} \gnorm{N-1}, 
	\\
\|\triangle_{j \ k}^{\ 0} \|_{H^{N-1}} 
&\lesssim e^{(2-q)\Omega} \gnorm{N-1},
	\\
\|\triangle_{j \ k}^{\ i} \|_{H^{N-1}} 
&\lesssim e^{(1-q)\Omega} \gnorm{N-1},
	\label{E:trianglekij}
\end{align}
\end{subequations}
\begin{subequations}
\begin{align}
\| \partial_i \triangle_{0 \ 0}^{\ 0} \|_{H^{N-1}} 
& \lesssim e^{(1-q)\Omega} 
	\left\lbrace
		\gnorm{N-1}
		+ \totalellipticnorm{N-1}
	\right\rbrace,
	\label{E:PARTIALItriangle000} 
	\\
\| \partial_i \triangle_{j \ 0}^{\ 0} \|_{H^{N-1}} 
&\lesssim e^{(2-q)\Omega}  
	\left\lbrace
		\gnorm{N-1}
		+ \totalellipticnorm{N-1}
	\right\rbrace,
	\label{E:PARITALItriangle0j0} \\
\| \partial_i \triangle_{0 \ 0}^{\ j} \|_{H^{N-1}} 
&\lesssim  
	e^{-q \Omega} 
	\left\lbrace
		\gnorm{N-1}
		+ \totalellipticnorm{N-1}
	\right\rbrace,
	\label{E:PARTIALIVERYIMPORTANTCHRISTOFFELSYMBOLERRORESTIMATE} 
		\\
\| \partial_i \triangle_{0 \ k}^{\ j} \|_{H^{N-1}} 
&\lesssim e^{(1-q) \Omega} 
	\left\lbrace
		\gnorm{N-1}
		+ \totalellipticnorm{N-1}
	\right\rbrace, 
	\label{E:PARITALItriangle0jk} \\
\| \partial_i \triangle_{j \ k}^{\ 0} \|_{H^{N-1}} 
&\lesssim e^{(3-q)\Omega}
	\left\lbrace
		\gnorm{N-1}
		+ \totalellipticnorm{N-1}
	\right\rbrace,
	\label{E:PARITALItrianglej0k} \\
\| \partial_i \triangle_{j \ k}^{\ l} \|_{H^{N-1}} 
&\lesssim e^{(2-q)\Omega} 
		\left\lbrace
		\gnorm{N-1}
		+ \totalellipticnorm{N-1}
	\right\rbrace.
	\label{E:PARTIALItrianglekij}
\end{align}
\end{subequations}
For the time and $\u$ derivatives of the Christoffel symbol error terms,
we have the following estimates on $[0,T):$
\begin{subequations}
\begin{align}
\|\g_t \triangle_{0 \ 0}^{\ 0} \|_{H^{N-1}}
+ 
\| \u \triangle_{0 \ 0}^{\ 0} \|_{H^{N-1}} 
& \lesssim e^{-q\Omega} \totalnorm,
	\label{E:partialtandutriangle000} 
	\\
\|\g_t \triangle_{j \ 0}^{\ 0} \|_{H^{N-1}} 
+ \|\u \triangle_{j \ 0}^{\ 0} \|_{H^{N-1}} 
&\lesssim e^{(1-q)\Omega} \totalnorm,
	\label{E:partialtandutriangle0j0} \\
\|\g_t \triangle_{0 \ 0}^{\ j} \|_{H^{N-1}} 
+ \|\u \triangle_{0 \ 0}^{\ j} \|_{H^{N-1}} 
&\lesssim  
	e^{-(1 + q)\Omega} \totalnorm,
	\label{E:partialtanduVERYIMPORTANTCHRISTOFFELSYMBOLERRORESTIMATE} 
		\\
\|\g_t \triangle_{0 \ k}^{\ j} \|_{H^{N-1}} 
+ \|\u \triangle_{0 \ k}^{\ j} \|_{H^{N-1}} 
&\lesssim e^{-q \Omega} \totalnorm, 
	\\
\|\g_t \triangle_{j \ k}^{\ 0} \|_{H^{N-1}} 
+ \|\u \triangle_{j \ k}^{\ 0} \|_{H^{N-1}}
&\lesssim e^{(2-q)\Omega} \totalnorm,
	\\
\|\g_t \triangle_{j \ k}^{\ i} \|_{H^{N-1}} 
	+ \|\u \triangle_{j \ k}^{\ i} \|_{H^{N-1}}
&\lesssim e^{(1-q)\Omega} \totalnorm.
	\label{E:partialtandutrianglekij}
\end{align}
\end{subequations}
For the linearly small error term
$g^{ab} \Gamma_{ajb}$
from the right-hand side of the wave Eq. \eqref{E:metric0j},
we have the following estimates on $[0,T):$
\begin{subequations}
\begin{align} 
\| g^{ab} \Gamma_{ajb} \|_{H^{N-1}}
& \lesssim e^{(1-q)\Omega} \hstarstarnorm{N-1},
	\label{E:gabupperGammaajblower} \\
\|g^{ab}\Gamma_{ajb}\|_{H^{N-2}}
& \lesssim \hstarstarnorm{N-1},
\label{E:gabupperGammaajblowerHN-2}
	\\
\left\| \partial_t \left(g^{ab} \Gamma_{ajb} \right) \right\|_{H^{N-2}}
& \lesssim e^{-q\Omega} \hstarstarnorm{N-1},
	\label{E:partialtgabupperGammaajblowerHN-2}
	\\
\| \partial_{\mathbf{u}} (g^{ab} \Gamma_{ajb}) \|_{H^{N-1}}
& \lesssim e^{(1-q)\Omega} \left(\combinedpartialupartialhstarstarnorm{N-1}
		+ \hstarstarnorm{N-1}\right).
	\label{E:gabupperGammaajblowerNEW} 
\end{align}
\end{subequations}
For the metric wave equation error terms $\triangle_{\mu\nu}$ defined in
~\eqref{AE:triangle00}-\eqref{AE:trianglejk}, 
we have the following estimates on $[0,T):$
\begin{subequations}
\begin{alignat}{2}
\label{E:triangleTWODOWNHNMINUSONE}
	\|\triangle_{00}\|_{H^{N-1}}
	+ e^{-\Omega}  \|\triangle_{0j}\|_{H^{N-1}}
	+ \|\triangle_{jk}\|_{H^{N-1}}
&\lesssim&&  e^{-2q\Omega} \totalbelowtopnorm{N-1},
	\\
	\|\u \triangle_{00} \|_{H^{N-1}}
	+ e^{-\Omega} \| \u \triangle_{0j}\|_{H^{N-1}}
	+ \| \u \triangle_{jk}\|_{H^{N-1}}
&\lesssim&&  e^{-2q\Omega} \totalnorm.
\label{E:PARTIALUtriangleTWODOWNHNMINUSONE}
\end{alignat}
\end{subequations}
For the fluid evolution equation error terms 
$\triangle$ and $\triangle^j$
defined in \eqref{AE:triangledef} and \eqref{AE:trianglejdef},
we have the following estimates on $[0,T):$
\begin{subequations}
\begin{align} \label{E:FLUIDTRIANGLEHNMINUSONE}
\left\| \triangle \right \|_{H^{N-1}} 
+ \left\| \frac{1}{u^0} \triangle \right \|_{H^{N-1}} 
& \lesssim e^{-q \Omega} \totalbelowtopnorm{N-1},
	\\
\left\| \triangle^j \right\|_{H^{N-1}}  
+ \left\| \frac{1}{u^0} \triangle^j \right\|_{H^{N-1}}
	& \lesssim e^{-(1+2q) \Omega} \totalbelowtopnorm{N-1},
	\label{E:FLUIDTRIANGLEJHNMINUSONE}
\end{align}
\end{subequations}
\begin{align} \label{E:PARTIALIFLUIDTRIANGLEJHNMINUSONE}
\left\| \partial_i \triangle^j \right\|_{H^{N-1}}  
+ \left\| \partial_i \left( \frac{1}{u^0} \triangle^j \right) \right\|_{H^{N-1}}	
	& \lesssim e^{-2q \Omega} 
	\left\lbrace 
		\totalbelowtopnorm{N-1} 
		+ \totalellipticnorm{N-1}
		+ \topvelocitynorm{N-1}
	\right\rbrace.
\end{align}
For the time derivative of the fluid evolution equation error terms 
$\triangle^j,$
we have the following estimates on $[0,T):$
\begin{align}
\left\| \g_t \triangle^j \right\|_{H^{N-1}}  
	& \lesssim e^{-(1+q) \Omega} \totalbelowtopnorm{N-1}.
	\label{E:PARTIALTFLUIDTRIANGLEJHNMINUONE}
\end{align}
For the time derivatives of the fluid four-velocity, we have the following estimates on $[0,T):$
\begin{subequations}
\begin{align}
\|\g_t u^0\|_{H^{N-1}} 
+ \|\g_t u_0\|_{H^{N-1}}
& \lesssim e^{-q\Omega}
	\left\lbrace
		\totalbelowtopnorm{N-1}
		+ \topvelocitynorm{N-1}
	\right\rbrace,             
	\label{E:dtuzeroN-1}
	\\
\|\g_t u^j\|_{H^{N-1}} 
+ e^{-2 \Omega} \|\g_t u_j\|_{H^{N-1}} 
& \lesssim e^{-(1+q)\Omega}
\left\lbrace
	\totalbelowtopnorm{N-1}
	+ \topvelocitynorm{N-1}
\right\rbrace.              
\label{E:dtujN-1}
\end{align}
\end{subequations}
For the second time derivatives of the fluid four-velocity, we have the following estimates on $[0,T):$
\begin{subequations}
\begin{alignat}{2}
	\|\g_{tt} u^0\|_{H^{N-2}} & \lesssim &&  e^{-q\Omega} \totalnorm \, \label{E:dttuzeroHNMINUSTWO},
		\\
	\|\g_{tt} u^j\|_{H^{N-2}} & \lesssim &&  e^{-(1+q)\Omega} \totalnorm \,. \label{E:dttujHNMINUSTWO} 
\end{alignat}
\end{subequations}
For the second time derivatives of the metric components, we have the following estimates on $[0,T):$
\begin{subequations}
\begin{align}
	\| \g_{tt} g_{00}\|_{H^{N-2}}
	+ e^{-\Omega} \| \g_{tt} g_{0i}\|_{H^{N-2}} 
	+ \| \g_{tt} h_{ij}\|_{H^{N-2}}
	& \lesssim  e^{-q \Omega} \totalbelowtopnorm{N-1},  
	\label{E:dttgHMMINUSTWO} 
		\\
\| \g_{tt} g_{00}\|_{H^{N-1}}
	+ e^{-\Omega} \| \g_{tt} g_{0i}\|_{H^{N-1}} 
	+ \| \g_{tt} h_{ij}\|_{H^{N-1}}
	& \lesssim  e^{-q \Omega} 
		\left\lbrace 
			\totalbelowtopnorm{N-1} 
			+ \totalellipticnorm{N-1}
		\right\rbrace.
	\label{E:dttgHMMINUSONE} 
\end{align}
\end{subequations}
For the time derivative of the metric wave equation error terms $\triangle_{\mu\nu},$
we have the following estimates on $[0,T):$
\begin{align}
\| \partial_t \triangle_{00} \|_{H^{N-2}}
+  e^{- \Omega} \| \partial_t \triangle_{0j}\|_{H^{N-2}}
+ \| \partial_t \triangle_{jk}\|_{H^{N-2}}
& \lesssim e^{-2q\Omega}\totalnorm.
	\label{E:dttriangleTWOLOWERHN-2} 
\end{align}
Finally, for the third time derivatives of the metric components, we have the following estimates on $[0,T):$
\begin{align}
\| \g_{ttt} g_{00}\|_{H^{N-2}} 
+ e^{- \Omega} \| \g_{ttt} g_{0i}\|_{H^{N-2}}
+ \| \g_{ttt} h_{ij}\|_{H^{N-2}} 
\lesssim  e^{-q\Omega} \totalnorm. 
	\label{E:dtttgHMMINUSTWO} 
\end{align}
\end{proposition}

The following proposition is a companion to Prop. \ref{P:Sobolev}.

\begin{proposition} [\textbf{Error term estimates connected to commutations}] \label{P:Sobolevtwo}
	Let $N \geq 4$ be an integer and let $(g_{\mu \nu}, u^{\mu}, \dens),$ $(\mu,\nu = 0,1,2,3),$ be a solution to the modified
	Eqs. \eqref{E:metric00}-\eqref{E:velevol} on the spacetime slab $[0,T) \times \mathbb{T}^3$. Assume that 
	the bootstrap assumption \eqref{E:gjkBAvsstandardmetric} holds on the same slab for some constant $c_1 \geq 
	1.$ Then there exists a constant $\epsilon' > 0$ such that if $\totalnorm(t) \leq \epsilon'$ for $t \in [0,T)$
	and $1 \leq |\vec{\alpha}| \leq N-1,$ then the following estimates also hold on $[0,T)$
	(and the implicit constants in the estimates can depend on $N$ and $c_1$):
\begin{subequations}
	\begin{align}
		\|
			g^{\mu \nu} (\g_{\mu} \g_{\nu} u^{\delta})
			\partial_{\delta} \g_{\vec{\alpha}} g_{00} 
		\|_{L^2}
		+ 
		\|
			g^{\mu \nu} (\g_{\mu} u^{\delta})
				\g_{\vec{\alpha}} \g_{\nu} \partial_{\delta} g_{00} 
		\|_{L^2}
		& \lesssim e^{-2q\Omega}\totalnorm, 
			\label{E:BOXUG00COMM} \\
		\|
			g^{\mu \nu} (\g_{\mu} \g_{\nu} u^{\delta})
			\partial_{\delta} \g_{\vec{\alpha}} g_{0j}
		\|_{L^2}
		+ 
		\|
			g^{\mu \nu} (\g_{\mu} u^{\delta})
				\g_{\vec{\alpha}} \g_{\nu} \partial_{\delta} g_{0j} 
		\|_{L^2}
		& \lesssim e^{(1-2q)\Omega}\totalnorm, 
			\\
		\|
			g^{\mu \nu} (\g_{\mu} \g_{\nu} u^{\delta})
			\partial_{\delta} \g_{\vec{\alpha}} h_{jk}
		\|_{L^2}
		+
		\|
			g^{\mu \nu} (\g_{\mu} u^{\delta})
				\g_{\vec{\alpha}} \g_{\nu} \partial_{\delta} h_{jk} 
		\|_{L^2}
		& \lesssim e^{-2q\Omega}\totalnorm.
		\label{E:BOXUGJKCOMM}
	\end{align}
\end{subequations}
For $1 \leq |\vec{\alpha}| \leq N-1,$
we have the following commutator estimates on $[0,T):$
\begin{subequations}
\begin{align} \label{E:UCOMMUTATORBOUND}
\sum_{|\vec{\alpha}|\le N-1}\left\| \{ \frac{u^a}{u^0} \partial_a, \a \} u^j \right\|_{L^2}	
+ e^{-\Omega} \sum_{|\vec{\alpha}|\le N-1}\left\| \{ \frac{u^a}{u^0} \partial_a, \g_i\a \} u^j \right\|_{L^2}
	& \lesssim e^{-2(1+q)\Omega} \totalnorm, \\
	\sum_{|\vec{\alpha}|\le N-1}\left\| \{ \frac{u^a}{u^0} \partial_a, \a \} \dens \right\| \lesssim 
e^{-(1 + q) \Omega} \totalnorm. \label{E:RHOCOMMUTATORBOUND}
\end{align}
\end{subequations}
For $1 \leq |\vec{\alpha}| \leq N-1,$
we have the following commutator estimates on $[0,T):$
\begin{align} \label{E:COMMUTATORPARTIALTPARTIALUMETRICHNMINUONE}
		\|
			\{\u, \partial_t \} \a g_{00} 
		\|_{L^2}
		+
		e^{-\Omega}
		\|
			\{\u, \partial_t \} \a g_{0j}
		\|_{L^2}
		+
		\|
			\{\u, \partial_t \} \a h_{jk}
		\|_{L^2}
		& \lesssim e^{-2q\Omega} \gnorm{N-1}. 
	\end{align}
For $1 \leq |\vec{\alpha}| \leq N-1,$
we have the following estimates on $[0,T):$
\begin{subequations}
	\begin{align}
		\left\|
			 \left(\u g^{00}\right) \g_{tt} \a g_{00} 
		\right\|_{L^2}
		+
		\left\|
			 \left(\u g^{0a}\right) \g_t \g_a \a g_{00} 
		\right\|_{L^2}
		& \lesssim e^{-2q\Omega}\totalnorm,  \label{E:commutatorserror1}
			\\
		\left\|
			 \left(\u g^{00}\right) \g_{tt} \a g_{0j} 
		\right\|_{L^2}
		+
		\left\|
			 \left(\u g^{0a}\right) \g_t \g_a \a g_{0j} 
		\right\|_{L^2}
		& \lesssim e^{(1-2q)\Omega}\totalnorm, 
			\label{E:commutatorserror2} \\
		\left\|
			 \left(\u g^{00}\right) \g_{tt} \a h_{jk} 
		\right\|_{L^2}
		+
		\left\|
			 \left(\u g^{0a}\right) \g_t \g_a \a h_{jk} 
		\right\|_{L^2}
		& \lesssim e^{-2q\Omega}\totalnorm. \label{E:commutatorserror3}
	\end{align}
\end{subequations}
For $1 \leq |\vec{\alpha}| \leq N-1,$
we have the following commutator estimate on $[0,T):$
\begin{align}
	\left\|
		\{\u, \a \} (g^{ab}\Gamma_{ajb})
	\right\|_{L^2}
	& \lesssim e^{-2q\Omega}\totalnorm. \label{E:commutatorgabGammaajb}
\end{align}
For $1 \leq |\vec{\alpha}| \leq N-1,$
we have the following commutator estimates on $[0,T):$
\begin{align}
	\left\|
		\{ \u, \a \} \triangle_{00} 
	\right\|_{L^2}
	+ e^{-\Omega} 
		\left\| 
			\{ \u,\a \} \triangle_{0j}
		\right\|_{L^2}
	+ \left\| 
			\{\u, \a \} \triangle_{jk}
		\right\|_{L^2} 
	& \lesssim 
		e^{-2q\Omega} \totalnorm.
		\label{E:COMMUTATORPARTIALUtriangleTWODOWNHNMINUSONE}
\end{align}
For $1 \leq |\vec{\alpha}| \leq N-1,$
we have the following commutator estimates on $[0,T):$
\begin{subequations}
 \begin{align}
	\left\|
		\{\hat{\square}_g,\g_{\vec{\alpha}}\} g_{00} 
	\right\|_{L^2}
	+ e^{- \Omega} 
		\left \|
			\{\hat{\square}_g,\g_{\vec{\alpha}}\}g_{0j}
		\right\|_{L^2}
	+ \left\| 
			\{\hat{\square}_g,\g_{\vec{\alpha}}\}h_{jk}
		\right\|_{L^2}
	& \lesssim  e^{-2q\Omega} \totalbelowtopnorm{N-1}, 
	\label{E:gcommutatorL2}
	\\
	\left\| 
		\u \{\hat{\square}_g,\g_{\vec{\alpha}}\} g_{00} 
	\right\|_{L^2}
	+ e^{- \Omega} 
		\left\|
			\u \{\hat{\square}_g,\g_{\vec{\alpha}}\}g_{0j}
		\right\|_{L^2}
	+ \left\|
			\u \{\hat{\square}_g,\g_{\vec{\alpha}}\}h_{jk}
		\right\|_{L^2}
	& \lesssim  e^{-2q\Omega} \totalnorm.
	\label{E:gpartialucommutatorL2}
\end{align}
\end{subequations}
For $1 \leq |\vec{\alpha}| \leq N-1,$
we have the following commutator estimates on $[0,T):$
\begin{subequations}
\begin{align}
\label{E:equivalencecommutator1}
 & \left\| 
 		(\g_t\u\a-\a\u\g_t)g_{00}
 	 \right \|_{L^2}  
 	+ \sum_{i=1}^3 
 		\left\|
 			(\g_i \u\a-\a\u \g_i)g_{00}
 		\right\|_{L^2}
	+ \left\|
			\{\u,\a\}g_{00}
		\right\|_{L^2} \\
  &   \ \ \qquad \lesssim  
  \sum_{|\vec{\beta}| \leq N - 2}
		\left\| 
			\g_t \u \g_{\vec{\beta}} g_{00} 
		\right\|_{L^2}
	+ e^{-q\Omega} \gzerozeronorm{N-1}, 
	\notag \\
\label{E:ALTERNATEequivalencecommutator1}
 & \left\|
 		(\g_t\u\a-\a\u\g_t)g_{00}
 	 \right\|_{L^2}  
 	+ \sum_{i=1}^3 \|(\g_i \u\a-\a\u \g_i)g_{00}\|_{L^2}
	+ \left\|
			\{\u,\a\}g_{00}
		\right\|_{L^2} \\
  &   \ \ \qquad \lesssim  
 	e^{-q\Omega} \gzerozerounorm{N-1}
 	+
 	e^{-q\Omega} \gzerozeronorm{N-1}, 
	\notag \\
\label{E:equivalencecommutator2}
& 
\left\|
	(\g_t\u\a-\a\u\g_t)g_{0j}
\right\|_{L^2} 
+ \sum_{i=1}^3 
	\left\|
		(\g_i \u\a-\a\u\g_i)g_{0j}
	\right\|_{L^2} 
+ \left\| 
		\{\u,\a\}g_{0j} 
	\right\|_{L^2}  
	\\
& \ \ \qquad  \lesssim 
		\sum_{|\vec{\beta}| \leq N - 2}
		\left\| 
			\g_t \u \g_{\vec{\beta}} g_{0j} 
		\right\|_{L^2}
		+ e^{-(1 + q)\Omega} \gzerostarnorm{N-1}, \notag \\
\label{E:ALTERNATEequivalencecommutator2}
& \|(\g_t\u\a-\a\u\g_t)g_{0j}\|_{L^2} 
+ \sum_{i=1}^3 
	\left\|
		(\g_i\u\a-\a\u\g_i)g_{0j}
	\right\|_{L^2}
+ \left\|
		\{\u,\a\}g_{0j}
	\right\|_{L^2}   \\
& \ \ \qquad  
\lesssim
	e^{-(1 + q)\Omega} \gzerostarunorm{N-1}
 	+
 	e^{-(1 + q)\Omega} \gzerostarnorm{N-1},
 \notag \\
\label{E:equivalencecommutator3}
& \left\|
		(\g_t\u\a-\a\u\g_t)h_{jk}
	\right\|_{L^2}  
	+ \sum_{i=1}^3 
		\left\|
			(\g_i \u\a-\a\u\g_i)h_{jk}
		\right\|_{L^2}
 \\ 
& \ \ \qquad \lesssim 
		\sum_{|\vec{\beta}| \leq N - 2}
		\left\| 
			\g_t \u \g_{\vec{\beta}} h_{jk} 
		\right\|_{L^2} 
+ e^{-2q\Omega}\hstarstarnorm{N-1},\notag 
	\\
\label{E:ALTERNATEequivalencecommutator3}
& \left\|
		(\g_t\u\a-\a\u\g_t)h_{jk}
	\right\|_{L^2} 
	+ \sum_{i=1}^3 
		\left\|
			(\g_i \u\a-\a\u\g_i)h_{jk}
		\right\|_{L^2}
	\\ 
& \ \ \qquad \lesssim e^{-q\Omega} \combinedpartialupartialhstarstarnorm{N-1} + e^{-q\Omega}\hstarstarnorm{N-1}.
\notag
\end{align}
\end{subequations}
 	For the error terms from Lemma \ref{L:metricfirstdifferentialenergyinequality},
	where $\triangle_{\mathcal{E};(\gamma, \delta)}[v,\partial v]$ is defined in \eqref{E:trianglemathscrEdef},
	we have the following estimates on $[0,T):$
	\begin{subequations}
	\begin{align}
		e^{2q \Omega} \left\| \triangle_{\mathcal{E};(\upgamma_{00}, \updelta_{00})}[\partial_{\vec{\alpha}} (g_{00} + 1)
			,\partial (\partial_{\vec{\alpha}} g_{00})] \right\|_{L^1}
		& \lesssim  e^{-q \Omega} \totalnorm^2,
				\label{E:triangleEgamma00delta00L1} \\
		e^{2q \Omega} \left\| \triangle_{\mathcal{E};(\widetilde{\upgamma}_{00}, \widetilde{\updelta}_{00})}[\u 		
				\partial_{\vec{\alpha}} g_{00}
			,\partial (\u \partial_{\vec{\alpha}} g_{00})] \right\|_{L^1}
		& \lesssim  e^{-q \Omega} \totalnorm^2,    
			\label{E:partialutriangleEgamma00delta00L1} \\
		e^{2(q-1)\Omega} \left\| \triangle_{\mathcal{E};(\upgamma_{0*}, \updelta_{0*})}[\partial_{\vec{\alpha}} g_{0j},
			\partial (\partial_{\vec{\alpha}} g_{0j})] \right\|_{L^1}
		& \lesssim  e^{-q \Omega} \totalnorm^2,	
			\label{E:triangleEgamma0jdelta0jL1} \\
		e^{2(q-1)\Omega} \left\| \triangle_{\mathcal{E};(\widetilde{\upgamma}_{0*}, \widetilde{\updelta}_{0*})}[\u \partial_{\vec{\alpha}} g_{0j},
			\partial (\u \partial_{\vec{\alpha}} g_{0j})] \right\|_{L^1}
		& \lesssim  e^{-q \Omega} \totalnorm^2, 
			\label{E:partialutriangleEgamma0jdelta0*L1} 
			\\
		e^{2q \Omega} \left \| \triangle_{\mathcal{E};(\upgamma_{**},\updelta_{**})}[0,\partial (\partial_{\vec{\alpha}} h_{jk})] \right\|_{L^1}    
		& \lesssim  e^{-q \Omega} \totalnorm^2,
				\label{E:triangleEgamma**delta**L1}
			 \\
		e^{2q \Omega} \left \| \triangle_{\mathcal{E};(\widetilde{\upgamma}_{**},\widetilde{\updelta}_{**})}[0,\partial (\u \partial_{\vec{\alpha}} 
			h_{jk})] \right\|_{L^1}
		& \lesssim  e^{-q \Omega} \totalnorm^2.
			\label{E:partialutriangleEgamma**delta**L1}
	\end{align}
	\end{subequations}
	Finally, for the error terms from Lemma \ref{L:COMMUTEDBASICSTRUCTURE},
	where $\triangle_{\text{Ell}}[\partial^{(2)} v]$ is defined in \eqref{E:NEWQUADRATICERROR},
	we have the following estimates on $[0,T):$
	\begin{align} \label{E:ELLIPTICERRORL2}
		\|
			 \triangle_{\text{Ell}}[\partial^{(2)} \a g_{00}]
		\|_{L^2}
		+
		e^{-\Omega}
		\|
			\triangle_{\text{Ell}}[\partial^{(2)} \a g_{0j}] 
		\|_{L^2}
		+
		\|
			 \triangle_{\text{Ell}}[\partial^{(2)} \a h_{jk}]
		\|_{L^2}
		& \lesssim e^{-2q\Omega}\totalnorm.
	\end{align}
\end{proposition}
\begin{remark} [\textbf{Dangerous error terms}]
The error estimates
\eqref{E:VERYIMPORTANTCHRISTOFFELSYMBOLERRORESTIMATE},
\eqref{E:PARTIALIVERYIMPORTANTCHRISTOFFELSYMBOLERRORESTIMATE},
\eqref{E:gabupperGammaajblower},
and
\eqref{E:gabupperGammaajblowerNEW}
and from Props. 
\ref{P:Sobolev} and \ref{P:Sobolevtwo}
are the bounds used for the ``dangerous linear terms" appearing in the differential energy inequalities of Lemmas~\ref{L:metricfirstdifferentialenergyinequality} and~\ref{L:fluidfirstdifferentialenergyinequality}. 
The precise structure of the right-hand sides of these estimates
plays a critical role in our proof of Prop.~\ref{P:integralinequalitiesmetric}.
\end{remark}
\subsection{Counting Principle}
Before proving Prop.~\ref{P:Sobolev}, we first introduce a {\em Counting Principle} that will largely simplify and facilitate the proof
of many of the estimates. \textbf{The main point is that many of the error term products can be estimated by simply counting the net number of downstairs spatial indices in the product.} However, our proofs of the higher-order time derivative estimates and the 
commutator term estimates require separate arguments that 
do not directly rely on the Counting Principle.
\begin{remark}
	The Counting Principle used in the present article differs from the ones
	used in ~\cite{iRjS2012,jS2012}. The main differences are 
	\textbf{i)} the Counting Principle in the present article only concerns the norms $\| \cdot \|_{H^{N-1}}$
	of various quantities;
	\textbf{ii)} we have to account for the operator $\partial_{\mathbf{u}}$ in the present article; and \textbf{iii)}
	the top-order derivatives of various quantities are allowed to behave worse by a factor of $e^{\Omega}$
	compared to the corresponding top-order derivatives in ~\cite{iRjS2012,jS2012}.
\end{remark}

Our Counting Principle involves the following two classes of terms.

\begin{definition}[\textbf{The sets $\mathcal{G}_{N-1}$ and $\mathcal{H}_{N-1}$}]\label{D:GandH}
Let $N\geq3$ be a given integer and let $q$ be the small positive constant defined in 
~\eqref{E:qdef}. Let $v:\T^3\to\R$ be a given product of tensorfield 
components and let $A$ denote its total number of downstairs spatial indices minus the total number
of its upstairs spatial indices (e.g. if $v=\g_jg^{kl}$, then $A=-1$).
Let $\totalnorm$ denote the total solution norm defined in \eqref{E:totalnorm}.
We say that $v$ belongs to the set 
$\mathcal{G}_{N-1}$ 
if there exists a constant $C > 0$ such that 
under the assumptions of Prop. \ref{P:Sobolev}
and the assumption that $\totalnorm(t)$ is sufficiently small for $t \in [0,T),$
the following estimate holds for $t \in [0,T):$
\[
\|v\|_{H^{N-1}}\leq Ce^{-q\Omega}e^{A\Omega}\totalnorm.
\]
We say that $v$ belongs to the set $\mathcal{H}_{N-1}$ if 
there exists a constant $C > 0$ such that 
under the assumptions of Prop. \ref{P:Sobolev}
and the assumption that $\totalnorm(t)$ is sufficiently small for $t \in [0,T),$
then either inequality \textbf{i)} or \textbf{ii)} below
holds for $t \in [0,T):$
\begin{align} \label{E:mathcalHdefpartone}
	& \textbf{i)} \qquad
	\|v \|_{H^{N-1}}\leq Ce^{A\Omega}\totalnorm
		\\
	& \textbf{ii)} \qquad
		\| v \|_{L^{\infty}}\leq C e^{A\Omega}
		\quad\text{ and }\quad
		\| \underpartial v \|_{H^{N-2}}\leq Ce^{A\Omega}\totalnorm.
		\notag
\end{align}
\end{definition}

\begin{remark}
	Note our use of the spatial coordinate gradient
	notation $\underpartial$ in \textbf{ii)} above.  
	In this context, it is understood that
	``$\underpartial$'' does not contribute an additional downstairs
	spatial index because we have not explicitly displayed 
	any such additional index.
	\end{remark}

%
%
\begin{remark} [\textbf{The $\mathcal{G}_{N-1}$ terms experience improved decay}]
	We view elements of $\mathcal{G}_{N-1}$ as being particularly ``good'' terms.
	Such terms have the extra decay factor $e^{-q\Omega}$ 
	compared to the decay one would infer from simply counting spatial indices.
\end{remark}
\begin{remark} [\textbf{A preliminary list of terms in $\mathcal{G}_{N-1}$ and $\mathcal{H}_{N-1}$}] \label{R:list}
For convenience, we now provide preliminary lists of quantities that, under the 
assumptions of Prop. \ref{P:Sobolev}, 
have already been shown to
belong to the sets $\mathcal{G}_{N-1}$ and $\mathcal{H}_{N-1}$
respectively.
\begin{center}
	\underline{\textbf{Already identified elements of $\mathcal{G}_{N-1}$}}
\end{center}
\begin{align*}
	& g_{00} + 1, \ \partial_i g_{00}, \ \partial_i \partial_j g_{00}, 
		\ \g_t g_{00}, \ \partial_i \g_t g_{00}, 
		\ \partial_{\mathbf{u}} g_{00}, \ \partial_{\mathbf{u}} \partial_t g_{00},
		\ \partial_{\mathbf{u}} \g_i g_{00}, 
			\\
	& g^{00} + 1, \ \partial_i g^{00}, 
		\ \g_t g^{00}, 
		\ \partial_{\mathbf{u}} g^{00}, 
		\\
	& g_{0j}, \ \partial_i g_{0j}, \ \partial_i \partial_j g_{0k}, 
		\ \g_t g_{0j}, \ \partial_i \g_t g_{0j}, 
		\ \partial_{\mathbf{u}} g_{0j}, \ \partial_{\mathbf{u}} \partial_t g_{0j},
		\ \partial_{\mathbf{u}} \g_i g_{0j},
			\\
	& g^{0j}, \ \partial_i g^{0j}, 
		\ \g_t g^{0j}, 
		\ \partial_{\mathbf{u}} g^{0j}, 
		\\
	& e^{2 \Omega} \partial_{\mathbf{u}} h_{jk},
		\ \partial_i g_{jk}, \ \partial_i \partial_j g_{kl}, 
		\ e^{2 \Omega} \partial_t h_{jk}, \ \partial_t \g_i g_{jk}, 
		\ e^{2 \Omega} \partial_{\mathbf{u}} \partial_t h_{jk},
		\ e^{2 \Omega} \partial_{\mathbf{u}} \partial_i h_{jk}, 
		\ \partial_{\mathbf{u}} \partial_i g_{jk}, 
		\ e^{2 \Omega} \partial_i \partial_j h_{jk}, 
		\ \partial_i \partial_j g_{jk}, 
			\\
	& \partial_i g^{jk},  
			\\
	& \Gamma_{000}, \ \Gamma_{j00}, \ \Gamma_{0j0}, \ \Gamma_{ijk},
		\ \partial_i \Gamma_{000}, \ \partial_i \Gamma_{j00}, \ \partial_i \Gamma_{0j0}, 
		\ \partial_i \Gamma_{j0k}, \ \partial_i \Gamma_{jk0}, \ \partial_i \Gamma_{jkl},
			\\
	& \partial_t \Gamma_{000}, \ \partial_t \Gamma_{j00}, \ \partial_t \Gamma_{0j0}, \ \partial_t \Gamma_{ijk},
 		\\
	& \partial_{\mathbf{u}} \Gamma_{000}, \ \partial_{\mathbf{u}} \Gamma_{j00}, \ \partial_{\mathbf{u}} \Gamma_{0j0}, \ \partial_{\mathbf{u}} \Gamma_{ijk},
 		\\
 & \u \dens, 
 	\\
 & u^0 - 1, 
 		\ \partial_i u^0, 
 		\ u_0 + 1,
 		\ \partial_i u_0,
 		\ u^j, 
 		\ \partial_i u^j,
 		\ u_j, 
 		\ \partial_i u_j,
 		\ \partial_{\mathbf{u}} u^0,
 		\ \partial_{\mathbf{u}} u_0,
 		\ \partial_{\mathbf{u}} u^j, 
 		\ \partial_{\mathbf{u}} u_j
 \end{align*}
\begin{center}
	\underline{\textbf{Already identified elements of $\mathcal{H}_{N-1}$}}
\end{center}
\begin{align*}
	& g_{jk}, \ g^{jk}, 
		\ e^{2 \Omega} h_{jk}, 
		\ \partial_{\mathbf{u}} g_{jk}, \ \partial_{\mathbf{u}} g^{jk},
		\ \g_t g_{jk}, \ \g_t g^{jk}, 
		\ \partial_{\mathbf{u}} \g_t g_{jk}, \ \partial_{\mathbf{u}} \g_t g^{jk},
			\\
	& \Gamma_{j0k}, \ \Gamma_{jk0},
		\ \partial_t \Gamma_{j0k}, \ \partial_t \Gamma_{jk0},
		\ \partial_{\mathbf{u}} \Gamma_{j0k}, \ \partial_{\mathbf{u}} \Gamma_{jk0},
		\\
	& \dens - \bar{\dens},
		\\
	& \omega, \ \g_t \omega, \ \u \omega
\end{align*}

\end{remark}

\begin{proof}
The metric terms in the list for $\mathcal{G}_{N-1}$ with lower indices, including the ones with $\u$-derivatives and the $h$-terms, 
belong there by virtue of the definition \eqref{E:totalnorm} of the norm $\totalnorm$ 
and the estimates of Prop.~\ref{P:metric}.
The metric terms with upper indices, including the ones with $\u$-derivatives, are elements of $\mathcal{G}_{N-1}$ 
by the estimates of Props.~\ref{P:metric} and~\ref{P:SobolevMetricNEW}. 
As for the Christoffel symbols of the form $\Gamma_{\mu \nu \kappa}$ in the list for $\mathcal{G}_{N-1}$ and their derivatives, each of them is a sum of the derivatives of metric terms that have already been shown to belong to 
$\mathcal{G}_{N-1}.$
The terms involving the components
$u^{\mu},$ 
$u_{\mu},$ 
and their spatial and $\u$ derivatives
are elements of $\mathcal{G}_{N-1}$ 
by the definition \eqref{E:totalnorm} of the norm $\totalnorm$
and the estimates of Prop. \ref{P:SobolevMetricNEW}.
In a similar fashion, 
Lemma \ref{L:backgroundaoftestimate},
Props.~\ref{P:metric} and~\ref{P:SobolevMetricNEW},
and the definition of the total solution norm $\totalnorm$ 
imply the preliminary list of elements belonging to $\mathcal{H}_{N-1}$ stated above
except for $\u \omega.$
The fact that $\u \omega \in \mathcal{H}_{N-1}$ 
follows from Lemma \ref{L:backgroundaoftestimate} and \eqref{E:u0MINUSONEHNMINUSONE}.
\end{proof}
\begin{lemma}[\textbf{Counting Principle}]\label{L:countingprinciple}
Let 
$N\geq4$
and $l\geq1$ be given integers, and assume that $v_{(i)} \in\mathcal{G}_{N-1}\cup\mathcal{H}_{N-1}$.
Assume that there exists an index $k$ verifying $1\leq k\leq l$ such that $v_{(k)}$ satisfies
the condition $\textbf{i})$ from
\eqref{E:mathcalHdefpartone}.
Then there exists a constant $C>0$ such that 
if $\totalnorm \leq \epsilon$ and $\epsilon$ is sufficiently small, 
then the following inequality holds for $t \in [0,T):$
\be\label{E:cpone}
\Big|\Big|\prod_{i=1}^lv_{(i)}\Big|\Big|_{H^{N-1}}
\leq Ce^{-n_{\mathcal{G}_{N-1}}q\Omega}e^{n_{\text{total}}\Omega}\totalnorm.
\ee
In inequality ~\eqref{E:cpone},

\begin{itemize}
 \item $n_{\text{total}}$ is the total number of downstairs spatial indices in the product minus the total
 	number of the upstairs spatial indices in the product.
 \item $n_{\mathcal{G}_{N-1}}$ is the total number of $v_{(i)}$-s that belong to $\mathcal{G}_{N-1}.$
 \item A single application of the operator $\partial_t$ or $\u$ 
 	is neutral from the point of view of counting spatial indices.
 \end{itemize}
\end{lemma}

\begin{remark}
It is sufficient to assume $N\ge 3$ for the above lemma to be true, but we assume that 
$N\ge 4$ since this assumption will be used in the proof of Prop.~\ref{P:Sobolev} 
to guarantee that $H^{N-2}(\T^3)$ forms an algebra. 
\end{remark}

\begin{proof}
The proof is essentially identical to the proof of \cite[Lemma 9.7]{iRjS2012}, but
we present it for the sake of clarity and completeness. Without loss of generality we may assume that
$v_{(k)}=v_{(l)}\in L^2$. From Prop.~\ref{P:F1FkLinfinityHN}, we deduce that

\[
\Big|\Big|\prod_{i=1}^lv_{(i)}\Big|\Big|_{H^{N-1}}
\lesssim \|v_{(l)}\|_{H^{N-1}}\prod_{i=1}^{l-1}\|v_{(i)}\|_{L^{\infty}}
+\sum_{j=1}^{l-1}\|\underline{\g}v_{(j)}\|_{H^{N-2}}\prod_{i\neq j}\|v_{(i)}\|_{L^{\infty}}.
\]
The estimate~(\ref{E:cpone}) now follows from the above estimate, the Sobolev embedding result
$H^2(\T^3) \hookrightarrow L^{\infty}(\T^3),$ and
the definitions of the sets $\mathcal{G}_{N-1}$ and $\mathcal{H}_{N-1}$.
\end{proof}
\begin{remark} [\textbf{Adjusting the Counting Principle for $h_{jk}$}] \label{R:countingprinciplehjk}
Because of the rescaling $h_{jk}=e^{-2\Omega}g_{jk}$, we have to adjust the Counting Principle by a factor of $e^{-2\Omega}$ whenever we apply it to $h-$dependent
quantities such as $h_{jk},$ $\partial_t h_{jk},$ or $\g_i h_{jk}$. For example, we have
\[
\|\g_ih_{jk}\| \lesssim e^{(1-q)\Omega} \totalnorm,
\]
since $e^{2\Omega}\g_ih_{jk} = \g_ig_{jk}\in\mathcal{G}_{N-1}$. That is, 
even though the net count of downstairs indices in $\g_ih_{jk}$ is three, 
the above inequality involves the factor $e^{(1-q)\Omega}$ 
precisely due to the correction factor $e^{-2\Omega}.$
\end{remark}

In our analysis, we will often make use of the simple corollary of the proof of
Lemma \ref{L:countingprinciple}, which we also refer to as the ``Counting Principle.''
The corollary simply states that on the right-hand side of \eqref{E:cpone},
we can omit the norms inherent in
the definition \eqref{E:totalnorm} $\totalnorm$
that are irrelevant for the terms on the left-hand side.

\begin{corollary}[\textbf{Counting Principle with a more precise use of the relevant norms}]
	\label{C:COUNTINGPRINCIPLE}
	Assume the hypotheses of Lemma \ref{L:countingprinciple}.
	Assume that each $v_{(i)}$ from the hypotheses
	is controlled by a sub-sum 
	of the norms whose total sum is $\totalnorm$ [see definition \eqref{E:totalnorm}].
	Then inequality \eqref{E:cpone} holds with $\totalnorm$
	replaced by the sub-sum of those norms.
\end{corollary}

For example, Cor. \ref{C:COUNTINGPRINCIPLE} yields that 
$\|(\g_t g^{a0}) \u \g_i g_{ab}\|_{H^{N-1}} 
\lesssim e^{(1-q) \Omega} \left\lbrace \gnorm{N-1} + \gunorm{N-1} \right\rbrace $ since this product contains the factor 
$\g_t g^{a0} \in \mathcal{G}_{N-1}$ and since only $g,$ $\partial g,$ and $\u \partial g$ 
are involved.

During our proof of Prop.~\ref{P:Sobolev}, we will have to bound various Sobolev norms of 
twice-in-time differentiated quantities such as the metric components.
Because such terms are not explicitly contained in the definition of our total norm $\totalnorm$~\eqref{E:totalnorm}, 
we will sometimes use the following simple, but important identity to derive the desired estimates.
\begin{lemma}\label{L:twotimederivatives}
For any  $v \in C^2([0,T), \, H^2(\T^3)),$ the following identity holds:
%
\begin{align*}
\g_{tt} v 
= \frac{1}{u^0}\u\g_t v
	- \frac{u^a}{(u^0)^2}\u\g_a v 
	+ \frac{u^a u^b}{(u^0)^2}\g_a \g_b v,
\end{align*}
where we recall the definition~\eqref{E:deltau}: $\u=u^{\mu}\g_{\mu}$.
\end{lemma}
\begin{proof}
We first use the decomposition $\g_t = \frac{1}{u^0}\u - \frac{u^a}{u^0} \g_a$
[see definition~\eqref{E:deltau}]
to deduce that
\begin{align}
\g_t  v 
& = \frac{1}{u^0}\u v - \frac{u^a}{u^0}\g_a  v,	
	\label{E:du1} \\
\g_t\g_i v 
&= \frac{1}{u^0}\u \g_i v - \frac{u^a}{u^0}\g_a \g_i v.
\label{E:du2}
\end{align}
Using~(\ref{E:du1}) and~(\ref{E:du2}), we obtain the desired relation as follows:
\begin{align}
\g_{tt} v 
	= \frac{1}{u^0}\u\g_t v - \frac{u^i}{u^0}\g_i\g_t v \nonumber 
& =\frac{1}{u^0}\u\g_t v-\frac{u^i}{u^0}
	\left(
		\frac{1}{u^0}\u\g_i v
		-\frac{u^j}{u^0}\g_j\g_i v
	\right) \nonumber \\
& =\frac{1}{u^0}\u\g_t v
- \frac{u^i}{(u^0)^2}\u\g_i v
+ \frac{u^iu^j}{(u^0)^2}\g_j\g_i v. \nonumber
\end{align}
\end{proof}
%
%
\noindent
{\bf Proof of Prop.~\ref{P:Sobolev}.}
Throughout this proof, we will often make use of 
Lemma \ref{L:backgroundaoftestimate},
the Sobolev embedding result $H^2(\mathbb{T}^3) \hookrightarrow L^{\infty}(\mathbb{T}^3),$
and the definition of the norms from Sect. \ref{S:norms}
without explicitly mentioning it every time. \emph{We also stress that the order in which 
we prove the estimates is important in some cases.}
In particular, the order in which we prove 
the estimates is generally different
than the order in which we have stated them in the proposition.

\noindent
{\em Proof of \eqref{E:TRIANGLELAHNMINUSONE}-\eqref{E:TRIANGLELCHNMINUSONE},
\eqref{E:PARTIALUTRIANGLELAHNMINUSONE}-\eqref{E:PARTIALUTRIANGLELCHNMINUSONE},
\eqref{E:triangle000}-\eqref{E:trianglekij},
and \eqref{E:PARTIALItriangle000}-\eqref{E:PARTIALItrianglekij}.}
To prove \eqref{E:TRIANGLELAHNMINUSONE}-\eqref{E:PARTIALUTRIANGLELCHNMINUSONE}, 
we make the following key observation:
each of the products on the right-hand side of \eqref{AE:triangleA00def}-\eqref{AE:triangleAjkdef}
and \eqref{AE:triangleC00def}-\eqref{AE:triangleC0jdef}
are quadratic in the elements of the set $\mathcal{G}_{N-1}$ 
both before and after the application of the operator $\u$
(this can easily be verified with the help of Remark~\ref{R:list}). 
Furthermore, all terms involve only the metric and its first derivatives. 
The desired bounds \eqref{E:TRIANGLELAHNMINUSONE}-\eqref{E:TRIANGLELCHNMINUSONE}
thus follow as a direct consequence of Cor. \ref{C:COUNTINGPRINCIPLE}. 
The desired bounds \eqref{E:PARTIALUTRIANGLELAHNMINUSONE}-\eqref{E:PARTIALUTRIANGLELCHNMINUSONE} 
follow similarly, except we now also include the norm 
$\gunorm{N-1}$ on the right-hand side because the terms of interest also involve
the $\u$ derivative of the metric.

In our proof of ~\eqref{E:triangle000}-\eqref{E:PARTIALItrianglekij}, we
only prove the desired bound for the term $\| \partial_i \triangle_{0 \ 0}^{\ 0} \|_{H^{N-1}};$
the proofs of the remaining estimates are essentially the same.
To derive the bound for $\partial_i \triangle_{0 \ 0}^{\ 0},$
we examine Eq. \eqref{AE:triangleGamma000} and observe that each product in
$\partial_i \triangle_{0 \ 0}^{\ 0}$ has a net downstairs spatial index count of $1$ and
is at least \emph{linear} in elements of $\mathcal{G}_{N-1}.$ 
Furthermore, all terms involve only $g,$ $\partial g,$ and $\underpartial \partial g.$
The desired estimate 
$\| \partial_i \triangle_{0 \ 0}^{\ 0} \|_{H^{N-1}} 
\lesssim e^{(1-q) \Omega} 
\left\lbrace
		\gnorm{N-1}
		+ \totalellipticnorm{N-1}
\right\rbrace$
thus follows directly from Cor. \ref{C:COUNTINGPRINCIPLE}. 

\vspace{0.1in}

\noindent
{\em Proof of \eqref{E:partialtandutriangle000}-\eqref{E:partialtandutrianglekij}.}
We prove only \eqref{E:partialtandutriangle000} since the proofs of the remaining estimates
are similar. We first prove the bound for 
$\| \u \triangle_{0 \ 0}^{\ 0} \|_{H^{N-1}}.$
The proof is almost identical to the proof of \eqref{E:PARTIALItriangle000} given above.
The only difference is that the $\u$ derivative is involved instead of the $\g_i$ derivative.
Hence, there is one fewer downstairs spatial index in the present case, and 
we now allow for the norm $\totalnorm,$
which in particular controls $\u$ derivatives, 
on the right-hand side of \eqref{E:partialtandutriangle000}.

We now prove the desired bound for $\|\g_t \triangle_{j \ 0}^{\ 0} \|_{H^{N-1}}.$
To proceed, we use the decomposition $\g_t = \frac{1}{u^0} \g_{\mathbf{u}} - \frac{u^a}{u^0} \g_a,$
Prop. \ref{P:F1FkLinfinityHN},
and
Cor. \ref{C:SobolevTaylor}
to deduce that 
\begin{align} \label{E:PARTIALTCHRISTOFFELERRORINTERMSOFPARTIALUCHRISTOFFELERROR}
	\|\g_t \triangle_{0 \ 0}^{\ 0} \|_{H^{N-1}}
	& \lesssim 
		(1 + \| u^0 - 1 \|_{H^{N-1}}) \| \u \triangle_{0 \ 0}^{\ 0} \|_{H^{N-1}}
		+ (1 + \| u^0-1 \|_{H^{N-1}}) \| u^a \|_{H^{N-1}} \| \g_a \triangle_{0 \ 0}^{\ 0} \|_{H^{N-1}}.
\end{align}
Inserting the estimates
\eqref{E:u0MINUSONEHNMINUSONE}
and
\eqref{E:PARTIALItriangle000}
and the inequality
$\| u^a \|_{H^{N-1}} \lesssim e^{-(1 + q) \Omega} \totalnorm$
into the right-hand side of \eqref{E:PARTIALTCHRISTOFFELERRORINTERMSOFPARTIALUCHRISTOFFELERROR}
and using the bound $\| \u \triangle_{0 \ 0}^{\ 0} \|_{H^{N-1}} \lesssim e^{-q\Omega} \totalnorm$
proved just above, we conclude the desired estimate \eqref{E:partialtandutriangle000}
for $\|\g_t \triangle_{0 \ 0}^{\ 0} \|_{H^{N-1}}.$

\vspace{0.1in}
\noindent
{\em Proof of \eqref{E:gabupperGammaajblower}-\eqref{E:gabupperGammaajblowerNEW}}.
To prove \eqref{E:gabupperGammaajblower}, we note that $g^{ab} \Gamma_{ajb}$ is the product of elements
of $\mathcal{H}_{N-1}$ and furthermore, $\Gamma_{ajb} \in \mathcal{G}_{N-1}$ is controlled by the norm
$\hstarstarnorm{N-1}.$ The desired estimate \eqref{E:gabupperGammaajblower} thus follows from Cor. \ref{C:COUNTINGPRINCIPLE}.
The same reasoning yields \eqref{E:gabupperGammaajblowerNEW}, except we now also include the 
norm $\combinedpartialupartialhstarstarnorm{N-1}$ on the right-hand side because the terms of interest also involve
the $\u$ derivative of the metric.

To prove \eqref{E:gabupperGammaajblowerHN-2},
we first use Prop. \ref{P:F1FkLinfinityHN} and the
definition of $\Gamma_{ajb}$ to deduce that
\begin{align} \label{E:KEYGAMMATERMLOWERORDERFIRSTESTIMATE}
	\|g^{ab} \Gamma_{ajb}\|_{H^{N-2}}
	& \lesssim 
		\sum_{a,b,j,k=1}^3
		\left\lbrace
			\|g^{ab} \|_{L^{\infty}}
			+
			\| \underpartial g^{ab} \|_{H^{N-3}}
		\right\rbrace
		\| \underpartial g_{jk} \|_{H^{N-2}}.
\end{align}
We now insert the estimates
\eqref{E:GJKUPPERLINFTY},
\eqref{E:DERIVATIVESOFGJKUPPERHNMINUS2},
and $\| \underpartial g_{jk} \|_{H^{N-2}} 
= e^{2 \Omega} \| \underpartial h_{jk} \|_{H^{N-2}}
\lesssim e^{2 \Omega} \hstarstarnorm{N-1}$
into the right-hand side of \eqref{E:KEYGAMMATERMLOWERORDERFIRSTESTIMATE}
and thus conclude \eqref{E:gabupperGammaajblowerHN-2}.

To prove \eqref{E:partialtgabupperGammaajblowerHN-2},
we first rewrite $g^{ab} \Gamma_{ajb} = (e^{2 \Omega} g^{ab}) (e^{-2 \Omega} \Gamma_{ajb}).$
Thus, we have
$\g_t (g^{ab} \Gamma_{ajb})
= e^{2 \Omega} g^{ab} \g_t(e^{-2 \Omega} \Gamma_{ajb})
+  e^{2 \Omega}(\g_t g^{ab} + 2 \omega g^{ab}) \Gamma_{ajb}.$
Hence, using the relation 
$2 e^{-2 \Omega} \Gamma_{ajb} = \g_a h_{jb} + \g_b h_{jb} - \g_j h_{ab}$ 
and Prop. \ref{P:F1FkLinfinityHN}, we deduce that
\begin{align} \label{E:KEYPARTIALTGAMMATERMLOWERORDERFIRSTESTIMATE}
	\|\g_t (g^{ab} \Gamma_{ajb}) \|_{H^{N-2}}
	& \lesssim 
		e^{2 \Omega}
		\sum_{a,b,j,k=1}^3
		\left\lbrace
			\| g^{ab} \|_{L^{\infty}}
			+
		\| \underpartial g^{ab} \|_{H^{N-3}}
		\right\rbrace
		\| \g_t h_{jk} \|_{H^{N-1}}.
			\\
	& \ \
		+ e^{2 \Omega}
		\sum_{a,b,j,k=1}^3
		\| \g_t g^{ab} + 2 \omega g^{ab}  \|_{H^{N-1}}
		\| \underpartial g_{jk} \|_{H^{N-2}}.
		\notag
\end{align}
We now insert the estimates
\eqref{E:GJKUPPERLINFTY},
\eqref{E:DERIVATIVESOFGJKUPPERHNMINUS2},
\eqref{E:SUPERIMPORTANTCOMMUTATORFACTOR},
$\| \underpartial g_{jk} \|_{H^{N-2}} 
= e^{2 \Omega} \| \underpartial h_{jk} \|_{H^{N-2}}
\lesssim e^{2 \Omega} \hstarstarnorm{N-1},$
and $\| \g_t h_{jk} \|_{H^{N-1}}
\lesssim e^{- q \Omega} \hstarstarnorm{N-1}$
into the right-hand side of \eqref{E:KEYPARTIALTGAMMATERMLOWERORDERFIRSTESTIMATE}
and thus conclude \eqref{E:partialtgabupperGammaajblowerHN-2}.

\vspace{0.1in}

\noindent
{\em Proof of \eqref{E:triangleTWODOWNHNMINUSONE}-\eqref{E:PARTIALUtriangleTWODOWNHNMINUSONE}.}
We first prove \eqref{E:triangleTWODOWNHNMINUSONE}.
We begin by examining the right-hand sides of \eqref{AE:triangle00}-\eqref{AE:trianglejk}.
The terms $\triangle_{A,\mu \nu}$ and $\triangle_{C,\mu \nu}$ have already been suitably bounded by
\eqref{E:TRIANGLELAHNMINUSONE}-\eqref{E:TRIANGLELCHNMINUSONE}.
Furthermore, 
we see that all of the remaining terms to be estimated in $\| \cdot \|_{H^{N-1}}$
are products of elements of $\mathcal{G}_{N-1} \cup \mathcal{H}_{N-1}.$
Furthermore, the products are either quadratic 
in elements of the set $\mathcal{G}_{N-1}$
or are linear in elements of $\mathcal{G}_{N-1}$ 
and are weighted by $e^{-3 \Omega}$ or 
$|\omega - H| \lesssim e^{- 3 \Omega}.$
In addition, none of the terms depend on the
first spatial derivatives of the fluid,
the second spatial derivatives of the metric,
or any $\u$ derivative of the metric or fluid
(that is, all terms are controlled by the norm $\totalbelowtopnorm{N-1}$).
The desired estimates for the quadratic-in-$\mathcal{G}_{N-1}$ terms 
therefore follow from Cor. \ref{C:COUNTINGPRINCIPLE}, while the 
desired estimates for the linear-in-$\mathcal{G}_{N-1}$ terms
follow from applying Cor. \ref{C:COUNTINGPRINCIPLE} 
and also accounting for the extra $e^{- 3 \Omega}$ factor
[without the extra $e^{-3 \Omega}$ factor, we would only be able to deduce a single factor of
$e^{- q \Omega}$ on the right-hand sides of \eqref{E:triangleTWODOWNHNMINUSONE}].
We stress again, that in the above application of the Counting Principle, we always adjust it by a factor of $e^{-2\Omega}$ whenever estimating terms $h_{jk}$ and its derivatives (see Remark~\ref{R:countingprinciplehjk}).

The proof of \eqref{E:PARTIALUtriangleTWODOWNHNMINUSONE} is similar, but
we now allow for the total norm $\totalnorm$ on the right-hand side because it is sufficient 
for controlling the $\u$ derivatives.

\vspace{0.1in}
\noindent
{\em Proof of \eqref{E:FLUIDTRIANGLEHNMINUSONE}-\eqref{E:FLUIDTRIANGLEJHNMINUSONE} and \eqref{E:PARTIALIFLUIDTRIANGLEJHNMINUSONE}.}
To prove the estimate \eqref{E:FLUIDTRIANGLEJHNMINUSONE}
for the first term on the left-hand side, we use
the formula \eqref{AE:trianglejdef},
\eqref{E:u0MINUSONEHNMINUSONE}, 
\eqref{E:triangle000}-\eqref{E:trianglekij},
and Cor. \ref{C:COUNTINGPRINCIPLE}.
In particular, we are using the fact that
$\triangle_{\alpha \ \beta}^{\ j} \in \mathcal{G}_{N-1},$
[that is, the estimates \eqref{E:triangle000}-\eqref{E:trianglekij}]
and that each of the summands in~\eqref{AE:trianglejdef}
is at least quadratic in the elements of $\mathcal{G}_{N-1},$ 
which accounts for the factor $2q$ in $e^{-(1+2q)\Omega}$ on the right-hand side
of \eqref{E:FLUIDTRIANGLEJHNMINUSONE}.
In addition, none of the terms depend on the
first spatial derivatives of the fluid,
the second spatial derivatives of the metric,
or any $\u$ derivative of the metric or fluid
(that is, all terms are controlled by the norm $\totalbelowtopnorm{N-1}$).
To prove the estimate \eqref{E:FLUIDTRIANGLEJHNMINUSONE}
for the second term on the left-hand side,
we first use Cor. \ref{C:SobolevTaylor} and \eqref{E:u0MINUSONEHNMINUSONE}
to deduce $\|\frac{1}{u^0} - 1 \|_{H^{N-1}} \lesssim e^{-q\Omega} \totalbelowtopnorm{N-1}.$ 
In particular, we have $\frac{1}{u^0} \in \mathcal{H}_{N-1}.$
Hence, with the help of Cor. \ref{C:COUNTINGPRINCIPLE},
we conclude the desired bound
using essentially the same reasoning that we used to derive the bound for the
first term on the left-hand side of \eqref{E:FLUIDTRIANGLEJHNMINUSONE}.

The proof of the two estimates in \eqref{E:FLUIDTRIANGLEHNMINUSONE} is essentially the same
and is based on Eq. \eqref{AE:triangledef}
and the already proven estimate \eqref{E:FLUIDTRIANGLEJHNMINUSONE},
which shows that $\triangle^j \in \mathcal{G}_{N-1}.$
However, some of the products on the 
right-hand side of \eqref{AE:triangledef} 
(for example, the first one)
are only \emph{linear} in the elements of $\mathcal{G}_{N-1},$ 
which explains the presence of the factor $e^{-q \Omega}$ on the right-hand side
of \eqref{E:FLUIDTRIANGLEHNMINUSONE}.

The proof of \eqref{E:PARTIALIFLUIDTRIANGLEJHNMINUSONE} is similar
and relies on the estimates \eqref{E:PARTIALIU0HNMINUSONE} 
and \eqref{E:PARTIALItriangle000}-\eqref{E:PARTIALItrianglekij},
which show that $\g_i \triangle_{\mu \ \nu}^{\ \lambda} \in \mathcal{G}_{N-1}.$
We stress that since none of the terms 
depend on any $\u$ derivative of the metric or fluid,
all terms are controlled by
the norms
$  \totalbelowtopnorm{N-1} 
		+ \totalellipticnorm{N-1}
		+ \topvelocitynorm{N-1}.$

\vspace{0.1in}
\noindent{\em Proof of \eqref{E:dtuzeroN-1}-\eqref{E:dtujN-1}.}
To prove \eqref{E:dtujN-1} for the first term on the left-hand side, 
we first use \eqref{E:velevol} to deduce 
\be\label{E:dtuj}
\g_t u^j = 
	\frac{-u^a\g_au^j}{u^0}
	- 2 \omega u^j
	- \triangle_{0 \ 0}^{\ j}
	+\frac{\triangle^j}{u^0}.
\ee
In our proof of \eqref{E:FLUIDTRIANGLEJHNMINUSONE}, we showed that
$\|\frac{1}{u^0} - 1 \|_{H^{N-1}} \lesssim e^{-q\Omega} \totalbelowtopnorm{N-1}$ 
and hence  $\frac{1}{u^0} \in \mathcal{H}_{N-1}.$
Furthermore, 
\eqref{E:VERYIMPORTANTCHRISTOFFELSYMBOLERRORESTIMATE}
and
\eqref{E:FLUIDTRIANGLEJHNMINUSONE} 
imply that 
$\triangle_{0 \ 0}^{\ j}, \triangle^j \in \mathcal{G}_{N-1}.$
It follows that all terms on the right-hand side of \eqref{E:dtuj} are products of elements of 
$\mathcal{G}_{N-1} \cup \mathcal{H}_{N-1}$ and that each product is at least linear in elements of 
$\mathcal{G}_{N-1}.$ Furthermore, 
no terms depend on any $\u$ derivative of the metric or fluid
or on the second derivatives of the metric.
Hence, all terms are controlled by
the norm
$	\totalbelowtopnorm{N-1} + \topvelocitynorm{N-1}.$
The desired estimate \eqref{E:dtujN-1} for $\g_t u^j$ thus follows from Cor. \ref{C:COUNTINGPRINCIPLE}.

To prove the estimate for $\g_t u^0$ in \eqref{E:dtuzeroN-1}, we first note that
the identity \eqref{E:PARTIALUU0FIRSTRELATION} holds with $\g_t$ in place of $\u.$
We have just shown that $\g_t u^j \in \mathcal{G}_{N-1},$
and hence all terms on the right-hand side of the identity for $\g_t u^0$
are products of elements of $\mathcal{G}_{N-1} \cup \mathcal{H}_{N-1}$ and each 
product is at least linear in elements of $\mathcal{G}_{N-1}.$ 
In addition, all terms involved are controlled
by the norm $\totalbelowtopnorm{N-1} + \topvelocitynorm{N-1}.$
Cor. \ref{C:COUNTINGPRINCIPLE} thus yields the desired estimate 
\eqref{E:dtuzeroN-1} for the first term on the left-hand side.

To prove the desired estimate \eqref{E:dtujN-1} for $\g_t u_j = \g_t (g_{j\alpha} u^{\alpha})$ 
and the desired estimate \eqref{E:dtuzeroN-1}  for $\g_t u_0 = \g_t (g_{0\alpha} u^{\alpha}),$ 
we use the fact that
$\g_t u^j \in \mathcal{G}_{N-1}$ and $\g_t u^0 \in \mathcal{G}_{N-1}$ 
(that is, the bounds we just proved for these quantities),
the fact that $\g_t g_{\alpha \beta} \in \mathcal{H}_{N-1},$
and Cor. \ref{C:COUNTINGPRINCIPLE}.

\vspace{0.1in}

\noindent
{\em Proof of \eqref{E:PARTIALTFLUIDTRIANGLEJHNMINUONE}}.
We first time-differentiate \eqref{AE:trianglejdef} 
and use Prop. \ref{P:F1FkLinfinityHN}, 
Lemma \ref{L:backgroundaoftestimate},
\eqref{E:u0MINUSONEHNMINUSONE},
and
\eqref{E:dtuzeroN-1} 
to deduce that
\begin{align}
\|\g_t \triangle^j\|_{H^{N-1}} 
& \lesssim 
	\|\g_t u^0 \|_{H^{N-1}} \| \triangle^{\ j}_{0 \ 0} \|_{H^{N-1}}
		+ \| \g_t u^0 \|_{H^{N-1}} \| u^a \|_{H^{N-1}} \| \triangle^{\ j}_{0 \ a} \|_{H^{N-1}}
		\label{E:auxdtt} 	\\
& \ \ +  \| \g_t u^a \|_{H^{N-1}} \| \triangle^{\ j}_{0 \ a} \|_{H^{N-1}}
		+ \| \g_t u^a \|_{H^{N-1}} \| u^b \|_{H^{N-1}} \| \triangle^{\ j}_{a \ b} \|_{H^{N-1}} 
		\notag \\
& \ \ + \| u^0 - 1 \|_{H^{N-1}} \| \g_t \triangle^{\ j}_{0 \ 0} \|_{H^{N-1}}
			+  \| u^a \|_{H^{N-1}} \| \g_t \triangle^{\ j}_{0 \ a} \|_{H^{N-1}}
			\notag \\
& \ \ + \| u^a \|_{H^{N-1}} \| u^b \|_{H^{N-1}} \| \g_t \triangle^{\ j}_{a \ b} \|_{H^{N-1}}. 
	\notag
\end{align}
We now use the bounds
\eqref{E:u0MINUSONEHNMINUSONE}, 
\eqref{E:triangle000}-\eqref{E:trianglekij},
\eqref{E:partialtandutriangle000}-\eqref{E:partialtandutrianglekij},
\eqref{E:dtuzeroN-1}-\eqref{E:dtujN-1},
and $\| u^a \|_{H^{N-1}} \lesssim e^{-(1 + q) \Omega} \totalbelowtopnorm{N-1}$
to deduce that the right-hand side of the above inequality is 
$\lesssim e^{-(1 + q) \Omega} \totalbelowtopnorm{N-1}$
as desired.

\noindent {\em Proof of \eqref{E:dttgHMMINUSTWO}-\eqref{E:dttgHMMINUSONE}.}
We will prove the estimates only for $g_{00};$ the estimates for 
$g_{0j}$ and $h_{jk}$ can be proved using very similar arguments.
To prove \eqref{E:dttgHMMINUSONE}, we first use Eq. \eqref{E:metric00}
to deduce
\begin{align} \label{E:PARTIALTTG00ISOLATED}
	\g_{tt} g_{00}
	& = - \frac{1}{g^{00}} g^{ab} \g_a \g_b g_{00}
		- \frac{2}{g^{00}} g^{a0} \g_a \g_t g_{00}
		+ \frac{1}{g^{00}}
			\left\lbrace
				5H\g_t g_{00}+6H^2(g_{00}+1)
				+\triangle_{00}
			\right\rbrace.
\end{align}
From \eqref{E:G00UPPERPLUSONEHNMINUSONE} and Cor. \ref{C:SobolevTaylor},
we deduce $\frac{1}{g^{00}} \in \mathcal{H}_{N-1}.$ Furthermore,
\eqref{E:triangleTWODOWNHNMINUSONE} implies that $\triangle_{00} \in \mathcal{G}_{N-1}.$
Therefore, all terms on the right-hand side of \eqref{E:PARTIALTTG00ISOLATED} 
are products of elements of $\mathcal{G}_{N-1} \cup \mathcal{H}_{N-1}$ and each 
product is at least linear in elements of $\mathcal{G}_{N-1}.$ 
In addition, none of the terms in \eqref{E:PARTIALTTG00ISOLATED} depend on 
the derivatives of the fluid or on the
$\u$ derivatives of the metric or fluid
[that is, all terms are controlled by 
the norms
$\totalbelowtopnorm{N-1} 
		+ \totalellipticnorm{N-1}$].
Cor. \ref{C:COUNTINGPRINCIPLE} thus yields the desired estimate \eqref{E:dttgHMMINUSONE}.

The proof of \eqref{E:dttgHMMINUSTWO} is similar, but there are a few important differences.
The main point is that we want to avoid using the elliptic norms $\totalellipticnorm{N-1}$
that we used in proving \eqref{E:dttgHMMINUSONE}. 
To handle the first term on the right-hand side of \eqref{E:PARTIALTTG00ISOLATED},
we first use Prop. \ref{P:derivativesofF1FkL2} and
the facts that $g^{ab}, \frac{1}{g^{00}} \in \mathcal{H}_{N-1},$
to bound its $H^{N-2}$ norm by
$\lesssim (\| g^{ab} \|_{L^{\infty}} + \| \underpartial g^{ab} \|_{H^{N-2}}) \| \g_a \g_b g_{00} \|_{H^{N-2}}
\lesssim e^{-2 \Omega} \sum_{a,b=1}^3 \| \g_a \g_b g_{00} \|_{H^{N-2}} 
\lesssim e^{-2 \Omega} \sum_{a=1}^3 \| \g_a g_{00} \|_{H^{N-1}}.$
We now note that it follows from the definition of $\totalbelowtopnorm{N-1}$ that 
$e^{-2 \Omega} \sum_{a=1}^3 \| \g_a g_{00} \|_{H^{N-1}} \lesssim e^{-(1 + q)\Omega} \totalbelowtopnorm{N-1}$
as desired. Similarly, we can bound the $H^{N-2}$ norm of
second term on the right-hand side of \eqref{E:PARTIALTTG00ISOLATED}
by $\lesssim e^{-(1 + q)\Omega} \totalbelowtopnorm{N-1}$ as desired.
The remaining terms on the right-hand side of \eqref{E:PARTIALTTG00ISOLATED}
do not depend on 
the derivatives of the fluid, the $\u$ derivatives of the metric or fluid,
or the second derivatives of the metric. Hence, Cor. \ref{C:COUNTINGPRINCIPLE}
and \eqref{E:triangleTWODOWNHNMINUSONE}
yield that their $H^{N-1}$ norms (which is a stronger norm than is needed)
are $\lesssim e^{-(1 + q)\Omega} \totalbelowtopnorm{N-1}$ as desired.

\vspace{0.1in}
\noindent
{\em Proof of~\eqref{E:dttuzeroHNMINUSTWO}-\eqref{E:dttujHNMINUSTWO}.}
To prove \eqref{E:dttujHNMINUSTWO}, we first time-differentiate 
\eqref{E:dtuj} 
and use Prop. \ref{P:derivativesofF1FkL2}, 
Cor. \ref{C:SobolevTaylor},
Lemma \ref{L:backgroundaoftestimate},
\eqref{E:u0MINUSONEHNMINUSONE},
and
\eqref{E:dtuzeroN-1} 
to deduce that
\begin{align*}
\|\g_{tt}u^j\|_{H^{N-2}} 
& \lesssim \sum_{a,j=1}^3 \|\g_t u^a\|_{H^{N-1}} \| u^j \|_{H^{N-1}} 
	+ \|\g_t u^j\|_{H^{N-2}} 
	+ \| \g_t \triangle_{0 \ 0}^{\ j} \|_{H^{N-2}}
	+ \| \g_t \triangle^j \|_{H^{N-2}}
	\\
& \ \ 
		+ \sum_{a,j=1}^3 \| u^a \|_{H^{N-2}} \| u^j \|_{H^{N-1}}
		+ \|u^j\|_{H^{N-2}}
		+ \| \triangle^j  \|_{H^{N-2}}.
\end{align*}
Inserting the estimates
\eqref{E:partialtanduVERYIMPORTANTCHRISTOFFELSYMBOLERRORESTIMATE},
\eqref{E:PARTIALTFLUIDTRIANGLEJHNMINUONE},
and
\eqref{E:dtujN-1}
into the above inequality and also using the
inequality $\| u^j \|_{H^{N-1}} \lesssim e^{-(1 + q) \Omega} \totalnorm,$
we conclude \eqref{E:dttujHNMINUSTWO}.

To prove \eqref{E:dttuzeroHNMINUSTWO}, we
twice time differentiate
the relation
\eqref{E:U0UPPERISOLATEDagain}
and apply similar reasoning.
Although we omit the details, we remark that we
in particular make use of the 
already proven bounds
\eqref{E:dttujHNMINUSTWO}
and
\eqref{E:dttgHMMINUSTWO}.

\vspace{0.1in}
\noindent
{\em Proof of \eqref{E:dttriangleTWOLOWERHN-2}.}
We prove only the estimate for $\| \partial_t \triangle_{00} \|_{H^{N-2}}$
in \eqref{E:dttriangleTWOLOWERHN-2} since the proofs of 
the other two estimates are similar. To simplify the proof, we will use the identity
\begin{align} \label{E:PARTIALTIDTRIANGLE00}
	\partial_t \triangle_{00} = \frac{1}{u^0} \partial_{\mathbf{u}} \triangle_{00} - \frac{u^a}{u^0} \partial_a \triangle_{00}.
\end{align}
From Prop.~\ref{P:F1FkLinfinityHN},
Cor. \ref{C:SobolevTaylor}, 
and the estimates 
\eqref{E:u0MINUSONEHNMINUSONE} and 
\eqref{E:triangleTWODOWNHNMINUSONE}, 
and \eqref{E:PARTIALUtriangleTWODOWNHNMINUSONE},
we deduce that
\begin{align}
	\left\| \frac{1}{u^0} \partial_{\mathbf{u}} \triangle_{00} \right\|_{H^{N-2}} 
	& \lesssim 
		\left\lbrace
			1
			+ \left\| u^0 - 1 \right\|_{^{N-1}}
		\right\rbrace
		\| \partial_{\mathbf{u}} \triangle_{00} \|_{H^{N-1}}
		\lesssim
		e^{-2q\Omega}\totalnorm, \nonumber
		\\
	\left\| \frac{u^a}{u^0} \partial_a \triangle_{00} \right\|_{H^{N-2}}  
	& \lesssim 
		\left\lbrace
			1
			+ \left\| u^0 - 1 \right\|_{^{N-1}}
		\right\rbrace
		\| u^a \|_{H^{N-2}}
		\| \triangle_{00} \|_{H^{N-1}}
		\lesssim
		e^{-2q\Omega}\totalnorm.	 \nonumber
\end{align}
The desired estimate ~\eqref{E:dttriangleTWOLOWERHN-2} now follows easily from the relation \eqref{E:PARTIALTIDTRIANGLE00} and the above two estimates.

\vspace{0.1in}

\noindent {\em Proof of \eqref{E:dtttgHMMINUSTWO}.}
We provide the proof only for the first term on the left-hand side of
\eqref{E:dtttgHMMINUSTWO} since the other estimates 
can be proved in a similar fashion. To proceed,
we first use the definition of $\hat{\Box}_g$ and equation~\eqref{E:metric00} to deduce 
the identity
\begin{align} \label{E:PARTIALTTTG00EXPRESSION}
\g_{ttt}g_{00}
& = \left(
			\g_t 
			\left(
				\frac{1}{g^{00}} 
			\right)
		\right)
		\left\lbrace
			- g^{ab} \partial_a \partial_b g_{00}
			- 2 g^{0a} \partial_a \partial_t g_{00}
			+ 5H\g_t g_{00}
			+ 6H^2(g_{00}+1)
			+ \triangle_{00}
		\right\rbrace
			\\
& \ \ + 
		\frac{1}{g^{00}} 
		\g_t
		\left\lbrace
			- g^{ab} \partial_a \partial_b g_{00}
			- 2 g^{0a} \partial_a \partial_t g_{00}
			+ 5H\g_t g_{00}
			+ 6H^2(g_{00}+1)
			+ \triangle_{00}
		\right\rbrace.
		\notag
\end{align}
Using Cor. \ref{C:SobolevTaylor}, 
Prop.~\ref{P:F1FkLinfinityHN}, 
and the estimate \eqref{E:G00UPPERPLUSONEHNMINUSONE},
we deduce that
$\left \| \g_t \left(\frac{1}{g^{00}} \right) \right\|_{H^{N-1}} \lesssim e^{-q\Omega} \totalnorm.$
Furthermore, the first group of terms in braces on the right-hand side of \eqref{E:PARTIALTTTG00EXPRESSION}
was shown during our proof of \eqref{E:dttgHMMINUSONE}
to be bounded in the norm $H^{N-1}$ by $\lesssim e^{-q\Omega} \totalnorm.$
Thus, by Prop.~\ref{P:F1FkLinfinityHN}, 
the first product on the right-hand side of \eqref{E:PARTIALTTTG00EXPRESSION} is 
$\lesssim e^{-2q\Omega} \totalnorm$ in the $H^{N-1}$ norm, which is a stronger estimate than we need.

We now address the second group of terms in braces on the right-hand side of \eqref{E:PARTIALTTTG00EXPRESSION}.
Since $\frac{1}{g^{00}} \in \mathcal{H}_{N-1},$ by Prop.~\ref{P:F1FkLinfinityHN}, 
we need only to show that
the second group of terms in braces is $\lesssim e^{-q\Omega} \totalnorm$ in the $H^{N-2}$ norm. The terms
$\g_{tt} g_{00}$ and $\g_t \triangle_{00}$ have already been suitably bounded in 
\eqref{E:dttgHMMINUSTWO}  and \eqref{E:dttriangleTWOLOWERHN-2}, while the bound 
$\| \g_t g_{00} \|_{H^{N-1}}\leq e^{-q \Omega} \totalnorm$ follows directly from the definition of $\totalnorm.$
To bound the term $\partial_t(g^{ab} \partial_a \partial_b g_{00}),$ we use 
Prop. \ref{P:F1FkLinfinityHN} and
the estimates \eqref{E:GJKUPPERLINFTY},
\eqref{E:DERIVATIVESOFGJKUPPERHNMINUS2}, 
and \eqref{E:SUPERIMPORTANTCOMMUTATORFACTOR}
to deduce that
\begin{align}
	\| \g_t (g^{ab} \partial_a \partial_b g_{00}) \|_{H^{N-2}}
	& \lesssim 
		\sum_{a,b=1}^3 \| \partial_t g^{ab} + 2 \omega g^{ab} \|_{H^{N-2}} \| \partial_b g_{00} \|_{H^{N-1}} 
		\\
	& \ \ + 
			\sum_{a,b,i=1}^3 
			\left\lbrace
				\| g^{ab} \|_{L^{\infty}} + \| \underpartial g^{ab} \|_{H^{N-3}}
			\right\rbrace 
		\left\lbrace
			\| \partial_i \partial_t g_{00} \|_{H^{N-1}}
			+ \| \partial_i g_{00} \|_{H^{N-1}}
		\right\rbrace
		\notag \\
	& \lesssim e^{- (1 + q) \Omega} \totalnorm. 
		\notag
\end{align}
To bound $\g_t(g^{0a} \partial_a \partial_t g_{00}),$ we use 
Prop. \ref{P:F1FkLinfinityHN}, 
\eqref{E:G0JUPPERHNMINUSONE}, 
and \eqref{E:dttgHMMINUSONE} to deduce that
\begin{align*}
	\| \g_t(g^{0a} \partial_a \partial_t g_{00})\|_{H^{N-2}}
	& \lesssim \sum_{a=1}^3 
											\left\lbrace
												\| \partial_t g^{0a} \|_{H^{N-2}} 
												+ \| g^{0a} \|_{H^{N-2}}
											\right\rbrace
											\left\lbrace 
												\| \partial_t g_{00} \|_{H^{N-1}} + \| \g_{tt} g_{00} \|_{H^{N-1}}
											\right\rbrace
							\\
	& \lesssim e^{-(1 + q) \Omega} \totalnorm.
		\notag
\end{align*}\prfe

\noindent
{\bf Proof of Prop.~\ref{P:Sobolevtwo}.}
Throughout this proof, we will often make use of 
Lemma \ref{L:backgroundaoftestimate},
the Sobolev embedding result $H^2(\mathbb{T}^3) \hookrightarrow L^{\infty}(\mathbb{T}^3),$
and the definition of the norms from Sect. \ref{S:norms}
without explicitly mentioning it every time. \emph{We also stress that the order in which 
we prove the estimates is important in some cases.}
In particular, order in which we prove 
the estimates is generally different
than the order in which we have stated them in the proposition.

\noindent
{\em Proof of~\eqref{E:BOXUG00COMM}-\eqref{E:BOXUGJKCOMM}.}
We prove only the estimate \eqref{E:BOXUG00COMM} since the proof of the other
two estimates is similar.

\vspace{0.1in}
\noindent{\em Step 1: Bounds for the term
			$g^{\mu \nu} (\g_{\mu} \g_{\nu} u^{\delta})
				\partial_{\delta} \g_{\vec{\alpha}} g_{00}.$
}

\noindent{\em Case 1A: $(\mu,\nu) = (0,0).$}				
In this case, we use 
the standard Sobolev calculus and
the estimates
\eqref{E:G00UPPERPLUSONEHNMINUSONE},
\eqref{E:dttuzeroHNMINUSTWO},
and \eqref{E:dttujHNMINUSTWO}
to deduce the desired bound:
\begin{align*}
		\left \|
				g^{00} 
				(\g_{tt} u^{\delta})
				\partial_{\delta} \g_{\vec{\alpha}} g_{00}
		\right\|_{L^2}
		& \lesssim 
			\left\lbrace
				1 +
				\left\|
					g^{00} + 1
				\right\|_{H^{N-1}}
			\right\rbrace
			\left\|
				\g_{tt} u^0
			\right\|_{H^{N-2}}
			\left\|
				\g_t g_{00} 
			\right\|_{H^{N-1}}
			\\
		& \ \ +
			\left\lbrace
				1 +
				\left\|
					g^{00} + 1
				\right\|_{H^{N-1}}
			\right\rbrace
			\left\|
				\g_{tt} u^a
			\right\|_{H^{N-2}}
			\left\|
				\g_a g_{00} 
			\right\|_{H^{N-1}}
	 \lesssim e^{-2q \Omega} \totalnorm.
		\notag
\end{align*}

\noindent{\em Case 1B: $(\mu,\nu) = (0,j).$}				
In this case, we use 
the standard Sobolev calculus and the estimates
\eqref{E:G0JUPPERHNMINUSONE},
\eqref{E:dtuzeroN-1},
and \eqref{E:dtujN-1}
to deduce the desired bound:
\begin{align*}
		\left \|
				g^{0j} 
				(\g_j \g_t u^{\delta})
				\partial_{\delta} \g_{\vec{\alpha}} g_{00}
		\right\|_{L^2}
		& \lesssim 
			\sum_{j=1}^3
			\left\|
					g^{0j}
			\right\|_{H^{N-1}}
			\left\|
				\g_t u^0
			\right\|_{H^{N-1}}
			\left\|
				\g_t g_{00} 
			\right\|_{H^{N-1}}
			\\
		& \ \ +
			\sum_{j=1}^3
			\left\|
				g^{0j}
			\right\|_{H^{N-1}}
			\left\|
				\g_t u^a
			\right\|_{H^{N-1}}
			\left\|
				\g_a g_{00} 
			\right\|_{H^{N-1}}
	 \lesssim e^{-2q \Omega} \totalnorm.
		\notag
\end{align*}

\noindent{\em Case 1C: $(\mu,\nu) = (j,k).$}				
In this case, we use 
the standard Sobolev calculus and the estimates
\eqref{E:GJKUPPERLINFTY}
and \eqref{E:PARTIALIU0HNMINUSONE} 
to deduce the desired bound:
\begin{align*}
		\left \|
				g^{jk} 
				(\g_j \g_k u^{\delta})
				\partial_{\delta} \g_{\vec{\alpha}} g_{00}
		\right\|_{L^2}
		& \lesssim 
			\sum_{i,j,k=1}^3
			\left\|
				g^{jk}
			\right\|_{L^{\infty}}
			\left\|
				\g_i u^0 
			\right\|_{H^{N-1}}
			\left\|
				\g_t g_{00} 
			\right\|_{H^{N-1}}
			\\
		& \ \ +
			\sum_{i,j,k=1}^3
			\left\|
				g^{jk}
			\right\|_{L^{\infty}}
			\left\|
				\g_i u^a
			\right\|_{H^{N-1}}
			\left\|
				\g_a g_{00} 
			\right\|_{H^{N-1}}
	 \lesssim e^{-2q \Omega} \totalnorm.
		\notag
\end{align*}

\noindent{\em Step 2: Bounds for the term
			$g^{\mu \nu} (\g_{\mu} u^{\delta})
				\g_{\vec{\alpha}} \g_{\nu} \partial_{\delta} g_{00}.$
}

\vspace{0.1in}

\noindent{\em Case 2A: $(\mu,\nu) = (0,0).$}				
In this case, we use 
the standard Sobolev calculus and the estimates
\eqref{E:G00UPPERPLUSONEHNMINUSONE},
\eqref{E:dtuzeroN-1},
\eqref{E:dtujN-1},
and
\eqref{E:dttgHMMINUSTWO} 
to deduce the desired bound:
\begin{align*}
		\left \|
			g^{00} (\g_t u^{\delta})
			\g_{\vec{\alpha}} \g_t \partial_{\delta} g_{00}
		\right\|_{L^2}
		& \lesssim 
			\left\lbrace
				1 +
				\left\|
					g^{00} + 1
				\right\|_{H^{N-1}}
			\right\rbrace
			\left\|
				\g_t u^0
			\right\|_{H^{N-1}}
			\left\|
				\g_{tt} g_{00} 
			\right\|_{H^{N-1}}
			\\
		& \ \ +
			\left\lbrace
				1 +
				\left\|
					g^{00} + 1
				\right\|_{H^{N-1}}
			\right\rbrace
			\left\|
				\g_t u^a
			\right\|_{H^{N-1}}
			\left\|
				\g_a \g_t g_{00} 
			\right\|_{H^{N-1}}
	 \lesssim e^{-2q \Omega} \totalnorm.
		\notag
\end{align*}

\noindent{\em Case 2B: $(\mu,\nu) = (0,j).$}				
In this case, we use 
the standard Sobolev calculus and the estimates
\eqref{E:G00UPPERPLUSONEHNMINUSONE}
and \eqref{E:u0MINUSONEHNMINUSONE} 
to deduce the desired bound:
\begin{align*}
		\left \|
			g^{0j} (\g_j u^{\delta})
			\g_{\vec{\alpha}} \g_j \partial_{\delta} g_{00}
		\right\|_{L^2}
		& \lesssim 
			\sum_{j=1}^3
			\left\|
				g^{0j} 
			\right\|_{H^{N-1}}
			\left\|
				u^0 - 1
			\right\|_{H^{N-1}}
			\left\|
				\g_j \g_t g_{00} 
			\right\|_{H^{N-1}}
			\\
		& \ \ +
			\sum_{a,j=1}^3
			\left\|
				g^{0j} 
			\right\|_{H^{N-1}}
			\left\|
				u^a
			\right\|_{H^{N-1}}
			\left\|
				\g_a \g_j g_{00} 
			\right\|_{H^{N-1}}
	 \lesssim e^{-2q \Omega} \totalnorm.
		\notag
\end{align*}

\noindent{\em Case 2C: $(\mu,\nu) = (j,k).$}				
In this case, we use 
the standard Sobolev calculus and the estimates
\eqref{E:GJKUPPERLINFTY}
and \eqref{E:u0MINUSONEHNMINUSONE} 
to deduce the desired bound:
\begin{align*}
		\left \|
			g^{jk} (\g_j u^{\delta})
			\g_{\vec{\alpha}} \g_k \partial_{\delta} g_{00}
		\right\|_{L^2}
		& \lesssim 
			\sum_{j,k=1}^3
			\left\|
				g^{jk}
			\right\|_{L^{\infty}}
			\left\|
				u^0 - 1
			\right\|_{H^{N-1}}
			\left\|
				\g_k \g_t g_{00} 
			\right\|_{H^{N-1}}
			\\
		& \ \ +
			\sum_{a,j,k=1}^3
			\left\|
				g^{jk}
			\right\|_{L^{\infty}}
			\left\|
				u^a
			\right\|_{H^{N-1}}
			\left\|
				\g_a \g_k g_{00} 
			\right\|_{H^{N-1}}
	 \lesssim e^{-2q \Omega} \totalnorm.
		\notag
\end{align*}
\vspace{0.1in}

\noindent
{\em Proof of \eqref{E:UCOMMUTATORBOUND}-\eqref{E:RHOCOMMUTATORBOUND}.}
For any $|\vec{\alpha}|\le N-1,$ we have 
\[
\{\frac{u^a}{u^0}\g_a,\,\a\}u^j = \sum_{\vec{\beta}<\vec{\alpha}}{\vec{\alpha} \choose \vec{\beta}} \left[\g_{\vec{\alpha}-\vec{\beta}} 
\left(\frac{u^a}{u^0}\right)\right]\g_a\g_{\vec{\beta}} u^j.
\]
Hence, using Prop.~\ref{P:derivativesofF1FkL2},
Cor. \ref{C:SobolevTaylor},
and \eqref{E:u0MINUSONEHNMINUSONE}, we obtain that
\begin{align*}
\sum_{|\vec{\alpha}|\le N-1}\left\| \{ \frac{u^a}{u^0} \partial_a, \a \} u^j \right\|_{L^2}
\lesssim 
\sum_{a,j=1}^3
\left\lbrace
	1 +
	\left\|\frac{1}{u^0} - 1 \right\|_{H^{N-1}}
\right\rbrace
\|u^a\|_{H^{N-1}}
\| u^j\|_{H^{N-1}} \lesssim e^{-2(1+q)\Omega} \totalnorm,
\end{align*}
The remaining estimate in \eqref{E:UCOMMUTATORBOUND} and the
estimate \eqref{E:RHOCOMMUTATORBOUND} follow similarly.

\vspace{0.1in}
\noindent
{\em Proof of \eqref{E:COMMUTATORPARTIALTPARTIALUMETRICHNMINUONE}.}
We prove the estimate only for $g_{00}$ since the other two estimates can be proved in a similar fashion.
To proceed, we use Prop. \ref{P:derivativesofF1FkL2},
the identity $\{\u,\g_t\} = -(\g_t u^{\delta})\g_{\delta},$
\eqref{E:dtuzeroN-1},
and
\eqref{E:dtujN-1}
to deduce that for $|\vec{\alpha}| \leq N-1,$ we have
\begin{align*}
\|\{\u,\g_t\}\a g_{00}\|_{L^2} & \lesssim \|\g_tu^0\|_{H^{N-1}} \|\g_t g_{00}\|_{H^{N-1}} + \|\g_t u^a\|_{H^{N-1}}\|\g_a g_{00}\|_{H^{N-1}} 
\lesssim e^{-2q\Omega} \gnorm{N-1}
\end{align*}
as claimed.

\vspace{0.1in}
\noindent
{\em Proof of \eqref{E:commutatorserror1}-\eqref{E:commutatorserror3}.}
To prove \eqref{E:commutatorserror1}, we use 
Prop. \ref{P:derivativesofF1FkL2},
\eqref{E:partialuNEW},
and \eqref{E:dttgHMMINUSTWO} 
to deduce that for $|\vec{\alpha}| \leq N-1,$ we have
\begin{align*}
 \left\|\left(\u g^{00}\right) \g_{tt}\a g_{00} \right \|_{L^2} + \left\| \left(\u g^{0a}\right)\g_t\g_a\a g_{00} \right\|_{L^2}
& \lesssim \|\u g^{00}\|_{H^{N-1}} \|\g_{tt}g_{00}\|_{H^{N-1}} + \|\u g^{0a}\|_{H^{N-1}} \|\g_t\g_a g_{00}\|_{H^{N-1}} \\
& \lesssim e^{-2q\Omega}\totalnorm
\end{align*}
as desired. The estimates \eqref{E:commutatorserror2} and \eqref{E:commutatorserror3} can be proved in a similar fashion.

\vspace{0.1in}
\noindent
{\em Proof of~\eqref{E:commutatorgabGammaajb} and \eqref{E:COMMUTATORPARTIALUtriangleTWODOWNHNMINUSONE}.}
To prove \eqref{E:commutatorgabGammaajb}, 
we use Prop.~\ref{P:derivativesofF1FkL2},
\eqref{E:u0MINUSONEHNMINUSONE},
\eqref{E:gabupperGammaajblower},
and
\eqref{E:partialtgabupperGammaajblowerHN-2}
to deduce that for $|\vec{\alpha}| \leq N-1,$ we have
\begin{align*}
\left\|\{\u,\a\}(g^{ab}\Gamma_{ajb})\right\|_{L^2}
&= \left\|\sum_{\vec{\beta} < \vec{\alpha}} {\vec{\alpha} \choose \vec{\beta}} \left(\g_{\vec{\alpha}-\vec{\beta}} u^{\delta} \right)
\g_{\vec{\beta}}\g_{\delta}(g^{ab}\Gamma_{ajb}) \right\|_{L^2} \\
& \lesssim 
	\sum_{i=1}^3\|u^i \|_{H^{N-1}} 
	\left\| g^{ab}\Gamma_{ajb} \right\|_{H^{N-1}}
	+ \|u^0 - 1 \|_{H^{N-1}} 
		\left\| \partial_t \left(g^{ab}\Gamma_{ajb} \right) \right\|_{H^{N-2}}
		\\
& \lesssim e^{-2q \Omega} \totalnorm
\end{align*}
as desired.

The proof of \eqref{E:COMMUTATORPARTIALUtriangleTWODOWNHNMINUSONE} is similar,
but we use the estimates
\eqref{E:triangleTWODOWNHNMINUSONE} 
and
\eqref{E:dttriangleTWOLOWERHN-2}
in place of
\eqref{E:gabupperGammaajblower}
and
\eqref{E:partialtgabupperGammaajblowerHN-2}.

\vspace{0.1in}

\noindent
{\em Proof of~\eqref{E:equivalencecommutator1}-\eqref{E:ALTERNATEequivalencecommutator3}.}
We first prove \eqref{E:equivalencecommutator1} for the first term on the left-hand side.
To this end, we use the decomposition $\g_t = \frac{1}{u^0} \g_{\mathbf{u}} - \frac{u^a}{u^0} \g_a$ 
to compute that 
\begin{align} \label{E:PARTIALTPARTIALUPARTIALSPATIALCOMMUTATORG00}
(\g_t\u\a-\a\u\g_t)g_{00}
& = 
	(\g_t u^{\mu}) \g_{\mu} \a g_{00}
- \sum_{\vec{\beta}<\vec{\alpha}}
	{\vec{\alpha} \choose \vec{\beta}}
	(\g_{\vec{\alpha} - \vec{\beta}}u^{\mu})
	\g_{\mu}\g_{\vec{\beta}} \g_t g_{00}
		\\
& = 
	(\g_t u^0) \a \g_t g_{00}
	+ (\g_t u^a) \a \g_a g_{00}
		\notag \\
& \ \ 
	- \frac{1}{u^0}
	\sum_{\vec{\beta}<\vec{\alpha}}
	{\vec{\alpha} \choose \vec{\beta}}
	(\g_{\vec{\alpha} - \vec{\beta}}u^0)
	\g_t \g_{\mathbf{u}} 
	\g_{\vec{\beta}}  g_{00}
	+ \frac{u^a}{u^0} 
		\sum_{\vec{\beta}<\vec{\alpha}}
		{\vec{\alpha} \choose \vec{\beta}}
		(\g_{\vec{\alpha} - \vec{\beta}}u^0)
		\g_a
		\g_{\vec{\beta}} \g_t g_{00}
		\notag \\
& \ \ 
	+ \frac{1}{u^0}
	\sum_{\vec{\beta}<\vec{\alpha}}
	{\vec{\alpha} \choose \vec{\beta}}
	(\g_{\vec{\alpha} - \vec{\beta}}u^0)
	(\g_t u^0)
	\g_{\vec{\beta}} \g_t g_{00}
	+ \frac{1}{u^0}
	\sum_{\vec{\beta}<\vec{\alpha}}
	{\vec{\alpha} \choose \vec{\beta}}
	(\g_{\vec{\alpha} - \vec{\beta}}u^0)
	(\g_t u^a)
	\g_a
	\g_{\vec{\beta}} g_{00}
		\notag \\
& \ \ 
	- \sum_{\vec{\beta}<\vec{\alpha}}
	{\vec{\alpha} \choose \vec{\beta}}
	(\g_{\vec{\alpha} - \vec{\beta}}u^a)
	\g_{\vec{\beta}} \g_a \g_t g_{00}.
	\notag
\end{align}
Hence, using Prop.~\ref{P:derivativesofF1FkL2},
\eqref{E:u0MINUSONEHNMINUSONE}, 
\eqref{E:dtuzeroN-1},
and
\eqref{E:dtujN-1}, 
for $|\vec{\alpha}| \leq N - 1,$ 
we have
\begin{align*}
\|(\g_t\u\a-\a\u\g_t)g_{00}\|_{L^2}
& \lesssim	
	e^{-q \Omega}
	\sum_{|\vec{\beta}| \leq N-2}
	\left\| \g_t \u \g_{\vec{\beta}} g_{00} \right\|_{L^2} 
	+ e^{-q \Omega} \| \g_t g_{00} \|_{H^{N-1}}
	+ \sum_{a=1}^3 e^{-(1 + q) \Omega} \| \g_a g_{00} \|_{H^{N-1}}
		\\
& \lesssim 
\sum_{|\vec{\beta}| \leq N-2}
\left\| \g_t \u \g_{\vec{\beta}} g_{00} \right\|_{L^2}
+ e^{-q \Omega} \gzerozeronorm{N-1}
\end{align*}
as desired.

To bound the second term on the left-hand side of
\eqref{E:equivalencecommutator1},
we first note the identity
\begin{align} \label{E:PARTIALIPARTIALUPARTIALSPATIALCOMMUTATORG00}
(\g_i \u\a-\a\u\g_i)g_{00}
& = 
	(\g_i u^0) \a \g_t g_{00}
	+ 	(\g_i u^a) \g_{\mu} \a \g_a g_{00}
- \sum_{\vec{\beta}<\vec{\alpha}}
	{\vec{\alpha} \choose \vec{\beta}}
	(\g_{\vec{\alpha} - \vec{\beta}}u^{\mu})
	\g_{\mu}\g_{\vec{\beta}} \g_i g_{00}.
\end{align}
Therefore, using Prop. \ref{P:derivativesofF1FkL2}
and
\eqref{E:u0MINUSONEHNMINUSONE}, 
for $|\vec{\alpha}| \leq N - 1,$ 
we have
\begin{align}
	\| (\g_i \u\a-\a\u\g_i)g_{00} \|_{L^2}
	& \lesssim 
		\| u^0 - 1 \|_{H^{N-1}}
		\| \g_t g_{00} \|_{H^{N-1}}
		+ \| u^a \|_{H^{N-1}}
			\| \g_a g_{00} \|_{H^{N-1}}
	\lesssim e^{-q\Omega} \gzerozeronorm{N-1}
\end{align}
as desired.
To bound the third term on the left-hand side of
\eqref{E:equivalencecommutator1}, we first note the identity
\begin{align}
\{\u,\a\}g_{00} = 
	- \sum_{\vec{\beta}<\vec{\alpha}}
	{\vec{\alpha} \choose \vec{\beta}}
	(\g_{\vec{\alpha} - \vec{\beta}}u^0)
	\g_{\vec{\beta}} \g_t g_{00}
	- \sum_{\vec{\beta}<\vec{\alpha}}
	{\vec{\alpha} \choose \vec{\beta}}
	(\g_{\vec{\alpha} - \vec{\beta}}u^a)
	\g_{\vec{\beta}} \g_a g_{00}.
\end{align}
Hence, 
by Prop. \ref{P:derivativesofF1FkL2} 
and
\eqref{E:u0MINUSONEHNMINUSONE}, 
for $|\vec{\alpha}| \leq N - 1,$ 
we have
\begin{align}
	\| \{\u,\a\} g_{00} \|_{L^2}
	& \lesssim 
		\| u^0 - 1 \|_{H^{N-1}}
		\| \g_t g_{00} \|_{H^{N-1}}
		+ \| u^a \|_{H^{N-1}}
			\| g_{00} + 1 \|_{H^{N-1}}
	\lesssim e^{-q \Omega} \gzerozeronorm{N-1}
\end{align}
as desired. We have thus proved \eqref{E:equivalencecommutator1}.

We now prove the estimate \eqref{E:ALTERNATEequivalencecommutator1}.
The second and third terms were suitably bounded in the proof of \eqref{E:equivalencecommutator1}.
Hence, we only have to bound the first term on the left-hand side of \eqref{E:ALTERNATEequivalencecommutator1}.
To proceed, we use Lemma \ref{L:twotimederivatives}
to rewrite the first line of 
\eqref{E:PARTIALTPARTIALUPARTIALSPATIALCOMMUTATORG00} as
\begin{align}
	(\g_t u^{\mu}) \g_{\mu} \a g_{00}
	& - \sum_{\vec{\beta}<\vec{\alpha}}
		{\vec{\alpha} \choose \vec{\beta}}
		(\g_{\vec{\alpha} - \vec{\beta}}u^0)
		\g_{\vec{\beta}} 
		\left\lbrace
			\frac{1}{u^0}\u\g_t g_{00}
		- \frac{u^a}{(u^0)^2}\u\g_a g_{00} 
		+ \frac{u^a u^b}{(u^0)^2}\g_a \g_b g_{00}
		\right\rbrace	
		\\
	& \ \ 
		- \sum_{\vec{\beta}<\vec{\alpha}}
		{\vec{\alpha} \choose \vec{\beta}}
		(\g_{\vec{\alpha} - \vec{\beta}}u^a)
		\g_a \g_{\vec{\beta}} \g_t g_{00}.
		\notag
\end{align}
Therefore, by Prop. \ref{P:derivativesofF1FkL2},
Cor. \ref{C:SobolevTaylor},
\eqref{E:u0MINUSONEHNMINUSONE}, 
\eqref{E:dtuzeroN-1},
and
\eqref{E:dtujN-1}, 
for $|\vec{\alpha}| \leq N - 1,$ 
we have
\begin{align*}
& \|(\g_t\u\a-\a\u\g_t)g_{00}\|_{L^2}
\lesssim 
e^{-q\Omega} \gzerozerounorm{N-1}
+ e^{-q\Omega} \gzerozeronorm{N-1}
\end{align*}
as desired.
We have thus proved \eqref{E:ALTERNATEequivalencecommutator1}.

The remaining bounds \eqref{E:equivalencecommutator2}-\eqref{E:ALTERNATEequivalencecommutator3} 
can be proved using similar arguments, and we omit the details.

\vspace{0.1in}
\noindent 
{\em Proof of \eqref{E:gcommutatorL2}.}
We only give the proof for the term
$\| \{\hat{\square}_g,\g_{\vec{\alpha}}\} g_{00} \|_{L^2}$
from the left-hand side of \eqref{E:gcommutatorL2}
since the proofs for the other terms are similar.
We begin by stating the following commutation formula, which is straightforward to verify
by direct computation:
\begin{align} \label{E:PARTIALALPHABOXCOMMUTATORACTINGONGOOEXPRESSION}
\{
	\hat{\square}_g, \g_{\vec{\alpha}}
\} g_{00} 
& = - \sum_{\vec{\beta} < \vec{\alpha}}
			{\vec{\alpha} \choose \vec{\beta}}
			(\g_{\vec{\alpha} - \vec{\beta}} g^{\mu\nu})
			\g_{\vec{\beta}} \g_{\mu} \g_{\nu} g_{00}.
\end{align}
%
\vspace{0.1in}

\noindent{\em Step 1: $(\mu,\nu) = (0,0).$}
In this case, we use Prop. \ref{P:derivativesofF1FkL2}
and the estimates
\eqref{E:G00UPPERPLUSONEHNMINUSONE}
and \eqref{E:dttgHMMINUSTWO}
to deduce the desired bound:
\begin{align*}
		\left \|
				(\g_{\vec{\alpha} - \vec{\beta}} g^{00})
				\g_{\vec{\beta}} \g_{tt} g_{00}
		\right\|_{L^2}
		& \lesssim 
			\left\|
				g^{00} + 1
			\right\|_{H^{N-1}}
			\left\|
				\g_{tt} g_{00}
			\right\|_{H^{N-2}}
		\lesssim e^{-2q \Omega} \totalbelowtopnorm{N-1}.
		\notag
\end{align*}

\noindent{\em Step 2: $(\mu,\nu) = (0,j).$}				
In this case, we use Prop. \ref{P:derivativesofF1FkL2} 
and the estimate \eqref{E:G0JUPPERHNMINUSONE}
to deduce the desired bound:
\begin{align*}
		\left \|
				(\g_{\vec{\alpha} - \vec{\beta}} g^{0j})
				\g_{\vec{\beta}} \g_t \g_j g_{00}
		\right\|_{L^2}
		& \lesssim 
			\sum_{j=1}^3
			\left\|
				g^{0j}
			\right\|_{H^{N-1}}
			\left\|
				\g_t g_{00}
			\right\|_{H^{N-1}}
		\lesssim e^{-2q \Omega} \totalbelowtopnorm{N-1}.
		\notag
\end{align*}
		
\noindent{\em Step 3: $(\mu,\nu) = (j,k).$}					
In this case, we use Prop. \ref{P:derivativesofF1FkL2} 
and the estimate 
\eqref{E:DERIVATIVESOFGJKUPPERHNMINUS2}
to deduce the desired bound:
\begin{align*}
		\left \|
				(\g_{\vec{\alpha} - \vec{\beta}} g^{jk})
				\g_{\vec{\beta}} \g_j \g_k g_{00}
		\right\|_{L^2}
		& \lesssim 
			\sum_{j,k=1}^3
			\left\|
				\underpartial g^{jk}
			\right\|_{H^{N-2}}
			\sum_{a=1}^3
			\left\|
				\g_a g_{00} 
			\right\|_{H^{N-1}}
		\lesssim e^{-2q \Omega} \totalbelowtopnorm{N-1}.
		\notag
\end{align*}			

\vspace{0.1in}
\noindent 
{\em Proof of \eqref{E:gpartialucommutatorL2}.}
We only give the proof for the term
$\| \u \{\hat{\square}_g,\g_{\vec{\alpha}}\} g_{00} \|_{L^2}$
from the left-hand side of \eqref{E:gpartialucommutatorL2} since the proofs for the other terms
are similar.
We begin by stating the following commutation formula, which is straightforward to verify
by direct computation:
\begin{align} \label{E:commagain}
\u 
\{
	\hat{\square}_g, \g_{\vec{\alpha}}
\} g_{00} 
& = - \sum_{\vec{\beta} < \vec{\alpha}}
			{\vec{\alpha} \choose \vec{\beta}}
			u^{\delta}
			(\g_{\vec{\alpha} - \vec{\beta}} g^{\mu\nu})
			\g_{\vec{\beta}} \g_{\mu} \g_{\nu} \partial_{\delta} g_{00}
			\\
& \ \ - \sum_{\vec{\beta} < \vec{\alpha}}
			{\vec{\alpha} \choose \vec{\beta}}
			(\g_{\vec{\alpha} - \vec{\beta}} \partial_{\mathbf{u}} g^{\mu\nu})
			\g_{\vec{\beta}} \g_{\mu} \g_{\nu} g_{00}
			\notag \\
& \ \ + \mathop{\sum_{\vec{\beta} \leq \vec{\alpha}}}_{\vec{\gamma} < \vec{\alpha} - \vec{\beta}}
				{\vec{\alpha} \choose \vec{\beta}}
				{ \vec{\alpha} - \vec{\beta} \choose \vec{\gamma}}
				(\g_{\vec{\alpha} - \vec{\beta} - \vec{\gamma}} u^{\delta})(\g_{\vec{\gamma}}\g_{\delta} g^{\mu\nu}) 
				\g_{\mu} \g_{\nu} \g_{\vec{\beta}}  g_{00}.
			 \notag
\end{align}
We will separately bound each of the $3$ sums on the right-hand side of \eqref{E:commagain}.
\vspace{0.1in}

\noindent{\em Step 1: Bounds for the term $- \sum_{\vec{\beta} < \vec{\alpha}}
			{\vec{\alpha} \choose \vec{\beta}}
			u^{\delta}
			(\g_{\vec{\alpha} - \vec{\beta}} g^{\mu\nu})
			\g_{\vec{\beta}} \g_{\mu} \g_{\nu} \partial_{\delta} g_{00}.$}

\noindent{\em Case 1A: $(\mu,\nu) = (0,0).$}				
In this case, we use Prop. \ref{P:derivativesofF1FkL2}
and the estimates
\eqref{E:G00UPPERPLUSONEHNMINUSONE},
\eqref{E:u0MINUSONEHNMINUSONE},
\eqref{E:dttgHMMINUSTWO},
and \eqref{E:dtttgHMMINUSTWO} 
to deduce the desired bound:
\begin{align*}
		\left \|
				u^{\delta}
				(\g_{\vec{\alpha} - \vec{\beta}} g^{00})
				\g_{\vec{\beta}} \g_{tt} \partial_{\delta} g_{00}
		\right\|_{L^2}
		& \lesssim 
			\left\lbrace
				1 + 
				\left\|
					u^0 - 1
				\right\|_{H^{N-1}}
			\right\rbrace
			\left\|
				g^{00} + 1
			\right\|_{H^{N-1}}
			\left\|
				\g_{ttt} g_{00}
			\right\|_{H^{N-2}}
			\\
		& \ \ +
			\sum_{a=1}^3
			\left\|
				u^a
			\right\|_{H^{N-1}}
			\left\|
				g^{00} + 1
			\right\|_{H^{N-1}}
			\left\|
				\g_{tt} g_{00}
			\right\|_{H^{N-1}}
			\notag 
	 \lesssim e^{-2q \Omega} \totalnorm.
		\notag
\end{align*}

\noindent{\em Case 1B: $(\mu,\nu) = (0,j).$}				
In this case, we use Prop. \ref{P:derivativesofF1FkL2} 
and the estimates
\eqref{E:G0JUPPERHNMINUSONE},
\eqref{E:u0MINUSONEHNMINUSONE},
and \eqref{E:dttgHMMINUSTWO}
to deduce the desired bound:
\begin{align*}
		\left \|
				u^{\delta}
				(\g_{\vec{\alpha} - \vec{\beta}} g^{0j})
				\g_{\vec{\beta}} \g_j \g_t \partial_{\delta} g_{00}
		\right\|_{L^2}
		& \lesssim 
			\sum_{j=1}^3
			\left\lbrace
				1 + 
				\left\|
					u^0 - 1
				\right\|_{H^{N-1}}
			\right\rbrace
			\left\|
				g^{0j}
			\right\|_{H^{N-1}}
			\left\|
				\g_{tt} g_{00}
			\right\|_{H^{N-1}}
			\\
		& \ \ +
			\sum_{a,j=1}^3
			\left\|
				u^a
			\right\|_{H^{N-1}}
			\left\|
				g^{0j}
			\right\|_{H^{N-1}}
			\left\|
				\partial_j \g_t g_{00}
			\right\|_{H^{N-1}}
	 \lesssim e^{-2q \Omega} \totalnorm.
		\notag
\end{align*}

\noindent{\em Case 1C: $(\mu,\nu) = (j,k).$}				
In this case, we use Prop. \ref{P:derivativesofF1FkL2} 
and the estimates
\eqref{E:DERIVATIVESOFGJKUPPERHNMINUS2}
and \eqref{E:u0MINUSONEHNMINUSONE}
to deduce the desired bound:
\begin{align*}
		\left \|
				u^{\delta}
				(\g_{\vec{\alpha} - \vec{\beta}} g^{jk})
				\g_{\vec{\beta}} \g_j \g_k \partial_{\delta} g_{00}
		\right\|_{L^2}
		& \lesssim 
			\sum_{j,k=1}^3
			\left\lbrace
				1 + 
				\left\|
					u^0 - 1
				\right\|_{H^{N-1}}
			\right\rbrace
			\left\|
				\underpartial g^{jk}
			\right\|_{H^{N-2}}
			\left\|
				\partial_j \g_t g_{00}
			\right\|_{H^{N-1}}
			\\
		& \ \ +
			\sum_{a,j,k=1}^3
			\left\|
				u^a
			\right\|_{H^{N-1}}
			\left\|
				\underpartial g^{jk}
			\right\|_{H^{N-2}}
			\left\|
				\g_j \g_k g_{00}
			\right\|_{H^{N-1}}
	 \lesssim e^{-2q \Omega} \totalnorm.
		\notag
\end{align*}			
\vspace{0.1in}

\noindent{\em Step 2: Bounds for the term
			$- \sum_{\vec{\beta} < \vec{\alpha}}
			{\vec{\alpha} \choose \vec{\beta}}
			(\g_{\vec{\alpha} - \vec{\beta}} \partial_{\mathbf{u}} g^{\mu\nu})
			\g_{\vec{\beta}} \g_{\mu} \g_{\nu} g_{00}.$
}

\vspace{0.1in}

\noindent{\em Case 2A: $(\mu,\nu) = (0,0).$}				
In this case, we use Prop. \ref{P:derivativesofF1FkL2}
and the estimates
\eqref{E:partialuNEW}
and \eqref{E:dttgHMMINUSTWO}
to deduce the desired bound:
\begin{align*}
		\left \|
				(\g_{\vec{\alpha} - \vec{\beta}} \partial_{\mathbf{u}} g^{00}) 
				\g_{tt} \g_{\vec{\beta}} g_{00}
		\right\|_{L^2}
		& \lesssim 
			\left\|
				\partial_{\mathbf{u}} g^{00}
			\right\|_{H^{N-1}}
			\left\|
				\g_{tt} g_{00}
			\right\|_{H^{N-2}}
	\lesssim e^{-2q \Omega} \totalnorm.
\end{align*}

\noindent{\em Case 2B: $(\mu,\nu) = (0,j).$}				
In this case, we use Prop. \ref{P:derivativesofF1FkL2}
and the estimate
\eqref{E:partialuNEW}
to deduce the desired bound:
\begin{align*}
		\left \|
				(\g_{\vec{\alpha} - \vec{\beta}} \partial_{\mathbf{u}} g^{0j}) 
				\g_t \g_j \g_{\vec{\beta}} g_{00}
		\right\|_{L^2}
		& \lesssim 
			\left\|
				\partial_{\mathbf{u}} g^{0j}
			\right\|_{H^{N-1}}
			\left\|
				\partial_j \g_t g_{00}
			\right\|_{H^{N-1}}
	\lesssim e^{-2q \Omega} \totalnorm.
\end{align*}

\noindent{\em Case 2C: $(\mu,\nu) = (j,k).$}				
In this case, we use Prop. \ref{P:derivativesofF1FkL2}
and the estimates
\eqref{E:partialugupperjk}
and
\eqref{E:partialugupperjk2}
to deduce the desired bound:
\begin{align*}
		\left \|
				(\g_{\vec{\alpha} - \vec{\beta}} \partial_{\mathbf{u}} g^{jk}) 
				\g_j \g_k \g_{\vec{\beta}} g_{00}
		\right\|_{L^2}
		& \lesssim 
			\sum_{j,k=1}^3
			\left\lbrace
			\left\|
					\partial_{\mathbf{u}} g^{jk}
			\right\|_{L^{\infty}}
			+ 
			\left\|
				\underpartial
				\partial_{\mathbf{u}} g^{jk}
			\right\|_{H^{N-2}}
			\right\rbrace
			\left\|
			 	\g_j g_{00} 
			\right\|_{H^{N-1}}
	\lesssim e^{-2q \Omega} \totalnorm.
\end{align*}
\vspace{0.1in}

\noindent{\em Step 3: Bounds for the term
			$\sum_{\vec{\beta} \leq \vec{\alpha}} \sum_{\vec{\gamma} < \vec{\alpha} - \vec{\beta}}
				{\vec{\alpha} \choose \vec{\beta}}
				{ \vec{\alpha} - \vec{\beta} \choose \vec{\gamma}}
				(\g_{\vec{\alpha} - \vec{\beta} - \vec{\gamma}} u^{\delta})(\g_{\vec{\gamma}}\g_{\delta} g^{\mu\nu}) 
				\g_{\mu} \g_{\nu} \g_{\vec{\beta}} g_{00}.$
}

\vspace{0.1in}
Step $3$ can be treated in essentially the same fashion as Step $1.$

\vspace{0.1in}

\noindent
{\em Proof of~\eqref{E:triangleEgamma00delta00L1}-\eqref{E:partialutriangleEgamma**delta**L1}}.
We cite the following bound \cite[Eq. (9.2.20)]{jS2012}, which can be proved by using
arguments similar to the ones we use to prove \eqref{E:ELLIPTICERRORL2}:
\begin{align} \label{E:trianglegammadeltaL1}
	\| \triangle_{\mathcal{E};(\upgamma, \updelta)}[v,\partial v] \|_{L^1} & \leq C \Big\lbrace e^{-q \Omega}  
		\| \partial_t v \|_{L^2}^2 + e^{-(2 + q) \Omega} 
		\| \underpartial v \|_{L^2}^2 + C_{(\upbeta)} e^{-q \Omega} \| v \|_{L^2}^2 \Big\rbrace,
\end{align}
where $C_{(\upbeta)}$ is defined in \eqref{E:mathcalEfirstlowerbound}. 
We now show how to use \eqref{E:trianglegammadeltaL1} to prove
\eqref{E:partialutriangleEgamma00delta00L1}; the remaining bounds in
\eqref{E:triangleEgamma00delta00L1}-\eqref{E:partialutriangleEgamma**delta**L1} 
can be proved by using essentially the same reasoning. 
To begin, we use \eqref{E:trianglegammadeltaL1} to deduce that
\begin{align} \label{E:NOTYETCOMMUTEDESTIMATEpartialutriangleEgamma00delta00L1}
		e^{2q \Omega} \left\| \triangle_{\mathcal{E};(\widetilde{\upgamma}_{00}, \widetilde{\updelta}_{00})}[\u 		
				\partial_{\vec{\alpha}} g_{00}
			,\partial (\u \partial_{\vec{\alpha}} g_{00})] \right\|_{L^1}
	& \lesssim 
		e^{q \Omega} \| \partial_t \u \partial_{\vec{\alpha}} g_{00} \|_{L^2}^2
			\\
	& \ \ 
		+ \sum_{a=1}^3 e^{(q - 2) \Omega} \| \g_a \u \partial_{\vec{\alpha}} g_{00} \|_{L^2}^2
		+ e^{q \Omega} \| \u \partial_{\vec{\alpha}} g_{00} \|_{L^2}^2.
		\notag
\end{align}
From the first inequality in \eqref{E:ALTERNATEequivalencecommutator1},
we deduce that the first term on the right-hand side of 
\eqref{E:NOTYETCOMMUTEDESTIMATEpartialutriangleEgamma00delta00L1} is equal to
$e^{q \Omega} \| \partial_{\vec{\alpha}} \u  \partial_t g_{00} \|_{L^2}^2$
plus an error term that is $\lesssim e^{-q \Omega} \totalnorm^2.$
Furthermore, from the definition of $\totalnorm,$ we have
$e^{q \Omega} \| \partial_{\vec{\alpha}} \u  \partial_t g_{00} \|_{L^2}^2 
\lesssim 
e^{-q \Omega} \| \u  \partial_t g_{00} \|_{H^{N-1}}^2\lesssim
e^{-q \Omega} \totalnorm^2$
as desired. Similarly, from the second and third inequalities in \eqref{E:ALTERNATEequivalencecommutator1},
we deduce that the second and third terms on
the right-hand side of \eqref{E:NOTYETCOMMUTEDESTIMATEpartialutriangleEgamma00delta00L1}
are $\lesssim e^{-q \Omega} \totalnorm^2$ as desired.

\vspace{0.1in}
\noindent
{\em Proof of \eqref{E:ELLIPTICERRORL2}.}
We prove \eqref{E:ELLIPTICERRORL2} only for the first term on the left-hand side since the remaining inequalities can be proved
by using similar arguments. We begin by referring to the definition \eqref{E:NEWQUADRATICERROR} of $\triangle_{\text{Ell}}$ 
with $v=g_{00}$.
The first four terms on the right-hand side of~\eqref{E:NEWQUADRATICERROR} are bounded by 
\eqref{E:BOXUG00COMM} and \eqref{E:commutatorserror1}.
To bound the next term, 
we use \eqref{E:SUPERIMPORTANT}
to deduce that
$\|(\u g^{ab} + 2\omega g^{ab})\g_{a}\g_b \a g_{00}\|_{L^2} 
\lesssim \|\u g^{ab} + 2\omega g^{ab}\|_{L^{\infty}} \|\g_a \g_b g_{00}\|_{H^{N-1}}
\lesssim  \sum_{a,b=1}^3 e^{-(2 + q)\Omega} \|\g_a \g_b g_{00}\|_{H^{N-1}}
\lesssim e^{-2 q \Omega} \totalnorm$ as desired.
To bound the term $\|(\omega-H)g^{ab}\g_a\g_b\a g_{00}\|_{L^2}$, we use 
a similar argument based on 
the estimates 
$\|\omega-H\|_{L^{\infty}} \lesssim e^{-3\Omega}$ and
\eqref{E:GJKUPPERLINFTY}.
To bound the next term, we use \eqref{E:G0JUPPERHNMINUSONE}
to deduce that 
$\|g^{0a}\g_a\g_t\a g_{00}\|_{L^2} \lesssim \| g^{0a}\|_{L^{\infty}} \| \g_a\g_t \a g_{00} \|_{H^{N-1}}
\lesssim \sum_{a=1}^3 e^{-(1 + q) \Omega} \| \g_a\g_t \a g_{00} \|_{H^{N-1}} 
\lesssim e^{-2q \Omega} \totalnorm$ as desired.
To bound the term $\|\{\hat{\Box}_g,\a\}g_{00}\|_{L^2},$ we use \eqref{E:gcommutatorL2}.
To bound the next term, we use
\eqref{E:G00UPPERPLUSONEHNMINUSONE}
and \eqref{E:dttgHMMINUSONE} to deduce that 
$\|(g^{00}+1)\g_{tt}\a g_{00}\|_{L^2} 
\lesssim \|(g^{00}+1) \|_{L^{\infty}} \| \g_{tt}\a g_{00}\|_{H^{N-1}}
\lesssim e^{-2q\Omega}\totalnorm$ as desired.
To bound the next term, we use Cor. \ref{C:SobolevTaylor}, 
\eqref{E:u0MINUSONEHNMINUSONE}, 
and the commutator estimate \eqref{E:ALTERNATEequivalencecommutator1}
to deduce that
$\|(\frac{1}{u^0}-1)\g_t\u \a g_{00}\|_{L^2}
\lesssim \|(\frac{1}{u^0}-1)\|_{L^{\infty}} \|\g_t \u \a g_{00}\|_{L^2}
\lesssim \| u^0 - 1 \|_{H^{N-1}} \left\lbrace \| \g_t \u g_{00}\|_{H^{N-1}} + e^{-q \Omega} \totalnorm \right\rbrace
\lesssim e^{-2q \Omega} \totalnorm$ as desired.
To bound the next term, we use
\eqref{E:u0MINUSONEHNMINUSONE},
\eqref{E:dtuzeroN-1},
and
\eqref{E:dtujN-1},
to deduce that
$\|\frac{\g_t u^{\delta}}{u^0}\g_{\delta}\a g_{00}\|_{L^2}
\lesssim \|\frac{\g_t u^0}{u^0}\|_{L^{\infty}} \| \g_t g_{00}\|_{H^{N-1}}
+ \|\frac{\g_t u^a}{u^0}\|_{L^{\infty}} \| \g_a g_{00}\|_{H^{N-1}}
\lesssim e^{-q \Omega} \| \g_t g_{00}\|_{H^{N-1}}
	+ \sum_{a=1}^3 e^{-(1 + q) \Omega} \| \g_a g_{00}\|_{H^{N-1}}
\lesssim e^{-2q \Omega} \totalnorm$ as desired.
To bound the last term, we use
\eqref{E:u0MINUSONEHNMINUSONE} to deduce that
$\|\frac{u^a}{u^0}\g_t\g_a\a g_{00}\|_{L^2}
\lesssim \| u^a \|_{H^{N-1}} \| \g_t  \g_a g_{00} \|_{H^{N-1}}
\lesssim e^{- (1 + q) \Omega} \sum_{a=1}^3 \| \g_t  \g_a g_{00} \|_{H^{N-1}}
\lesssim e^{-2q \Omega} \totalnorm$ as desired.

\section{Comparison of norms and energies}\label{S:equivalence}
\setcounter{equation}{0}
In this section, we show that the norms we used in formulating the bootstrap assumption~(\ref{E:bootstrap})
are controlled by the energies. As a preparatory step, we
first prove a crucially important auxiliary lemma that will, by virtue of elliptic estimates, 
allow us to control the elliptic norms 
$\gzerozeroellipticnorm{N-1},$
$\gzerostarellipticnorm{N-1},$
and
$\hstarstarellipticnorm{N-1}$
in terms of
the energies.

\begin{lemma} [\textbf{The key elliptic estimate}] \label{L:elliptic}
Assume the hypotheses and conclusions of Prop. \ref{P:Sobolev},
and let $v \in C^2([0,T),\,H^2(\T^3))$. 
We define the inverse Riemannian metric $H^{ij}$ on $\mathbb{T}^3$ by  
\begin{align} \label{E:Hij}
H^{ij}
& \eqdef e^{2 \Omega} g^{ij} 
	+ e^{2\Omega} \frac{g^{00} u^iu^j}{(u^0)^2}
	- e^{2 \Omega} g^{0i} 
		\frac{u^j}{u^0}
	- e^{2\Omega} g^{0j} 
		\frac{u^i}{u^0}, && (i,j = 1,2,3).
\end{align}
Then the following identity holds:
\begin{align}\label{E:elliptic}
H^{ab}\g_a \g_b v
&=e^{2\Omega} \hat{\square}_g v
	- e^{2\Omega} \frac{g^{00}}{u^0}\u\g_tv
	+ e^{2\Omega}
		\left\lbrace
			\frac{g^{00}u^a}{(u^0)^2}
			- 2 \frac{g^{0a}}{u^0}
		\right\rbrace 
		\u\g_a v.
\end{align}
Furthermore, the following estimate also holds for $t \in [0,T):$
\begin{align}\label{E:ell}
\sum_{a,b=1}^3 \| \g_a \g_b v \|_{L^2} 
& \lesssim 
	e^{2\Omega} \|\hat{\square}_g v\|_{L^2}
	+ e^{2\Omega} \|\g_t \u v\|_{L^2} 
	+  \sum_{a=1}^3 e^{(1-q)\Omega} \| \g_a \u v\|_{L^2}
		\\
	& \ \ + e^{(2-q)\Omega} \| \g_t v\|_{L^2} 
	+  \sum_{a=1}^3 e^{(1-q)\Omega} \| \g_a v\|_{L^2}.
		\notag
\end{align}
\end{lemma} 
\prf
Using Lemma~\ref{L:twotimederivatives}, we decompose $\hat{\square}_g v$ as follows: 
\beas
\hat{\square}_g v&=&g^{00}\g_{tt}v+ 2 g^{0a} \g_t\g_av+g^{ab}\g_a\g_bv\\
&=&\frac{g^{00}}{u^0}\u\g_tv-\frac{g^{00}u^a}{(u^0)^2}\u\g_av
+\frac{g^{00}u^au^b}{(u^0)^2}\g_a\g_b v
+2 g^{0a}\left\lbrace\frac{1}{u^0}\u\g_av
-\frac{u^b}{u^0}\g_b\g_av\right\rbrace
+g^{ab}\g_a\g_bv\\
&=&\frac{g^{00}}{u^0}\u\g_tv
+\left\lbrace-\frac{g^{00}u^a}{(u^0)^2}
+2 g^{0a} \frac{1}{u^0}\right\rbrace\u\g_av
+\left\lbrace\frac{g^{00}u^au^b}{(u^0)^2}- 2 g^{0a} \frac{u^b}{u^0}+g^{ab}\right\rbrace\g_a\g_bv\\
&=&\frac{g^{00}}{u^0}\u\g_tv
+\left\lbrace-\frac{g^{00}u^a}{(u^0)^2}
+ 2 g^{0a} \frac{1}{u^0}\right\rbrace\u\g_av
+e^{-2\Omega}H^{ab}\g_a\g_bv.
\eeas
The desired relation~\eqref{E:elliptic} follows directly from the above decomposition.

From definition \eqref{E:Hij}, 
\eqref{E:G00UPPERPLUSONEHNMINUSONE}, 
\eqref{E:G0JUPPERHNMINUSONE}, 
\eqref{E:u0MINUSONEHNMINUSONE},
and the inequality $\| u^a \|_{H^{N-1}} \lesssim e^{-(1 + q)\Omega}\totalnorm,$
we deduce that $\|H^{ij}-e^{2\Omega}g^{ij}\|_{H^{N-1}} \lesssim \epsilon e^{-2q \Omega}.$ 
Also using the estimates \eqref{E:gjkuppercomparetostandard}
and \eqref{E:DERIVATIVESOFGJKUPPERHNMINUS2},
we deduce that $H^{ij} \in C^1([0,T) \times \mathbb{T}^3)$
with uniformly bounded spatial derivatives up to first order
and that
\eqref{E:elliptic} can be viewed as a uniformly elliptic PDE in $v.$
Hence, we can apply the standard $H^2$-regularity bound for second-order elliptic PDEs 
(see, for example, \cite[Theorem 8.12]{dGnT1998}) 
to deduce that
\begin{align} \label{E:H2ELLIPTIC}
\sum_{a,b=1}^3 \| \g_a \g_b v \|_{L^2} 
& \lesssim  
	\left\|
		H^{ab} \g_a \g_b v	
	\right\|_{L^2}
	+ \sum_{a,b,i,j} \| \g_i H^{ab} \|_{C_b^1} \| \g_j v \|_{L^2}
		\\
& \lesssim 	
	\left\|	
		H^{ab} \g_a \g_b v	
	\right\|_{L^2}
	+ \sum_{j=1}^3 \| \g_j v \|_{L^2}.
	\notag
\end{align}
Using
\eqref{E:G00UPPERPLUSONEHNMINUSONE}, 
\eqref{E:G0JUPPERHNMINUSONE}, 
\eqref{E:u0MINUSONEHNMINUSONE},
\eqref{E:dtuzeroN-1},
\eqref{E:dtujN-1},
and the inequality $\| u^a \|_{H^{N-1}} \lesssim e^{-(1 + q)\Omega}\totalnorm,$
we bound the right-hand side of \eqref{E:elliptic} in the $L^2$ norm by
$\lesssim$ the right-hand side of \eqref{E:ell}. We stress that this bound involves
the commutators $\{\u, \g_t \} v$ and $\{\u, \g_a \} v,$ which are not difficult to estimate.
Combining this bound with \eqref{E:H2ELLIPTIC}, we conclude \eqref{E:ell}.
\prfe


\begin{proposition}[\textbf{Equivalence of the norms and energies}]\label{P:energynormcomparison}
 Let $\E_{(\upgamma,\updelta)}[v,\g v]$ be the metric component building block energy defined in \eqref{E:mathcalEdef},
 and let $C_{(\upbeta)} \geq 0$ be the constant associated to the constants
 $\upgamma,\updelta$ from the statement of Lemma \ref{L:buildingblockmetricenergy}.
 Let 
 $\gzerozeroenergy{N-1},$
 $\gzerozeronorm{N-1},$
 $\cdots$
 be the norms and energies defined in Sect. \ref{S:norms}.
 Assume that the hypotheses and conclusions of 
 Props. \ref{P:metric}-\ref{P:Sobolevtwo} hold on the spacetime slab $[0,T) \times\T^3.$
 Then there exists a constant $\epsilon' > 0$ such that if $\totalnorm(t) \leq \epsilon'$ for $t \in [0,T),$ 
 then the following inequalities also hold for $t \in [0,T):$
 \begin{subequations}
 \begin{align}
 	\label{E:mathcalEcomparison}
  \|\g_tv\|_{L^2}+e^{-\Omega} \sum_{i=1}^3 \|\g_i v\|_{L^2}+C_{(\upbeta)}\|v\|_{L^2} 
  & \approx \E_{(\upgamma,\updelta)}[v,\g v],
  		\\
  \gzerozeronorm{N-1} & \approx \gzerozeroenergy{N-1},  
  	\label{E:gzerozeronormequivalence}
  	\\
  \gzerostarnorm{N-1} & \approx \gzerostarenergy{N-1}, 
  	\label{E:gzerostarnormequivalence}
  	\\
  	\hstarstarnorm{N-1}
  	& \approx
   \hstarstarenergy{N-2} + \partialhstarstarenergy{N-1}, 
   	\label{E:hstarstarnormequivalence}
  	\\
  \gnorm{N-1} & \approx \genergy{N-1},
  	\label{E:GNORMANDENEGYEQUIVALENCE} \\
  \gzerozeronorm{N-1} + \gzerozerounorm{N-1}
  & \approx
  \gzerozeroenergy{N-1} + \gzerozeropartialuenergy{N-1}, 
  	\label{E:partialufirst} \\
  \gzerostarnorm{N-1} + \gzerostarunorm{N-1}
  & \approx
  	\gzerostarenergy{N-1} + \gzerostarpartialuenergy{N-1},
   	\label{E:partialusecond} \\
 \hstarstarnorm{N-1} + \combinedpartialupartialhstarstarnorm{N-1}
  & \approx
 	\hstarstarenergy{N-2} + \partialhstarstarenergy{N-1} 
 	+ \partialhstarstarpartialuenergy{N-1},
		\label{E:partialufourth} \\
 \totalbelowtopnorm{N-1} 
 	& \approx \totalbelowtopenergy{N-1},
		\label{E:BELOWTOPORDERNORMCOMPARABLETOENERGY}
  		\\
 \totalbelowtopnorm{N-1} 
 +
 \totalbelowtopunorm{N-1} 
 	& \approx 
 		\totalbelowtopenergy{N-1} 
 		+ \guenergy{N-1},
		\label{E:BELOWTOPORDERUNORMCOMPARABLETOPRECISEENERGY}
  		\\
  \totalellipticnorm{N-1} 
 & \lesssim 
 		\totalbelowtopenergy{N-1} 
 		+ \guenergy{N-1}, 
  			\label{E:ELLIPTICNORMSBOUNDEDBYTOTALENERGY} \\
  \topvelocitynorm{N-1} 
  & \approx  \topordervelocityenergy{N-1},
  	\label{E:TOPORDERVELOCITYCOMPARABLETOTOPORDERVELOCITYENERGY} \\
  \totalnorm
  	& \approx
  	\totalenergy.
  	\label{E:TOTALNORMENERGYEQUIVALENCE}
 \end{align}
 \end{subequations}
\end{proposition} 
 
\begin{proof}
With the help of Prop. \ref{P:metric}, inequality \eqref{E:mathcalEcomparison} is not difficult to prove.
It was proved in detail as the first inequality in \cite[Prop. 10.0.1]{jS2012}, and we will not repeat the 
simple proof here. We stress that the constant $C_{(\upbeta)}$ is either $0$ or positive and that when 
it is positive, it yields control over $\| v \|_{L^2}.$ Furthermore, the positivity of $C_{(\upbeta)}$
is decided by the constants $\upgamma,\updelta$ as described in the statement of Lemma \ref{L:buildingblockmetricenergy}.
The estimates~\eqref{E:gzerozeronormequivalence}-\eqref{E:GNORMANDENEGYEQUIVALENCE}
then follow as a direct consequence
of~\eqref{E:mathcalEcomparison} and the definitions of the energies and norms in question.
Note that the energy $\hstarstarenergy{N-2}$ on the right-hand side of \eqref{E:hstarstarnormequivalence} is 
needed because of the second sum in the norm \eqref{E:HSTARSTARNORM}; this second
sum is not controlled by the energy $\partialhstarstarenergy{N-1}$ because the constant
$C_{(\upbeta)}$ is $0$ for the energies corresponding to the wave equations verified by $h_{ij}.$

To prove the bound $\lesssim$ in~\eqref{E:partialufirst}, we note that 
by \eqref{E:gzerozeronormequivalence}, it suffices to show that
$\gzerozerounorm{N-1} \lesssim$ the right-hand side of \eqref{E:partialufirst}.
To this end, we first use definitions \eqref{E:G00UNORM} and \eqref{E:enpartialu00} 
and the comparison estimate \eqref{E:mathcalEcomparison}
to deduce that $\gzerozerounorm{N-1}$ and $\gzerozeropartialuenergy{N-1}$ 
are uniformly comparable up to commutator error terms.
Specifically, we deduce that
\begin{align} \label{E:GOOUNORMVSENERGYEQUIVALENCEFIRSTBOUND}
\gzerozerounorm{N-1} 
& \lesssim \gzerozeropartialuenergy{N-1}
	+ \sum_{|\vec{\alpha}|\le N-1}
		e^{q\Omega} 
		\left\|(\g_t\u\a-\a\u\g_t)g_{00} \right\|_{L^2}
			\\
& \ \ + \sum_{|\vec{\alpha}|\le N-1}
			e^{q\Omega} 
			\left\|\{\u,\a\}g_{00} \right\|_{L^2}
			+ 
			\sum_{|\vec{\alpha}|\le N-1}
			\sum_{i=1}^3
			e^{(q-1)\Omega}
			\left\|(\g_i\u\a-\a\u\g_i) g_{00} \right\|_{L^2}.
	\notag
\end{align}
Inserting the estimate \eqref{E:equivalencecommutator1} 
into the right-hand side of \eqref{E:GOOUNORMVSENERGYEQUIVALENCEFIRSTBOUND}
and using
\eqref{E:enpartialu00}
and \eqref{E:mathcalEcomparison}, 
we deduce that $\gzerozerounorm{N-1} \lesssim \gzerozeropartialuenergy{N-1} + \gzerozeronorm{N-1}.$
The desired bound $\lesssim$ in \eqref{E:partialufirst} now follows from the
previous bound and \eqref{E:gzerozeronormequivalence}.
To prove the bound $\gtrsim$ in \eqref{E:partialufirst},
we apply the same reasoning that we used to prove 
\eqref{E:GOOUNORMVSENERGYEQUIVALENCEFIRSTBOUND}, 
but we use the commutator estimate \eqref{E:ALTERNATEequivalencecommutator1}
in place of \eqref{E:equivalencecommutator1}.

The proof of \eqref{E:partialusecond}
is essentially the same as the proof of \eqref{E:partialufirst}
and relies instead on the commutator estimates
\eqref{E:equivalencecommutator2}-\eqref{E:ALTERNATEequivalencecommutator2}. 

The proof of the bound $\gtrsim$ in \eqref{E:partialufourth}
is also essentially the same as the proof of the bound $\gtrsim$ in \eqref{E:partialufirst} and
relies on the commutator estimate \eqref{E:ALTERNATEequivalencecommutator3}.
To prove the bound $\lesssim$ in \eqref{E:partialufourth},
we first note that \eqref{E:hstarstarnormequivalence} implies that we only have to show that
$\combinedpartialupartialhstarstarnorm{N-1}$ 
is $\lesssim$ the right-hand side of \eqref{E:partialufourth}.
To this end, we examine definition \eqref{E:DERIVATIVEHSTARSTARUNORM}.
To see that the first and the fourth sums are $\lesssim$
the right-hand side of \eqref{E:partialufourth},
we use the same argument that we used to prove the bound
$\lesssim$ in \eqref{E:partialufirst}, but we use 
the commutator estimate \eqref{E:equivalencecommutator3}  
in place of \eqref{E:equivalencecommutator1}.
To bound the second sum in \eqref{E:DERIVATIVEHSTARSTARUNORM}
by $\lesssim$ the right-hand side of \eqref{E:partialufourth},
we use the definition $\u = u^{\mu} \g_{\mu},$
Prop. \ref{P:F1FkLinfinityHN},
and the estimates
\eqref{E:u0MINUSONEHNMINUSONE} 
and $\| u^a \|_{H^{N-1}} \lesssim e^{-(1 + q) \Omega}$ to deduce that
\begin{align}
	e^{q \Omega} \| \u h_{jk} \|_{H^{N-1}}
	& \lesssim 
		e^{q \Omega}
		\left\lbrace 
			1 + \|u^0 - 1\|_{H^{N-1}} 
		\right\rbrace
		\| \g_t h_{jk} \|_{H^{N-1}}
		+ e^{q \Omega}
			\sum_{a=1}^3
			\| u^a \|_{H^{N-1}}
			\| \g_a h_{jk} \|_{H^{N-1}}
				\\
	& \lesssim 
		e^{q \Omega} \| \g_t h_{jk} \|_{H^{N-1}}
		+ \sum_{a=1}^3
			e^{-\Omega} \| \g_a h_{jk} \|_{H^{N-1}}
		\lesssim \hstarstarnorm{N-1}.
		\notag
\end{align}
The desired bound now follows from \eqref{E:hstarstarnormequivalence}.
To bound the third sum in \eqref{E:DERIVATIVEHSTARSTARUNORM}
by $\lesssim$ the right-hand side of \eqref{E:partialufourth},
we use essentially the same proof. We have thus proved \eqref{E:partialufourth}.

To prove \eqref{E:BELOWTOPORDERNORMCOMPARABLETOENERGY},
we first refer to Defs.
\ref{D:TOTALBELOWTOPSOLUTIONNORM},
\ref{D:TOTALMETRICENERGIES},
and
\ref{D:TOTALBELOWTOPSOLUTIONENERGY}.
Using the already proven estimates 
\eqref{E:partialufirst}-\eqref{E:partialufourth},
we immediately conclude the desired result.

To prove \eqref{E:BELOWTOPORDERUNORMCOMPARABLETOPRECISEENERGY},
we first refer to Defs. 
\ref{D:TOTALBELOWTOPSOLUTIONNORM},
\ref{D:TOTALMETRICENERGIES},
and
\ref{D:TOTALBELOWTOPSOLUTIONENERGY}.
Using the already proven estimates 
\eqref{E:partialufirst}-\eqref{E:BELOWTOPORDERNORMCOMPARABLETOENERGY},
we deduce that the desired bound \eqref{E:BELOWTOPORDERUNORMCOMPARABLETOPRECISEENERGY}
will follow once we prove
\begin{align}  \label{E:FLUIDUNORMSBOUNDEDBYENERGY}
\sum_{j=1}^3 e^{(1 + q) \Omega} \| \u u^j \|_{H^{N-1}}
	+ e^{q\Omega} \|\partial_{\mathbf{u}} \dens \|_{H^{N-1}}
	 	& \lesssim 
	 	\totalbelowtopenergy{N-1}.
\end{align}	 		
To prove the desired estimate \eqref{E:FLUIDUNORMSBOUNDEDBYENERGY}
for $e^{(1 + q) \Omega} \|\u u^j\|_{H^{N-1}},$
we first use Eq. \eqref{E:velevol} and Prop. \ref{P:derivativesofF1FkL2} to deduce that
\begin{align} \label{E:PARTIALUUJHNMINUSONEFIRSTBOUND}
\|\u u^j\|_{H^{N-1}}
& \lesssim (1 + \|u^0 - 1 \|_{H^{N-1}}) \| u^j \|_{H^{N-1}} 
+ (1 + \|u^0 - 1 \|_{H^{N-1}}) \| \triangle^{\ j}_{0 \ 0}\|_{H^{N-1}} 
+ \|\triangle^j \|_{H^{N-1}}.
\end{align}
Inserting the estimates
\eqref{E:u0MINUSONEHNMINUSONE},
\eqref{E:VERYIMPORTANTCHRISTOFFELSYMBOLERRORESTIMATE},
and
\eqref{E:FLUIDTRIANGLEJHNMINUSONE}
into the right-hand side of \eqref{E:PARTIALUUJHNMINUSONEFIRSTBOUND},
we deduce that the right-hand side of \eqref{E:PARTIALUUJHNMINUSONEFIRSTBOUND}
is $\lesssim e^{-(1+q) \Omega} \totalbelowtopnorm{N-1}.$
The desired bound $e^{(1+q) \Omega} \|\g_{\bf u} u^j \|_{H^{N-1}} \lesssim \totalbelowtopenergy{N-1}$
thus follows from the previously proven bound \eqref{E:BELOWTOPORDERNORMCOMPARABLETOENERGY}.
To prove the desired estimate \eqref{E:FLUIDUNORMSBOUNDEDBYENERGY} for 
$e^{q\Omega} \|\partial_{\mathbf{u}} \dens \|_{H^{N-1}},$ 
we first take the $H^{N-1}$ norm of the right-hand side of \eqref{E:revol}
and use the estimate \eqref{E:FLUIDTRIANGLEHNMINUSONE} to deduce that
$e^{q\Omega} \|\partial_{\mathbf{u}} \dens \|_{H^{N-1}} \lesssim \totalbelowtopnorm{N-1}.$
The desired bound now follows from \eqref{E:BELOWTOPORDERNORMCOMPARABLETOENERGY}.

We now prove \eqref{E:ELLIPTICNORMSBOUNDEDBYTOTALENERGY}. 
We recall that 
$	\totalellipticnorm{N-1}  
= \gzerozeroellipticnorm{N-1} 
  	+ \gzerostarellipticnorm{N-1} 
  + \hstarstarellipticnorm{N-1}.$
We will show only how to prove
$\gzerostarellipticnorm{N-1} \lesssim \totalbelowtopenergy{N-1} + \guenergy{N-1};$ 
the bounds for 
$\gzerozeroellipticnorm{N-1}$ and 
$\hstarstarellipticnorm{N-1}$
can be proved in a similar fashion.
In view of definition \eqref{E:G0STARELLIPTICNORM}, we see that
we have to prove that
$\sum_{a,b=1}^3 e^{(q-2)\Omega} \| \g_a \g_t g_{0b}\|_{H^{N-1}} 
\lesssim 
\totalbelowtopenergy{N-1} 
+ \guenergy{N-1}$
and
that 
$\sum_{a,b,j=1}^3 e^{(q-3)\Omega}\| \g_a \g_b g_{0j}\|_{H^{N-1}} 
\lesssim 
\totalbelowtopenergy{N-1} 
+ \guenergy{N-1}.$
We first bound the second sum. 
We begin by using the elliptic regularity estimate \eqref{E:ell} to deduce that
\begin{align} \label{E:G0JELLIPTICBOUNDED}
\sum_{a,b,j=1}^3 e^{(q-3)\Omega}\| \g_a \g_b g_{0j}\|_{H^{N-1}}
& \lesssim 
\sum_{|\vec{\alpha}| \leq N-1}
\sum_{j=1}^3	
	e^{(q - 1)\Omega} \|\hat{\square}_g \a g_{0j}\|_{L^2}
+ \sum_{|\vec{\alpha}| \leq N-1}
	\sum_{j=1}^3
	e^{(q - 1)\Omega}  \|\g_t \u \a g_{0j}\|_{L^2}
		\\
& \ \ 
	+  
	\sum_{|\vec{\alpha}| \leq N - 1} 
	\sum_{i,j=1}^3
	e^{- 2 \Omega} \|\g_i \u \a  g_{0j}\|_{L^2}
	+ \sum_{j=1}^3
	 	e^{-\Omega}  \| \g_t g_{0j} \|_{H^{N-1}}
	 		\notag \\
& \ \ + \sum_{i,j=1}^3
	 	e^{-2\Omega}  \| \g_i g_{0j} \|_{H^{N-1}}.	
	\notag
\end{align}
From the commutator estimate \eqref{E:ALTERNATEequivalencecommutator2} 
and \eqref{E:partialusecond},
we deduce that the last five sums on the right-hand side of
\eqref{E:G0JELLIPTICBOUNDED} are $\lesssim \gzerostarenergy{N-1} + \gzerostarpartialuenergy{N-1}$ as desired.
To show that the remaining (first) term verifies
$e^{(q - 1)\Omega} \|\hat{\square}_g \a g_{0j}\|_{L^2}$
$\lesssim 
\totalbelowtopenergy{N-1} 
+ \guenergy{N-1}$
whenever $|\vec{\alpha}| \leq N-1,$ 
we first apply $\a$ to \eqref{E:metric0j},
commute $\hat{\square}_g$ and $\a,$ 
and take the $L^2$-norm to deduce that 
\begin{align} \label{E:ANNOYINGBOXPARTIALALPHAL2COMMUTATORESTIMATE}
e^{(q - 1)\Omega} \|\hat{\square}_g\a g_{0j}\|_{L^2}
&\lesssim 
e^{(q - 1)\Omega} \|\{\hat{\square}_g,\a\} g_{0j}\|_{L^2}
+ e^{(q - 1)\Omega} \|\g_tg_{0j} \|_{H^{N-1}}
+ e^{(q - 1)\Omega} \| g_{0j}\|_{H^{N-1}}
	\\
& \ \ 
+  e^{(q - 1)\Omega} \left\| g^{ab}\Gamma_{ajb} \right\|_{H^{N-1}}
+ e^{(q - 1)\Omega} \|\triangle_{0j} \|_{H^{N-1}}.
\notag
\end{align}
From \eqref{E:gabupperGammaajblower} 
and the already proven estimate 
\eqref{E:BELOWTOPORDERNORMCOMPARABLETOENERGY},
we conclude that the second, third, and fourth terms on the right-hand side of
\eqref{E:ANNOYINGBOXPARTIALALPHAL2COMMUTATORESTIMATE}
are
$\lesssim 
\totalbelowtopenergy{N-1}$ as desired.
To bound the first and last term on the right-hand side of
\eqref{E:ANNOYINGBOXPARTIALALPHAL2COMMUTATORESTIMATE},
we use 
\eqref{E:triangleTWODOWNHNMINUSONE},
\eqref{E:gcommutatorL2},
and the already proven estimate 
\eqref{E:BELOWTOPORDERNORMCOMPARABLETOENERGY}.
We have thus shown that $\sum_{a,b,j=1}^3 e^{(q-3)\Omega}\| \g_a \g_b g_{0j}\|_{H^{N-1}} 
\lesssim 
\totalbelowtopenergy{N-1} 
+ \guenergy{N-1}$
as desired. It remains for us to show that
$\sum_{a,b=1}^3 e^{(q-2) \Omega} \| \g_a \g_t g_{0b}\|_{H^{N-1}} 
\lesssim 
\totalbelowtopenergy{N-1} 
+ \guenergy{N-1}.$
To this end, we first use the decomposition $\g_t = \frac{1}{u^0}\u - \frac{u^a}{u^0} \g_a,$
Prop. \ref{P:F1FkLinfinityHN},
and Cor. \ref{C:SobolevTaylor}
to deduce that
\begin{align} \label{E:PARTIALIPARTIALTG0JHNMINUSONEFIRSTBOUND}
	e^{(q-2)\Omega} \| \g_i \g_t g_{0j} \|_{H^{N-1}}
	& \lesssim e^{(q-2)\Omega} (1 + \| u^0 - 1 \|_{H^{N-1}}) \| \u \g_i g_{0j} \|_{H^{N-1}}
		\\
	& \ \ + 
		e^{(q-2)\Omega} (1 + \| u^0 - 1 \|_{H^{N-1}}) \sum_{a,b=1}^3 \| u^a \|_{H^{N-1}} \| \g_a \g_b g_{0j} \|_{H^{N-1}}.
			\notag
\end{align}
Using the estimate 
\eqref{E:u0MINUSONEHNMINUSONE}
and the bound $\| u^a \|_{H^{N-1}} \lesssim e^{-(1 + q)\Omega} \totalbelowtopnorm{N-1} \lesssim e^{-(1 + q)\Omega},$
we deduce that the right-hand side of
\eqref{E:PARTIALIPARTIALTG0JHNMINUSONEFIRSTBOUND}
is $\lesssim \gzerostarunorm{N-1}$
plus $e^{-3 \Omega} \sum_{a,b=1}^3 \| \g_a \g_b g_{0j} \|_{H^{N-1}}.$
Hence, using \eqref{E:partialusecond} and 
the bound for $\sum_{a,b=1}^3 \| \g_a \g_b g_{0j} \|_{H^{N-1}}$
proved just above, we deduce that the right-hand side of
\eqref{E:PARTIALIPARTIALTG0JHNMINUSONEFIRSTBOUND}
is $\lesssim 
\totalbelowtopenergy{N-1} 
+ \guenergy{N-1}.$ Summing this inequality over $i$ and $j,$
we conclude the desired result. 

Next, we note that estimate \eqref{E:TOPORDERVELOCITYCOMPARABLETOTOPORDERVELOCITYENERGY}
follows trivially from definitions 
\eqref{E:VELOCITYNORMTOPORDER}
and
\eqref{E:TOPORDERVELOCITYENERGYDEF}.

Finally, we prove \eqref{E:TOTALNORMENERGYEQUIVALENCE}.
We will prove only the bound $\totalnorm \lesssim \totalenergy$
in detail. The reverse inequality can be proved using a similar argument and
is in fact much easier to prove because it does not require the use of elliptic 
estimates.
To show that $\totalnorm \lesssim \totalenergy,$
we first appeal to Defs. \ref{D:TOTALSOLUTIONNORM} and \ref{D:TOTALENERGY}.
We then simply combine the bounds
\eqref{E:BELOWTOPORDERUNORMCOMPARABLETOPRECISEENERGY},
\eqref{E:ELLIPTICNORMSBOUNDEDBYTOTALENERGY},
and
\eqref{E:TOPORDERVELOCITYCOMPARABLETOTOPORDERVELOCITYENERGY},
and the desired estimate follows.
\end{proof}
\section{Global existence and future-causal geodesic completeness}\label{S:global}
In this section, we state and prove our main future stability theorem (see Sect.~\ref{SS:maintheorem}).  
As a crucial preparatory step, we first derive a system of integral inequalities for the energies.

\subsection{Integral inequalities for the energies}
We now use Lemma~\ref{L:metricfirstdifferentialenergyinequality}, Lemma~\ref{L:fluidfirstdifferentialenergyinequality},
the error term estimates of Sect.~\ref{S:sobolev},
and Prop. \ref{P:energynormcomparison} 
to derive our 
main energy integral inequalities. The main point is that the inequalities 
are amenable to an intricate hierarchy of Gronwall-type estimates 
that will allow us to improve our Sobolev norm bootstrap assumption \eqref{E:bootstrap}. 
These estimates are the main ingredient in the proof of our main future stability theorem 
(which is provided in Sect.~\ref{SS:maintheorem}).
\begin{proposition}[\textbf{Integral inequalities for the fluid energies}]\label{P:integralinequalities}
Let 
$\velocityenergy{N-1},$ 
$\topordervelocityenergy{N-1},$
and $\densenergy{N-1}$ be the fluid energies from Def.~\ref{D:FLUIDENERGIES},
let 
$\genergy{N-1}$
and
$\guenergy{N-1}$
be the aggregate metric energies from Def. 
\ref{D:TOTALMETRICENERGIES},
let $\totalbelowtopenergy{N-1}$ be the aggregate metric$+$fluid energy from
Def. \ref{D:TOTALBELOWTOPSOLUTIONENERGY},
and let $\totalenergy$ be the total solution energy from Def. \ref{D:TOTALENERGY}.
Let $q$ be the small positive constant defined in~\eqref{E:qdef}.
Assume that the hypotheses of Prop.~\ref{P:Sobolev}, including the smallness assumption~\eqref{E:bootstrap},
hold on the spacetime slab $[0,T) \times \mathbb{T}^3.$ Then there exists 
a large constant $C > 1$ such that 
if $\epsilon$ is sufficiently small,
then the following integral inequalities hold for $0 \leq t_1 \leq t < T:$
\begin{subequations}
 \begin{align}
\velocityenergy{N-1}^2(t) 
& \leq 
	\velocityenergy{N-1}^2(t_1) 
 	+	\int_{t_1}^t  
 				C
 				e^{- q H \tau}
 				\totalenergy^2(\tau)
 		\, d \tau,	
 			\label{E:intun-1}
 			\\ 
 	& \ \ 
 		+
 		\int_{t_1}^t 
 			\Big\lbrace
 				\underbrace{- 2(1 - q) H }_{< 0} \velocityenergy{N-1}^2(\tau) 
 				+ C \underbrace{\genergy{N-1}(\tau)}_{\mbox{\textnormal{dangerous}}}
 				\velocityenergy{N-1}(\tau) 
 			\Big\rbrace
 		\, d\tau,
 			\notag \\		
 \topordervelocityenergy{N-1}^2(t) 
& \leq 
	\topordervelocityenergy{N-1}^2(t_1) 
 	+	\int_{t_1}^t  
 				C
 				e^{- q H \tau}
 				\totalenergy^2(\tau)
 		\, d \tau,	
 			\label{E:TOPORDERintun-1}
 			\\ 
 	& \ \ 
 		+
 		\int_{t_1}^t 
 			\Big\lbrace
 				\underbrace{-2(2 - q)H}_{< 0} 
 				 \topordervelocityenergy{N-1}^2(\tau) 
 				+ C \underbrace{\left[
 				\totalbelowtopenergy{N-1}(\tau) 
 				+ \guenergy{N-1}(\tau)\right]}_{\mbox{\textnormal{dangerous}}}
 				\topordervelocityenergy{N-1}(\tau) 
 			\Big\rbrace
 		\, d\tau,
 			\notag \\		
\densenergy{N-1}^2(t)
& \leq \densenergy{N-1}^2(t_1)  
		+ \int_{t_1}^t  
				C
				e^{-q H \tau}
				\totalenergy^2(\tau) 
			\,d \tau. \label{E:intun} 
\end{align}
\end{subequations}
\end{proposition}

\begin{remark}
	In \cite{jS2012}, because of a typographical error that morphed into an actual error, 
	the second author overlooked the presence of dangerous linear terms
	present in the energy estimates for the lower-order derivatives of the fluid velocity.
	The term that should have been
	included is analogous to the integrand term
	$C \genergy{N-1}(\tau) \velocityenergy{N-1}(\tau)$
 	in \eqref{E:intun-1}. The argument in present article 
 	corrects this oversight and shows that it is not difficult to 
 	completely fix the error in \cite{jS2012}. 
 	In particular, the main results of \cite{jS2012} are true exactly 
 	as stated there. The main idea of the fix is to exploit the effective
 	decoupling and to derive suitable estimates for the dangerous factor
 	$\genergy{N-1}(\tau)$
 	before estimating $\velocityenergy{N-1}(\tau).$ We carry this out in complete detail
 	in our proof of Prop. \ref{P:auxiliary} below.
 \end{remark}

\begin{proof}
Throughout this proof, we freely use the comparison estimates of Prop. \ref{P:energynormcomparison} without 
explicitly mentioning it each time. To prove~\eqref{E:intun-1}, we integrate \eqref{E:VELOCITYDIFFINEQ} in time over the interval $[t_1,t].$
Using Lemma \ref{L:backgroundaoftestimate} to replace $\omega$ with $H$ up to an
error term bounded by $ \int_{t_1}^t  
 				C
 				e^{- q H \tau}
 				\totalenergy^2(\tau)
 			\,d\tau,$
we see that the first term on the right-hand side of \eqref{E:VELOCITYDIFFINEQ} 
yields the negative definite integral
$- \int 2(1 - q) \cdots$ on the right-hand side of \eqref{E:intun-1}. 
To handle the integral on the second line of~\eqref{E:VELOCITYDIFFINEQ}, we use the Cauchy-Schwarz inequality,
the estimate
$e^{(1 + q)\Omega} \| u^j \|_{H^{N-1}} \lesssim \velocityenergy{N-1}$
(which follows trivially from the definition of $\velocityenergy{N-1}$),
and the estimate 
\eqref{E:VERYIMPORTANTCHRISTOFFELSYMBOLERRORESTIMATE}
to deduce that it is bounded by the dangerous integral
$\int_{t_1}^t C \genergy{N-1}(\tau) \velocityenergy{N-1}(\tau) \,d\tau $ 
appearing on the right-hand side of~\eqref{E:intun-1}.
The former integral is dangerous because it lacks an exponentially decaying factor in the integrand. 
We emphasize that in proving this bound,  
we have used \eqref{E:GNORMANDENEGYEQUIVALENCE}.
All of the remaining integrals 
on the right-hand side of \eqref{E:VELOCITYDIFFINEQ}
are bounded by the
non-dangerous integral $C\int_{t_1}^t e^{-q H \tau} \totalenergy^2(\tau) \,d\tau$
on the right-hand side of \eqref{E:intun-1}.
To see that this is the case, we use the Cauchy-Schwarz inequality 
(in the form 
	$|\int_{\mathbb{T}^3} v_1 v_2 v_3 \, dx| \lesssim \| v_1 \|_{L^{\infty}} \| v_2 \|_{L^2} \| v_3 \|_{L^2}$
for the integrals on the last line),
the estimates
\eqref{E:FLUIDTRIANGLEJHNMINUSONE}
and
\eqref{E:UCOMMUTATORBOUND}
as well as the estimate
$\left\| \partial_a \left(\frac{u^a}{u^0} \right) \right\|_{L^{\infty}} 
\lesssim \sum_{a=1}^3 (1 + \| u^0 - 1 \|_{H^{N-1}}) \| u^a \|_{H^{N-1}} 
\lesssim e^{-(1 + q) \Omega},$
which follows easily with the help of \eqref{E:u0MINUSONEHNMINUSONE}.
We have thus proved \eqref{E:intun-1}.

The proof of \eqref{E:TOPORDERintun-1} is very similar and is based on 
the relation \eqref{E:TOPORDERVELOCITYDIFFINEQ}.
We remark that in place of the estimates
\eqref{E:VERYIMPORTANTCHRISTOFFELSYMBOLERRORESTIMATE},
$e^{(1 + q)\Omega} \| u^j \|_{H^{N-1}} \lesssim \velocityenergy{N-1},$
and
\eqref{E:FLUIDTRIANGLEJHNMINUSONE} used above, we use
\eqref{E:PARTIALIVERYIMPORTANTCHRISTOFFELSYMBOLERRORESTIMATE},
$e^{q\Omega} \| \g_i u^j \|_{H^{N-1}} \lesssim \topordervelocityenergy{N-1},$
and
\eqref{E:PARTIALIFLUIDTRIANGLEJHNMINUSONE}.
Furthermore, in deriving the structure of the dangerous integrands,
which arise from the estimate \eqref{E:PARTIALIVERYIMPORTANTCHRISTOFFELSYMBOLERRORESTIMATE},
we use \eqref{E:ELLIPTICNORMSBOUNDEDBYTOTALENERGY}
in addition to \eqref{E:GNORMANDENEGYEQUIVALENCE}.

The bound~\eqref{E:intun} is much simpler to prove. 
Specifically, it follows from
\eqref{E:DENSITYDIFFINEQ},
the Cauchy-Schwarz inequality,
the bound $\left\| \partial_a \left(\frac{u^a}{u^0} \right) \right\|_{L^{\infty}} \lesssim e^{-(1 + q) \Omega}$
mentioned above,
and the error estimates
\eqref{E:FLUIDTRIANGLEHNMINUSONE} 
and
\eqref{E:RHOCOMMUTATORBOUND}.

\end{proof}
\begin{proposition}[\textbf{Integral inequalities for the metric energies}]\label{P:integralinequalitiesmetric}
Assume that the hypotheses of Prop.~\ref{P:Sobolev}, including the smallness assumption~\eqref{E:bootstrap},
hold on the spacetime slab $[0,T) \times \mathbb{T}^3.$ Let $\totalenergy$ be the total solution energy from Def. \ref{D:TOTALENERGY}.
Let $q$ be the small positive constant defined in~\eqref{E:qdef}. Then there exists a constant $C > 0$ such 
that if $\epsilon$ is sufficiently small,
then the following integral inequalities hold for the energies 
$\gzerozeroenergy{N-1}$ and $\gzerozeropartialuenergy{N-1}$ from Def.~\ref{D:energiesforg}
whenever $0 \leq t_1 \leq t < T:$
\begin{subequations}
\begin{align}
\gzerozeroenergy{N-1}^2(t) 
& \leq 
	\gzerozeroenergy{N-1}^2(t_1) \label{E:mathfrakENg00integral} 
	+ \int_{t_1}^t  
			  C  e^{-q H \tau} \totalenergy^2(\tau) 
		\, d \tau, 
		\\
\gzerozeropartialuenergy{N-1}^2(t) 
& \le \gzerozeropartialuenergy{N-1}^2(t_1)
	+ \int_{t_1}^t  
			  C  e^{-q H \tau} \totalenergy^2(\tau) 
		\, d \tau
	\label{E:mathfrakENg00integralpartialu} \\
& \ \ 
	+\int_{t_1}^t
		\Big\lbrace
			\underbrace{- 4qH }_{< 0}\gzerozeropartialuenergy{N-1}^2(\tau) 
			+
			C \underbrace{\gzerozeroenergy{N-1}(\tau)}_{\mbox{\textnormal{dangerous}}} \gzerozeropartialuenergy{N-1}(\tau) 
		\Big\rbrace
\, d\tau.  \notag
\end{align}
\end{subequations}

Furthermore, the following integral inequalities hold for 
the energies $\gzerostarenergy{N-1}$ and $\gzerostarpartialuenergy{N-1}$ from Def.~\ref{D:energiesforg}:
\begin{subequations}
\begin{align}
		\gzerostarenergy{N-1}^2(t)
		& \le \gzerostarenergy{N-1}^2(t_1) 
		 + \int_{t_1}^t 
				C e^{-q H \tau} \totalenergy^2(\tau) 		
			\, d \tau
			\label{E:mathfrakENg0*integral} \\
		& 	+ \int_{t_1}^t 
							\underbrace{- 4qH}_{< 0} \gzerostarenergy{N-1}^2(\tau)
					\, d \tau,
					\notag \\
		& 	+ \int_{t_1}^t 
						C \underbrace{
							\Big[
								\partialhstarstarenergy{N-1}(\tau) + \hstarstarenergy{N-2}(\tau)
							\Big]
							}_{\mbox{\textnormal{dangerous}}}
							\gzerostarenergy{N-1}(\tau)
					 \, d \tau,
					\notag \\
\gzerostarpartialuenergy{N-1}^2(t) \label{E:mathfrakENg0*integralpartialu}
& \le \gzerostarpartialuenergy{N-1}^2(t_1)  
			+ \int_{t_1}^t C e^{-q H \tau} \totalenergy^2(\tau)  \, d \tau
			    \\
& + \int_{t_1}^t 
						\underbrace{- 4qH}_{< 0}  \gzerostarpartialuenergy{N-1}^2(\tau)
			\, d \tau
			\notag \\
& + \int_{t_1}^t 
						C
						\underbrace{
						\Big[
						\partialhstarstarpartialuenergy{N}(\tau)
						+ \partialhstarstarenergy{N-1}(\tau)
						+ \hstarstarenergy{N-2}
						+\gzerostarenergy{N-1}(\tau)
					\Big]}_{\mbox{\textnormal{dangerous}}}
			\notag \\				
&    	\ \ \ \ \ \ \ \ \ \ \ \ \ 
			\times		
			\gzerostarpartialuenergy{N-1}(\tau)
			\, d \tau.
			\notag
\end{align}
\end{subequations}

Finally, the following integral inequalities hold for the metric energies 
$\hstarstarenergy{N-2},$ 
$\partialhstarstarenergy{N-1},$ 
and $\partialhstarstarpartialuenergy{N-1}$ from Def.~\ref{D:energiesforg}:
\begin{subequations}
\begin{align}
			\hstarstarenergy{N-2}^2(t) & \leq \hstarstarenergy{N-2}^2(t_1)	
			+ \int_{t_1}^t 
					C e^{-q H \tau}\totalenergy^2(\tau) 
				\, d \tau, 
			\label{E:ENh**integral} \\
		\partialhstarstarenergy{N-1}^2(t) & \leq \partialhstarstarenergy{N-1}^2(t_1)	
			+ \int_{t_1}^t 
					C e^{-q H \tau} \totalenergy^2(\tau) 
				\, d \tau, 
			\label{E:ENpartialh**integral} \\
		\partialhstarstarpartialuenergy{N-1}^2(t)
		& \le \partialhstarstarpartialuenergy{N-1}^2(t_1) 
					+ \int_{t_1}^t 
					 		C e^{-q H \tau} \totalenergy^2(\tau)
						\, d \tau 
					\label{E:ENpartialh**integralpartialu} \\
		& \ \ +
					\int_{t_1}^t
						\Big\lbrace
							\underbrace{- 4qH}_{< 0} \partialhstarstarpartialuenergy{N-1}^2(\tau)
							+ C \underbrace{\partialhstarstarenergy{N-1}(\tau)}_{\mbox{\textnormal{dangerous}}} \partialhstarstarpartialuenergy{N-1}(\tau)
						\Big\rbrace
					\,d\tau. 
					\notag
			 \end{align}
	\end{subequations}
	\end{proposition}
\begin{proof}
We stress that throughout this proof, 
\emph{we freely use the comparison estimates of Prop. \ref{P:energynormcomparison} without 
explicitly mentioning it each time.}

\noindent
{\em Proof of~\eqref{E:mathfrakENg00integral}-\eqref{E:mathfrakENg00integralpartialu}.}
We first prove~\eqref{E:mathfrakENg00integralpartialu}.
We integrate inequality~\eqref{E:partialu00est} in time over the interval $[t_1,t]$ 
and apply Cauchy-Schwarz.
We note that
$
(2q-\tilde{\eta}_{00})H \le -4qH
$
by our choice of $q$ [see~\eqref{E:qdef}]
and that by Lemma \ref{L:backgroundaoftestimate},
$|2q(\omega-H)\gzerozeroenergy{N-1}^2| \lesssim e^{-qHt} \totalenergy^2.$
Also using the estimates
$\sum_{|\vec{\alpha}| \leq N - 1} e^{q \Omega} \| \g_t \u \a g_{00} \|_{L^2} 
\lesssim \gzerozeropartialuenergy{N-1}$
and
$\sum_{|\vec{\alpha}| \leq N - 1} e^{q \Omega} \| \u \a g_{00} \|_{L^2} 
\lesssim \gzerozeropartialuenergy{N-1},$
which follow from \eqref{E:mathcalEcomparison} and 
the definition \eqref{E:enpartialu00} of $\gzerozeropartialuenergy{N-1},$
we deduce that
\begin{align*}
& \gzerozeropartialuenergy{N-1}^2(t) 
\le \gzerozeropartialuenergy{N-1}^2(t_1) 
		\\
& \ \ -\int_{t_1}^t 
				4qH \gzerozeropartialuenergy{N-1}^2
				\, d\tau 
			+ \int_{t_1}^t 
					C e^{-q H \tau} \totalenergy^2
				\, d \tau 
					\\
& \ \ +  
			\sum_{|\vec{\alpha}| \leq N-1}
			\int_{t_1}^t C \gzerozeropartialuenergy{N-1} 
			\Big(\underbrace{e^{q\Omega} \|\g_t \a g_{00} \|_{L^2} + e^{q\Omega} \| g_{00} + 1 \|_{L^2}}_{\mbox{\textnormal{dangerous}}}\Big) 
			\, d\tau
			\\
& \ \ +
		\sum_{|\vec{\alpha}| \leq N-1}
		\int_{t_1}^t C \gzerozeropartialuenergy{N-1}
			\Big(e^{q\Omega} \left\|\a \u \triangle_{00} \right\|_{L^2} + e^{q\Omega} \left\|\a \triangle_{00} \right\|_{L^2}\Big) 
		\, d \tau
			\\
& \ \ + 
		\sum_{|\vec{\alpha}| \leq N-1}
		\int_{t_1}^t 
			C \gzerozeropartialuenergy{N-1}	
				\Big(e^{q\Omega} \left\|\{\u,\g_t\}\a g_{00} \right\|_{L^2} 
					+ e^{q\Omega} \left\| \{\u, \a\} \triangle_{00} \right\|_{L^2}
			 	\Big) 
		\, d\tau \\
& \ \ + 
		\sum_{|\vec{\alpha}| \leq N-1}
		\int_{t_1}^t 
			C \gzerozeropartialuenergy{N-1}	
				\Big(e^{q\Omega}  \left\|\u\{\hat{\Box}_g,\a\}g_{00} \right\|_{L^2} 
					+ e^{q\Omega} \left\|\triangle_{\text{Ell}}[\partial^{(2)}\a g_{00}] \right\|_{L^2} 
				\Big) 
		\, d\tau \\
& \ \ + 
		\sum_{|\vec{\alpha}| \leq N-1}
		\int_{t_1}^t 
			C e^{2q\Omega} 	
			\left\|\triangle_{\E;(\widetilde{\upgamma}_{00}, \widetilde{\updelta}_{00})}[\u\a g_{00},\g(\u\a g_{00})] \right\|_{L^1} 
		\, d\tau.
\end{align*}
The two ``dangerous" terms on the third line above are bounded by 
$ \gzerozeronorm{N-1} \lesssim \gzerozeroenergy{N-1}$ 
and hence the corresponding quadratic integral
is bounded by the dangerous integral
$\int_{t_1}^t C \gzerozeroenergy{N-1}(\tau)\, \gzerozeropartialuenergy{N-1}(\tau)\,d\tau$ 
on the right-hand side of~\eqref{E:mathfrakENg00integralpartialu}.
As we noted in our proof of \eqref{E:intun-1},
this integral is dangerous because it lacks an exponentially decaying factor in the integrand.
The remaining integrals are not dangerous and are in fact bounded
by the integral on the right-hand side of \eqref{E:mathfrakENg00integralpartialu}
involving the integrand $e^{-q H \tau} \totalenergy^2(\tau),$
which contains the stabilizing factor $e^{-q H \tau}.$
To see this, we use the estimates
\eqref{E:triangleTWODOWNHNMINUSONE},
\eqref{E:PARTIALUtriangleTWODOWNHNMINUSONE},
\eqref{E:COMMUTATORPARTIALTPARTIALUMETRICHNMINUONE},
\eqref{E:COMMUTATORPARTIALUtriangleTWODOWNHNMINUSONE},
\eqref{E:gpartialucommutatorL2},
\eqref{E:partialutriangleEgamma00delta00L1},
and
\eqref{E:ELLIPTICERRORL2}.
We have thus proved \eqref{E:mathfrakENg00integralpartialu}. 

The proof of~\eqref{E:mathfrakENg00integral} is completely analogous to the proof of~\eqref{E:mathfrakENg00integralpartialu} 
and relies instead on the differential
inequality~\eqref{E:00est} 
and the error term estimates
\eqref{E:triangleTWODOWNHNMINUSONE},
\eqref{E:gcommutatorL2},
and \eqref{E:triangleEgamma00delta00L1}.
In particular, the estimates are less delicate because {\em no} 
``dangerous linear terms" are present 
and hence {\em no} cross term of the type 
$\int_{t_1}^t \gzerozeroenergy{N-1}(\tau)\, \gzerozeropartialuenergy{N-1}(\tau)\,d\tau$ 
appears on the right-hand side of~\eqref{E:mathfrakENg00integral}.

\noindent
{\em Proof of~\eqref{E:mathfrakENg0*integral}-\eqref{E:mathfrakENg0*integralpartialu}.}
To prove \eqref{E:mathfrakENg0*integral}, we first 
integrate \eqref{E:0*est} in time over $[t_1,t],$ use the definition of $q,$ 
and argue as above to obtain
\begin{align*}
\gzerostarenergy{N-1}^2(t) 
& \le \gzerostarenergy{N-1}^2(t_1) 
	\\
& \ \
	-  \int_{t_1}^t 
		4 q H \gzerostarenergy{N-1}^2(\tau)
	\,d\tau 
	+ \int_{t_1}^t 
			C e^{-q H \tau} \totalenergy^2
		\, d \tau 
			\\
& \ \ 
	+ C\int_{t_1}^t \gzerostarenergy{N-1} \underbrace{e^{(q-1)\Omega}\|g^{ab}\Gamma_{ajb}\|_{H^{N-1}}}_{\mbox{\textnormal{dangerous}}}\,d\tau 
	\\
& \ \ + C \sum_{|\vec{\alpha}|\leq N - 1} 
				\int_{t_1}^t 
					\gzerostarenergy{N-1}
					\Big(e^{(q-1)\Omega} \left\|\a\triangle_{0j} \right\|_{L^2}
					+e^{(q-1)\Omega} \left\|\{\hat{\Box}_g,\a\}g_{0j} \right\|_{L^2}
					\Big)
				\,d\tau \\
& \ \ 
+ C \sum_{|\vec{\alpha}|\leq N - 1} 
		\int_{t_1}^t 
			e^{2(q-1)\Omega} \left\|\triangle_{\E;(\upgamma_{0*},\updelta_{0*})}[\a g_{0j}, \g(\a g_{0j})] \right\|_{L^1}
		\,d\tau.
\end{align*}
Using the estimates
\eqref{E:triangleTWODOWNHNMINUSONE},
\eqref{E:gcommutatorL2},
and
\eqref{E:triangleEgamma0jdelta0jL1},
we can bound the third and fourth lines above 
by the non-dangerous integral $C\int_{t_1}^t e^{-q H \tau} \totalenergy^2(\tau) \,d\tau$ as desired.
Next, using the estimate \eqref{E:gabupperGammaajblower},
we bound the integral involving the ``dangerous linear term" on the second line above
by the dangerous integral $\int_{t_1}^t C \left[\partialhstarstarenergy{N-1}(\tau) + \hstarstarenergy{N-2}(\tau)\right]
\gzerostarenergy{N-1}(\tau) \, d \tau$
on the right-hand side of~\eqref{E:mathfrakENg0*integral}. 

The proof of~\eqref{E:mathfrakENg0*integralpartialu} is based on inequality \eqref{E:partialu0*est}
and is very similar to the previous proof and the proof of~\eqref{E:mathfrakENg00integralpartialu}. We omit the full details
and instead simply highlight the fact that
the ``dangerous linear terms" on the right-hand side of the estimate~\eqref{E:partialu0*est} are precisely the source of the
integral
\[
\int_{t_1}^t C 
\Big\{ 
\partialhstarstarpartialuenergy{N}(\tau)
+ \partialhstarstarenergy{N-1}(\tau)
+ \hstarstarenergy{N-2}
+\gzerostarenergy{N-1}(\tau)
\Big\}\,\gzerostarpartialuenergy{N-1}(\tau) \,d\tau 
\]
on the right-hand side of~\eqref{E:mathfrakENg0*integralpartialu}.

\noindent
{\em Proof of~\eqref{E:ENh**integral}-\eqref{E:ENpartialh**integralpartialu}.}
Inequality~\eqref{E:ENh**integral} follows easily from integrating 
the differential inequality~\eqref{**est}, using the Cauchy-Schwarz inequality,
and the estimate $\| \g_t h_{jk}\|_{H^{N-1}} \lesssim e^{-qHt} \hstarstarnorm{N-1} \lesssim e^{-qHt}\totalenergy.$


To prove~\eqref{E:ENpartialh**integral}, we time integrate the differential inequality~\eqref{E:partial**est}. We then use the bound $2q-\upeta_{**}<0,$ 
which follows from our choice
of $q$, as well as $|\omega-H|\lesssim e^{-qHt}$. To bound the remaining
integrals, we apply the Cauchy-Schwarz inequality 
and use the error term estimates
\eqref{E:triangleTWODOWNHNMINUSONE},
\eqref{E:gcommutatorL2}, 
and
\eqref{E:triangleEgamma**delta**L1}.

Finally, to prove~\eqref{E:ENpartialh**integralpartialu}, 
we integrate the differential inequality~\eqref{E:partialpartialu**est} over the time interval $[t_1,t]$. 
As above, we use the estimate
$2q-\upeta_{**}<-4q$ to infer that
\[
\int_{t_1}^t(2q-\upeta_{**}) H \partialhstarstarpartialuenergy{N-1}^2(\tau) \, d\tau
< \int_{t_1}^t- 4 q H \partialhstarstarpartialuenergy{N-1}^2(\tau) \, d\tau.
\]
The term $6H^2\g_t\a h_{jk}$ [denoted ``dangerous" on the right-hand side of~\eqref{E:partialpartialu**est}] 
is bounded by 
$6H^2 \| \g_t\a h_{jk} \|_{L^2} 
\lesssim 
e^{-qH t}\partialhstarstarenergy{N-1},$
where we have used \eqref{E:mathcalEcomparison}
and the definition \eqref{E:partialen**} of $\partialhstarstarenergy{N-1}.$
Also using the 
Cauchy-Schwarz inequality and
the estimate
$\sum_{|\vec{\alpha}| \leq N - 1} e^{q \Omega} \| \g_t \u \a h_{jk} \|_{L^2} 
\lesssim \hstarstarpartialuenergy{N-1},$
which follows from \eqref{E:mathcalEcomparison} and the definition \eqref{E:partialenpartialu**}  
of $\hstarstarpartialuenergy{N-1},$
we bound the integral corresponding to the dangerous linear term by
the integral $\int_{t_1}^t C \partialhstarstarenergy{N-1}(\tau)
\partialhstarstarpartialuenergy{N-1}(\tau) \,d\tau$ 
on the right-hand side of~\eqref{E:ENpartialh**integralpartialu}. 
We bound all remaining integrals
by $\int_{t_1}^tCe^{-qH\tau}\totalenergy^2(\tau)\,d\tau$
with the help of the Cauchy-Schwarz inequality, 
the estimate $|\omega-H| \lesssim e^{-qHt},$
and the error term estimates
\eqref{E:triangleTWODOWNHNMINUSONE},
\eqref{E:PARTIALUtriangleTWODOWNHNMINUSONE},
\eqref{E:COMMUTATORPARTIALTPARTIALUMETRICHNMINUONE},
\eqref{E:COMMUTATORPARTIALUtriangleTWODOWNHNMINUSONE},
\eqref{E:gpartialucommutatorL2},
\eqref{E:partialutriangleEgamma**delta**L1},
and
\eqref{E:ELLIPTICERRORL2}.
\end{proof}
\subsection{Statement and proof of the main theorem} \label{SS:maintheorem}
We begin this section by providing two technical results that will be used in our proof of the main future stability theorem. 
The first is a simple Gronwall-type lemma. The second is a proposition  
\emph{that contains the most important estimates in the article.}
Specifically, based on the integral inequalities of Prop.~\ref{P:integralinequalities} and~\ref{P:integralinequalitiesmetric},
the proposition provides suitable a priori estimates for the total solution norm $\totalnorm.$
This is the main step in our proof of future-global existence.
We now provide the Gronwall-type lemma, which we will use to bound the ``dangerous'' error integrals
appearing in Props.~\ref{P:integralinequalities} and~\ref{P:integralinequalitiesmetric}.
\begin{lemma} [\textbf{An integral estimate}] \cite[Lemma 11.4]{iRjS2012} \label{L:integralinequality}
	Let $b(t) > 0$ be a continuous \textbf{non-decreasing} function on the interval $[0,T],$ and let $\epsilon > 0.$
	Suppose that for each $t_1 \in [0,T],$ $y(t) \geq 0$ is a continuous function satisfying the inequality
	\begin{align} 
		y^2(t) \leq y^2(t_1) + \int_{\tau = t_1}^t \Big\lbrace - b(\tau)y^2(\tau) + \epsilon y(\tau) \Big\rbrace \, d \tau \notag
	\end{align}
	for $t \in [t_1,T].$ Then for any $t_1, t \in [0,T]$ with $t_1 \leq t,$
	we have that
	\begin{align} 
		y(t) \leq y(t_1) + \frac{\epsilon}{b(t_1)}. \notag
	\end{align}
\end{lemma}

\begin{remark} [\textbf{Comment on $b(t)$}]
In our proof of Prop. \ref{P:auxiliary}, the function $b(t)$ from Lemma \ref{L:integralinequality}
will be a small, positive constant.
\end{remark}

\begin{proposition} [\textbf{A priori estimates for the total solution norm}] \label{P:auxiliary}
Assume that the hypotheses of Prop.~\ref{P:Sobolev}, including the smallness assumption~\eqref{E:bootstrap},
hold on the spacetime slab $[0,T) \times \mathbb{T}^3.$ In particular, assume that the 
solution to the modified dust-Einstein system~\eqref{E:metric00}-\eqref{E:metricjk} exists on the interval $[0,T)$ 
and that $\totalnorm(t) \le \epsilon$ for $t \in [0,T).$ Let $\totalnorm(0) \eqdef \mathring{\epsilon},$
and assume that $\mathring{\epsilon} \leq \epsilon.$
Then there exist large constants $c,\,C>1$ such that for any 
$t_1 \in [0,T),$ the following inequality holds:
\begin{align} \label{E:FUNDAMENTALTOTALSOLUTIONNORMAPRIORIESTIMATE}
\totalnorm(t) \le C \left\lbrace\mathring{\epsilon} e^{ct_1}+\epsilon e^{-qHt_1/2}\right\rbrace, \quad t\in[0,\,T).
\end{align}
\end{proposition}
\begin{proof}
Throughout this proof, we freely use the comparison estimates of Prop. \ref{P:energynormcomparison} without 
explicitly mentioning it each time.
We start with a very crude application of Props.~\ref{P:integralinequalities} and~\ref{P:integralinequalitiesmetric} (with $t_1=0$)
to conclude that there exists a constant $c > 0$ such that
\begin{align*}
\totalnorm^2(t) \le \totalnorm^2(0) + 2 c \int_0^t \totalnorm^2(\tau)\,d\tau.
\end{align*}
Applying the standard Gronwall inequality and using $\totalnorm(0)=\mathring{\epsilon}$, we deduce that
$
\totalnorm(t)  \le C \mathring{\epsilon} e^{ct}  
$
and therefore
\begin{align}
\totalenergy(t) & \le C \mathring{\epsilon} e^{ct} \label{E:Gronwall2}. 
\end{align}
Now let $t_1 \in [0,T)$ be an arbitrary time. 
We {\em first} estimate the term $\hstarstarenergy{N-2}$. Using
the integral inequality~\eqref{E:ENh**integral} and the smallness assumption
$\totalnorm(t) \le \epsilon$, we deduce that
\[
\hstarstarenergy{N-2}^2(t) \le \hstarstarenergy{N-2}^2(t_1) + C \epsilon^2 \int_{t_1}^t e^{-qH \tau}\, d\tau.
\]
Using~\eqref{E:Gronwall2} at time $t=t_1$ and carrying out the integration above, we deduce that for any $t\in[t_1,T),$
we have
$
\hstarstarenergy{N-2}(t) \le C \{\mathring{\epsilon} e^{ct_1} + \epsilon e^{-q H t_1 /2}\}.
$
Using~\eqref{E:Gronwall2} to bound $\hstarstarenergy{N-2}(t)$ for $t \in [0,t_1),$  
we deduce that for any $t\in [0,T),$ we have
\begin{align}\label{E:Gronwall3}
\hstarstarenergy{N-2}(t) \le C \{\mathring{\epsilon} e^{ct_1} + \epsilon e^{-q H t_1 /2}\}.
\end{align}
A similar argument based on the integral inequalities 
\eqref{E:ENpartialh**integral},
\eqref{E:mathfrakENg00integral}, 
and 
\eqref{E:intun}
leads to the following estimates, which hold for $t\in[0,T)$:
\begin{align} \label{E:Gronwall4}
\partialhstarstarenergy{N-1}(t) & \le C \{\mathring{\epsilon} e^{ct_1} + \epsilon e^{-q H t_1 /2}\},
	\\
\gzerozeroenergy{N-1}(t) & \le C \{\mathring{\epsilon} e^{ct_1} + \epsilon e^{-q H t_1 /2}\},
	\label{E:Gronwall6}
	\\
\densenergy{N-1}(t) 
 & \le C \{\mathring{\epsilon} e^{ct_1} + \epsilon e^{-q H t_1 /2}\}.	
 \label{E:Gronwall7}
\end{align}
We have thus estimated all of the energies that do not involve any dangerous terms. We now estimate the energies
involving dangerous terms, and we stress that the order in which we proceed is important.

{\em Next}, we estimate $\gzerostarenergy{N-1}(t).$ To this end, we first use the integral inequality~\eqref{E:mathfrakENg0*integral}, 
the smallness assumption $\totalnorm(t) \le \epsilon,$
~\eqref{E:Gronwall2} at time $t=t_1,$
\eqref{E:Gronwall3},
and~\eqref{E:Gronwall4} to obtain the following inequality for $t\in[t_1,T)$:
\begin{align} \label{E:preGronwall8}
\gzerostarenergy{N-1}^2(t) 
& \le \gzerostarenergy{N-1}^2(t_1) 
	+ C \epsilon^2 \int_{t_1}^t e^{-qH \tau}\, d\tau
	\\
& \ \ + \int_{t_1}^t -4 q H \gzerostarenergy{N-1}^2(\tau) + C \{\mathring{\epsilon} e^{ct_1} + \epsilon e^{-q H t_1 /2}\} \gzerostarenergy{N-1}(\tau)\,d\tau.
	\notag \\
& \le C \{ \mathring{\epsilon}^2 e^{2 ct_1} + \epsilon^2 e^{-qHt_1} \}
	\notag \\
& \ \ 
+ \int_{t_1}^t -4 q H \gzerostarenergy{N-1}^2(\tau) + C \{\mathring{\epsilon} e^{ct_1} + \epsilon e^{-q H t_1 /2}\} 	
	\gzerostarenergy{N-1}(\tau)\,d\tau.
	\notag
\end{align}
To estimate $\gzerostarenergy{N-1}(t)$ for $t \in [t_1,T),$ we apply
Lemma \ref{L:integralinequality} to the last inequality in~\eqref{E:preGronwall8},
while to estimate $\gzerostarenergy{N-1}(t)$ for $t \in [0,t_1),$
we use ~\eqref{E:Gronwall2}. In total, we deduce that
the following inequality holds for $t\in[0,T):$
\begin{align} \label{E:Gronwall8}
\gzerostarenergy{N-1}(t) \le C \{\mathring{\epsilon} e^{ct_1} + \epsilon e^{-q H t_1 /2}\}.
\end{align}
Summing 
\eqref{E:Gronwall3}, 
\eqref{E:Gronwall4},
\eqref{E:Gronwall6},
and \eqref{E:Gronwall8},
and referring to definition
\eqref{E:METRICTOTALBELOWTOPORDERSPATIALDERIVATIVEENERGY}, we deduce that
\begin{align} \label{E:GENERGYSUITABLYBOUNDED}
	\genergy{N-1}
	\le C \{\mathring{\epsilon} e^{ct_1} + \epsilon e^{-q H t_1 /2}\}.
\end{align}

We now estimate the terms $\partialhstarstarpartialuenergy{N}, \gzerozeropartialuenergy{N}, \gzerostarpartialuenergy{N}$ 
\emph{precisely in this order} in an analogous fashion. In particular, we apply Lemma \ref{L:integralinequality}
and crucially use the {\em already} established bounds
\eqref{E:Gronwall3},
\eqref{E:Gronwall4},
\eqref{E:Gronwall6},
and
\eqref{E:Gronwall8}
to bound the potentially dangerous integrands on the right-hand side of
the integral inequalities~\eqref{E:ENpartialh**integralpartialu},~\eqref{E:mathfrakENg00integralpartialu}, and~\eqref{E:mathfrakENg0*integralpartialu}.
In total, this line of reasoning leads to the following estimates, which are valid for $t\in[0,T):$
\begin{align}\label{E:Gronwall9}
\partialhstarstarpartialuenergy{N-1}(t), 
 \ 
\gzerozeropartialuenergy{N-1}(t), 
\ \gzerostarpartialuenergy{N-1}(t)
\le C \{\mathring{\epsilon} e^{ct_1} + \epsilon e^{-q H t_1 /2}\}.
\end{align}

We now estimate $\velocityenergy{N-1}.$ To this end, we use the already proven estimate
\eqref{E:GENERGYSUITABLYBOUNDED} to bound the potentially dangerous term $\genergy{N-1}$ 
in the integral inequality \eqref{E:intun-1}
and argue as in our proof of \eqref{E:Gronwall8} to deduce that the following inequality holds for $t\in[0,T):$
\begin{align} \label{E:Gronwall10}
\velocityenergy{N-1}(t) \le C \{\mathring{\epsilon} e^{ct_1} + \epsilon e^{-q H t_1 /2}\}.
\end{align}
Next, we sum the bounds
\eqref{E:Gronwall7},
\eqref{E:GENERGYSUITABLYBOUNDED},
\eqref{E:Gronwall9},
and
\eqref{E:Gronwall10}
and refer to definitions
\eqref{E:METRICTOTALBELOWTOPORDERSPATIALDERIVATIVEENERGY},
\eqref{E:TOTALMETRICUDERIVATIVEENERGY},
and \eqref{E:totalbelowtopenergy}
to deduce that
\begin{align} \label{E:BOUNDFORALLBUTTOPORDERSPATIALUDERIVATIVES}
	\totalbelowtopenergy{N-1}(t) 
	+ \guenergy{N-1}(t)
	\le C \{\mathring{\epsilon} e^{ct_1} + \epsilon e^{-q H t_1 /2}\}.
\end{align}

Finally, we estimate $\topordervelocityenergy{N-1}.$ To this end, we use the already proven estimate
\eqref{E:BOUNDFORALLBUTTOPORDERSPATIALUDERIVATIVES} 
to bound the potentially dangerous terms 
$	\totalbelowtopenergy{N-1} + \guenergy{N-1}$ 
in the integral inequality \eqref{E:TOPORDERintun-1}
and argue as in our proof of \eqref{E:Gronwall8} to deduce that the following inequality holds for $t\in[0,T):$
\begin{align} \label{E:Gronwall11}
\topordervelocityenergy{N-1}(t) \le C \{\mathring{\epsilon} e^{ct_1} + \epsilon e^{-q H t_1 /2}\}.
\end{align}
Summing the bounds 
\eqref{E:BOUNDFORALLBUTTOPORDERSPATIALUDERIVATIVES} and \eqref{E:Gronwall11},
referring to the definition \eqref{E:TOTALENERGYDEF} of $\totalenergy,$
and using the comparison estimate \eqref{E:TOTALNORMENERGYEQUIVALENCE},
we finally arrive at the desired estimate \eqref{E:FUNDAMENTALTOTALSOLUTIONNORMAPRIORIESTIMATE}.
\end{proof}
\begin{theorem}[\textbf{Global stability of the FLRW solutions}]\label{T:maintheorem}
Let $\Lambda>0$ be a fixed cosmological constant, and let $N\geq4$ be an integer. 
Let $(\mathring{g}_{\mu\nu} = g_{\mu \nu}|_{t=0}, 2 \mathring{K}_{\mu\nu} = \partial_t g_{\mu \nu}|_{t=0}, 
\mathring{u}^{\mu} = u^{\mu}|_{t=0}, \mathring{\dens} = e^{3 \Omega(0)} \rho|_{t=0}),$
$(\mu, \nu = 0,1,2,3),$
be an initial data set for the modified dust-Einstein system 
\eqref{E:metric00}-\eqref{E:velevol}
generated by the initial data set 
$(\T^3, \mathring{\underline{g}}_{jk}, \mathring{\underline{K}}_{jk}, \mathring{\rho}, \underline{\mathring{u}}^j),$ $(j,k = 1,2,3)$, for the unmodified system as explained in Sect.~\ref{S:modified}. 
Assume that there exists a constant $c_1 > 2$ such that
	\begin{align}
	\label{E:initialassumption}
		\frac{2}{c_1} \delta_{ab} X^a X^b \leq \mathring{g}_{ab}X^a X^b \leq \frac{c_1}{2} \delta_{ab} X^a X^b,  \ \
		\forall(X^1,X^2,X^3) \in \mathbb{R}^3.
	\end{align}
Let $\totalnorm$ be the total norm defined in~\eqref{E:totalnorm}.	
Then there exist constants $\epsilon_0>0$ and $C_*> 1$ such that if 
$\epsilon \le \epsilon_0$ and $\totalnorm(0) \le C_*^{-1} \epsilon$, then
there exists a unique classical solution $(g_{\mu\nu}, u^{\mu},\dens)$, $\mu,\nu=0,1,2,3,$ to 
the dust-Einstein equations. The solution exists on
$[0,\infty)\times\T^3$ and furthermore, the following bound holds:
\begin{align*}
	\totalnorm(t) \leq \epsilon, \quad t\in[0,\infty).
\end{align*}
Moreover, the time $T_{\text{max}}$ from Prop.~\ref{P:continuationprinciple} is infinite and the
spacetime-with-boundary $([0,\infty)\times\T^3,\,g_{\mu\nu})$
is future geodesically complete. Furthermore, the solution to the modified equations also solves 
the unmodified dust-Einstein system.
\end{theorem}
\begin{remark}[\textbf{Alternate stability characterization}]
As stated above, the stability condition is phrased in terms of the initial data for the unmodified system.
However, it is also possible to formulate a near-FLRW smallness condition solely in terms of
the pure Einstein data $(\T^3, \mathring{\underline{g}}_{jk}, \mathring{\underline{K}}_{jk}, \mathring{\dens}, \underline{\mathring{u}}^j),$ $(j,k = 1,2,3).$
The condition would imply that $\totalnorm(0)\le \epsilon;$ 
see Remark 11.2.1 in~\cite{jS2012} for additional details.
\end{remark}
%


\noindent
{\bf Proof of Theorem~\ref{T:maintheorem}.}
Let $\totalnorm(0)=\mathring{\epsilon}<\frac{\epsilon}{2}$. Then by the local well-posedness Theorem~\ref{T:localwellposedness},
if $\epsilon$ is sufficiently small,
then there exists a $T>0$ such that a unique solution exists on the time interval $[0,T)$ and satisfies 
\begin{align} 
		\totalnorm(t) & \leq \epsilon, \label{E:proofbootstrapQN} \\
		c_1^{-1} \delta_{ab}X^{a}X^{b} & \leq e^{-2 \Omega} g_{ab} X^{a}X^{b} 
			\leq c_1 \delta_{ab}X^{a}X^{b},&& \forall (X^1,X^2,X^3) \in \mathbb{R}^3. 
			\label{E:gjproofBAkvsstandardmetric} 
		\end{align}

Observe that by Remark \ref{R:RoughBootstrapAutomatic}, \eqref{E:proofbootstrapQN} implies that
	the rough bootstrap assumptions \eqref{E:metricBAeta} and \eqref{E:g0jBALinfinity} are satisfied if $\epsilon$ is sufficiently small.
Let now
\[
\mathcal{T}\eqdef\sup_{T > 0}\{\text{ solution exists on} \  [0,T) \times \T^3\ \text{and~\eqref{E:proofbootstrapQN}-\eqref{E:gjproofBAkvsstandardmetric} hold } \}.
\]
Using Prop.~\ref{P:auxiliary}, we infer that there exist constants $c,C>0$ such that for any $t_1<\mathcal{T}$ the following bound holds:
\[
\totalnorm(t) \le C\left\lbrace\mathring{\epsilon} e^{ct_1}+\epsilon e^{-qHt_1/2}\right\rbrace, \qquad t\in[0,\,\mathcal{T}).
\]
Choosing $t_1$ sufficiently large so that $Ce^{-qHt_1/2}\le 1/4$ and $\mathring{\epsilon}/\epsilon$ sufficiently small so that $\mathring{\epsilon} e^{ct_1} \le \epsilon/4$, we ensure that $\totalnorm(t) \le \epsilon/2$. This demonstrates the improvement of the inequality~\eqref{E:proofbootstrapQN} on the time interval $[0,\mathcal{T})$. 
In order to show that the time $\mathcal{T}>0$ is sufficiently large so that $\mathcal{T}>t_1$, we use the standard Cauchy-stability bound
$\totalnorm(t)\le \mathring{\epsilon}e^{ct}$ to infer that the time of existence is at the 
order of $c^{-1} \mbox{ln} \big[\epsilon/(\widetilde{C}\mathring{\epsilon})\big]$ if $\epsilon$ and $\mathring{\epsilon}/\epsilon$ are sufficiently small.
Since $\mathring{\epsilon}/\epsilon$ is chosen after fixing $t_1$, we can ensure that $t_1 <\mathcal{T}$.
Next, we observe that $\|\g_t(e^{-2\Omega}g_{jk})\|_{L^{\infty}}=\|\g_th_{jk}\|_{L^{\infty}}\lesssim \epsilon e^{-qHt}.$ 
Hence, through integration in time, we deduce that
\[
\|e^{-2\Omega}g_{jk}(t,\cdot)-\mathring{g}_{jk}(\cdot)\|_{L^{\infty}} \lesssim \epsilon.
\]
It is easy to see that the previous bound and~\eqref{E:initialassumption} lead to the
following improvement of~\eqref{E:gjproofBAkvsstandardmetric} on the time interval $[0,\mathcal{T})$
when $\epsilon$ is sufficiently small:
\be\label{E:improvedassumption}
\frac{3}{2c_1} \delta_{ab}X^{a}X^{b}  \leq e^{-2 \Omega} g_{ab} X^{a}X^{b} 
			\leq \frac{2c_1}{3} \delta_{ab}X^{a}X^{b}, \ \  \forall (X^1,X^2,X^3) \in \mathbb{R}^3. 
\ee
To prove future-global existence, we will assume that $\mathcal{T}<\infty$ and
show that this leads to a contradiction.
Specifically, we will use
\eqref{E:proofbootstrapQN}
and
\eqref{E:gjproofBAkvsstandardmetric}
to deduce that
\emph{none of the three breakdown scenarios from Prop.~\ref{P:continuationprinciple} can occur},
which will easily lead to the contradiction.
The first breakdown scenario is easily excluded since on $[0,\mathcal{T}),$
we have the Sobolev embedding estimate
$\|g_{00}+1\|_{L^{\infty}} \lesssim e^{-q \Omega} \totalnorm \le \epsilon.$ 
The second one is excluded by~\eqref{E:gjproofBAkvsstandardmetric}.
Finally, the third breakdown scenario is excluded by the Sobolev embedding estimates
$\|g_{00}+1\|_{C_b^2} + \|\g_t g_{00}\|_{C_b^1}\lesssim e^{(1-q)\Omega} \totalnorm,$ 
$\sum_{j=1}^3(\|g_{0j}\|_{C_b^2}+\|\g_tg_{0j}\|_{C_b^1})\lesssim e^{(2-q)\Omega}\totalnorm,$
$\sum_{j,k=1}^3 \|\g_t h_{jk}\|_{C_b^1}\lesssim e^{-q\Omega} \totalnorm,$
$\sum_{j,k=1}^3 \|\underpartial h_{jk}\|_{C_b^1}\lesssim e^{(1-q)\Omega} \totalnorm,$
$\|\dens-\bar{\dens}\|_{C_b^1}\lesssim \totalnorm,$ 
$\sum_{j=1}^3\|u^j\|_{C_b^1}\lesssim e^{-(1+q)\Omega}\totalnorm,$ 
the relations $g_{jk} = e^{2 \Omega} h_{jk}$
and $\dens = e^{3 \Omega} \rho,$
and~\eqref{E:improvedassumption}.
Using the continuity of $\totalnorm(\cdot),$ 
the improved bounds $\totalnorm(\tau)\le \epsilon/2$ and~\eqref{E:improvedassumption}, 
we apply Theorem~\ref{T:localwellposedness}
and Prop.~\ref{P:continuationprinciple}
to deduce that there exists a $\delta>0$ such that we can extend the solution to the time interval $[0,\mathcal{T}+\delta)$
on which ~\eqref{E:proofbootstrapQN}-\eqref{E:gjproofBAkvsstandardmetric} hold. 
This contradicts the definition of $\mathcal{T}$ and implies that $\mathcal{T} = \infty.$

Finally, we note that based on the estimates of Prop. \ref{P:metric},
future geodesic completeness can be proved using the same argument given in~\cite[Theorem~11.7]{iRjS2012};
we omit the details.
\prfe

\section*{Acknowledgments}
MH gratefully acknowledges support from NSF grant \# DMS-1211517.
JS gratefully acknowledges support from NSF grant \# DMS-1162211 
and from a Solomon Buchsbaum grant administered by the Massachusetts Institute of Technology.
Any opinions, findings, and conclusions or recommendations expressed in this material are those of the authors and do not necessarily reflect the views of the National Science Foundation.

\appendix

\setcounter{section}{0}
   \setcounter{subsection}{0}
   \setcounter{subsubsection}{0}
   \setcounter{paragraph}{0}
   \setcounter{subparagraph}{0}
   \setcounter{figure}{0}
   \setcounter{table}{0}
   \setcounter{equation}{0}
   \setcounter{theorem}{0}
   \setcounter{definition}{0}
   \setcounter{remark}{0}
   \renewcommand{\thesection}{\Alph{section}}
   \renewcommand{\theequation}{\Alph{section}.\arabic{equation}}
   \renewcommand{\theproposition}{\Alph{section}.\arabic{proposition}}
   \renewcommand{\thecorollary}{\Alph{section}.\arabic{corollary}}
   \renewcommand{\thedefinition}{\Alph{section}.\arabic{definition}}
   \renewcommand{\thetheorem}{\Alph{section}.\arabic{theorem}}
   \renewcommand{\theremark}{\Alph{section}.\arabic{remark}}
   \renewcommand{\thelemma}{\Alph{section}.\arabic{lemma}}

\section{Derivation of the Modified System} \label{S:APPENDIXA}

In Appendix \ref{S:APPENDIXA}, we sketch a derivation of the modified system \eqref{E:metric00}-\eqref{E:velevol}.

\begin{proposition}[\textbf{Decomposition of the modified equations}] \label{AP:Decomposition}
	The Eqs. \eqref{E:Rhat00}-\eqref{E:secondEulersummary} in the unknowns $(g_{\mu \nu},\dens,u^j),$ $(\mu, \nu = 0,1,2,3),$
	$(j = 1,2,3),$ can be written as 
\begin{subequations}
\begin{align}
	\hat{\square}_g (g_{00} + 1) & = 5 H \partial_t g_{00} + 6 H^2 (g_{00} + 1) + \triangle_{00},
		\label{AE:finalg00equation} \\
	\hat{\square}_g g_{0j} & = 3H \partial_t g_{0j} + 2 H^2 g_{0j} - 2Hg^{ab}\Gamma_{a j b} + \triangle_{0j}, 
		\label{AE:finalg0jequation} \\
	\hat{\square}_g h_{jk} & = 3H \partial_t h_{jk} + \triangle_{jk}, \label{AE:finalhjkequation}  \\
u^{\alpha}\g_{\alpha}(\dens-\bar{\dens})
& =  \triangle\,,\label{AE:rhoevolution}\\
u^{\alpha} \partial_{\alpha} u^j & = 
-2 \omega u^0 u^j 
- u^0 \triangle^{\ j}_{0 \ 0}
+ \triangle^j\,,\label{AE:velocityevolution}
\end{align}
\end{subequations}
where $\omega(t),$ which is uniquely determined by the parameters $\Lambda > 0$ and $\bar{\dens} \geq 0$,
is the function from \eqref{E:LITTLEOMEGADEFINED},
\begin{align*} 
	u^0 & = - \frac{g_{0a}u^a}{g_{00}} + \sqrt{1 + \Big(\frac{g_{0a}u^a}{g_{00}}\Big)^2 - \frac{g_{ab}u^a u^b}{g_{00}} 
		- \Big(\frac{g_{00} + 1}{g_{00}}\Big)}, \\
	\dens & \eqdef e^{3\Omega}\rho, \ \
	h_{jk}  \eqdef e^{-2\Omega}g_{jk}, \ \
		H  \eqdef \sqrt{\frac{\Lambda}{3}}.
	\end{align*}
	The error terms $\triangle_{\mu \nu},$ $\triangle,$ $\triangle^j$ can be expressed as
\begin{subequations} 
\begin{align}
	\triangle_{00} & = -(g_{00}+1)e^{-3\Omega}\bar{\dens} - e^{-3\Omega}(\dens-\bar{\dens})
                           +2e^{-3\Omega}\dens(1-u_0^2)-e^{-3\Omega}\dens(g_{00}+1)\label{AE:triangle00}\\
                       & \ \  +2(\triangle_{A,00}+\triangle_{C,00})
                           +5(\omega-H)\g_tg_{00}+6(\omega^2-H^2)(g_{00}+1),\notag\\
	\triangle_{0j} & = \frac{1}{2}e^{-3\Omega}\bar{\dens}g_{0j}
                                       -2e^{-3\Omega} \dens u_0u_j-e^{-3\Omega} \dens g_{0j}\label{AE:triangle0j}\\
                                   & \ \  +2(H^2-\omega^2)g_{0j}+3(\omega-H)\g_tg_{0j}
                                        -2(\omega-H)g^{ab}\Gamma_{ajb}+2(\triangle_{A,0j}+\triangle_{C,0j}),\notag\\
	\triangle_{jk} & = \bar{\dens}e^{-3\Omega}(g^{00}+1)h_{jk}-e^{-3\Omega}(\dens-\bar{\dens})h_{jk}\label{AE:trianglejk}\\
                                   & \ \  -2e^{-5\Omega} \dens u_ju_k
                                        +3(\omega-H)\g_th_{jk}-4\omega g^{0a}\g_ah_{jk}+2e^{-2\Omega}\triangle_{A,jk},\notag \\
	\triangle & =  - \dens \partial_a u^a  
								- \frac{\dens }{u_0 u^0} u_a u^b \partial_b u^a	
								- 2 \omega \frac{\dens }{u_0} u_a u^a
								-\frac{\dens}{u_0}u_a\triangle^{\ a}_{0 \ 0}
								+ \frac{\dens }{u_0 u^0} u_a \triangle^a						
								 \label{AE:triangledef} \\
						& \ \ - \dens \triangle_{\alpha \ \beta}^{\ \alpha}u^{\beta}
                  + \frac{\dens}{2u_0} 
                  	\Big( (\g_tg_{00})(u^0)^2
                     	+2 (\g_tg_{0a}) u^au^0
            					+ \underbrace{(\g_tg_{ab}) u^au^b}_{(e^{2\Omega}\g_th_{ab}+2\omega g_{ab})u^au^b} 
            				\Big),
                       \notag \\
	\triangle^j & = 
		- u^0 (u^0 - 1) \triangle^{\ j}_{0 \ 0} 
		- 2  u^0 u^a \triangle^{\ j}_{0 \ a}
		- u^a u^b \triangle^{\ j}_{a \ b} 
	\label{AE:trianglejdef},
\end{align}
\end{subequations}
where the $\triangle_{A,\mu \nu}$ are defined in \eqref{AE:triangleA00def}-\eqref{AE:triangleAjkdef}, 
$\triangle_{C,00},$ $\triangle_{C,0j}$ are defined in \eqref{AE:triangleC00def}-\eqref{AE:triangleC0jdef},
and the $\triangle_{\mu \ \nu}^{\ \alpha}$ are defined in \eqref{AE:triangleGamma000}-\eqref{AE:triangleGammaikj}.
\end{proposition}

\begin{proof}
	The proof is a series of tedious computations based on Lemmas 
	\ref{AL:modifiedRiccidecomposition}-\ref{AL:Christoffeldecomposition}. We provide the proofs of 
	\eqref{AE:finalhjkequation}-\eqref{AE:velocityevolution} 
	and leave the remaining details to the reader. To obtain \eqref{AE:finalhjkequation}, we first use Eq. \eqref{E:Rhatjk}, 
	Lemma \ref{AL:modifiedRiccidecomposition}, and Lemma \ref{AL:Amunudecomposition} to obtain the following equation for $h_{jk} 
	= e^{-2\Omega} g_{jk}:$	
	\begin{align} \label{AE:boxhjkfirstformula}
		\hat{\square}_g h_{jk} & = 3 \omega \partial_t h_{jk}
			+ 2[3 \omega^2 + \frac{d}{dt} \omega - \Lambda]h_{jk} 
			- 4 \omega g^{0a} \partial_a h_{jk} \\
		& \ \ + 2 e^{-2\Omega} \triangle_{A;jk} 
			- 2\left(\frac{d}{dt} \omega\right)(g^{00} + 1)h_{jk} \notag \\
		& \ \ - 2e^{-5\Omega} \dens u_j u_k - e^{-3\Omega} \dens h_{jk}. \notag
		\end{align}
		Using \eqref{E:backgroundrhoafact}-\eqref{E:omegadotidentity2}, we have that
		$2[3 \omega^2 + \frac{d}{dt} \omega - \Lambda]h_{jk}$ $= -\bar{\dens}e^{-3\Omega} 
		h_{jk}.$ Substituting into \eqref{AE:boxhjkfirstformula}, and using 
		$\frac{d}{dt} \omega = - \frac{1}{2}e^{-3\Omega}\bar{\dens}$ [that is, 
		\eqref{E:omegadotidentity}], we find that
		\begin{align} \label{AE:boxhjksecondformula}
			\hat{\square}_g h_{jk} & = 3 \omega \partial_t h_{jk}
				- e^{-3\Omega}(\dens-\bar{\dens}) h_{jk} 
				- 4 \omega g^{0a} \partial_a h_{jk} \\
			& \ \ + 2 e^{-2\Omega} \triangle_{A;jk} 
				+ \bar{\dens} e^{-3\Omega}(g^{00} + 1)h_{jk} 
			 - 2e^{-5\Omega} \dens u_j u_k. \notag
		\end{align}	
		Eq. \eqref{AE:finalhjkequation} now easily follows from \eqref{AE:boxhjksecondformula}. We remark that the proofs of 
		\eqref{AE:finalg00equation} and \eqref{AE:finalg0jequation} require the use of Lemma \ref{AL:AmunuplusImunudecomposition}.
		To obtain \eqref{AE:velocityevolution}, we first recall Eq. (\ref{E:secondEulersummary}): 		
		\begin{align}  \label{AE:finalEulerChristoffel}
			u^{\alpha} \partial_{\alpha} u^j +\Gamma_{\alpha \ \beta}^{\ j} u^{\alpha} u^{\beta}
				& = 0.
		\end{align}
		Lemma \ref{AL:Christoffeldecomposition} implies that 
		$\Gamma_{\alpha \ \beta}^{\ j} u^{\alpha} u^{\beta} 
		= 2 \omega u^0 u^j 
		+
		\triangle_{\alpha \ \beta}^{\ j} u^{\alpha} u^{\beta}
		=2 \omega u^0 u^j 
		+ u^0 \triangle_{0 \ 0}^{\ j}
		+ u^0(u^0-1)\triangle_{0 \ 0}^{\ j} 
		+ 2 u^0 u^a \triangle_{0 \ a}^{\ j}
		+ u^a u^b \triangle_{a \ b}^{\ j}.$ Eq. 
		\eqref{AE:velocityevolution}
		now follows from plugging the previous formula into~\eqref{AE:finalEulerChristoffel}.
		To obtain
		\eqref{AE:rhoevolution}, we first recall Eq. \eqref{E:firstEulersummary}:
		\begin{align} \label{AE:firstEulerChristoffel}
			u^{\alpha} \partial_{\alpha} \rho + \rho \partial_{\alpha}u^{\alpha} 
			+ \rho \Gamma_{\alpha \ \beta}^{\ \alpha} u^{\beta}  = 0.
		\end{align}
		Lemma \ref{AL:Christoffeldecomposition} implies that 
		$\Gamma_{\alpha \ \beta}^{\ \alpha} u^{\beta} = 3 \omega u^0 + \triangle_{\alpha \ \beta}^{\ \alpha} u^{\beta},$
		while the normalization condition $g_{\alpha \beta} u^{\alpha} u^{\beta} = -1$ implies that 
		$\partial_t u^0 = - \frac{1}{u_0} \big\lbrace u_a \partial_t u^a 
		+ \frac{1}{2}(\partial_t g_{\alpha \beta})u^{\alpha} u^{\beta} \big\rbrace.$ 
		Eq. \eqref{AE:rhoevolution} now follows from multiplying both sides of
		\eqref{AE:firstEulerChristoffel} by $e^{3 \Omega},$ using the fact that
		$e^{3 \Omega}\rho = \dens,$ and using Eq. \eqref{AE:velocityevolution}
		to replace $\partial_t u^a$ with spatial derivatives of $u$
		plus inhomogeneous terms.
		\end{proof}

We now state the four lemmas used in the proof of Prop. \ref{AP:Decomposition}.
\begin{lemma} \cite[Lemma 4]{hR2008} \label{AL:modifiedRiccidecomposition}
	The modified Ricci tensor from \eqref{E:modifiedRicci} can be decomposed as follows:
	\begin{align}
		\widehat{\mbox{Ric}}_{\mu \nu} & = - \frac{1}{2} \hat{\square}_g g_{\mu \nu}
			+ \frac{3}{2}(g_{0 \mu} \partial_{\nu} \omega + g_{0 \nu} \partial_{\mu} \omega)
			+ \frac{3}{2} \omega \partial_t g_{\mu \nu} + A_{\mu \nu}, &&  (\mu, \nu = 0,1,2,3), \nonumber \\
			\mbox{where} \notag \\
		A_{\mu \nu} & \eqdef g^{\alpha \beta} g^{\kappa \lambda}
			\big[(\partial_{\alpha} g_{\nu \kappa})(\partial_{\beta} g_{\mu \lambda})
			- \Gamma_{\alpha \nu \kappa} \Gamma_{\beta \mu \lambda}\big], && (\mu, \nu = 0,1,2,3).	\label{AE:Amunudef}
	\end{align}
	\end{lemma}
\begin{lemma} \cite[Lemma 5]{hR2008} \label{AL:Amunudecomposition}
	The term $A_{\mu \nu}$ ($\mu, \nu = 0,1,2,3$) defined in \eqref{AE:Amunudef} can be decomposed into principal terms and error 
	terms $\triangle_{A,\mu \nu}$ as follows:
	\begin{subequations}
	\begin{align*}
		A_{00} & = 3 \omega^2 - \omega g^{ab}\partial_t g_{ab} + 2 \omega g^{ab} \partial_a g_{0b} + \triangle_{A,00}, \\
		A_{0j} & = 2 \omega g^{00} \partial_t g_{0j} - 2 \omega^2 g^{00}g_{0j} - \omega g^{00} \partial_j g_{00} 
			+ \omega g^{ab}\Gamma_{ajb} + \triangle_{A,0j}, && (j=1,2,3), \\
		A_{jk} & = 2 \omega g^{00} \partial_t g_{jk} - 2 \omega^2 g^{00} g_{jk} + \triangle_{A,jk}, && (j,k = 1,2,3),
	\end{align*}
	\end{subequations}
	where
	\begin{subequations}
	\begin{align}
		\triangle_{A,00} & \eqdef (g^{00})^2 \Big\lbrace(\partial_t g_{00})^2 - (\Gamma_{000})^2 \Big\rbrace 
			 \label{AE:triangleA00def} \\
		& \ \ + g^{00}g^{0a} \Big\lbrace 2 (\partial_t g_{00})(\partial_t g_{0a} + \partial_a g_{00}) 
			- 4 \Gamma_{000} \Gamma_{00a} \Big\rbrace \notag \\
		& \ \ + g^{00}g^{ab} \Big\lbrace (\partial_t g_{0a})(\partial_t g_{0b}) 
				+ (\partial_a g_{00}) (\partial_b g_{00}) 
				- 2 \Gamma_{00a} \Gamma_{00b} \Big\rbrace \notag \\
		& \ \ + g^{0a} g^{0b} \Big\lbrace 2(\partial_t g_{00})(\partial_a g_{0b}) + 2(\partial_t g_{0b})(\partial_a g_{00}) 
				- 2 \Gamma_{000} \Gamma_{a0b} - 2 \Gamma_{00b} \Gamma_{00a} \Big\rbrace \notag \\
		& \ \ + g^{ab} g^{0l} \Big\lbrace 2(\partial_t g_{0a})(\partial_l g_{0b}) + 2(\partial_b g_{00})(\partial_a g_{0l}) 
				- 4\Gamma_{00a} \Gamma_{l0b}  \Big\rbrace \notag \\
		& \ \ + g^{ab}g^{lm}(\partial_a g_{0l})(\partial_b g_{0m}) 
				+ \frac{1}{2} g^{lm}(\underbrace{g^{ab} \partial_t g_{al} - 2\omega \delta_l^b}_{
				e^{2 \Omega} g^{ab} \partial_t h_{al}- 2 \omega g^{0b} g_{0l}})(\partial_b g_{0m} + \partial_m g_{0b}) \notag \\
		& \ \ - \frac{1}{4} g^{ab} g^{lm}(\partial_a g_{0l} + \partial_l g_{0a})(\partial_b g_{0m} + \partial_m g_{0b}) \notag \\
		& \ \ - \frac{1}{4}(\underbrace{g^{ab} \partial_t g_{al} - 2\omega \delta_l^b}_{
				e^{2 \Omega} g^{ab} \partial_t h_{al} - 2 \omega g^{0b} g_{0l}}) 
				(\underbrace{g^{lm} \partial_t g_{bm} - 2 \omega \delta_b^l}_{e^{2 \Omega} g^{lm} \partial_t h_{bm} 
				- 2 \omega g^{0l} g_{0b}}), \notag 
		\end{align}
		\begin{align}
		\triangle_{A,0j} & \eqdef (g^{00})^2 \Big\lbrace (\partial_t g_{00}) (\partial_t g_{0j}) - \Gamma_{000} \Gamma_{0j0} \Big\rbrace
			\label{AE:triangleA0jdef} \\
		& \ \ + g^{00}g^{0a} \Big\lbrace (\partial_t g_{00})(\partial_t g_{aj} + \partial_a g_{0j}) 
			+(\partial_t g_{0j})(\partial_t g_{0a} + \partial_a g_{00}) \notag \\
		& \hspace{1in} - 2 \Gamma_{000} \Gamma_{0ja} - 2 \Gamma_{0j0} \Gamma_{00a} \Big\rbrace \notag \\
		& \ \ + g^{00} (\underbrace{g^{ab} \partial_t g_{bj} - 2 \omega \delta_j^a}_{e^{2 \Omega} g^{ab} \partial_t h_{bj}
			- 2 \omega g^{0a}g_{0j}}) \Big(\partial_t g_{0a} - \frac{1}{2} \partial_a g_{00}\Big) 
		 \ \ + \frac{1}{2} g^{00} g^{ab}(\partial_a g_{00})(\partial_b g_{0j} + \partial_j g_{0b}) \notag \\
		& \ \ + g^{0a} g^{0b} \Big\lbrace (\partial_t g_{00})(\partial_a g_{bj}) + (\partial_t g_{0b})(\partial_a g_{0j})  			
			+ (\partial_a g_{00})(\partial_t g_{bj}) + (\partial_a g_{0b})(\partial_t g_{0j})  \notag  \\
		& \hspace{1 in} - \Gamma_{000} \Gamma_{ajb} - 2 \Gamma_{00b} \Gamma_{0ja} - \Gamma_{a0b} \Gamma_{0j0} \Big\rbrace \notag \\
		& \ \ + g^{ab} g^{0l} \Big\lbrace (\partial_t g_{0a})(\partial_l g_{bj}) + (\partial_l g_{0a})(\partial_t g_{bj}) 
			+ (\partial_b g_{00})(\partial_a g_{lj}) 
		 + (\partial_b g_{0l})(\partial_a g_{0j}) - 2 \Gamma_{00a} \Gamma_{ljb} 
			\Big\rbrace \notag \\
		& \ \ - g^{ab} g^{0l} \Big\lbrace (\partial_l g_{0a} + \partial_a g_{0l})\Gamma_{0jb} 
			-\frac{1}{2}(\partial_t g_{la})(\partial_b g_{0j} - \partial_j g_{0b}) \Big\rbrace \notag \\
		& \ \ + \omega g^{0a}(\underbrace{\partial_t g_{aj} - 2 \omega g_{aj}}_{e^{2 \Omega} \partial_t h_{aj}})
			+ \frac{1}{2}g^{0l}(\underbrace{g^{ab} \partial_t g_{la} - 2\omega \delta_{l}^b}_{e^{2 \Omega} g^{ab} \partial_t h_{la} 
				- 2 \omega g^{0b}g_{0l}})\partial_t g_{bj} \notag \\
		& \ \ + g^{ab} g^{lm} \Big\lbrace (\partial_a g_{0l})(\partial_b g_{mj}) 
			- \frac{1}{2} (\partial_a g_{0l} + \partial_l g_{0a})\Gamma_{bjm} \Big\rbrace 
		+ \frac{1}{2} g^{ab} (\underbrace{g^{lm} \partial_t g_{la} - 2\omega \delta_a^m}_{e^{2 \Omega} g^{lm} \partial_t h_{la} - 
			2 \omega g^{0m}g_{0a}})\Gamma_{bjm}, \notag
		\end{align}
		\begin{align}
		\triangle_{A,jk} & \eqdef (g^{00})^2 \Big\lbrace (\partial_t g_{0j}) (\partial_t g_{0k}) 
			- \Gamma_{0j0} \Gamma_{0k0} \Big\rbrace \label{AE:triangleAjkdef} \\
		& \ \ + g^{00}g^{0a} \Big\lbrace (\partial_t g_{0j})(\partial_t g_{ak} + \partial_a g_{0k}) 
			+ (\partial_t g_{0k})(\partial_t g_{aj} + \partial_a g_{0j}) \notag \\
		& \hspace{1in} - 2 \Gamma_{0j0} \Gamma_{0ka} - 2 \Gamma_{0k0} \Gamma_{0ja} \Big\rbrace \notag \\
		& \ \ + g^{00}g^{ab} \Big\lbrace (\partial_a g_{0j})(\partial_b g_{0k})  
			- \frac{1}{2}(\partial_a g_{0j} - \partial_j g_{0a})(\partial_b g_{0k} - \partial_k g_{0b}) \Big\rbrace	\notag \\
		& \ \ - \frac{1}{2} g^{00} \Big\lbrace (\underbrace{g^{ab} \partial_t g_{aj} - 2 \omega \delta_j^b}_{
			e^{2 \Omega} g^{ab} \partial_t h_{aj} - 2 \omega g^{0b}g_{0j}})(\partial_b g_{0k} - \partial_k g_{0b}) \notag \\
		& \hspace{1in} + (\underbrace{g^{ab} \partial_t g_{bk} - 2 \omega \delta_k^a}_{e^{2 \Omega}g^{ab} \partial_t 
			h_{bk} - 2 \omega g^{0a}g_{0k}})(\partial_a g_{0j} - \partial_j g_{0a}) \Big\rbrace \notag \\
		& \ \ + \omega g^{00}(\underbrace{g_{bk}g^{ab} - \delta_k^a}_{-g_{0k}g^{0a}})\partial_t g_{aj} 
			+ \frac{1}{2} g^{00}(\underbrace{g^{ab}\partial_t g_{aj} - 2 \omega \delta_j^b}_{e^{2 \Omega} g^{ab} \partial_t 
			h_{aj} - 2 \omega g^{0b}g_{0j}})(\underbrace{\partial_t g_{bk} - 2 \omega g_{bk}}_{e^{2 \Omega}\partial_t h_{bk}}) 
			\notag \\
		& \ \ + g^{0a} g^{0b} \Big\lbrace (\partial_t g_{0j})(\partial_a g_{bk}) + (\partial_t g_{bj})(\partial_a g_{0k})  			+ 
			(\partial_a g_{0j})(\partial_t g_{bk})  \notag  \\
		& \hspace{1in} + (\partial_a g_{bj})(\partial_t g_{0k})
			- \Gamma_{0j0} \Gamma_{akb} - 2 \Gamma_{0jb} \Gamma_{0ka} - \Gamma_{ajb} \Gamma_{0k0} \Big\rbrace \notag \\
		& \ \ + g^{ab} g^{0l} \Big\lbrace (\partial_t g_{aj})(\partial_l g_{bk}) + (\partial_l g_{aj})(\partial_t g_{bk}) 
			+ (\partial_b g_{0j})(\partial_a g_{lk})  \notag \\
		& \hspace{1in} + (\partial_b g_{lj})(\partial_a g_{0k})
			- 2 \Gamma_{0ja} \Gamma_{lkb} - 2 \Gamma_{lja} \Gamma_{0kb} \Big\rbrace \notag \\
		& \ \ + g^{ab} g^{ml} \Big\lbrace (\partial_a g_{lj})(\partial_b g_{mk}) - \Gamma_{ajl} \Gamma_{bkm} \Big\rbrace. 
		\notag
	\end{align}
	\end{subequations}
	\end{lemma}

\begin{lemma} \cite[Lemma 6]{hR2008} \label{AL:AmunuplusImunudecomposition}
	The following identities hold, where 
	$A_{\mu \nu}$ is defined in \eqref{AE:Amunudef}
	and
	$\triangle_{A,00},$ 
	$\triangle_{A,0j},$ 
	are defined in \eqref{AE:triangleA00def}-\eqref{AE:triangleA0jdef}:
	\begin{subequations}
	\begin{align*}
		A_{00} + 2 \omega \Gamma^0 - 6 \omega^2 & = \omega \partial_t g_{00} + 3 \omega^2(g_{00} + 1)
			+ 3 \omega^2 g_{00} + \triangle_{A,00} + \triangle_{C,00},  \\
		A_{0j} + 2\omega(3 \omega g_{0j} - \Gamma_j) & = 4 \omega^2 g_{0j} - \omega g^{ab} \Gamma_{ajb} + \triangle_{A,0j} 
			+ \triangle_{C,0j}, 
	\end{align*}
	\end{subequations}
	and
	\begin{subequations}
	\begin{align}
		\triangle_{C,00} & \eqdef -6 (g_{00})^{-1} \omega^2 \Big\lbrace (g_{00} + 1)^2 - g^{0a}g_{0a} \Big\rbrace 
			- \omega (g^{00} + 1)(\underbrace{g^{ab} \partial_t g_{ab} - 6 \omega}_{e^{2\Omega}g^{ab} \partial_t h_{ab} - 2 \omega 
			g^{0a}g_{0a}})  \label{AE:triangleC00def} \\
		& \ \ + 2\omega(g^{00} + 1)g^{ab} \partial_a g_{0b} 
			+ 4\omega g^{0a} g^{0b} \Gamma_{0ab} + 2 \omega g^{ab} g^{0l} \Gamma_{alb}, \notag \\
		\triangle_{C,0j} & \eqdef 2 \omega^2(g^{00} + 1) g_{0j}
		 - 2 \omega g^{0a} \Big\lbrace (\underbrace{\partial_t g_{aj} - 2\omega g_{aj}}_{e^{2 \Omega} \partial_t h_{aj}}) + 
		 \partial_a g_{0j} - \partial_j g_{0a} \Big\rbrace. \label{AE:triangleC0jdef}
	\end{align}
	and $A_{00},A_{0j}$ are defined in \eqref{AE:Amunudef}.
	\end{subequations}
\end{lemma}
\begin{remark}
Terms $2 \omega \Gamma^0 - 6 \omega^2$ and $ 2\omega(3 \omega g_{0j} - \Gamma_j)$ are exactly the terms $I_{00}$ and $I_{0j}$ from 
\cite[Sect. 5]{jS2012}.
They are given by $I_{00}=-2 \omega (\tilde{\Gamma}^0-\Gamma^0)$ and $I_{0j}=2\omega (\tilde{\Gamma}^j-\Gamma^j)$ and vanish if the wave coordinate condition is satisfied.
The above lemma thus states that
the sums $A_{00} + I_{00}$ and $A_{0j} + I_{0j},$ $(j=1,2,3),$ can be decomposed into principal terms and error terms. 
\end{remark}

\begin{lemma} \label{AL:Christoffeldecomposition}
	The Christoffel symbols $\Gamma_{\mu \ \nu}^{\ \alpha}$ can be decomposed into principal terms and
	error terms $\triangle_{\mu \ \nu}^{\ \alpha}$ as follows:
	\begin{subequations}
	\begin{align}
		\Gamma_{0 \ 0}^{\ 0} & = \triangle_{0 \ 0}^{\ 0},  \label{AE:triangleGamma000} \\
		\Gamma_{j \ 0}^{\ 0} = \Gamma_{0 \ j}^{\ 0} & = \triangle_{j \ 0}^{\ 0} = \triangle_{0 \ j}^{\ 0}, 
			 \\
		\Gamma_{0 \ 0}^{\ j} & = \triangle_{0 \ 0}^{\ j},  \\
		\Gamma_{0 \ k}^{\ j} = \Gamma_{k \ 0}^{\ j} & = \omega \delta_k^j +  \triangle_{0 \ k}^{\ j}  
			= \omega \delta_k^j + \triangle_{k \ 0}^{\ j},  \\
		\Gamma_{j \ k}^{\ 0} & =\omega g_{jk} + \triangle_{j \ k}^{\ 0}, \\
		\Gamma_{i \ j}^{\ k} & = \triangle_{i \ j}^{\ k},  \label{AE:triangleGammaikj}
\end{align}
\end{subequations}
where
	\begin{subequations}
	\begin{align}
		2 \triangle_{0 \ 0}^{\ 0} & \eqdef g^{00} \partial_t g_{00} + 2 g^{0a} \partial_t g_{0a} - g^{0a}\partial_a g_{00}, 
			\label{AE:triangle000} \\
		2 \triangle_{j \ 0}^{\ 0} & \eqdef g^{00} \partial_j g_{00} + g^{0a}(\partial_j g_{a0} - \partial_a g_{j0}) 
			+ 2 \omega g^{0a} g_{ja}
			+ g^{0a}(\underbrace{\partial_t g_{ja} - 2 \omega g_{ja}}_{e^{2 \Omega} \partial_t h_{aj}}), 
			\\
		2 \triangle_{0 \ 0}^{\ j} & \eqdef g^{0j} \partial_t g_{00} + 2 g^{ja} \partial_t g_{0a} - g^{ja} \partial_a g_{00},  
			\label{AE:TRIANGLE0JUPPER0DEF} \\
		2 \triangle_{0 \ k}^{\ j} & \eqdef g^{0j} \partial_k g_{00} + g^{ja} \partial_k g_{0a} - g^{ja} \partial_a g_{0k}
			+ (\underbrace{g^{ja} \partial_t g_{ak} - 2 \omega \delta_k^j}_{e^{2 \Omega} g^{ja} \partial_t 
			h_{ak} - 2 \omega g^{0j}g_{0k}}),  \\
		2 \triangle_{j \ k}^{\ 0} & \eqdef g^{00}(\partial_j g_{0k} + \partial_k g_{0j}) + g^{0a}(\partial_j g_{ak} 
			+ \partial_k g_{aj} - \partial_a g_{jk}) \\
		& \ \ + (\underbrace{\partial_t g_{jk} - 2 \omega g_{jk}}_{e^{2 \Omega} \partial_t h_{jk}}) - 2 \omega (g^{00} + 1) g_{jk} 
			- (g^{00} + 1)(\underbrace{\partial_t g_{jk} - 2 \omega g_{jk}}_{e^{2 \Omega} \partial_t h_{jk}}), \notag \\
		2 \triangle_{i \ j}^{\ k} & \eqdef g^{0k}(\partial_i g_{0j} + \partial_j g_{0i})
				- g^{0k} \underbrace{(\partial_t g_{ij} - 2 \omega g_{ij})}_{e^{2\Omega} \partial_t h_{ij}} 
				- \underbrace{2 \omega g^{0k} g_{ij}}_{2 \omega e^{2 \Omega} g^{0k} h_{ij}}  \\
			& \ \ + g^{ka}(\partial_i g_{aj} + \partial_j g_{ia} - \partial_a g_{ij}). \label{AE:trianglekij}
		\end{align}
	\end{subequations}
\end{lemma}

\begin{proof}
	The proof is again a series of tedious computations that follow from the definition 
	$\Gamma_{\mu \ \nu}^{\ \alpha} \eqdef \frac{1}{2} g^{\alpha \lambda} 
	(\partial_{\mu} g_{\lambda \nu} + \partial_{\nu} g_{\mu \lambda} - \partial_{\lambda} g_{\mu \nu}).$ 
\end{proof}
\setcounter{equation}{0}
   \setcounter{theorem}{0}
   \setcounter{definition}{0}
   \setcounter{remark}{0}

\section{Sobolev-Moser Inequalities} \label{B:SobolevMoser}
		In Appendix \ref{B:SobolevMoser}, we recall the Sobolev-Moser inequalities stated in the Appendix of \cite{iRjS2012}. The propositions 			and corollaries stated below can be proved using standard methods found in e.g. \cite[Chapter 6]{lH1997} and  
		\cite{sKaM1981}. Throughout we abbreviate $L^p=L^p(\mathbb{T}^3),$ and $H^M=H^M(\mathbb{T}^3).$


\begin{proposition} \label{P:derivativesofF1FkL2}
	Let $M \geq 0$ be an integer. If $\lbrace v_a \rbrace_{1 \leq a \leq l}$ are functions such that $v_a \in
    L^{\infty}, \|\underpartial^{(M)} v_a \|_{L^2} < \infty$ for $1 \leq a \leq l,$ and
	$\vec{\alpha}_1, \cdots, \vec{\alpha}_l$ are spatial derivative multi-indices with 
	$|\vec{\alpha}_1| + \cdots + |\vec{\alpha}_l| = M,$ then
	\begin{align*}
		\| (\partial_{\vec{\alpha}_1}v_1) (\partial_{\vec{\alpha}_2}v_2) \cdots (\partial_{\vec{\alpha}_l}v_l)\|_{L^2}
		& \leq C(l,M) \sum_{a=1}^l \Big( \| \underpartial^{(M)} v_a  \|_{L^2} \prod_{b \neq a} \|v_{b} \|_{L^{\infty}} \Big).
	\end{align*}
\end{proposition}

\begin{corollary}                                                  \label{C:DifferentiatedSobolevComposition}
    Let $M \geq 1$ be an integer, let $\mathfrak{K}$ be a compact set, and let $F \in C_b^M(\mathfrak{K})$ be a 
    function. Assume that $v$ is a function such that $v(\mathbb{T}^3) \subset \mathfrak{K}$ and $ \underpartial v \in H^{M-1}.$
    Then $\underpartial (F \circ v) \in H^{M-1},$ and
    \begin{align*} 														
    	\| \underpartial (F \circ v) \|_{H^{M-1}} 
    		& \leq C(M) \| \underpartial v \|_{H^{M-1}} \sum_{l=1}^M |F^{(l)}|_{\mathfrak{K}} 
    		\| v \|_{L^{\infty}}^{l - 1}.
    \end{align*}
\end{corollary}

\begin{corollary}                                                                                             \label{C:SobolevTaylor}
     Let $M \geq 1$ be an integer, let $\mathfrak{K}$ be a compact, convex set, and let $F \in C_b^M(\mathfrak{K})$ be a 
     function. Assume that $v$ is a function such that $v(\mathbb{T}^3) \subset \mathfrak{K}$ and $v - \bar{v} \in H^M,$
    where $\bar{v} \in \mathfrak{K}$ is a constant. Then $F \circ v - F \circ \bar{v} \in H^M,$ and
    \begin{align*} 														
    	\|F \circ v - F \circ \bar{v} \|_{H^M} 
    		\leq C(M) \Big\lbrace |F^{(1)}|_{\mathfrak{K}}\| v - \bar{v} \|_{L^2} 
    		+ \| \underpartial v \|_{H^{M-1}} \sum_{l=1}^M  |F^{(l)}|_{\mathfrak{K}} 
    		\| v \|_{L^{\infty}}^{l - 1} \Big\rbrace.
    \end{align*}
\end{corollary}

\begin{proposition} \label{P:F1FkLinfinityHN}
	Let $M \geq 1, l \geq 2$ be integers. Suppose that $\lbrace v_a \rbrace_{1 \leq a \leq l}$ are functions such that $v_a \in
    L^{\infty}$ for $1 \leq a \leq l,$ that $v_l \in H^M,$ and that
	$\underpartial v_a \in H^{M-1}$ for $1 \leq a \leq l - 1.$
	Then
	\begin{align*}
		\| v_1 v_2 \cdots v_l \|_{H^M} \leq C(l,M) \Big\lbrace \| v_l \|_{H^M} \prod_{a=1}^{l-1} \| v_a \|_{L^{\infty}}  
		+ \sum_{a=1}^{l-1} \| \underpartial v_a \|_{H^{M-1}} \prod_{b \neq a} \| v_b \|_{L^{\infty}} \Big\rbrace.
	\end{align*}	
\end{proposition}

\begin{remark}
	Note that $v_l$ is the only function that is estimated with the $L^2$ norm.
\end{remark}

%
   

\end{document}